%% file: main.tex
\documentclass[openany]{amsbook}
\usepackage{preamble}
\begin{document}


\title{Group Chunks in Model Theory and Algebraic Geometry}
\author{Ronan O'Gorman}

\input{Chapters/Abstract/Abstract}

\maketitle

\tableofcontents
\pagestyle{headings}

\chapter*{Acknowledgements}
I am fortunate to have had the support of not one but two inspiring mathematicians in Professors Thomas Scanlon and Martin Olsson. I am very grateful for their advice and encouragement. Thanks also to all the graduate students in the UC Berkeley Department of Mathematics who have helped make the past six years so enriching and enjoyable. Thanks to my parents for their unwavering support, and thanks to my partner Tianhe for being a continual source of joy in my life for the past four years.

\include{Chapters/Introduction/main_introduction}

\include{Chapters/Category-Theoretic_Prerequisites/Main_Prerequisites}

\include{Chapters/Partial_Morphisms_and_Partial_Magmas/Main_Partial_Morphisms_and_Partial_Magmas}

\include{Chapters/Group_Chunks_On_Presheaves/Main_Group_Chunks_on_Presheaves}

\include{Chapters/First-order_structures/Main_First_order_structures}

\include{Chapters/Scheme-Theoretic_Preliminaries/Main_Schemes}

\include{Chapters/Group_Chunks_on_Schemes}

\include{Chapters/S-Rational_Morphisms/Main_S-Rational_Morphisms}

\include{Chapters/Rational_Families/Main_Rational_Families}

\include{Chapters/Group_Chunks_from_S-Rational_Families/Main_Group_Chunks_from_S-Rational_Families}

\printbibliography

\end{document}

%% file: Chapters/Abstract/Abstract.tex


\begin{abstract}
  We formulate a group chunk theorem in the context of sheaves on sites which generalizes many similar results in model theory and algebraic geometry. Secondly, we develop an algebro-geometric analogue of Hrushovski's method of producing a group chunk from germs of definable functions on stationary types. The use of the model-theoretic tools of canonical bases and elimination of imaginaries is replaced with the use of Hilbert schemes to study ``canonical'' families of rational morphisms, allowing us to extend the previously known results over more general base schemes.

  The proofs of these results involve some technical work which may be of independent interest. First, we study partial morphisms in arbitrary categories, and show that a presheaf of partial magmas on a small category admits a universal morphism to a group. On the model theory side, we show that type-definable sets of $M^{eq}$ can be interpreted as sheaf quotients of type-definable sets. On the algebraic geometry side, we develop a theory of rational morphisms and families of rational morphisms of schemes over an arbitrary base.
\end{abstract}

%% file: Chapters/Introduction/main_introduction.tex
\chapter{Introduction}
\input{Chapters/Introduction/Motivation}

\input{Chapters/Introduction/Background.tex}

\input{Chapters/Introduction/Summary_of_main_results}

\input{Chapters/Introduction/Future_questions}
\input{Chapters/Introduction/Notation.tex}

%% file: Chapters/Introduction/Motivation.tex
\section{Motivation}
Model theory has had striking applications to algebraic geometry. Central to many of these applications are the tools of a subfield of model theory known as geometric stability theory. These tools have found application in, for example, Hrushovski's proofs of the Manin-Mumford and Mordell-Lang conjectures, (\cite{HrushovskiEhud1996TMordell_lang}, \cite{HrushovskiEhud2001Manin_mumford}), combinatorial algebraic geometry (\cite{EvansHrushovski91}, \cite{BAYSMartin2021Pgaf}), differential algebraic geometry (\cite{FreitagJames2022Wats}), difference algebraic geometry (\cite{ChatzidakisZoé2008Dfad}), and in reconstruction results - see for example Zilber's theorem below.

Unfortunately, the use of model-theoretic tools can sometimes obscure the underlying geometry. In some cases, although the results may be fundamentally algebro-geometric in nature, the model-theoretic arguments have no known algebro-geometric alternative. Consider, for example, the following theorem of Zilber (\cite{ZilberBoris2014ACai}):
\begin{thm}[Zilber]\label{Zilber's jacobian theorem}
    Let $k_1$, $k_2$ be algebraically closed fields, and $C_1$ and $C_2$ be smooth projective curves of genus $g \geq 2$ over $k_1,k_2$ respectively. For $i = 1,2$, fix points $c_i \in C_i(k_i)$ and consider the corresponding embeddings $C_i \hookrightarrow J_i = \Pic^0(C_i)$.

    Suppose we have a group isomorphism $\Theta$ of the rational points of the Jacobians respecting the images of the embeddings of the curves:
    \[\begin{tikzcd}[ampersand replacement=\&, row sep = small]
            J_1(k_1) \arrow[r,"\sim"', "\Theta"]\& J_2(k_2)\\
            C_1(k_1) \arrow[u, hook] \& C_2(k_2)\arrow[u,hook].
        \end{tikzcd}\]
    Then there exists a field isomorphism $\phi\colon k_1 \to k_2$, and a bijective isogeny $f\colon \phi(J_1) \to J_2$ over $k_2$ such that $\Theta = f \circ \phi$.
\end{thm}
The result was generalized in the recent work \cite{castle2026curveabstractgeneralizedjacobian}.

Zilber proved this result in answer to a conjecture of Bogomolov, Korotiaev, and Tschinkel \cite{BogomolovFedor2010ATTf}. Significant improvements to the proof (or rather the key technical result on which it relies, originally due to \cite{Rabinovich}) were made in \cite{HassonAssaf2024IsoC} and \cite{CastleBenjamin2023Zrti}. However, as far as we are aware, the proof remains poorly understood by algebraic geometers, despite some amount of interest. In the recent work \cite{klos}, the authors pose the question of developing an algebro-geometric proof of Zilber's theorem; this was the original motivation for our work.

Moreover, the tools of geometric stability theory are restricted to the world of varieties over fields, and often in practice the field is algebraically closed. This makes it difficult to see to what extent the arguments might generalize beyond the case of algebraically closed fields. For example, Zilber's \cref{Zilber's jacobian theorem} is known to fail in general if the fields are not algebraically closed (\cite[Proposition~5.3.2]{klos}), but it would be interesting to see if Zilber's methods could be used to prove the theorem, or similar reconstruction results, under weaker hypotheses.

The aim of this work is to begin to develop the tools of geometric stability theory from an algebro-geometric perspective. In particular, we focus on the so-called ``group chunk theorem'' and its model-theoretic extensions (see the next section for the statements). In the particular case of the group chunk theorem, there are already well-understood algebro-geometric analogues (see below) which have their own applications in algebraic geometry; they are used, for example, to study generalized Jacobians (\cite{rosenlicht}), N\'eron models of abelian varieties (\cite{blr90}), and split reductive groups over the integers (\cite{Demazure}).  We seek to connect the algebro-geometric and model-theoretic statements, and to this end we formulate a general group chunk type statement in the context of presheaves on a site which generalizes both.

On the other hand, there are several model-theoretic extensions of the group chunk theorem that do not have satisfactory algebro-geometric analogues - most notably Hrushovski's group configuration theorem, which is central to all the applications listed in the opening paragraph. In light of these applications, it seems reasonable to expect that scheme-theoretic analogues of these results would find many useful applications in algebraic geometry, in addition to clarifying (from a geometric perspective) the proofs which already exist. To this end, we give an algebro-geometric analogue of Hrushovski's construction of a group chunk from germs of definable bijections on stationary types, which is an important step in the proof of the group configuration.

%% file: Chapters/Introduction/Background.tex
\section{Background}
In this section we give some background on the history and various iterations of group chunk theorems in different fields.

A ``group chunk theorem" gives a universal procedure for constructing a group out of something which looks like a sufficiently large ``chunk" of a group. Informally, a group chunk is an object $X$, in whatever category one is working in, together with a ``partially defined morphism," referred to as a partial multiplication operation, from $X \times X$ to $X$ which is cancellative, associative, and defined on a sufficiently large domain. The goal is to find a universal map from $X$ into a group object which respects the partial multiplication; in other words, we want to construct a left adjoint to the forgetful functor from groups to group chunks. The general recipe for constructing the group goes back to Weil (see \cite{WeilAndre1955OAGo}). One considers the subgroup of the group of partial automorphisms of $X$ generated by the left action of $X$ on itself. One shows that the group is generated in two steps, and in particular it can be realized as a quotient of $X^2$. One then argues that the quotient is representable.

\begin{rem}\label{pre group remark}
    In applications, one often starts with a slightly larger object, sometimes known as a \emph{pre-group}. One must then find a group chunk embedded in the pre-group, as is done in, for example \cite[Proposition~3.2]{ArtinGroupChunk}.
\end{rem}


The group chunk construction has a long history. It was conceived by Weil in the 1940s to construct the Jacobian of a curve (\cite{Weil_varietes_abeliennes}). Weil's original construction predates modern scheme theory. At the time, Weil was unable to prove that the Jacobian he produced using the group chunk construction was projective, and in fact this was one of the original motivations for his definition of an ``abstract" variety as a collection of affine varieties glued together (see \cite{Weil_foundations}).

Today, there are several group chunk type theorems in different fields, all of which more or less follow Weil's recipe. In model theory, Weil's result was generalized to the context of stable theories by Hrushovski (\cite{Hrushovski-phd}), and it has since become a central tool of geometric stability theory. We give here a more recent generalization of Hrushovski's statement to the context of definable types:
\begin{thm}[{\cite[Proposition~3.15, Proposition~3.16]{Hrushovski2019}}]\label{Hrushovski group chunk intro statement}
    Let $p$ be a global partial type definable over a small set of parameters $C$. Let $(\cdot)$ be a $C$-definable map from $p^{\otimes 2}$ to $p$. Assume $(\cdot)$ is:
    \begin{enumerate}
        \item \textbf{Generically Cancellative}: For $a,b\models p^{\otimes 2}|_C$ we have $b\in \dcl(a,a\cdot b, C)$ and $a \in \dcl(b,a\cdot b, C)$.
        \item \textbf{Generically Associative}: For $a,b,c \models p^{\otimes 3}|_C$, we have $(a\cdot b)\cdot c = a\cdot (b\cdot c)$.
        \item \textbf{Generically Transitive}: For $a \models p|_C$, we have $(\lambda_a)_* p = p$, where $\lambda_a$ denotes the left multiplication by $a$.
    \end{enumerate}
    Let $X = p|_C(\bU)$. Then there exists a universal $C$-definable embedding from $X$ into a $C$-type-interpretable group $G$ which respects the multiplication operation.
\end{thm}
Group chunks arise in many other model-theoretic settings. Group chunks in o-minimal structures are studied in \cite{PillayAnand1988Ogaf}. See also \cite{PillayHrushovski-local-fields}, \cite{PeterzilYaacov2022Odgi}, \cite{pillay2025groupsdefinablegeometricfields}, \cite{chernikov2025externallydefinablefsggroups}.

In algebraic geometry, Weil's result was generalized in a different direction by Artin to the context of schemes (\cite{ArtinGroupChunk}). Again following Weil's recipe, Artin produces a group algebraic space, although  algebraic spaces had not been invented at the time, so this was only recognized later. Artin's result was improved in \cite{blr90} and \cite{ER15}. The most up-to-date statement of which we are aware is the following:
\begin{thm}\label{artin group chunk introduction}(See \cite{ER15}).
    Let $S$ be a scheme, and $X$ a scheme faithfully flat and locally of finite presentation over $S$, whose fibers over $S$ are separated and have no embedded components. Let $W$ be a locally closed subscheme of $X \times_S X \times_S X$ such that, if we denote $X_1 = X_2 = X_3 = X$, then for all permutations $(i,j,k)$ of  $\{1,2,3\}^3$ with $i<j$, we have that
    \begin{enumerate}
        \item The projection $\rho_{ij}\colon W \to X_i\times_S X_j$ is an open immersion whose image $U_{ij}$ is both $X_i$-dense and $X_j$-dense.
        \item The projection $\rho_{ij}$  exhibits $W$ as the graph of a morphism $m_{ij}$ from $U_{ij} \subset X_i\times_S X_j$ to $X_k$.
        \item We have
              \begin{align*}
                  m_{12}(m_{12}(x_1,x_2),x_3) = m_{12}(x_1,m_{12}(x_2,x_3))
              \end{align*}
              whenever both sides of the above equation are defined.
    \end{enumerate}
    Then $X$ embeds into a group algebraic space, which is a scheme in good circumstances (see \cref{Artin Group chunk theorem} for more details)
\end{thm}

In \cite{vandenDriesL.P.D.1990Wgct}, Van den Dries studies group chunks in a topological setting with a structure sheaf, although his version does not generalize Artin's result. In \cite{Zaitsev1995}, Weil's construction is generalized to reducible varieties without using scheme theory and applied to the study of real Lie groups. See also \cite{GoldbringVandenDries} for a similar-looking result where topological groups are constructed from local groups.

The model-theoretic group chunk theorems have several extensions which do not have satisfactory preexisting geometric analogues. In particular, we consider the problem of reconstructing groups from ``generically defined families of bijections,'' which is not addressed in Artin's work, and which is an important step in the proof of the group configuration. The result is proved (though not stated explicitly) in \cite{Hrushovski-phd}. We paraphrase the statement from \cite[Lemma~5.4]{Bays:GeometricStabilityTheory2016}:
\begin{thm}[Hrushovski]\label{model theoretic translation chunk}
    Fix a stable theory $T$, and suppose $\acl^{eq}(\emptyset) = \dcl^{eq}(\emptyset)$ (so that all types over the empty set are stationary). Let $p,q, r$ be types over $\emptyset$, and suppose $f_r\colon p \to q$ is a generically transitive canonical family of invertible germs of definable functions from $p$ to $q$. Let $b_1,b_2$ be independent realizations of $r$, and let $c$ be the canonical base of the germ $f_{b_1}\inv \circ f_{b_2}$ from $p$ to $p$. Let $s$ be the type of $c$, and $g_s$ be the induced canonical family of germs from $p$ to $p$. Suppose $c \forkindep b_i$ for each $i$. Then composition of germs induces a generically presented group (in the sense of \cref{Hrushovski group chunk intro statement}) on $s$, and hence a definable group.
\end{thm}
In the model-theoretic setting this is a straightforward extension of the group chunk construction in the stable context, as observed in \cite{Hrushovski-phd}. However in the scheme-theoretic world the construction is more difficult since many model-theoretic tools such as elimination of imaginaries are unavailable, and taking scheme-theoretic quotients is more complicated. The closest algebro-geometric analogue of which we are aware is the following theorem of Weil, which works for varieties over a field (not necessarily algebraically closed):
\begin{thm}(Paraphrased from \cite[Proposition~2]{WeilAndre1955OAGo})\label{Weil translation chunk}
    Let $V,W$ be two integral varieties defined over a field $k$. Let $g$ be a rational morphism from $V \times W$ into $W$, and suppose there exist generic points $x,y$ of $V$, independent over $k$, and a generic point $z$ of $V$ over $k$ such that
    \begin{enumerate}
        \item The rational morphism $u \mapsto g(x,u)$ is birational, and
        \item  $g(z,u) = g(x,g(y,u))$ for $u$ generic on $W$ over $k(x,y,z)$.\label{z existence assumption in weil's theorem}
    \end{enumerate}
    Then there exists a variety $\overline{V}$, and rational morphisms $\phi\colon V \to \overline{V}$, $\overline{g}\colon \overline{V}\times W \to W$ and $f\colon \overline{V}\times \overline{V} \dashrightarrow \overline{V}$, all defined over $k$, such that
    \begin{itemize}
        \item $g(x,u) = \overline{g}(\phi(x),u)$ for generic $u$,
        \item $f(\phi(x), \phi(y)) = \phi(z)$,
        \item $f$ and $\overline{g}$ define a pre-group structure (as in \cref{pre group remark}) on $\overline{V}$, and a pre-transformation space structure on $W$ with respect to $\overline{V}$.
    \end{itemize}
    In particular, $f$ induces a group chunk structure on a subvariety of $\overline{V}$.
\end{thm}

%% file: Chapters/Introduction/Summary_of_main_results.tex
\section{Overview of Main Results}\label{section: overview of main results}
The main goal of the first part of this work is to show that Weil's group chunk construction can be carried out in a very general setting which generalizes \cref{artin group chunk introduction}, \cref{Hrushovski group chunk intro statement} and many similar statements. In later chapters, we develop a scheme-theoretic analogue of \cref{model theoretic translation chunk} which works over more general base schemes. Along the way, we prove numerous smaller results which we believe may be of independent interest.

In this section we discuss our results in some detail and describe how they fit in with existing literature.

\subsection{Partial Morphisms and Partial Magmas}
We begin in \cref{chapter: Partial Morphisms and Partial Magmas} by studying objects with partially defined operations abstractly. A \emph{partial morphism} from $X$ to $Y$ in a category $\scrC$ is a morphism from a subobject of $X$ to $Y$.  We say that two partial morphisms from $X$ to $Y$ are \emph{compatible} if they agree on the intersection of their domains.

By the Yoneda lemma, a partial morphism in any category $\scrC$ can be viewed as a partial morphism in the category $\Presh(\scrC)$ of presheaves on $\scrC$. For any $\scrC$, there is a bifunctor $\iHom(-,-)$ on $\Presh(\scrC)$, called the \emph{internal hom}, which is determined up to natural isomorphism by the adjunction
\begin{align*}
    \Hom(X \times Z ,Y) \simeq \Hom(X,\iHom(Z,Y)).
\end{align*}
In \cref{section: partial morphisms of presheaves}, we define an ``internal partial hom'' bifunctor $\parhom(-,-)$ on $\Presh(\scrC)$. We give two interpretations of $\parhom(-,-)$, and show that it satisfies an analogous adjunction:
\begin{prop}[see \cref{partial hom adjunction}]
    Let $\scrC$ be a small category, and $X,Y,Z \in \Presh(\scrC)$. We have isomorphisms
    \begin{align*}
        \lambda\colon \rparhom(X\times Z, Y) \simeq \Hom(X,\parhom(Z,Y))
    \end{align*}
    which are natural in $X$,$Y$ and $Z$.
\end{prop}
This shows in particular that a partial binary operation on an object $X$ induces a morphism from $X$ into $\parhom(X,X)$; this map sends an element $x \in X(T)$ to the left translation by $x$. We will make use of this in the next chapter for the group chunk construction.

Next, we move on to studying partial magmas; i.e., objects with partial binary operations. We show that in essentially all situations, and for largely formal reasons, a presheaf of partial magmas $(X,\mu)$ admits a universal morphism to a presheaf of groups (see \cref{Universal maps to groups - general case}).
As a corollary, we obtain a universal map from $X$ into a sheaf of groups.

\subsection{Group Chunks on Sites}
Having shown that a universal morphism to a presheaf of groups always exists, the question then becomes under what circumstances does this presheaf of groups have a manageable form. In \cref{chapter: group chunks on presheaves} we formulate conditions under which Weil's recipe can be implemented, and show that it gives an explicit description both of the associated presheaf of groups and the associated sheaf of groups.

Roughly, we define a group chunk on a site $\scrC$ to be a presheaf $X$ on $\scrC$ together with a partial binary operation $m_{12}\colon X \times X \dashrightarrow X$ which is cancellative, associative, and which has sufficiently large domain (see \cref{Group Chunk definition} for the precise definition). Surprisingly, formulating the condition that the domain be sufficiently large does not require any analogue of the notion of density (as in the algebro-geometric case) or independence (as in the model-theoretic case); instead it is enough to specify that one may locally find points $z$ such that finite collections of iterated products with $z$ of depth at most 2 are defined (\cref{large domain definition}).

Given a group chunk, we first show that iterated products of arbitrary depth can be defined (\cref{Long products lemma}), and then show that the relation of compatibility, which we denote using the symbol $\ssim$, gives an equivalence relation on the subsemigroup of partial endomorphisms of $X$ generated by the left translation maps described above (\cref{compatibility is an equivalence relation in the cases we care about}), which we call $\tXinfty$. The universal presheaf of groups can then be realized as the quotient of $\tXinfty$ by the compatibility relation (\cref{section: second universal morphisms to groups}). Moreover, we show that the universal presheaf of groups is locally given by pairs of elements of $X$, and therefore when we sheafify we get a universal sheaf of groups which can be viewed as a quotient of $X^2$ (\cref{section: third universal maps to groups}).

The precise statement summarizing these results is:
\begin{thm}\label{main abstract group chunk result}
    Let $(X,\mu)$ be a presheaf of partial magmas on a category $\scrC$.
    \begin{enumerate}
        \item If $\scrC$ is small, then $X$ admits a universal map to a presheaf of groups (\cref{Universal maps to groups - general case}).
        \item If $\scrC$ is a site, and $(X,\mu)$ is a locally nontrivial group chunk, then the universal presheaf of groups $G(X)$ is separated and can be identified with the quotient $\tXinfty/\ssim$ (\cref{left adjoint for group chunks}).
        \item If $(X,\mu)$ is a locally nontrivial group chunk, then the universal sheaf of groups $G(X)^+$ can be identified with the sheaf quotient $(X^2/\ssim)^+$ (\cref{group chunks to sheaves left adjoint}).
    \end{enumerate}
\end{thm}

One should note that, at this level of generality, we cannot hope to say anything meaningful regarding representability of the sheaf of groups $G(X)^+$. Comparing the algebro-geometric and model-theoretic results one sees that in this respect things work out rather differently. In the model-theoretic setting one has elimination of imaginaries (at least after passing to $\bM^{eq}$), which guarantees that quotients by definable equivalence relations are representable; this is discussed in detail in \cref{chapter: model theory}. In the scheme-theoretic setting the question is more subtle; see \cref{Artin Group chunk theorem}. The relationship between \cref{main abstract group chunk result} and Artin's \cref{artin group chunk introduction} is discussed in \cref{Group Chunks on Schemes}.

\subsection{Interpretable Sets}
In \cref{chapter: model theory} we relate our abstract group chunk construction to Hrushovski's theorem on generically presented groups on definable types (\cref{Hrushovski group chunk intro statement}). In order to do this, we define a Grothendieck topology on the category of type-definable sets of a theory with definable maps (see \cref{Grothendieck topology on type-definable sets definition} for the precise definition) where covers are definable surjections, and show that the hypotheses of \cref{Hrushovski group chunk intro statement} induce a group chunk on this site.

In the course of the group chunk construction, we need to take a sheaf quotient of a representable presheaf by an equivalence relation representable by a definable set. However in the model-theoretic construction, one takes the same quotient by passing to a type-interpretable set (i.e., a type-definable set in $M^{eq}$). In \cref{section: Sheaf Quotients and type-interpretable sets} we show that these two quotients are essentially the same. More precisely, we show:
\begin{prop}[see \cref{equivalence of interpretable sets and presheaf quotients}]
    \begin{enumerate}
        \item     There is an equivalence of categories $j$ from the category of $0$-type-interpretable sets $\typInt(\calT)$ to the category of sheaves on $\typDef(\calT)$ which are sheaf quotients of $0$-type-definable sets by $0$-definable equivalence relations, such that if a type-interpretable set $D$ is a model-theoretic quotient of a $0$-type-definable set $X$ by a $0$-definable equivalence relation $E$ on $X$ then $j(D)$ is naturally isomorphic to the sheaf quotient $(X/E)^+$.
        \item  There is an equivalence of categories $j$ from the category of $0$-interpretable sets $\Int(\calT)$ to the category of sheaves on $\Def(\calT)$ which are sheaf quotients of $0$-definable sets by $0$-definable equivalence relations, such that if a $0$-interpretable set $D$ is a model-theoretic quotient of a $0$-definable set $X$ by a $0$-definable equivalence relation $E$ on $X$ then $j(D)$ is naturally isomorphic to the sheaf quotient $(X/E)^+$.
    \end{enumerate}
\end{prop}

After writing this section, we were informed by Johnson that the Grothendieck topology we define is a special case of the regular coverage of a regular category; see \cite{nlab:regular_coverage} or \cite{sketches_of_an_elephant}. We also discovered that, at least for the case of interpretable sets, \cref{equivalence of interpretable sets and presheaf quotients} follows quickly from preexisting work of Makkai (see \cite{mrg_categorical_logic}, \cite{Categorical_Meq}) which identifies the $\bM^{eq}$ construction with the pretopos completion of the category of definable sets, together with work of Shulman \cite[Theorem~9.2]{Shulman}, where a generalization of the notion of a pretopos completion is described and characterized as a subcategory of a category of sheaves. However, as far as we are aware, the result in the case of type-interpretable sets is new. Even in the case of interpretable sets the result has never been written down explicitly and is not well known in the logic community. In this work we provide a direct proof that requires less category theory than the aforementioned references.

\subsection{\texorpdfstring{$S$}{S}-Rational Morphisms}
Starting in \cref{chapter: S-rational morphisms}, we begin to develop the machinery that we will need to state our analogue of \cref{model theoretic translation chunk} over a general base scheme $S$. In order to formulate our results in the most general context, it is necessary to develop the theory of $S$-rational morphisms beyond what is normally presented (as in, for example, \cite{blr90} or \cite[IV~20]{EGA}).



Firstly, the composition $f\circ g$ of $S$-rational morphisms $f$ and $g$ is usually only considered either when the schemes are integral, or when $f$ is a morphism. In \cref{section: diffuse morphisms}, we introduce the class of \emph{diffuse} morphisms, and in \cref{composition of diffuse S-rational morphisms is well defined} we show that the composition of diffuse $S$-rational morphisms is well-defined.

Secondly, in order to run descent arguments with rational morphisms, we need a notion of faithful flatness that makes sense in the context of $S$-rational morphisms. We declare an $S$-rational morphism $\phi\colon X \dashrightarrow Y$ to be faithfully flat if there exist $S$-dense opens $U\subset X$, $V \subset Y$, and a representative $\phi_U$ of $\phi$ such that $\phi_{[U \to V]}$ is faithfully flat. In \cref{chapter: S-rational morphisms}, we use this to show that the class of faithfully flat $S$-rational morphisms reasonably well-behaved, and in particular is closed under taking products and compositions. Faithfully flat $S$-rational morphisms of locally Noetherian schemes are automatically diffuse.



As a key example, a rational morphism of integral varieties over fields is dominant if and only if it is faithfully flat. One drawback of this notion is that, outside of the case where $S$ is the spectrum of a field, we are not aware of good criteria for guaranteeing that a given $S$-rational morphism will be faithfully flat.

\subsection{Canonical Families of Rational Maps}
In \cref{chapter: S-rational families}, we define the algebro-geometric analogue of the ``family of germs of definable functions'' from \cref{model theoretic translation chunk}; this is an $S$-rational map $\phi\colon A\times X \dashrightarrow Y$. We call this an $S$-rational family from $X$ to $Y$ parametrized by $A$. Given another $S$-rational family $B\times Y \dashrightarrow Z$, one can in good circumstances compose the two families to get a third family $B\times A \times X \dashrightarrow Z$ over $B\times A$. Similarly, given a faithfully flat $S$-rational morphism $f\colon C \dashrightarrow A$, one can take the pullback $f^*(\phi)\colon C\times X \dashrightarrow Y$; this should be thought of as a ``reparametrization'' of $\phi$.

In the model theoretic setting, one proceeds by taking the ``canonical parameter space'' of a family of germs, and showing that the operation of composition of families induces definable maps on the canonical parameter spaces. The main results from \cref{chapter: S-rational families} show that this can be done geometrically. In the case of a $k$-rational families $\phi\colon A \times_k X \dashrightarrow Y$ where $A,X,Y$ are varieties over $k$ with $A$ geometrically integral, the construction runs as follows: first, the graph of $\phi$ induces a rational morphism $\xi$ from $A$ to the Hilbert scheme of $X \times_k Y$ over $k$ (\cref{xi existence result}). The schematic image of $\xi$, which we denote $Z(\phi)$, is the ``canonical parameter space'' for $\phi$, and it follows by a descent argument (\cref{Existence of canonical families}) that there is an associated canonical $k$-rational family $\widetilde{\phi}\colon Z(\phi)\times_k X \dashrightarrow Y$ which pulls back to $\phi$ along $\xi$. One then shows using the properties of the Hilbert scheme that this family is universal among families that pull back to $\phi$.

In the more general setting, we define a class of $S$-rational families which we call \emph{tame} (\cref{tame definition}), and which includes the case described above. We show that tame rational families can be canonically parametrized:
\begin{thm}[see \cref{canonical families for families on S-birationally projective schemes}]
    Let $\phi\colon AX \dashrightarrow Y$ be a tame $S$-rational family with $A$ Noetherian and $X,Y$ $S$-birationally projective. Then there exists a ``canonical'' rational family $\tphi\colon Z(\phi)X \dashrightarrow Y$, together with a faithfully flat $S$-rational morphism $\zeta_\phi\colon A \to Z(\phi)$, which satisfies the following universal property: for any $S$-rational family $\psi\colon BX \dashrightarrow Y$ with $B$ Noetherian, and faithfully flat $S$-rational morphism $f\colon A \dashrightarrow B$ such that $\phi = f^*(\psi)$, there exists a unique faithfully flat $S$-rational morphism $\eta\colon B \dashrightarrow Z(\phi)$ with $\psi = \eta^*(\tphi)$. The family $\tphi$ is unique up to birational equivalence.
\end{thm}
Moreover, in \cref{Canonical families and composition}, we show that composition of rational families induces rational morphisms on the canonical parameter spaces, and that the canonical parameter spaces are well-behaved under taking inverses.

\subsection{Constructing Group Chunks from Rational Families}\label{summary of constructing group chunks from rational families}
Finally, we give our geometric analogue of \cref{model theoretic translation chunk}. In the simpler and probably more important case of varieties over fields, the result reads as follows:

\begin{thm}[see \cref{the hypotheses for the main situation are not so bad over fields}]
    Let $A$ be a geometrically integral variety over a field $k$, and $X,Y$ be varieties over $k$.

    Let $\phi\colon AX \dashrightarrow Y$ be an invertible $k$-rational family over $A$ (i.e., assume the induced rational morphism $\overline{\phi}\colon A\times_k X \dashrightarrow A\times_k Y$ is birational) with inverse $\phi^\dagger$. Let $\psi = \phi^\dagger * \phi$. Suppose the rational morphisms
    \begin{align*}
        q_1\colon A^2 \dashrightarrow AZ(\psi); &  & (a_1,a_2) \in A^2(T) \mapsto (a_1, \zeta_\psi(a_1,a_2)) \\
        q_2\colon A^2 \dashrightarrow AZ(\psi); &  & (a_1,a_2) \in A^2(T) \mapsto (a_2, \zeta_\psi(a_1,a_2))
    \end{align*}
    are both dominant.

    Then composition of $\psi$ with itself induces an $S$-rational morphism $m_{12}\colon Z(\psi)^2 \dashrightarrow Z(\psi)$, and there exist $S$-dense opens $U\subset \dom(m_{12})$, $D \subset Z(\psi)$ and $W\subset D^3 \cap \Gamma_{m_{12U}}$ (where $\Gamma_{m_{12U}}$ denotes the graph of the representative of $m_{12}$ with domain $U$) such that the pair $(D,W)$ is an Artin group chunk (\cref{Artin Group Chunk definition}). Moreover, in this case the associated group algebraic space is a variety over $k$.
\end{thm}

The assumption that the $q_i$ are both dominant corresponds in the model-theoretic setting to assuming $c \forkindep b_i$ for each $i$, or in the setting of \cref{Weil translation chunk} to assuming $g(z,u) = g(x,g(y,u))$ for generic $u$.

A few words about how the above construction relates to the results of Weil and Hrushovski. Starting with the same setup $\phi\colon A \times_k X \dashrightarrow Y$, one can, at least so long as $X$ and $Y$ are assumed irreducible, produce a group by appealing to \cref{model theoretic translation chunk} and using model-theoretic methods (see for example \cite{NesinPillay-Model-theory-of-groups}). However the relationship between the family $\phi$ and the group produced this way is less clear. In particular, the rational morphisms from $A$ to the group will be definable maps rather than explicit rational morphisms of schemes, and these maps will be somewhat arbitrary because of the choices involved in elimination of imaginaries and in obtaining isomorphisms between definable groups and algebraic ones in positive characteristic. Moreover, in positive characteristic, the group may only be defined over the perfect closure of $k$.

The construction outlined above can also be viewed as an extension of Weil's result where instead of starting with a family of automorphisms of an irreducible variety $W$ (as in \cref{Weil translation chunk}), one starts with a family of birational morphisms $W_1 \to W_2$. Weil uses a different method of obtaining the canonical parameter space, using fields of definition of varieties, which only works over fields. Our construction uses more modern technology, and it has the advantage that it makes it clear how the construction can run over more general base schemes, with some additional hypotheses.

The general statement is:
\begin{thm}[see \cref{main result: most general case}]
    Let $S$ be a Noetherian scheme, and let $A,X,Y$ be schemes over $S$, with $A$ fppf and with geometrically integral fibers over $S$ and such that $Y$ is locally Noetherian and $X$ is $S$-birationally projective.

    Let $\phi\colon AX \dashrightarrow Y$ be an invertible $S$-rational family over $A$, with inverse $\phi^\dagger$. Let $\psi = \phi^\dagger * \phi$. Suppose
    \begin{enumerate}[label = (\roman*).]
        \item The family $\psi$ is tame (\cref{tame definition})
        \item The closed graphs $\Lambda_{\widetilde{\psi} * \widetilde{\psi}}$, $\Lambda_{(\widetilde{\psi})^\dagger * \widetilde{\psi}}$ and $\Lambda_{\widetilde{\psi} * (\widetilde{\psi})^\dagger}$ are $S$-generically flat over $Z(\psi)^2$, for some choice of $S$-birational projective model of $X$ witnessing tameness of $\psi$.
        \item The $S$-rational morphisms
              \begin{align*}
                  q_1\colon A^2 \dashrightarrow AZ(\psi); &  & (a_1,a_2) \in A^2(T) \mapsto (a_1, \zeta_\psi(a_1,a_2)) \\
                  q_2\colon A^2 \dashrightarrow AZ(\psi); &  & (a_1,a_2) \in A^2(T) \mapsto (a_2, \zeta_\psi(a_1,a_2))
              \end{align*}
              are both faithfully flat.
    \end{enumerate}

    Then composition of $\psi$ with itself induces an $S$-rational morphism $m_{12}\colon Z(\psi)^2 \dashrightarrow Z(\psi)$, and there exist $S$-dense opens $U\subset \dom(m_{12})$, $D \subset Z(\psi)$ and $W\subset D^3 \cap \Gamma_{m_{12U}}$ (where $\Gamma_{m_{12U}}$ denotes the graph of the representative of $m_{12}$ with domain $U$), such that the pair $(D,W)$ is an Artin group chunk (\cref{Artin Group Chunk definition}).
\end{thm}





%% file: Chapters/Introduction/Future_questions.tex
\section{Future Directions}
There are several.

Regarding applications of the abstract group chunk result:
\begin{enumerate}
    \item In \cite{ArtinGroupChunk}, Artin imposes the technical assumption that fibers over the base be separated and have no embedded components. We suspect that, starting from our formulation of a group chunk on schemes in \cref{Group Chunks on Schemes}, one can apply Artin's methods to obtain a reasonable representability result without this assumption.
    \item There are many instances of group chunk type theorems we have not studied, particularly in model theory. To name one example, in \cite{PillayAnand1988Ogaf}, Pillay uses group chunks in an o-minimal context to study definable groups. There is also an o-minimal group configuration. It would be interesting to see to what extent our abstract group chunk result can be applied in these cases, and if it can shed light on new cases.
\end{enumerate}

Regarding the identification of type-interpretable sets with sheaf quotients (\cref{equivalence of interpretable sets and presheaf quotients}):
\begin{enumerate}
    \item One should be able to use sheaf quotients to describe other kinds of model-theoretic quotients; for example quotients of type-definable sets by type-definable equivalence relations (hyperimaginaries) or quotients of pro-definable and $\infty$-definable sets.
\end{enumerate}

Regarding technical improvements to \cref{main result: most general case}:
\begin{enumerate}
    \item In \cref{Weil translation chunk}, Weil also produces a homogeneous space on which the group acts. One also ultimately gets a homogeneous space from \cref{model theoretic translation chunk}. It should not be too difficult to extend our main construction \cref{main result: most general case} to obtain a homogeneous space also.
    \item It should be possible to relax the assumption that the fibers of $Z(\phi)$ be geometrically integral in \cref{main result: most general case} by adapting Artin's argument for \cref{geometrically integral fibers gives X dense open} to allow, for example, finite unions of varieties with geometrically integral fibers.
    \item It would also be potentially interesting to formulate a purely categorical, presheaf-theoretic version of \cref{main result: most general case} which also generalizes the model-theoretic construction, as we have done for the group chunk. However the complexity of the scheme-theoretic argument suggests that this may be more challenging.
\end{enumerate}

Working towards the group configuration and applications:
\begin{enumerate}
    \item We aim to extend \cref{main result: most general case} to recover a geometric account of the full group configuration theorem, at least in the case of varieties over fields, where things work out more neatly. We believe it should be possible to replace much of the forking calculus in the model-theoretic proof with geometric arguments which will hopefully be more conceptual.
    \item With a geometric group configuration in hand, we would hope to return to the proof of Zilber's theorem \cref{Zilber's jacobian theorem}. We conjecture that the arguments of Hasson and Sustretov (\cite{HassonAssaf2024IsoC}) should simplify in this particular case, and one should be able to apply them to obtain a purely geometric proof. It should then be possible to analyze to what extent the argument will generalize beyond algebraically closed fields.
\end{enumerate}

%% file: Chapters/Introduction/Notation.tex
\section{Notation}\label{prerequisites notation section}

In this section we list our notational conventions for future reference. Some of our conventions are not entirely standard (we highlight in particular \cref{products and Yoneda combine}) so the reader is advised not to skip this section.

\begin{notation}[Tuples of elements in sets]
    If $X$ is a set, we often write $a_1,\dots,a_n \in X$ to mean that  $a_1,\dots,a_n$ is a collection of elements of $X$ where this will not lead to ambiguity.
\end{notation}

\begin{notation}[Size issues]
    \begin{enumerate}
        \item     We allow our categories $\scrC$ to be large (i.e., not small) simply because many of our definitions make sense at this level of generality. However in practice all our results will only apply if $\scrC$ is small.

        \item     All sites are assumed to be small.  This will ensure that sheafification always exists.
    \end{enumerate}
\end{notation}

\begin{notation}[Restrictions of Morphisms]\label{restrictions of morphisms} Let $\scrC$ be a category.
    \begin{enumerate}
        \item Let $\phi\colon X \to Y$ and $\psi \colon U \to X$ be morphisms in $\scrC$. Where $\psi$ is clear from context we will sometimes denote the composition $\phi\circ \psi$ by $\phi|_U$, or ``the restriction of $\phi$ to $U$''.

        \item Let $\phi,\psi$ be as above and suppose we have a monomorphism $\iota\colon V\hookrightarrow Y$ such that $\phi|_U$ factors through $\iota$. We denote by $\phi|_{[U \to V]}$ the induced morphism from $U$ to $V$.

        \item We sometimes omit the subscripts $[U \to V]$ where the meaning should be clear from context; for example, given $\phi,\psi,\iota$ as above we permit ourselves commutative diagrams labelled as follows:
              \[
                  \begin{tikzcd}
                      U \arrow[r, "\phi"] & V
                  \end{tikzcd}
              \]
    \end{enumerate}
\end{notation}

\begin{notation}[Presheaves]\label{presheaf notation} Let $\scrC$ be a category.
    \begin{enumerate}
        \item      We denote by $\Presh(\scrC)$ the category of presheaves on $\scrC$.

        \item     If $\scrC$ is a site, we denote by $\Sep(\scrC)$ and $\Sh(\scrC)$ the categories of separated presheaves and sheaves on $\scrC$ respectively.

        \item     Given a presheaf $F\in \Presh(\scrC)$, a morphism $T' \to T$ in $\scrC$ and a section $a \in F(T)$, and  we will often denote the restriction of $a$ to $F(T')$ again by $a$. Where we wish to make this explicit, we will write $a|_{T'}$.

        \item   If $\phi\colon X \to Y$ is a morphism of presheaves on $\scrC$, we will write $\phi_T$ for the corresponding morphisms $X(T) \to Y(T)$ for each $T$.

        \item     If $\iota\colon G \hookrightarrow F$ is a monomorphism, then given $t \in F(T)$, we abusively write $t \in G(T)$ to mean $t = \iota(t')$ for some (necessarily unique, since $\iota$ is monic) $t' \in G(T)$.
    \end{enumerate}
\end{notation}

\begin{notation}[Yoneda Embedding]\label{Yoneda Notation} Let $\scrC$ be a category.
    \begin{enumerate}
        \item   Given an object $X \in \scrC$, we denote by $h_X$ the image of $X$ under the Yoneda embedding; i.e. $h_X$ is the presheaf $T \mapsto \Hom(T,X)$ and acting on morphisms $T' \to T$ by precomposition.

        \item   We frequently identify an object $X \in \scrC$ with its image $h_X$ under the Yoneda embedding. In practice this means that we write $t \in X(T)$ to mean that $t\colon T \to X$ is a morphism. We refer to $t$ as a \emph{$T$-point of $X$}.

        \item     We will also identify morphisms in $\scrC$ with their images under the Yoneda embedding. In particular, for $t \in X(T)$ as above, we write $\phi(t)$ to denote the morphism $\phi \circ t$, which is a $T$-point of $Y$. This is less standard, but it will be particularly useful in \cref{chapter: S-rational families}.
    \end{enumerate}
\end{notation}

\begin{notation}[Products]\label{products and Yoneda combine}

    Let $\scrC$ be a category with products.
    \begin{enumerate}
        \item  We will sometimes (especially in later chapters) use concatenation to denote products. In particular, if $X$ and $Y$ are objects in a category $\scrC$ (this will usually be the category of schemes over a base $S$) we denote by $XY$ (or sometimes $X\cdot Y$) the product of $X$ and $Y$ in $\scrC$.

        \item Given a product $X_1\cdots X_n$, we often write $\rho_{X_{i_1},\dots,X_{i_m}}$, or $\rho_{i_1,\dots,i_n}$ where the $X_i$ are clear from context, to denote the projection $X_1\cdots X_n\to X_{i_1}\cdots X_{i_m}$.

        \item If $t_1 \in X_1(T)$ and $t_2 \in X_2(T)$ are morphisms in $\scrC$, we denote by ${(t_1,t_2)}$ the induced morphism from $T$ to $X_1\times X_2$.

        \item If $\phi_1\colon X_1 \to Y_1$ and $\phi_2\colon X_2 \to Y_2$ are morphisms, we denote by $\phi_1\times \phi_2$ the induced morphism from $X_1\times X_2$ to $Y_1 \times Y_2$.

        \item  The above plays rather nicely with \cref{Yoneda Notation}. For example, in the above situation, we have $(\phi_1\times \phi_2)\circ (t_1,t_2) = (\phi_1(t_1),\phi_2(t_2))$. If $x\colon XY \to X$ and $y\colon XY \to Y$ denote the co-ordinate projections, then $(x,y) = \id_{XY}$ and we have
              \begin{align*}
                  \phi_1\times \phi_2 = (\phi_1\times \phi_2)(x,y) = (\phi_1(x),\phi_2(y)).
              \end{align*}
              We will use this frequently in later sections.

    \end{enumerate}
\end{notation}

\begin{notation}[Slice Categories]\label{slice categories notation}
    Let $\scrC$ be a category and $S$ an object of $\scrC$. We denote by $\scrC_S$ the \emph{slice category of $\scrC$ over $S$} whose objects are morphisms $\pi\colon T \to S$ in $\scrC$, and where a morphism $(\pi\colon T \to S) \to (\pi'\colon T' \to S)$ is a morphism $\phi\colon T \to T'$ in $\scrC$ which fits into a commutative triangle:
    \[
        \begin{tikzcd}
            T \arrow[rr,"\phi"] \arrow[rd, "\pi"'] &   & T' \arrow[ld, "\pi'"] \\
            & S &
        \end{tikzcd}
    \]
    We will usually identify objects $\pi\colon T \to S$ in $\scrC_S$ with the corresponding object $T \in \scrC$ in situations where the morphism $\pi$ (sometimes called the \emph{structure morphism}) is clear from context. In these situations, morphisms in $\scrC_S$ will be called \emph{morphisms over $S$}.
\end{notation}

%% file: Chapters/Category-Theoretic_Prerequisites/Main_Prerequisites.tex
\chapter{Category-Theoretic Prerequisites}
We fix notational conventions and collect some standard results from category theory that we will need for the remainder of this work.

This chapter is probably best skipped on first reading and referred back to as needed.

\input{Chapters/Category-Theoretic_Prerequisites/Generalities.tex}

\input{Chapters/Category-Theoretic_Prerequisites/Subobjects.tex}

\input{Chapters/Category-Theoretic_Prerequisites/Products.tex}
\input{Chapters/Category-Theoretic_Prerequisites/Equivalence_Relations.tex}

\input{Chapters/Category-Theoretic_Prerequisites/Adjoint_functors.tex}
\input{Chapters/Category-Theoretic_Prerequisites/Slice_Categories.tex}

\input{Chapters/Category-Theoretic_Prerequisites/Graphs.tex}

\input{Chapters/Category-Theoretic_Prerequisites/Presheaves.tex}

\input{Chapters/Category-Theoretic_Prerequisites/Presheaves_on_Slice_Categories.tex}
\input{Chapters/Category-Theoretic_Prerequisites/Sites_and_Sheaves.tex}

%% file: Chapters/Category-Theoretic_Prerequisites/Generalities.tex
\section{Generalities}
We recall some basic definitions  from category theory.

\begin{dfn} \phantom{.} \label{definitions from category theory}
    \begin{enumerate}
        \item A category is said to be \emph{locally small} if the collection of morphisms between any two objects is a set.

        \item A category is \emph{small} if it is locally small and the collection of objects is a set.

        \item A category $\scrC$ is said to be \emph{complete} if it admits small limits (i.e., limits over set-sized diagrams), and \emph{cocomplete} if it admits small colimits.

        \item A functor is \emph{continuous} if it preserves all small limits, and cocontinuous if it preserves small colimits. Note that this is unrelated to the notion of continuous used in \cite[\href{https://stacks.math.columbia.edu/tag/00WU}{Tag 00WU}]{stacks-project}.

        \item A category is \emph{Cartesian closed} if it has finite products (including a terminal object) and if for any $T \in \scrC$ the functor $-\times T$ has a right adjoint; this right-adjoint is called the exponential functor and often denoted $(-)^T$.

        \item A functor is \emph{exact} if it preserves colimits and finite limits.
    \end{enumerate}
\end{dfn}

\begin{lem}[{\cite[Proposition~4.2]{nlab:yoneda_embedding}}]\label{Yoneda is exact}
    The Yoneda embedding is continuous. In particular, given any $X\in \scrC$, we have $h_X \times h_X = h_{X\times X}$, so \cref{Yoneda Notation} introduces no ambiguity.
\end{lem}
\begin{lem}\label{a functor which preserves finite limits preserves monomorphisms}
    A functor which preserves finite limits preserves monomorphisms.
\end{lem}
\begin{proof}
    Immediate, since the property of being a monomorphism can be expressed using a limit diagram - indeed $X \to Y$ is a monomorphism if and only if $X \simeq X \times_Y X$.
\end{proof}

%% file: Chapters/Category-Theoretic_Prerequisites/Subobjects.tex
\section{Subobjects}\label{section: subobjects}
\begin{dfn}
    \begin{enumerate}
        \item   Let $\scrC$ be a category and $X\in \scrC$. We define the category $\Sub'(X)$ to be the full subcategory of the slice category $\scrC_X$ consisting of monomorphisms. We define the category $\Sub(X)$ to be the category whose objects are isomorphism classes in $\Sub'(X)$ and whose morphisms are those induced by morphisms in $\Sub'(X)$. A subobject of $X$ will be an object of $\Sub(X)$.

        \item We will usually refer to an object in $\Sub(X)$ by one of its representatives in $\Sub'(X)$.

        \item If $Y,Y'$ are two subobjects of $X$, we denote by $Y \cap Y'$ and $Y\cup Y'$ the product and coproduct, respectively, of $Y$ and $Y'$ in the category $\Sub(X)$, where they exist. We refer to $Y \cap Y'$ as the \emph{intersection} of $Y$ and $Y'$, and $Y\cup Y'$ as the \emph{union} of $Y$ and $Y'$.
    \end{enumerate}
\end{dfn}

\begin{rem}[$\Sub(X)$ vs $\Sub'(X)$]
    We will need to work with $\Sub(X)$ instead of $\Sub'(X)$ to get associativity of composition of partial morphisms.

    In general, $\Sub(X)$ may be small even if $\Sub'(X)$ is not. This leads to the notion of a \emph{well-powered} category. Well-poweredness will not be important for us since we will usually assume our category is small.

    In any case, since the categories $\Sub'(X)$ and $\Sub(X)$ are equivalent, one should not worry overmuch about the distinction.
\end{rem}

\begin{lem}\label{Intersections and unions of subobjects exist in topoi}
    Let $\scrC$ be a category, $X \in \scrC$, and $Y,Y' \in \Sub(X)$.
    \begin{enumerate}
        \item  If $\scrC$ has fiber products, then the intersection $Y\cap Y'$ exists and is given by the product $Y\times_X Y'$.
        \item  If $\scrC$ is a topos (for example, $\scrC = \textit{Set}$ or $\scrC = \Presh(\scrD)$) then the union $Y \cup Y'$ exists and is given by the pushout $Y \sqcup_{(Y \cap Y')} Y'$. In particular the latter is a subobject of $X$.
    \end{enumerate}
\end{lem}
See \cite{nlab:subobject} for the proof.

%% file: Chapters/Category-Theoretic_Prerequisites/Products.tex
\section{Products}

We will make frequent use of the following standard results:
\begin{lem}[Cartesian rectangles] \label{Cartesian rectangles}
    Let
    \[
        \begin{tikzcd}
            A \arrow[r] \arrow[d] & B \arrow[r] \arrow[d] & C \arrow[d] \\
            X \arrow[r]           & Y \arrow[r]           & Z
        \end{tikzcd}
    \]
    be a commutative rectangle in $C$, and suppose the right-hand square is Cartesian. Then the left-hand square is Cartesian if and only if the outer rectangle is Cartesian.
\end{lem}
\begin{lem}[Cartesian Cubes]\label{cartesian cubes}
    Let
    \[
        \begin{tikzcd}
            A_1 \arrow[rr] \arrow[dd] \arrow[rd] &                           & B_1 \arrow[dd] \arrow[rd] &                \\
            & A_0 \arrow[rr] \arrow[dd] &                           & B_0 \arrow[dd] \\
            C_1 \arrow[rr] \arrow[rd]            &                           & D_1 \arrow[rd]            &                \\
            & C_0 \arrow[rr]            &                           & D_0
        \end{tikzcd}
    \]
    be a commutative cube in $\scrC$. Suppose the front, left and right faces are Cartesian. Then the back face is Cartesian.
\end{lem}

Recall the following, which we shall frequently use implicitly:
\begin{lem}[Pullbacks of monomorphisms]\label{pullbakcs of monomorphisms lemma}
    Suppose
    \[
        \begin{tikzcd}
            A \arrow[r] \arrow[d] & B \arrow[d] \\
            C \arrow[r, "\phi"]           & D
        \end{tikzcd}
    \]
    is a Cartesian diagram, and $B \to D$ is a monomorphism. Then $A \to C$ is a monomorphism, and $A$ represents the subpresheaf
    \begin{align*}
        T \mapsto \{c \in C(T) : \phi(c) \in B(T) \}
    \end{align*}

    In particular, if $B \to D$ is an isomorphism, then so is $A \to C$.
\end{lem}

\begin{lem}\label{associativity of fiber product}
    Fiber product is associative (up to canonical isomorphism). In other words, given a diagram
    \[
        \begin{tikzcd}
            X \arrow[rd] &   & Y \arrow[ld] \arrow[rd] &   & Z \arrow[ld] \\
            & A &                         & B &
        \end{tikzcd}\]
    we have an isomorphism $(X\times_A Y) \times_B Z \simeq X\times_A (Y \times_B Z)$ which commutes with the projections to $X,Y$ and $Z$.
\end{lem}

%% file: Chapters/Category-Theoretic_Prerequisites/Equivalence_Relations.tex
\section{Equivalence Relations}
The following definition comes from \cite{nlab:congruence}. The reader should consult {\cite[2.5]{Borceux_1994}} for more details on equivalence relations in categories.

\begin{dfn}[Equivalence relations]\label{equivalence relation definition}
    Let $\scrC$ be a category and $X \in \scrC$. An \emph{equivalence relation} $E$ on $X$ is a representable subpresheaf $h_E \hookrightarrow h_X \times h_X$ such that for every $T\in \scrC$, the induced embedding $E(T) \hookrightarrow X(T) \times X(T)$ exhibits $E(T)$ as an equivalence relation on $X(T)$.
\end{dfn}

\begin{lem}
    Let $\scrC$ be a category with finite products. An equivalence relation $E$ on $X$ is equivalently a subobject
    $    \begin{tikzcd}
            E \ar[r,hook,"{(p_1,p_2)}"]& X\times X
        \end{tikzcd}$
    satisfying:
    \begin{enumerate}
        \item \textbf{Reflexivity}: There exists a morphism $r \colon X \to E$ such that
              \[ p_1 \circ r = \mathrm{id}_X \quad \text{and} \quad p_2 \circ r = \mathrm{id}_X. \]

        \item \textbf{Symmetry}: There exists a morphism $s \colon E \to E$ such that
              \[ p_1 \circ s = p_2 \quad \text{and} \quad p_2 \circ s = p_1. \]

        \item \textbf{Transitivity}: Let $E \times_X E$ be the pullback of $p_1 \colon E \to X$ along $p_2 \colon E \to X$, with projections $\pi_1, \pi_2 \colon E \times_X E \to E$. There exists a morphism $t \colon E \times_X E \to E$ such that
              \[ p_1 \circ t = p_1 \circ \pi_1 \quad \text{and} \quad p_2 \circ t = p_2 \circ \pi_2. \]
    \end{enumerate}
\end{lem}
\begin{proof}
    See for example \cite[Proposition~2.5.4]{Borceux_1994}.
\end{proof}

\begin{cor}[The Yoneda embedding preserves equivalence relations]\label{Yoneda preserves equivalence relations}
    Let $\scrC$ be a category, $X\in \scrC$, and $E$ an equivalence relation on $X$. Then $h_E$ is an equivalence relation on $h_X$.
\end{cor}
\begin{proof}
    Since $\Presh(\scrC)$ has finite products, we just need to verify the hypotheses of the above lemma. For every $T$ we have by assumption that $E(T)$ is an equivalence relation on $X(T)$; in particular we have a well-defined reflexivity map $r_T\colon X(T) \to E(T)$ sending $x \to (x,x)$, and similarly have symmetry and transitivity maps $s_T$ and $t_T$. From the definitions we see immediately that the functions $r_T, s_T, t_T$ commute with restriction maps; therefore they induce well-defined maps of presheaves $r,s,t$ which satisfy the required properties.
\end{proof}

%% file: Chapters/Category-Theoretic_Prerequisites/Adjoint_functors.tex
\section{Adjoint Functors}
All the results in this section (with the possible exception of \cref{adjoint functors commutative square} for which we provide a proof) are standard. We refer the reader to \cite{Saunders_categories} for background on adjoint functors.

Let $F\colon\scrD \to \scrC$ and $G\colon \scrC \to \scrD$ be functors. Recall that $F$ is left adjoint to $G$, or equivalently $G$ is right adjoint to $F$, if the bifunctors $\Hom_\scrC(F(-),-)$ and $\Hom_\scrD(-,G(-))$ are naturally isomorphic. Adjoint functors are determined uniquely up to natural isomorphism.


The following results are proved in \cite{Saunders_categories}. First we give a characterization of adjoint functors:
\begin{lem}\label{universal property characterization of adjoint functors}
    Let $G\colon \scrC \to \scrD$ be a functor. Then $G$ has a left adjoint if and only if for every $Y$ in $\scrD$ there exists an object $F(Y) \in \scrC$ and a morphism $\eta_Y\colon Y \to G(F(Y))$ in $\scrD$ such that any morphism $Y \to  G(X)$ with $X\in \scrC$ factors uniquely through $\eta_Y$. In this case the mapping $Y \mapsto F(Y)$ can be extended uniquely to give a functor $F\colon \scrD \to \scrC$ which is left adjoint to $G$, and moreover the $\eta_Y$ cohere to give a natural transformation of functors $\eta\colon \id_\scrD \to GF$, called the unit of adjunction.
\end{lem}

\begin{lem}\label{unit and counit of adjunction}
    An adjunction $\Hom_\scrC(F(-),-) \simeq \Hom_\scrD(-,G(-))$ gives rise to natural transformations of functors
    \begin{align*}
        \eta\colon \id_\scrD \to GF \quad \epsilon \colon FG \to \id_C,
    \end{align*} called the unit and counit of adjunction respectively, such that for any $X,Y$ and $\phi \in \Hom_\scrC(F(X),Y)$, the corresponding morphism in $\Hom_\scrD(X,G(Y))$ is given by the composition
    \begin{align*}
        X \xrightarrow{\eta_X} GF(X) \xrightarrow{G(\phi)} G(Y),
    \end{align*}
    and given $\psi \in \Hom_{\scrD}(X,G(Y))$ the corresponding map is the composition
    \begin{align*}
        F(X) \xrightarrow{F(\psi)} FG(Y) \xrightarrow{\epsilon_Y} Y.
    \end{align*}

    The unit and counit satisfy the so-called triangle identities:
    \begin{align*}
        \id_F = \epsilon \cdot F \circ F\cdot\eta \\
        \id_G = G\cdot \epsilon \circ \eta \cdot G.
    \end{align*}
\end{lem}

The adjoint of a composition is the composition of the adjoints:
\begin{lem}\label{The adjoint of a composition is the composition of the adjoints}
    Let $G\colon \scrC \to \scrD$, $G'\colon \scrD \to \scrE$, $F'\colon\scrE \to \scrD$, $F\colon\scrD \to \scrC$ be functors. Suppose $F,F'$ are left adjoint to $G,G'$ respectively with respective units and counits of adjunction $\eta, \eta', \epsilon,\epsilon'$ Then $F'\circ F$ is left adjoint to $G \circ G'$. The unit and counit of the adjunction of the composition are given by
    \begin{align*}
        \id_\scrE \xrightarrow{\eta'} G'F' \xrightarrow{G'\eta F'} G'GFF' \\
        FF'G'G \xrightarrow{F\epsilon'G} FG \xrightarrow{\epsilon} \id_\scrC
    \end{align*}
\end{lem}

\begin{lem}\label{adjoint functors commutative square}
    In the situation of \cref{The adjoint of a composition is the composition of the adjoints}, we have a commutative diagram of bifunctors:
    \[
        \begin{tikzcd}
            \Hom_\scrC(F(-), -)  \ar[r,"\sim"] \ar[d,"\circ F\cdot \eta'"]     & \Hom_\scrD(-,G(-)) \ar[d, "G'"]\\
            \Hom_\scrC(FF'G'(-), -)    \ar[r, "\sim"]          & \Hom_\scrE(G'(-), G'G(-))
        \end{tikzcd}
    \]
    where the horizontal arrows are given by the adjunctions.
\end{lem}
\begin{proof}
    Indeed, given any $X \in \scrD$, $Y\in \scrC$, and $\phi \in \Hom_\scrC(F(X),Y)$, then applying \cref{unit and counit of adjunction}, going around clockwise we get the composition
    \begin{align*}
        G'(X) \xrightarrow{G'\cdot \eta_X} G'GF \xrightarrow{G'G(\phi)} G'G(Y)
    \end{align*}
    Going counterclockwise, applying the vertical arrow we get
    \begin{align*}
        FF'G'(X) \xrightarrow{F\cdot \epsilon'_X} F(X) \xrightarrow{\phi} Y.
    \end{align*}
    Then applying the horizontal arrow, by the characterization of the unit of adjunction from \cref{The adjoint of a composition is the composition of the adjoints}, we get the composition
    \begin{align}\label{many arrows adjoint computation}
        G'(X) \xrightarrow{\eta'_X\cdot G'} G'F'G'(X) \xrightarrow{G'\eta_XF'G'} G'GFF'G'(X) \xrightarrow{G'GF\epsilon'_X} G'GF(X) \xrightarrow{G'G(\phi)} G'G(Y)
    \end{align}
    Consider the middle two arrows, and note that we get the same result if we apply $\epsilon'_X$ first and then $\eta_X$; so the middle two arrows are the same as the composition
    \begin{align*}
        G'F'G'(X) \xrightarrow{G'\epsilon'_X} G'(X) \xrightarrow{G'\eta_X} G'GF(X).
    \end{align*}
    Then by the triangle identities we see that the first three arrows of \ref{many arrows adjoint computation} compose to give the map $G'(X) \xrightarrow{G'\cdot \eta_X} G'GF(X)$ and we conclude.
\end{proof}

The following lemma is not hard, but we will often use it:
\begin{lem}\label{adjoint functors and full subcategories}
    Suppose $F\colon\scrD \to \scrC$ and $G\colon \scrC \to \scrD$ are adjoint functors, and $\scrC' \subset \scrC$, $\scrD' \subset \scrD$ are full subcategories such that $F|_{C'}$ and $G|_{D'}$ factor (necessarily uniquely) through morphisms $F'\colon \scrC' \to \scrD'$ and $G'\colon \scrD' \to \scrC'$ respectively.

    Then $F'$ is left-adjoint to $G'$.
\end{lem}
\begin{proof}
    Since $\scrC'$ and $\scrD'$ are full subcategories, for any $C \in \scrC'$ and $D \in \scrD'$ we have
    \begin{align*}
        \Hom_{\scrD'}(F'(C),D) = \Hom_{\scrD}(F(C),D),\qquad \Hom_{\scrC'}(C,G'(D)) = \Hom_\scrC(C,G(D)).
    \end{align*}

    Therefore the natural isomorphism $\Hom_\scrD(F(-),-) \simeq \Hom_\scrC(-,G(-))$ induces a natural isomorphism $\Hom_{\scrD'}(F'(-),-) \simeq \Hom_{\scrC'}(-,G'(-))$.
\end{proof}


A very useful result is
\begin{lem}
    Right adjoints are continuous. Left adjoints are cocontinuous.
\end{lem}

A partial converse to the above is given by Freyd's adjoint functor theorem.
\begin{dfn}[Solution Set Condition]\label{solution set condition}
    A functor $G\colon \scrD \to \scrC$ is said to satisfy the solution set condition if for every $C\in \scrC$ there is a \emph{set} of arrows $(f_i\colon C \to G(D_i))_{i\in I}$ such that every arrow $h\colon C \to G(D)$ factors through one of the $f_i$.
\end{dfn}
\begin{thm}[Freyd's Adjoint Functor Theorem, {\cite[Chapter V.6]{Saunders_categories}}]\label{Freyd's adjoint functor theorem}
    A continuous functor from a complete locally small category which satisfies the solution set condition admits a left adjoint.
\end{thm}


%% file: Chapters/Category-Theoretic_Prerequisites/Slice_Categories.tex
\section{Slice Categories}
We refer the reader to \cite{nlab:over_category} or \cite{Mac_Lane_Moerdijk_1994} for background on slice categories. Recall \cref{slice categories notation}.

The following material is all standard.

\begin{rem}[A Slice of a Slice is a Slice]\label{a slice of a slice is a slice}
    Let $\scrC$ be a category and $\phi\colon T' \to T$ a morphism in $\scrC$. Unraveling the definitions, we see that objects in the ``double slice'' $(\scrC_T)_{\phi}$ correspond to morphisms $T'' \to T'$, and morphisms in $(\scrC_T)_{\phi}$ are the same as morphisms over $T'$; i.e., we have an equivalence of categories $(\scrC_T)_{\phi}\simeq \scrC_{T'}$.
\end{rem}

\begin{dfn}[Forgetful functors on slice categories]\label{forgetful functor on slice categories}
    Let $\scrC$ be a category and $\phi\colon S\to T$ a morphism in $\scrC$.
    \begin{enumerate}
        \item We denote by $j_T$ the forgetful functor from $\scrC_T$ to $\scrC$ sending an object $(\psi\colon X \to T )$ to $X$.
        \item We denote by $j_{S/T}$ the postcomposition functor sending each object $(\psi\colon X \to S)\in \scrC$ to the composition $\phi \circ \psi\colon X \to T$.
    \end{enumerate}
\end{dfn}
\begin{rem}
    Observe that $j_{S/T}$ corresponds via the equivalence of \cref{a slice of a slice is a slice} to the forgetful functor from $(\scrC_T)_{\phi}$ to $\scrC_T$. For this reason we sometimes call $j_{S/T}$ a forgetful functor.

    Note also that we have $j_T = j_{S} \circ j_{T/S}$.
\end{rem}

For the following result, see \cite[I.9,~Theorem~4]{Mac_Lane_Moerdijk_1994} and its proof:
\begin{prop}[Adjoint Functors on Slice Categories]\label{adjoint functors on slice categories}
    Let $\scrC$ be a category, and $\phi\colon T\to S$ a morphism in $\scrC$.
    \begin{enumerate}
        \item     If $\scrC$ admits products of pairs of elements, then the forgetful functor $j_S\colon \scrC_S \to \scrC$ has a right adjoint $j_S^*$, called the \emph{base change map} or \emph{pullback}, which sends an object $X\in \scrC$ to the projection $X \times T \mapsto T$, which we denote by $X_T$, and which sends morphisms $X \to Y$ to the induced morphisms $X\times T \to Y\times T$.

        \item \label{right adjoint functor existence on slice category}  If $\scrC$ admits fiber products, then the postcomposition $j_{T/S}\colon \scrC_T \to \scrC_S$ has a right adjoint $j_{T/S}^*$, also called the base change or pullback, which sends an object $X \to S$ to the projection $X \times_S T \mapsto T$, and which sends morphisms $X \to Y$ to the induced morphisms $X\times_S T \to Y\times_S T$.

        \item \label{right adjoint to right adjoint on slice category} If $\scrC$ is Cartesian closed then $j_S^*$ has a right adjoint $(j_S)_*$ such that for any $f\colon Y \to S$ in $\scrC_S$, we have that $(j_S)_*(f)$ is equal to fiber product
              \[
                  \begin{tikzcd}
                      & Y^S \ar[d,"f^S"]\\
                      1 \ar[r,"\iota"] & S^S,
                  \end{tikzcd}
              \]
              where $\iota\colon 1 \to S^S$ is the map from the terminal object induced by $\id_T$ and the adjunction.

        \item  If $\scrC_S$ is Cartesian closed, then $j_{T/S}^*$ has a right adjoint $(j_{T/S})_*$ defined similarly to $(j_S)_*$ above.
    \end{enumerate}
\end{prop}

\begin{notation}\label{base change notation}
    Let $\scrC$ be a category with products and $X,S \in \scrC$. We write $X_S$ to denote the base change $j_S^*(X)$. Similarly, for any $T \to S$ and $Y \in \scrC_S$, we write $Y_T$ to denote the base change $j_{T/S}^*Y$.
\end{notation}

\begin{rem}\label{slice categories extra right adjoint remarks}
    In statement \ref{right adjoint to right adjoint on slice category}, observe that it is enough for the category to allow exponentiation by $S$; i.e. that the functor $X \to X\times S$ has a right adjoint.
\end{rem}

\begin{rem}
    Note that statements 2 and 4 above follow from statements 1 and 3 by \cref{a slice of a slice is a slice}.
\end{rem}

\begin{rem}
    It follows from \cref{associativity of fiber product} that pullback is associative up to canonical isomorphism; i.e., for any $X \in \scrC$ and $T \to S$ as above we have $X_T \simeq (X_S)_T$ naturally in $X$.

    Similarly, for any $T' \to T$ and $Y \in \scrC/S$ we have $Y_{T'}\simeq (Y_T)_{T'}$ naturally in $Y$.
\end{rem}

\begin{cor}\label{pullback preserves limits in categories with fiber products}
    Let $\scrC$ be a category with finite limits. Then for any $T \to S$ in $\scrC$, the pullbacks $j_{S}^*$ and $j_{T/S}^*$ preserve limits.
\end{cor}
\begin{proof}
    Immediate since right adjoints preserve limits.
\end{proof}

\begin{notation}
    Let $\scrC$ be a category with fiber products, and $f\colon T \to S$ a morphism in $\scrC$. In this situation the pullback $j_{T/S}^*$ is usually denoted by $f^*$.
\end{notation}

\begin{lem}\label{pullback is the same as pullback by restriction to subobject of the target}
    Let $\scrC$ be a category with fiber products, $X \to S$ a morphism in $\scrC$, and $f\colon T\to S$ be a morphism in $\scrC$. Suppose that $X \to S$ and $f$ both factor through a sub-object $U$ of $S$. Then $f^*(X) = f|_{[T \to U]}^*(X)$.
\end{lem}
\begin{proof}
    Indeed, we have a diagram
    \[
        \begin{tikzcd}
            {(f|_{[T \to U]})^*(X)} \arrow[d] \arrow[r] & T \arrow[d, "{f|_{[T \to U]}}"] \arrow[r] & T \arrow[d, "f"] \\
            X \arrow[r]                                 & U \arrow[r, hook]                               & S
        \end{tikzcd}
    \]
    where both squares are Cartesian; we conclude by the Cartesian rectangle lemma (\cref{Cartesian rectangles}).
\end{proof}

%% file: Chapters/Category-Theoretic_Prerequisites/Graphs.tex
\section{Graphs}
The material in this section is largely standard. It will not be used until \cref{chapter: S-rational families}.

Let $\scrC$ be a category with products and fiber products, and $\phi\colon X \to Y$ a morphism in $\scrC$

Recall that the graph $\Gamma_\phi$ of $\phi$ is defined to be the fiber product
\[
    \begin{tikzcd}
        X \times_Y Y \arrow[r] \arrow[d] & Y \arrow[d, "\id_Y"] \\
        X \arrow[r, "\phi"]              & Y
    \end{tikzcd}
\]
In particular, the graph is isomorphic to $X$, and using \cref{pullbakcs of monomorphisms lemma} we see that it represents the presheaf
\begin{align*}
    T \mapsto \{ (x,y) \in X(T) \times Y(T) : \phi(x) = y\}.
\end{align*}
From this we get:
\begin{lem}\label{graph of morphis criterion}
    Let $\Gamma \subset X \times Y$ be a subobject, and consider the projections $\rho_X\colon \Gamma \to X$ and $\rho_Y\colon \Gamma \to Y$. Suppose $\rho_X$ is an isomorphism. Then $\Gamma$ is the graph of the morphism $\rho_Y \circ \rho_X\inv$.
\end{lem}
Also by \cref{pullbakcs of monomorphisms lemma} we have:
\begin{lem}\label{Diagonal cartesian square for graphs}
    For any morphism $\phi\colon X \to Y$, we have a commutative diagram
    \[
        \begin{tikzcd}
            X \arrow[rd, dashed] \arrow[rdd, "{(\id_X,\phi)}"', bend right] \arrow[rrd, "\phi", bend left] &                       &    \\
            & \Gamma_\phi \arrow[r] \arrow[d, hook]     & Y \arrow[d, "\Delta_Y", hook] \\                                          & X \times Y \arrow[r, "\phi \times \id_Y"] & Y \times Y
        \end{tikzcd}
    \]
    where $\Delta_Y$ denotes the diagonal morphism, the square is Cartesian, and the dashed arrow is an isomorphism.
\end{lem}

From the above we get:
\begin{enumerate}\label{Inverse and morphism graph lemma}
    \item Morphisms are determined uniquely by their graphs
    \item  Let $\phi\colon X \to Y$, $\psi\colon Y \to X$  be morphisms in $\scrC$. Then $\phi$ and $\psi$ are mutually inverse isomorphisms in $\scrC$ if and only if we have a Cartesian diagram:
          \[
              \begin{tikzcd}
                  \Gamma_\phi \arrow[r] \arrow[d, hook] & \Gamma_\psi \arrow[d, hook] \\
                  X \times Y \arrow[r, "\rho_{YX}", "\sim"']   & Y\times X
              \end{tikzcd}
          \]
          where $\rho_{YX}$ denotes the coordinate permutation.
\end{enumerate}

Taking the graph of a morphism commutes with pullback:
\begin{prop}\label{Taking the graph of a morphism commutes with pullback:}
    Let $\scrC$ be a category with fiber products and let $f\colon T \to S$ be a morphism in $\scrC$. Let $\phi\colon X \to Y$ be a morphism in $\scrC_S$ with graph $\Gamma_\phi$.

    Then $f^*(\Gamma_\phi) = \Gamma_{f^*(\phi)}$.
\end{prop}
\begin{proof}
    Immediate since pullbacks preserve fiber products (\cref{pullback preserves limits in categories with fiber products}), and the Cartesian square
    \[
        \begin{tikzcd}
            \Gamma_\phi \arrow[d] \arrow[r] & Y\arrow[d]\\
            X \arrow[r] & Y
        \end{tikzcd}
    \]
    pulls back to the Cartesian square defining $\Gamma_{f^*(\phi)}$.
\end{proof}

%% file: Chapters/Category-Theoretic_Prerequisites/Presheaves.tex
\section{Presheaf Categories}\label{Presheaves prerequisites section}
In this section we collect various results on presheaf categories. In particular, we discuss the internal hom functor and recall the precomposition functor and its adjoints.

\begin{lem}
    Let $\scrC$ be a small category. Then $\Presh(\scrC)$ is locally small.
\end{lem}
\begin{proof}
    Apply the remarks of \cite[Chapter~II.4]{Saunders_categories}, observing that the category of sets is locally small.
\end{proof}

\begin{rem}[Iterated Presheaves]\label{Iterated presheaves remark}
    Let $\scrC$ be a category. Then by the Yoneda embedding, there is a restriction map $r\colon \Presh(\Presh(\scrC)) \to \Presh(\scrC)$ such that for any presheaf $F$ on $\Presh(\scrC)$ and $X \in \Presh(\scrC)$ we have $r(F)(X) = F(h_X)$.

    Observe that, by the Yoneda lemma, for any $X\in \Presh(\scrC)$ we have $r(h_X) \simeq X$.
\end{rem}

\begin{dfn}[Internal Hom]\label{internal Hom definition}
    Let $\scrC$ be a small category, and $X,Y \in \Presh(\scrC)$. We denote by $\iHom(X,Y)$, or the \emph{internal hom of $X$ and $Y$}, the functor defined by
    \begin{align*}
        \iHom(X,Y)(T) = \Hom_{\Presh(\scrC)}(h_T \times X, Y),
    \end{align*}
    where for morphisms $T' \to T$ the corresponding restriction map is given by precomposition with the induced map $X \times h_{T'} \to X \times h_T$.
\end{dfn}

\begin{rem}\label{internal Hom in locally small category}
    Note that for objects $X,Y \in \scrC$ we have by the Yoneda lemma (and since the Yoneda embedding preserves products) that
    \begin{align*}
        \iHom(h_X,h_Y)(T) = \Hom_\scrC(T \times X, Y),
    \end{align*}
    and restriction along $T' \to T$ is given by precomposition with $X \times T' \to X\times T$.

    In particular, even if the category is only locally small, then one can still define the internal hom of representable presheaves.
\end{rem}

\begin{rem}\label{internal Hom of iterated presheaves}
    For any presheaf $X \in \Presh(\scrC)$, one can consider the presheaf $\iHom(h_X,h_X) \in \Presh(\Presh(\scrC))$. By the Yoneda lemma, we have that $r(\iHom(h_X,h_X)) = \iHom(X,X)$, where $r$ is the restriction defined in \cref{Iterated presheaves remark}.
\end{rem}

\begin{rem}
    In fact, the assignment $X,Y \mapsto \iHom(X,Y)$ defines a functor from the product category $\Presh(\scrC)^{op} \times \Presh(\scrC) $ to $\Presh(\scrC)$, called the \emph{internal hom functor}, where given morphisms $\phi\colon X' \to X$ and $\psi\colon Y \to Y'$ the morphism
    \begin{align*}
        \iHom(\phi,\psi)\colon \iHom(X,Y)\to \iHom(X',Y')
    \end{align*}
    is induced by precomposition with $\phi_T\times \id_T$ and postcomposition with $\psi_T$ for every $T$. See \cite{nlab:closed_monoidal_structure_on_presheaves}.
\end{rem}

\begin{rem}[Presheaf categories are Cartesian closed]\label{The category of presheaves is Cartesian closed}
    If $\scrC$ is small, then the category of presheaves on $\scrC$ is Cartesian closed. Indeed, $\Presh(\scrC)$ is complete (i.e., admits small limits), and the limits are computed pointwise (see \cite[Chapter~5]{Saunders_categories}). Moreover, since $\scrC$ is small, for any $Y \in \Presh(\scrC)$, $\Presh(\scrC)$ admits an exponential functor defined using the internal hom: $(-)^Y = \iHom(Y,-)$. See \cite[Section~I.6]{Mac_Lane_Moerdijk_1994}.
\end{rem}

Morphisms between categories induce morphisms between the corresponding presheaf categories:
\begin{dfn}[The Precomposition Functor, {\cite[\href{https://stacks.math.columbia.edu/tag/00VC}{Tag 00VC}]{stacks-project}}]\label{functors induce precomposition functors on presheaf categories}
    Let $\scrC$, $\scrD$ be categories. Given any functor $f\colon \scrC \to \scrD$, we denote by $f\inv\colon \Presh(\scrD) \to \Presh(\scrC)$ (or $f^p$ in the parlance of \cite[\href{https://stacks.math.columbia.edu/tag/00VC}{Tag 00VC}]{stacks-project}) the ``precomposition functor' or ``pullback'' sending any presheaf $X$ in $\Presh(\scrD)$ to the presheaf $X \circ f$, and sending any morphism of presheaves $\phi\colon X \to Y$ to the induced morphism $f\inv(\phi)\colon X \circ f \to Y \circ f$; i.e., for any $T \in \scrD$, we have $f\inv(\phi)_T = \phi_{f(T)}$ as a map from $X(f(T))$ to $Y(f(T))$.
\end{dfn}



\begin{prop}\label{adjoints to precomposition functor}
    \begin{enumerate}
        \item    Assuming all the relevant limits exist, the functor $f\inv$ has a left adjoint $f_!$ sending any presheaf $X \in \Presh(\scrC)$ to the presheaf
              \begin{align*}
                  T \mapsto \varinjlim_{\phi\colon T \to F(T')}F(T').
              \end{align*}
              See \cite[Expos\'e~i, Proposition~5.1]{SGA4} or \cite[\href{https://stacks.math.columbia.edu/tag/00XZ}{Tag 00XZ}]{stacks-project} for details.

        \item Similarly, assuming all the relevant limits exist, the functor $f\inv$ has a right adjoint $f_*$ sending any $X \in \Presh(\scrC)$ to the presheaf
              \begin{align*}
                  T \mapsto \varprojlim_{\phi\colon F(T') \to T}F(T').
              \end{align*}
              See \cite[\href{https://stacks.math.columbia.edu/tag/00XF}{Tag 00XF}]{stacks-project}.
    \end{enumerate}
\end{prop}

%% file: Chapters/Category-Theoretic_Prerequisites/Presheaves_on_Slice_Categories.tex

\section{Presheaves on Slice Categories}
We recall that presheaves on slice categories are equivalent to slices of presheaf categories, we discuss pullbacks of presheaves on slice categories, we give a second interpretation of the internal hom on presheaves, and we show for any $X\in \Presh(\scrC)$ that $\iHom(X,X)$ is a presheaf of semigroups.

First let us discuss pullbacks of presheaves on slice categories. There are two notions of pullback for slice categories; we show that they are equivalent.
\begin{dfn}\label{pullback of presheaves on slice categories}
    Let $\scrC$ be a category and $T \in \scrC$. Recall the forgetful functor $j_T\colon \scrC_T \to \scrC$ from \cref{forgetful functor on slice categories}.
    \begin{enumerate}
        \item    By \cref{functors induce precomposition functors on presheaf categories}, we have a precomposition map, which (following \cite[\href{https://stacks.math.columbia.edu/tag/00XZ}{Tag 00XZ}]{stacks-project}) we will denote by $j_T\inv$, from $\Presh(\scrC)$ to $\Presh(\scrC_T)$, sending $X \in \Presh(\scrC)$ to $X\circ j_T$. For any $X \in \Presh(\scrC)$ we denote by $X_T$ the pullback $j_T\inv(X)$.

        \item   By \cref{adjoints to precomposition functor}, in good situations, $j_T\inv$ has left and right adjoints $(j_T)_!$ and $(j_T)_*$ respectively. Unraveling the definitions, we see that if $\scrC$ is locally small    then the left adjoint $(j_T)_!$ exists and is defined as follows: for any $X \in \Presh(\scrC_T)$ and $S \in \scrC$, we have
              \begin{align*}
                  (j_T)_!(X)(S) = \bigsqcup_{S \to T}X(S\to T)
              \end{align*}
              (see \cite[\href{https://stacks.math.columbia.edu/tag/03CD}{Tag 03CD}]{stacks-project} or \cite[Expos\'e~i, Proposition~5.11]{SGA4}). We denote by $\alpha_T$ the adjunction
              \begin{align*}
                  \alpha_T\colon \Hom_{\Presh(\scrC)}((j_T)_!(-),) \simeq \Hom_{\Presh(\scrC_T)}(-,j_T\inv(-))
              \end{align*}
              The adjunction is given as follows: for any map $\phi\colon (j_T)_!(X) \to Y$, the corresponding morphism $\alpha(\phi)\colon X_T \to Y_T$ sends any $x \in X_{T}(S \to T)$ (which is contained in $(j_T)_!(X)(S)$) to $\phi(x)$ (note $Y(S) = Y_T(S \to T)$ by definition).

        \item   Moreover, if $\scrC$ admits products of pairs of elements, then $(j_T)_*$ exists and is defined as follows: for any $X \in \Presh(\scrC_T)$ and $S \in \scrC$, we have
              \begin{align*}
                  (j_T)_*(X)(S) = X(S \times T \to T)
              \end{align*}
              (see \cite[\href{https://stacks.math.columbia.edu/tag/03CE}{Tag 03CE}]{stacks-project}).

        \item    By \cref{a slice of a slice is a slice}, for any $S \to T$ in $\scrC$ we get an analogous adjoint triple $(j_{T/S})_!$, $j_{T/S}\inv$, $(j_{T/S})_*$.
    \end{enumerate}
\end{dfn}

\begin{cor}\label{presheaf pullback preserves limits}
    For any $T \to S$ in $\scrC$, the pullbacks $j_{S}\inv$ and $j_{T/S}\inv$ preserve limits.
\end{cor}
\begin{proof}
    Immediate, since right adjoints preserve limits.
\end{proof}

\begin{rem}\label{interpretation of j_!}
    Observe that if $\scrC$ is locally small then for any $T \in \scrC$ and $Y \in \Presh(\scrC)$ we have $(j_T)_!(Y_T) \simeq Y \times h_T$. Under this isomorphism, for any $\phi\colon Y\times h_T \to Y$, and $y \in Y_T(S \to T)$, we have $\alpha_T(\phi)(y) = \phi(y,S \to T)$. Moreover the counit of adjunction $(j_T)_!(Y_T) \to Y$ is given by the projection $Y \times h_T \to Y$.

    Similarly, for $T' \to T$ in $\scrC$ and $X \in \Presh(\scrC_T)$ we have $(j_{T'/T})_!(X_{T'}) = X \times h_{T' \to T}$, and the counit of adjunction is the projection to $\rho_X\colon X \times h_{T' \to T} \to X$. Moreover, we have $(j_{T'})_! = (j_{T})_!\circ (j_{T'/T})_!$, and therefore applying $(j_T)_!$ to $\rho_X$ we get a canonical map $(j_{T'})_!(X_{T'}) \to (j_T)_!(X)$, which is given by the maps
    \begin{align*}
        \bigsqcup_{S \to T'}X(S \to T' \to T) \to \bigsqcup_{S \to T}X(S \to T)
    \end{align*}
    for each $S \in \scrC$.
\end{rem}

\begin{lem}\label{precomposition on the left commutes with pullback on the right}
    For any locally small $\scrC$ and $T \in \scrC$ the adjunction of $(j_T)_!$ and $j_T\inv$ is natural in $T$ in the following sense: given  $Y \in \Presh(\scrC)$, $X\in \Presh(\scrC_T)$ and $T' \to T \in \scrC$ we have a commutative diagram:
    \[
        \begin{tikzcd}
            {\Hom_{\Presh_\scrC}((j_{T})_!(X), Y)} \arrow[r, "\alpha_T"] \arrow[d]                             & {\Hom_{\Presh(\scrC_T)}(X,Y_T)} \arrow[d,"j_{T'/T}\inv"] \\
            {\Hom_{\Presh_\scrC}((j_{T'})_!(X_{T'}), Y)}         \arrow[r, "\alpha_{T'}"]                           & {\Hom_{\Presh(\scrC_{T'})}(X_{T'}, Y_{T'} )}
        \end{tikzcd}
    \]
    where the left vertical arrow is induced by precomposition with the map $(j_{T'})_!(X_{T'}) \to (j_T)_!(X)$.
\end{lem}
\begin{proof}
    This follows immediately from \cref{adjoint functors commutative square}. Alternatively, for any $\phi \in \Hom((j_{T})_!(X), Y)$ then following the definitions we see that going both ways around the diagram gives the morphism sending elements $t \in X_{T'}(T'' \to T')$ to $\phi(t,T'' \to T)$.
\end{proof}

\begin{lem}\label{internal hom equivalence}
    Let $X$, $Y$ be presheaves on a small category $\scrC$. Then the internal hom $\iHom(X,Y)$ from \cref{internal Hom definition} is naturally (in both $X$ and $Y$) isomorphic to a presheaf defined by
    \[
        T \mapsto \Hom_{\Presh(\scrC_T)}(X_T,Y_T)
    \]
    where for any morphism $T' \to T$ the corresponding restriction
    \begin{align*}
        \Hom_{\Presh(\scrC)_{h_T}}(X_T,Y_T) \to \Hom_{\Presh(\scrC)_{h_T}}(X_{T'},Y_{T'})
    \end{align*}
    is given by the pullback $j_{T'/T}\inv$.
\end{lem}
\begin{proof}
    Indeed, by \cref{pullback of presheaves on slice categories} and \cref{precomposition on the left commutes with pullback on the right} for any $T\in \scrC$ we have an isomorphism given by the adjunction
    \begin{align*}
        \alpha_T \colon \Hom_\scrC(X\times h_T, Y) = \Hom_\scrC((j_T)_!(X_T),Y) \xrightarrow{\sim} \Hom_{\Presh(\scrC_T)}(X_T,Y_T),
    \end{align*}
    which is natural in $X$ and $Y$. It is also natural in $T$ by \cref{precomposition on the left commutes with pullback on the right}. Therefore we get a well-defined isomorphism of presheaves.
\end{proof}

\begin{cor}\label{internal hom is a presheaf of semigroups}
    For any $X\in \Presh(\scrC)$, composition of morphisms endows $\iHom(X,X)$ with the structure of a presheaf of semigroups.
\end{cor}
\begin{proof}
    Immediate, since the pullback commutes with composition.
\end{proof}

There is an alternative notion of pullback for presheaves on slice categories:
\begin{rem}\label{pullback of slices of presheaf categories}
    \begin{enumerate}
        \item By \cref{adjoint functors on slice categories}, we have a pullback functor $j_{h_T}^*\colon \Presh(\scrC) \to \Presh(\scrC)_{h_T}$.

        \item The functor $j_{h_T}^*$ admits a left adjoint, which is equal to the forgetful functor $j_{h_T}$.

        \item If $\scrC$ is locally small, then by \cref{internal Hom in locally small category} we can define the exponential $(-)^{h_T}$ and then by \cref{slice categories extra right adjoint remarks} we also have a right adjoint.

        \item Similarly, by \cref{a slice of a slice is a slice}, for any $S \to T$ in $\scrC$ we have that $j_{T/S}$ is left adjoint to $j_{T/S}^*$, and if $\scrC$ is locally small then $j_{T/S}^*$ has a right adjoint.
    \end{enumerate}
\end{rem}

The above remarks give us two potential notions of pullback for presheaves of slice categories, namely $j_S\inv$ and $j_S^*$. The next result shows that these notions are compatible via an equivalence of categories.
The following result may be found at \cite[Expos\'e~i, Proposition~5.11]{SGA4}, or (for the more general case of sheaves on a site) at \cite[\href{https://stacks.math.columbia.edu/tag/00XZ}{Tag 00XZ}]{stacks-project}:
\begin{lem}[Presheaves on Slice Categories]\label{Presheaves on Slice categories}
    Let $\scrC$ be a locally small category, and $T\in \scrC$. There is an equivalence of categories $e$ between the category of presheaves on the slice category $\Presh(\scrC_T)$ and the slice category $\Presh(\scrC)_{h_T}$ of $\Presh(\scrC)$ over $h_T$. The functor $e$ sends a presheaf $F$ to the presheaf
    \begin{align*}
        F'\colon T' \mapsto \bigsqcup_{f\colon T' \to T}F(f),
    \end{align*}
    where the restriction maps are those induced by $F$, together with the canonical morphism $F' \to h_T$ defined as follows: for every morphism $f\colon T' \to T$ and $x\in F(f) \subset F'(T')$, $x$ is sent to $f$.

    The functor $e$ has a weak inverse $e\inv$ sending any $\phi\colon F' \to h_T$ in $\Presh(\scrC)_{h_T}$ to the presheaf $w(\phi)$ on $\scrC_T$ which sends an object $f\colon T' \to T$ to the fiber product of the diagram
    \[
        \begin{tikzcd}
            & F'(T') \arrow[d, "\phi"] \\
            pt \arrow[r, "f"]        & h_T(T').
        \end{tikzcd}
    \]

    Moreover, the equivalence $e$ respects pullbacks: we have $e\circ j_T\inv = j_{h_T}^*$. The equivalence $e$ also respects the left and right adjoints outlined above. In particular, we have a commutative diagram
    \[
        \begin{tikzcd}
            \Presh(\scrC_T) \arrow[rrrr, "e", bend left] \arrow[rrdd, "(j_T)_!"', bend right] \arrow[rrdd, "(j_T)_*", bend left] &  &     &  & \Presh(\scrC)_{h_T} \arrow[lldd, "j_{h_T}"', bend right] \arrow[lldd, "(j_{h_T})_r", bend left] \\
            &  &                                                                 &  &                                                                                                  \\
            &  & \Presh(\scrC) \arrow[lluu, "j_T\inv"] \arrow[rruu, "j_{h_T}^*"] &  &
        \end{tikzcd}
    \]
    where $(j_{h_T})_r := (j_T)_*\circ e\inv $ is a right adjoint to $j_{h_T}^*$.

    Similarly, for any $T \to S$ in $\scrC$ we have $e \circ j_{T/S}\inv = j_{h_T/h_S}^* \circ e$, and we have an analogous diagram.
\end{lem}

\begin{cor}\label{j lower shriek preserves limits}
    For any locally small $\scrC$ and $T \in \scrC$, the functor $(j_T)_!$ preserves fiber products.
\end{cor}
\begin{proof}
    Immediate from above since the same is true of $j_{h_T}$. Alternatively one can proceed by direct computation using \cref{interpretation of j_!}. See also \cite[\href{https://stacks.math.columbia.edu/tag/04BB}{Tag 04BB}]{stacks-project}.
\end{proof}

\begin{lem}\label{relative j lower shriek cartesian diagram}
    Let $X\in \Presh(\scrC_T)$ and $T' \to T$ in $\scrC$. Given a subpresheaf $U \subset X$, we have a Cartesian diagram
    \[
        \begin{tikzcd}
            (j_{T'/T})_!(U_{T'}) \arrow[r] \arrow[d, hook] & U \arrow[d, hook] \\
            (j_{T'/T})_!(X_{T'}) \arrow[r]                   & X.
        \end{tikzcd}
    \]
\end{lem}
\begin{proof}
    Immediate from \cref{pullbakcs of monomorphisms lemma} and the explicit description given in \cref{interpretation of j_!}.
\end{proof}

%% file: Chapters/Category-Theoretic_Prerequisites/Sites_and_Sheaves.tex
\section{Sites and Sheaves}\label{Sites and Sheaves}
We refer the reader to \cite{Mac_Lane_Moerdijk_1994} for details on the following definitions and results.

Let us first recall the definitions of a site and a sheaf on a site for future reference.
\begin{dfn}[Sieves and Sites]
    Let $\scrC$ be a small category.
    \begin{enumerate}
        \item    A \emph{sieve} $S$ on an object $T$ of $\scrC$ is a subobject of the presheaf $h_T$.

        \item    A \emph{Grothendieck topology} $J$ on $\scrC$ is an assignment to each object $T \in \scrC$ a collection $J(T)$ of sieves on $T$, called \emph{covering sieves}, such that the following three axioms are satisfied:
              \begin{enumerate}
                  \item \emph{Maximality (Identity):} For any object $T$, the maximal sieve $h_T$ belongs to $J(T)$.
                  \item \emph{Stability (Base Change):} If $S \in J(T)$ and $h: T' \to T$ is any morphism in $\scrC$, then the pullback sieve $h^*S = \{ g \mid h \circ g \in S \}$ belongs to $J(T')$.
                  \item \emph{Transitivity (Local Character):} Let $S \in J(T)$, and let $R$ be any sieve on $T$. If for every morphism $h: T' \to T$ in $S$, the pullback sieve $h^*R \in J(T')$, then $R \in J(T)$.
              \end{enumerate}
              The pair $(\scrC, J)$ is called a \emph{site}. We will usually denote the pair $(\scrC, J)$ simply by $\scrC$.
    \end{enumerate}
\end{dfn}

\begin{notation}
    We will often identify a sieve $S$ with the collection of morphisms $\sqcup_{T\in \scrC}S(T)$.
\end{notation}

\begin{dfn}[Covering families]
    Let $\scrC$ be a site, $T\in \scrC$, and $(f_i)_{i\in I}$ be a collection of morphisms in $\scrC$ with target $\scrC$.
    \begin{enumerate}
        \item We define the sieve generated by the $f_i$ to be the smallest sieve on $T$ containing the $f_i$.
        \item We say that $(f_i)_{i\in I}$ is a \emph{covering family}, or \emph{covers} $T$, if the sieve generated by the $f_i$ is a covering sieve.
    \end{enumerate}
\end{dfn}

\begin{dfn}
    Let $\scrC$ be a category with all fiber products. A \emph{Grothendieck pretopology} on $\scrC$ is an assignment to each object $U \in \scrC$ a collection of families of morphisms $\{U_i \to U \}_{i \in I}$, called \emph{covering families} or \emph{covers}, such that the following properties are satisfied:
    \begin{enumerate}
        \item (\emph{Isomorphisms}) If $V \xrightarrow{\sim} U$, then $\{V \to U\}$ is a cover.
        \item (\emph{Stability under pullback}) If $\{U_i \to U\}_{i\in I}$ is a cover, then for any $Y \mapsto U$, the fiber product $\{U_i \times_U Y \to Y\}_{i\in I}$ is a cover.
        \item (\emph{Local character}) If $\{U_i \to U\}_{i\in I}$ is a cover, and for every $i \in I$ we have a cover $\{V_{ij} \to U_i\}_{j\in J_i}$ then the composition $\{V_{ij} \to U\}_{i\in I,j\in J_i}$ is a cover.
    \end{enumerate}
\end{dfn}
\begin{rem}
    Given a Grothendieck pretopology on a category $\scrC$ with fiber products, there is a unique associated Grothendieck topology on $\scrC$ where the covering sieves are defined to be those sieves which contain covering families. See for example \cite{Mac_Lane_Moerdijk_1994}.
\end{rem}

\begin{dfn}
    Let $\scrC$ be a site and $T\in \scrC$. Let $S$ be a cover of $T$, and $X\in \Presh(\scrC)$ be a presheaf of $\scrC$.

    We denote by $\Match(S,X)$ the set of matching families for the cover - i.e. if $S$ consists of the cover $(T_i \to T)_{i \in I}$, then $\Match(S,X)$ is the set of families $(t_i \in X(T_i))_{i \in I}$ such that, if we denote by $p_1, p_2$ the natural restriction maps
    \[
        \begin{tikzcd}
            \prod_i X(T_i) \arrow[r, "p_2"'] \arrow[r, "p_1", shift left] & {\prod_{i,j}X(T_i \times_T T_j)}
        \end{tikzcd}
    \]
    then $p_1((t_i)_{i \in I}) = p_2((t_i)_{i \in I})$.
\end{dfn}

\begin{dfn}[Separated Presheaves and Sheaves]
    Let $\scrC$ be a site, and $F\in \Presh(\scrC)$ be a presheaf on $\scrC$.

    \begin{enumerate}
        \item     We say $F$ is \emph{separated} if, for every object $T \in \scrC$ and every cover $S$ of $T$, any matching family of elements for $S$ has at most one amalgamation in $F(T)$.

        \item     We say $F$ is a \emph{sheaf} if, for every object $T \in \scrC$ and every cover $S$ of $T$, any matching family of elements for $S$ has exactly one amalgamation in $F(T)$.
    \end{enumerate}
\end{dfn}

\begin{dfn}
    Let $\scrC$ and $\scrD$ be sites, and $f\colon \scrC \to \scrD$ be a functor. We say $f$ is \emph{cover-preserving} if for any covering family $\{U_i \to U\}_{i\in I}$ of $\scrC$ we have that $\{ f(U_i) \to f(U) \}_{i\in I}$ is a covering family of $\scrD$.
\end{dfn}
\begin{lem}
    If $f$ is cover preserving, then for any sheaf $X\in \Sh(\scrD)$ we have that the pullback $f\inv(X)$ is a sheaf on $\scrC$.
\end{lem}
\begin{proof}
    Indeed, following the definitions, matching families for $f\inv(X)$ induce matching families for $X$ and the unique amalgamation for $X$ gives an amalgamation for $f\inv(X)$.
\end{proof}

\begin{dfn}
    Recall that a morphism of presheaves $\phi\colon X \to Y$ is \emph{locally surjective} if for every $T\in \scrC$ and $t \in Y(T)$ there exists a cover $S$ of $T$ such that for every $T' \in S$ there exists $s \in X(T')$ such that $\phi(s) = t$.
\end{dfn}

\begin{dfn}[The $(-)^+$ operation]
    Let $\scrC$ be a site. Recall that given a presheaf $X\in \Presh(\scrC)$, the presheaf $X^+$ is defined by
    \begin{align*}
        X^+(T) = \varinjlim_{S \in J(T)}\Match(S, X)
    \end{align*}
    where $J(T)$ denotes the collection of covers of $T$ and $\Match(S,X)$ denotes the set of matching families for the cover.

\end{dfn}

\begin{lem}\label{properties of the plus operation}
    Let $\scrC$ be a site and $X\in \Presh(\scrC)$.
    \begin{enumerate}
        \item The presheaf $X^+$ is separated.
        \item If $X$ is separated, then $X^+$ is a sheaf.
        \item If $X$ is a sheaf, then $X^+ \simeq X$.
        \item The functor $(-)^+$ is exact (\cref{definitions from category theory}).
    \end{enumerate}
\end{lem}

\begin{dfn}[Sheafification]
    Recall that the forgetful functor from sheaves on a site $\scrC$ to presheaves on $\scrC$ has a left adjoint, called the sheafification functor, defined by $(-)^{++}$.
\end{dfn}

\begin{rem}
    It follows from the characterization in terms of the $(-)^+$ operation that sheafification is exact.
\end{rem}

\begin{prop}\label{locally surjective morphism of presheaves leades to isomorphism of sheaves}
    Let $\phi\colon X \hookrightarrow Y$ be a locally surjective embedding of presheaves. Then the induced morphism $X^{++} \to Y^{++}$ is an isomorphism.
\end{prop}
\begin{proof}
    First observe that the induced map $X^{++} \to Y^{++}$ is an embedding since sheafification is exact, and by \cref{a functor which preserves finite limits preserves monomorphisms}. It remains to show that $X^{++} \to Y^{++}$ is surjective; but this follows immediately by \cite[III.7,~Corollary 6]{Mac_Lane_Moerdijk_1994}.
\end{proof}

See \cite[\href{https://stacks.math.columbia.edu/tag/00YR}{Tag 00YR}]{stacks-project} for the proof of the following result:
\begin{lem}\label{sheafification and groups}
    Let $X$ be a presheaf of groups on a site $\scrC$. Then $X^{++}$ is naturally a sheaf of groups on $\scrC$. Moreover, sheafification gives a left adjoint to the forgetful functor from sheaves of groups to presheaves of groups.
\end{lem}


%% file: Chapters/Partial_Morphisms_and_Partial_Magmas/Main_Partial_Morphisms_and_Partial_Magmas.tex
\chapter{Partial Morphisms and Partial Magmas}\label{chapter: Partial Morphisms and Partial Magmas}
In this chapter we introduce partial morphisms, the internal partial hom functor, and partial magmas. The main results are an adjunction for the internal partial hom, and the existence of a universal embedding from presheaves of partial magmas to presheaves of groups.

In \cref{section:partial morphisms} we introduce partial morphisms and compatibility, and discuss composition of partial morphisms. In \cref{section: partial morphisms of presheaves}, we define the pullback of partial morphisms on presheaves of slice categories. We also give two definitions of the internal partial hom functor, and show that it satisfies an adjunction. In \cref{section: partial magmas} we introduce partial magmas and define the notions of cancellativity and (strong) associativity. In \cref{section: presheaves of partial magmas} we define the ``left translation'', and ``left translation by inverse'' maps associated to a cancellative presheaf of partial magmas, and also define the subsemigroup of partial morphisms generated by the left translations (\cref{definition of tX}), which will be the key object of consideration in the next chapter. In \cref{section: first universal maps to groups}, we show that every presheaf of partial magmas over a small category admits a universal morphism to a presheaf of groups.

\input{Chapters/Partial_Morphisms_and_Partial_Magmas/Partial_Morphisms.tex}

\input{Chapters/Partial_Morphisms_and_Partial_Magmas/Partial_Morphisms_of_Presheaves.tex}

\input{Chapters/Partial_Morphisms_and_Partial_Magmas/Partial_Magmas.tex}
\input{Chapters/Partial_Morphisms_and_Partial_Magmas/Presheaves_of_Partial_Magmas.tex}

\input{Chapters/Partial_Morphisms_and_Partial_Magmas/Universal_Morphisms_to_Groups.tex}

%% file: Chapters/Partial_Morphisms_and_Partial_Magmas/Partial_Morphisms.tex
\section{Partial Morphisms} \label{section:partial morphisms}


We introduce partial morphisms, and define composition and compatibility of partial morphisms.

\begin{dfn}[Partial morphisms and compatibility]\label{partial morphism in a category definition}
    Let $\mathscr{C}$ be a category, and $X, Y$ be objects in $\scrC$.
    \begin{enumerate}
        \item A \emph{partial morphism $\phi\colon X \dashrightarrow Y$} is a pair $(\phi, \dom\phi)$, where $\dom\phi$ is a subobject of $X$, called the \emph{domain} of the partial morphism, and $\phi\colon \dom\phi \to Y$ is a morphism in $\scrC$.  We will usually refer to the pair $(\phi,\dom\phi)$ simply by $\phi$ where $\dom\phi$ is clear from context.

        \item For any $X,Y \in \scrC$, we denote by $\rparhom(X,Y)$ the collection of partial morphisms from $X$ to $Y$, i.e.,
              \begin{align*}
                  \rparhom(X,Y) := \bigsqcup_{\substack{\text{Subobjects} \\ U\hookrightarrow X}}\Hom(U,Y).
              \end{align*}
              Observe that if $\scrC$ is small then $\rparhom(X,Y)$ is a set.

        \item If $\phi$ and $\psi$ are partial morphisms from $X$ to $Y$, we say $\phi$ and $\psi$ are \emph{compatible} and write $\phi \ssim \psi$ if the restrictions of the corresponding morphisms to $(\dom\phi \cap \dom\psi)$ agree.

        \item We say that a partial morphism $\phi\colon X \dashrightarrow Y$ is \emph{injective} or \emph{monic} if the morphism $\phi\colon \dom\phi \to Y$ is monic.

        \item Suppose $\scrC$ has fiber products. If $\phi\colon X \dashrightarrow Y$ and $\psi\colon Y \dashrightarrow Z$ are partial morphisms, we define the composition $\psi\circ\phi\colon X\dashrightarrow Z$ to be the partial morphism $(\psi\circ\phi, V)$, where the domain $V$ is given by the Cartesian square:
              \[\begin{tikzcd}
                      V \arrow[r] \arrow[d,hook] & \dom\psi \arrow[d,hook] \\
                      \dom\phi \arrow[r,"\phi"] & Y
                  \end{tikzcd}\]
              In other words, $\psi\circ\phi$ is the composition of morphisms $V \xrightarrow{\phi|_V} \dom \psi \xrightarrow{\psi} Z$.

              We will make use of \cref{Yoneda Notation} and \cref{restrictions of morphisms} when referring to composition of partial morphisms, so the composition $\psi\circ\phi$ may also be denoted by $\psi(\phi)$ or $\psi|_X$ and denoted ``the restriction of $\psi$ to $X$''.

        \item Suppose $\scrC$ has fiber products. We define the \emph{graph} of a partial morphism $\phi\colon X \dashrightarrow Y$ to be the graph of the morphism $\phi\colon \dom \phi \to Y$.
    \end{enumerate}
\end{dfn}

\begin{rem}
    \begin{enumerate}
        \item  Observe that if $\scrC$ is a topos (for example, $\scrC = \textit{Set}$ or $\scrC = \Presh(\scrD)$; see \cite{Mac_Lane_Moerdijk_1994} for the definition of a topos), and $\phi$ and $\psi$ are compatible, then by \cref{Intersections and unions of subobjects exist in topoi} they ``glue'' to give a partial morphism with domain $\dom \phi \cup \dom \psi$.

        \item  Note that compatibility ($\ssim$) is not an equivalence relation in general; however, we will prove that it becomes an equivalence relation when restricted to a particular collection of partial morphisms we care about.
    \end{enumerate}
\end{rem}

\begin{ex}
    \begin{enumerate}
        \item Any morphism in $\scrC$ is a partial morphism in $\scrC$.
        \item Let $\scrC$ be a category with an initial object which is a subobject of any other object (e.g. the category of groups). Then for any $X,Y \in \scrC$, there is a partial morphism from $X$ to $Y$ given by the morphism from the initial object to $Y$.
    \end{enumerate}
\end{ex}

Composition of partial morphisms is associative:
\begin{lem} \label{Composition of partial morphisms is associative.}
    Let $\scrC$ be a category with fiber products. Then for any partial morphisms $\phi_1\colon W \dashrightarrow X$, $\phi_2\colon X \dashrightarrow Y$, $\phi_3\colon Y \dashrightarrow Z$ in $\scrC$, we have $(\phi_1\circ \phi_2) \circ \phi_3 = \phi_1 \circ (\phi_2\circ \phi_3)$.
\end{lem}
\begin{proof}
    Since composition of morphisms is associative, it is enough to show that $\dom(\phi_1\circ \phi_2) \circ \phi_3 = \dom \phi_1 \circ (\phi_2\circ \phi_3)$ as subobjects of $X$. But by definition we have
    \begin{align*}
        \dom(\phi_1\circ \phi_2) \circ \phi_3 = (\dom \phi_1 \times_X \dom \phi_2)\times_Y \dom \phi_3, \\
        \dom\phi_1\circ (\phi_2 \circ \phi_3) = \dom \phi_1 \times_X (\dom \phi_2\times_Y \dom \phi_3),
    \end{align*}
    so we conclude by associativity of the fiber product (\cref{associativity of fiber product}).
\end{proof}

The fact that composition is associative, and that for any $\phi\colon X \dashrightarrow Y$ we have $\phi\circ\id_X = \id_Y \circ \phi$, implies that the collection of objects of $\scrC$ with partial morphisms gives a well-defined category:
\begin{dfn}[$\Par(\scrC)$]
    Let $\scrC$ be a category with fiber products. We denote by $\Par(\scrC)$ the category whose objects are those of $\scrC$ and where morphisms are partial morphisms.
\end{dfn}

Let us observe also that compatibility respects composition of partial morphisms:
\begin{lem}\label{compatibility respects composition}
    Given partial morphisms $\phi,\phi'\colon X \dashrightarrow Y$ and $\psi,\psi'\colon Y \dashrightarrow Z$, suppose $\phi \ssim \phi'$ and $\psi \ssim \psi'$. Then $\phi\circ\psi \ssim \phi'\circ\psi'$.
\end{lem}
\begin{proof}
    Take any $T \in \scrC$ and $x \in (\dom \phi\circ\psi \cap \dom \phi'\circ\psi')(T)$. Then in particular $x \in (\dom \psi \cap \dom\psi')(T)$ and $\psi(x) = \psi'(x)$. Let $y = \psi(x)$. Then $y \in (\dom \phi \cap \dom \phi')(T)$, so
    \begin{align*}
        \phi\circ\psi(x) = \phi(y) = \phi'(y) = \phi'\circ \psi'(x).
    \end{align*}
\end{proof}

\begin{rem}\label{Yoneda induces an embedding of partial morphisms}
    Observe that, since the Yoneda embedding preserves monomorphisms (\cref{Yoneda is exact}, \cref{a functor which preserves finite limits preserves monomorphisms}), for any partial morphism $\phi\colon X \dashrightarrow Y$ with domain $U$ the Yoneda embedding induces a partial morphism $h_X \dashrightarrow h_Y$ with domain $h_U$, where for any $t \in U(T)$ we have $\phi(t) = \phi\circ t$. In particular, the Yoneda embedding induces an embedding from $\Par(\scrC)$ to $\Par(\Presh(\scrC))$; note that this will not be fully faithful in general. Note also that the Yoneda embedding preserves the compatibility relation between partial morphisms.
\end{rem}

%% file: Chapters/Partial_Morphisms_and_Partial_Magmas/Partial_Morphisms_of_Presheaves.tex
\section{Partial Morphisms of Presheaves}\label{section: partial morphisms of presheaves}


We introduce the pullback of a partial morphism of presheaves on a slice category, and define the ``inner partial hom'' presheaf.


\begin{notation}
    If $\phi\colon X \dashrightarrow Y$ is a partial morphism of presheaves on $\scrC$, we will write $\phi_T$ for the corresponding partial maps $X(T) \dashrightarrow Y(T)$ for each $T$.
\end{notation}

First, for any morphism $T' \to T$ and $X\in \Presh(\scrC)$, recall the definitions of $j_T\inv$, $j_{T'/T}\inv$ and $X_T$ from \cref{pullback of presheaves on slice categories}. We can extend the pullback functors to functors on the categories $\Par(\Presh(\scrC))$, $\Par(\Presh(\scrC_T))$ of presheaves with partial morphisms:
\begin{dfn}[Pullback of a partial morphism of presheaves on a slice category]\label{pullback of a partial morphism of presheaves on a slice category}
    Let $\scrC$ be a category, $T \in \scrC$, $X,Y \in \Presh(\scrC)$, and $\phi\colon X \dashrightarrow Y$ a partial morphism in $\Presh(\scrC)$. Note that $(\dom \phi)_T$ is a subobject of $X_T$. We define the pullback $j_{T}\inv(\phi)$ to be the pullback of the morphism $\phi|_{\dom \phi}$; i.e., $j_{T}\inv(\phi) := j_{T}\inv(\phi|_{\dom \phi})$.

    Similarly, if $T' \to T$ is a morphism in $\scrC$, for any $\phi\colon X \dashrightarrow Y$ on $\Presh(\scrC_T)$, we define the pullback $j_{T'/T}\inv(\phi)$ by $j_{T'/T}\inv(\phi) := j_{T'/T}\inv(\phi|_{\dom \phi})$.
\end{dfn}
\begin{rem}\label{remarks on pullback of partial morphisms on slice categories}
    \begin{enumerate}
        \item     Note that $j_{T'/T}\inv$ is induced by the definition of $j_T\inv$ via \cref{a slice of a slice is a slice}.

        \item     For any partial morphism $\phi\colon X \dashrightarrow Y$, it follows from the definitions that we have $j_{T'}\inv = j_{T'/T}\inv \circ j_{T}\inv$. Similarly for any $T \to S$ we have $j_{T'/S}\inv = j_{T'/T}\inv \circ j_{T/S}\inv$.
    \end{enumerate}
\end{rem}
Composition of partial morphisms commutes with pullback:
\begin{lem}\label{composition of partial morphisms commutes with pullback}
    Let $T\in \scrC$. Then for any partial morphisms $\phi\colon X \dashrightarrow Y$, $\psi\colon Y \dashrightarrow Z$ in $\Presh(\scrC)$, we have $j_T\inv(\psi\circ \phi) = j_T\inv(\psi)\circ j_T\inv(\phi)$.

    In particular, the pullback $j_T\inv$ induces a well-defined functor from $\Par(\scrC_T)$ to $\Par(\scrC_{T'})$.
\end{lem}
\begin{proof}
    Since composition commutes with pullback of morphisms, it is enough to verify that $\dom j_T\inv(\psi\circ\phi) = \dom j_T\inv(\psi)\circ j_T\inv(\phi)$. By definition, and since pullback preserves limits (\cref{presheaf pullback preserves limits}), we have
    \begin{align*}
        \dom j_T\inv(\psi\circ\phi) & = j_T\inv(\dom \psi\circ \phi) = j_T\inv( \dom \psi \times_Y \dom \phi)                               \\
                                    & = j_T\inv(\dom \psi) \times_{j_T\inv(Y)} j_T\inv(\dom \phi) = \dom j_T\inv(\psi) \circ j_T\inv(\phi).
    \end{align*}
    We conclude that $j_T\inv$ induces a well-defined functor.
\end{proof}

Next we introduce the ``inner partial hom'' between two presheaves. Recall \cref{internal Hom definition} and \cref{internal hom equivalence}:
\begin{dfn}[Inner partial hom]\label{Parhomdefinition}
    Let $\scrC$ be a small category, and $X,Y \in \Presh(\scrC)$. We denote by $\parhom(X,Y)$, or the \emph{inner partial hom of $X$ and \(Y\)} as follows:
    \begin{align*}
        \parhom(X,Y)(T) = \rparhom_{\Presh(\scrC)}(h_T \times X,Y),
    \end{align*}
    where restriction maps for $T' \to T$ are given by precomposition of partial morphisms with the induced map $X\times h_{T'} \to X\times h_T$. Equivalently, if $\phi$ has domain $U$, the restriction of $\phi$ is given by precomposition (of morphisms) of the top arrow in the induced Cartesian square:
    \[
        \begin{tikzcd}
            U' \arrow[r] \arrow[d]   & U \arrow[d] \\
            X\times h_{T'} \arrow[r] & X\times h_T
        \end{tikzcd}
    \]
\end{dfn}
\begin{rem}
    Observe that this gives a well-defined presheaf by associativity of composition of partial morphisms (\cref{Composition of partial morphisms is associative.}). Observe also that the construction is natural in $X$ and $Y$, and therefore we get a bifunctor $\parhom(-,-)$.
\end{rem}

The next result shows that the adjunction of \cref{The category of presheaves is Cartesian closed} extends to partial morphisms:
\begin{prop}\label{partial hom adjunction}
    Let $\scrC$ be a small category, and $X,Y,Z \in \Presh(\scrC)$. We have isomorphisms
    \begin{align*}
        \lambda\colon  \rparhom(X\times Z, Y) \simeq \Hom(X,\parhom(Z,Y))
    \end{align*}
    which are natural in $X$,$Y$ and $Z$. The isomorphism sends a partial morphism $\mu\colon X\times Z \dashrightarrow Y$ to the morphism $\lambda(\mu)$, which sends an element $x \in X(T)$ to the induced composition of partial morphisms
    \begin{align*}
        \lambda_x(\mu)\colon h_T \times Z \xrightarrow{(x\times \id_Z)} X\times Z \dashrightarrow Y,
    \end{align*}
    where we identify $x$ with the corresponding morphism given by the Yoneda embedding:
    \begin{align*}
        x\colon h_T \to X; \qquad f \in h_T(T') \mapsto X(f)(x).
    \end{align*}

    Explicitly, the domain of $\lambda_x(\mu)$ is the preimage of $\dom\mu$ under $(x,\id_Z)$, i.e., for any $T'$,
    \begin{align*}
        (\dom \lambda_x(\mu))(T') = \{ (f,z) \in (h_T\times Z)(T') : (X(f)(x),z) \in (\dom \mu) (T) \}
    \end{align*}
    and for $(f,z) \in (\dom\lambda_x(\mu))(T')$ we have $\lambda_x(\mu)(f,z) = \mu(X(f)(x),z)$.
\end{prop}
\begin{proof}
    Let us first verify that each $\lambda(\mu)$ is a morphism of presheaves; indeed, for any $T' \to T$ and $x \in X(T)$, the restriction $x|_{T'}$ corresponds to the composition $h_{T'} \to h_T \xrightarrow{x} X$, so $\lambda_{(x|_{T'})}(\mu)$ is the composition
    \begin{align*}
        h_{T'} \times Z \to h_T \times Z \xrightarrow{(x\times \id_Z)} X\times Z \dashrightarrow Y,
    \end{align*}
    and this coincides with the restriction $(\lambda_x(\mu))|_{T'}$ by definition of $\parhom$.

    Naturality in $Y$ is immediate. For naturality in $X$ and $Z$, given any $X'\to X$ and $Z' \to Z$ observe that precomposing on the left and applying $\lambda$ gives the map which sends $x \in X'(T)$ to the composition
    \begin{align*}
        h_{T} \times Z' \xrightarrow{(x\times \id_Z)} X'\times Z' \to X \times Z \dashrightarrow Y,
    \end{align*}
    and this coincides with applying $\lambda$ and precomposing on the right.

    Finally we construct an inverse to $\lambda$. For this, we define the ``partial evaluation map'' of presheaves
    \begin{align*}
        \ev_{Y,Z}\colon \parhom(Z,Y) \times Z \dashrightarrow Y\quad  (\phi,z) \mapsto \phi(\id_T,z).
    \end{align*}
    Explicitly, the domain of $\ev_{Y,Z}$ is given by
    \begin{align*}
        (\dom \ev_{Y,Z}) (T) = \{(\phi,z) \in (\parhom(Z,Y) \times Z)(T) : (\id_T,z) \in (\dom \phi)(T) \}.
    \end{align*}
    Observe that this is a well-defined partial morphism of presheaves; indeed for any $T' \to T$ and $(\phi,z) \in (\parhom(Z,Y) \times Z)(T)$, we have
    \begin{align*}
        \ev_{Y,Z}((\phi,z)|_{T'}) & = \phi|_{T'}(z|_{T'},\id_{T'})                                          \\
                                  & = \phi(z|_{T'}, T' \to T) \tag{definition of restriction in $\parhom$}  \\
                                  & = (\phi(z,\id_T))|_{T'}. \tag{since $\phi$ is a morphism of presheaves}
    \end{align*}
    Then we define the inverse $\beta$ to $\lambda$ as follows: to any $\theta \in \Hom(X,\parhom(Z,Y))$ we associate the composition of partial morphisms
    \[
        \begin{tikzcd}
            \beta(\theta)\colon X\times Z \arrow[r,"\theta \times \id_Z"]& \parhom(Z,Y) \times Z \arrow[r,dashed,"\ev_{Y,Z}"] & Y.
        \end{tikzcd}
    \]
    Let us verify that $\beta$ does indeed give an inverse. For any $x \in X(T)$ and $\theta \in \Hom(X,\parhom(Z,Y))$ we have that $\lambda_x(\beta(\theta))$ is the precomposition of the above with $h_T\times Z \xrightarrow{x\times \id_Z} X\times Z$. This sends pairs $(T' \to T,z)$ in the domain to $\theta(x|_{T'})(z,\id_{T'})$, and since $\theta$ is a morphism of presheaves the latter is equal to $\theta(x)(z,T' \to T)$. By inspection we see also that $\dom \lambda_x(\beta(\theta)) = \dom \theta(x)$. We conclude $\lambda_x(\beta(\theta)) = \theta(x)$, and hence $\lambda\circ\beta(\theta) = \theta$.

    For any $\mu \colon X\times Z \dashrightarrow Y$ one shows by a similar computation that $\beta(\lambda(\mu)) = \mu$; this concludes the proof.
\end{proof}
\begin{rem}\label{adjoint partial hom size issues}
    The only reason we need $\scrC$ to be small in the above result is that so that both sides of the equation are sets; even if $\scrC$ is not small, we still get a natural correspondence.
\end{rem}

\begin{dfn}[Equivalent Formulation of Inner Partial Hom]
    Let $\scrC$ be a small category, and $X,Y \in \Presh(\scrC)$. We define the presheaf $\parhom'(X,Y)$ by defining
    \begin{align*}
        \parhom'(X,Y)(T) = \rparhom_{\Presh(\scrC_T)}(X_T, Y_T),
    \end{align*}
    for each $T\in \scrC$, where restriction maps are given by pullback of partial morphisms of presheaves.
\end{dfn}
\begin{rem}
    Observe that the above mapping gives a well-defined presheaf since taking restrictions commutes with composition by \cref{remarks on pullback of partial morphisms on slice categories}. Observe also that the construction is natural in $X$ and $Y$, and therefore we get a bifunctor $\parhom'(-,-)$
\end{rem}

\begin{prop}[{cf. \cref{internal hom equivalence}}]\label{parhomequivalence}
    We have a natural isomorphism of bifunctors
    \begin{align*}
        \alpha\colon \parhom(-,-) \simeq \parhom'(-,-).
    \end{align*}
    Moreover, for any $\mu\colon X\times Z \dashrightarrow Y$, $T \in \scrC$ and $x \in X(T)$, the image under $\alpha$ of $\lambda_x(\mu)$ from \cref{partial hom adjunction} is the partial morphism from $Z_T$ to $Y_T$ sending elements $z \in Z_T(T' \to T)$ to $\mu(x|_{T'}, z)$. Explicitly, the domain is given by
    \begin{align*}
        (\dom \lambda_x(\mu))(T' \to T) = \{ z \in Z_T(T' \to T) : (x|_{T'},z) \in (\dom \mu)(T')\}.
    \end{align*}

\end{prop}
\begin{proof}
    We need to show
    \begin{align} \label{goal equation for equivalence of partial homs definitions}
        \rparhom_{\Presh(\scrC)}(X\times h_T, Y) \simeq \rparhom_{\Presh(\scrC_T)}(X_T,Y_T)
    \end{align}
    naturally in $X$, $Y$ and $T$.

    Recall \cref{pullback of presheaves on slice categories} and \cref{Presheaves on Slice categories}, and note that $X\times h_T = (j_{T})_!(X_T)$.

    First, for any presheaf $X$, observe that $(j_T)_!$ induces an equivalence of categories between $\Sub(X_T)$ and $\Sub(X \times h_T)$; indeed the functor $j_{h_T}$ induces an equivalence of categories of subobjects between $\Sub(X_{h_T})$ and $\Sub(X\times h_T)$, and we get the claim by precomposing with the equivalence $e$. In particular, every subobject of $X\times h_T$ is of the form $(j_T)_!(U)$ for some $U \hookrightarrow X_T$.

    Next, for any subobject $(j_T)_!(U) \hookrightarrow X \times h_T$, by \cref{pullback of presheaves on slice categories} and \cref{precomposition on the left commutes with pullback on the right} we have an isomorphism given by the adjunction
    \begin{align}\label{given adjunction for equivalence of partial homs}
        \alpha \colon \Hom_\scrC((j_T)_!(U),Y) \xrightarrow{\sim} \Hom_{\Presh(\scrC_T)}(U,Y_T),
    \end{align}
    which is natural $T$, $U$ and $Y$, such that for any $\phi \in \Hom_\scrC((j_T)_!(U),Y)$ and $u \in U(T')$ we have $\alpha(\phi)(u) = \phi(u)$. This immediately gives that \cref{goal equation for equivalence of partial homs definitions} is natural in $Y$.

    For any $T' \to T$, by \cref{relative j lower shriek cartesian diagram} and \cref{j lower shriek preserves limits} we have a Cartesian square
    \[
        \begin{tikzcd}
            (j_{T'})_!(U_{T'}) \arrow[r] \arrow[d]   & (j_T)_!(U)\arrow[d] \\
            X\times h_{T'} \arrow[r] & X\times h_T
        \end{tikzcd}
    \]
    so by definition restriction of partial morphisms along $T \to T'$ on the left side of \cref{goal equation for equivalence of partial homs definitions} agrees with precomposition (of morphisms) by the canonical map $(j_{T'})_!(U_{T'}) \to (j_T)_!(U)$; since \cref{given adjunction for equivalence of partial homs} is natural in $T$ we conclude \cref{goal equation for equivalence of partial homs definitions} is natural in $T$.

    Finally, for any $X' \to X$ and subobject $(j_T)_!(U) \hookrightarrow X\times h_T$, if $U'$ is the pullback of $U \hookrightarrow X_T$ to $X_{T'}$, then by \cref{j lower shriek preserves limits} we have a Cartesian diagram
    \[
        \begin{tikzcd}
            (j_T)_!(U') \ar[r] \ar[d,hook] & (j_T)_!(U) \ar[d, hook]\\
            X'\times h_T \ar[r]& X\times h_T
        \end{tikzcd}
    \]
    so in particular the precomposition of partial morphisms on the left side of \cref{goal equation for equivalence of partial homs definitions} arises from the precomposition by $(j_T)_!(U') \to (j_T)_!(U)$; since the adjunction \cref{given adjunction for equivalence of partial homs} is natural in $U$ we conclude that \cref{goal equation for equivalence of partial homs definitions} is natural in $X$.

    Finally, for any $\mu\colon X\times Z \dashrightarrow Y$, and $x \in X(T)$, the fact that the image of $\lambda_x(\mu)$ from \cref{partial hom adjunction} under the equivalence is as claimed follows from the remarks under \cref{given adjunction for equivalence of partial homs} and the definition of $\lambda_x(\mu)$.
\end{proof}
\begin{notation}
    Going forward, we will also use $\lambda$ to denote the induced adjunction
    \begin{align*}
        \lambda\colon  \rparhom(X\times Z, Y) \simeq \Hom(X,\parhom'(Z,Y)).
    \end{align*}
\end{notation}

Just as in the internal hom case, $\parhom(X,X)$ is a semigroup:
\begin{cor}\label{parhom is a presheaf of semigroups}
    For any $X \in \scrC$, composition of partial morphisms endows the presheaf $\parhom(X,X)$ with the structure of a presheaf of semigroups.
\end{cor}
\begin{proof}
    Immediate, since composition of partial morphisms is associative (\cref{Composition of partial morphisms is associative.}) and pullback respects composition (\cref{composition of partial morphisms commutes with pullback}).
\end{proof}



%% file: Chapters/Partial_Morphisms_and_Partial_Magmas/Partial_Magmas.tex
\section{Partial Magmas}\label{section: partial magmas}
For this section we fix a category $\scrC$ with products (necessary to define partial binary operations) and fiber products (necessary for composition of partial morphisms, and to define graphs).

\begin{dfn}\label{partial binary operations in a category definition}
    \begin{enumerate}
        \item   A \emph{partial binary operation} on an object $X\in \scrC$ is a partial morphism $\mu\colon X \times X \dashrightarrow X$.

        \item  A \emph{partial magma} on $\scrC$ is a pair $(X,\mu)$, where $X\in \scrC$ and $\mu$ is a partial binary operation on $X$.

        \item   A \emph{morphism of partial magmas} $\phi\colon (X,\mu) \to (X',\mu')$ is a morphism $\phi\colon X\to X'$ in $\scrC$ which preserves the partial binary operation; i.e., the restriction of $\phi\times \phi$ to $\dom \mu$ factors through $\dom \mu'$, and we have $\mu'((\phi\times \phi)|_U) = \phi(\mu)$.
    \end{enumerate}
    We will usually refer to the pair $(X,\mu)$ simply by $X$ where $\mu$ is clear from context.
\end{dfn}

\begin{ex}
    \begin{enumerate}
        \item  A group object in $\scrC$ is certainly a partial magma.

        \item Any object $X\in \Presh(\scrC)$ can be endowed with the structure of a magma by defining the partial operation $X\times X \dashrightarrow X$ to be the projection to the first factor.

        \item Let $\scrC$ be a category with an initial object which is a subobject of any other object (e.g. the category of groups). Then, for any $X\in\scrC$, the monomorphism from the initial object $e$ to $X\times X$, together with the map $e\to X$, define a partial morphism $X\times X\dashrightarrow X$, which gives the structure of a partial magma on $X$.
    \end{enumerate}
\end{ex}

The above examples show that the notion of a presheaf of partial magmas is very general, and that one should not hope to be able to say much about them without imposing extra conditions.

\begin{dfn}[Associativity]\label{associativity definition}
    Let $(X,\mu)$ be a partial magma on $\scrC$. We say $\mu$ is \emph{associative} if we have
    \begin{align*}
        \mu \circ (id_X \times \mu) \ssim \mu \circ (\mu \times \id_X),
    \end{align*}
    where $\ssim$ is defined as in \cref{partial morphism in a category definition}.

    Using the notation of \cref{products and Yoneda combine}, if we denote by $x_i\colon X^3 \to X$ the projection to the $i^{\text{th}}$ factor for $i = 1,2,3$, then $(x_1,x_2,x_3) = \id_{X^3}$ and we have
    \begin{align*}
        \mu(x_1,\mu(x_2,x_3)) \ssim \mu(\mu(x_1,x_2),x_3),
    \end{align*}
    We say $(X,\mu)$ is associative if $\mu$ is.
\end{dfn}

\begin{dfn}[Cancellative Partial Magmas]\label{cancellativity definition}
    Let \((X,\mu)\) be a partial magma in $\scrC$. Let \(X_i = X\) for \(i = 1,2,3\), and \(\Gamma_\mu \subset X^3\) be the graph of \(\mu\).

    We say that $\mu$ is \emph{cancellative} if for each permutation $(i,j,k)$ of $\{1,2,3\}$ with \(i<j\), the projection \(s_{ij}\colon \Gamma_\mu \to X_i \times X_j\) factors through an isomorphism with a subobject $U_{ij}$ of the product $X_i\times X_j$; i.e., $\rho_{ij}$ realizes \(\Gamma_\mu\) as the graph of a partial morphism \(\mu_{ij}\colon X_i\times X_j \dashrightarrow X_k\) with domain $U_{ij}$, where $\mu = \mu_{12}$.

    We say that $(X,\mu)$ is cancellative if $\mu$ is.
\end{dfn}

\begin{dfn}[Strong Associativity]\label{strong associativity definition}
    Let $(X,\mu)$ be a cancellative partial magma in $\scrC$. We say $\mu$ is \emph{strongly associative} if it is associative and if in addition we have that if $(x_1,x_2,x_3) = \id_{X^3}$ then
    \begin{align*}
        \mu_{13}(x_1,\mu(x_2,x_3)) \ssim \mu(\mu_{13}(x_1,x_2),x_3).
    \end{align*}

    We say $(X,\mu)$ is associative if $\mu$ is.
\end{dfn}

\begin{lem}\label{associativity and cancellativity preserved by passing to subobjects.}
    Let $(X,\mu)$ be a partial magma in $\scrC$, and $Y \hookrightarrow X$ be a subobject. If $\mu$ is cancellative or (strongly) associative then so is the partial binary operation on $Y$ given by restricting the graph of $\mu$ to $Y^3$.
\end{lem}
\begin{proof}
    It follows immediately from the definition that the associativity conditions are preserved when passing to subobjects; for cancellativity note that pulling back to $Y$ preserves monomorphisms and isomorphisms, so if the relevant projections factor as isomorphisms followed by monomorphisms this remains true after pullback.
\end{proof}

%% file: Chapters/Partial_Morphisms_and_Partial_Magmas/Presheaves_of_Partial_Magmas.tex
\section{Presheaves of Partial Magmas}\label{section: presheaves of partial magmas}
We introduce presheaves of partial magmas, which are presheaves with a partial binary operation. We discuss the associated left/right translation maps in this setting, and we verify that the associativity and cancellation properties can be verified pointwise.


First we observe that the Yoneda embedding preserves partial magmas:
\begin{lem}\label{Yoneda preserves partial magmas}
    Let $\scrC$ be a category with products and fiber products and let $(X,\mu)$ be a partial magma on $\scrC$. Then $(h_X,h_\mu)$ is a partial magma on $\Presh(\scrC)$. Moreover, if $\mu$ is cancellative, associative, or strongly associative, respectively, then so is $h_\mu$.
\end{lem}
\begin{proof}
    Follows from the fact that the Yoneda embedding preserves products, subobjects, partial morphisms and compatibility.
\end{proof}

\begin{rem}\label{mu is well be haved on presheaves of partial magmas}
    Let $(X,\mu)$ be a presheaf of partial magmas on $\scrC$. Let $T \in \scrC$, and $(a,b) \in X^2(T)$. By the Yoneda lemma, the pair $(a,b)$ corresponds to a map $(a',b')\colon h_T \to X^2$. Then the composition $\mu(a',b')$ is a partial morphism from $h_T$ to $X$. Observe that we have $(a,b) \in \dom \mu(T)$ if and only if $\dom \mu(a',b') = h_T$, and if this holds then the corresponding element of $X(T)$ is $\mu(a,b)$.
\end{rem}

\begin{notation}\label{lambda notation}
    Let $(X,\mu)$ be a presheaf of partial magmas on a small category $\scrC$. Recall the definition of $X_T$ from \cref{pullback of presheaves on slice categories}.

    \begin{enumerate}
        \item   Following \cref{presheaf notation}, for any $x \in X(T)$ and $y \in X_T(T' \to T)$ such that $(x|_{T'},y) \in \dom \mu(T')$, or $(y,x|_{T'}) \in \dom \mu(T')$, we will usually write $\mu(x,y)$ for $\mu(x|_{T'},y)$ or $\mu(y,x)$ for $\mu(y,x|_{T'})$, as appropriate.

        \item  By \cref{partial hom adjunction} and \cref{parhomequivalence}, the binary operation $\mu$ induces a morphism $\lambda(\mu)\colon X \to \parhom(X,X)$ such that for any $x \in X(T)$ and $y \in X_T(T' \to T)$ with $(x|_{T'},y) \in \dom \mu(T')$ we have $\lambda_x(\mu)(y) = \mu(x,y)$. Where $\mu$ is clear from context, we will usually write $\lambda_x$ for $\lambda_x(\mu)$, and $x(y)$ for $\lambda_x(y)$.

        \item Similarly, if $\mu$ is cancellative, the induced binary operation $\mu_{13}$ induces a morphism from $X$ to $\parhom(X,X)$, which we will denote by $\lambda^\dagger(\mu)$, such that for any $x \in X(T)$ and $y \in X_T(T' \to T)$ such that $(x|_{T'},y) \in \dom \mu_{13}(T')$ we have $\lambda_x^\dagger(\mu)(y) = \mu(x,y)$. Again, if $\mu$ is clear from context, we write $\lambda^\dagger_x$ for $\lambda_x(\mu)$, and $x^\dagger(y)$ for $\lambda^\dagger_x(y)$.
    \end{enumerate}
\end{notation}
\begin{rem}
    One could of course define \emph{left cancellativity} and define $\lambda^\dagger(\mu)$ under this weaker assumption; we will not worry about this since we will also need right cancellativity for our purposes, for example in \cref{cancellation on covers}.
\end{rem}

\begin{rem}
    Of course, one can still define the partial morphisms $\lambda_x, \lambda^\dagger_x$ if $\scrC$ is not small, see \cref{adjoint partial hom size issues}. We will not worry about this.
\end{rem}

\begin{rem}\label{dagger behaves linke an inverse}
    Observe that, if $(X,\mu)$ is cancellative, then for $z \in \dom \lambda_x(T')$ we have $x(z) \in \dom \lambda_x^\dagger$ and $x^\dagger(x(z)) = z$.
\end{rem}

In what follows we frequently use the symbol \((-)^\dagger\) where the reader should informally imagine an inverse \((-)\inv\). The next result justifies this intuition:
\begin{lem}\label{dagger maps to inverse}
    Let \(\phi\colon (X,\mu) \to (G,\cdot)\) be a morphism of presheaves of partial magmas on \(\scrC\), where \((G,\cdot)\) is a presheaf of groups, and let \(x,z \in X(T)\). If \(z \in \dom \lambda_x\) then \(\phi(x(z)) = \phi(x)\cdot \phi(z)\). Similarly, if \(z \in \dom \lambda_x^\dagger\) then \(\phi(x^\dagger(z)) = \phi(x)\inv \cdot \phi(z)\).
\end{lem}
\begin{proof}
    The first statement is immediate since \(\phi\) respects the partial multiplication by assumption.

    Suppose \(z \in \dom \lambda_x^\dagger\) and \(x^\dagger(z) = y\). Then by definition we have \(x(y) = z\) and therefore \(\phi(x)\cdot \phi(y) = \phi(z)\) by above. Therefore \(\phi(y) = \phi(x)\inv \cdot \phi(z)\) as required.
\end{proof}

\begin{lem}\label{associativity criterion}
    A partial magma $(X,\mu)$ of presheaves on $\scrC$ is associative if and only if (using the notation of \cref{lambda notation}) for any $a \in X(T)$ and $b \in (\dom \lambda_a)(T)$ we have
    \begin{align*}
        \lambda_{a(b)} \ssim \lambda_a\circ \lambda_b.
    \end{align*}

    Similarly, a cancellative presheaf of partial magmas $(X,\mu)$ is strongly associative if and only if it is associative and for any $a \in X(T)$ and $b \in (\dom \lambda^\dagger_a)(T)$ we have
    \begin{align*}
        \lambda_{a^\dagger(b)} \ssim \lambda^\dagger_a\circ \lambda_b.
    \end{align*}
\end{lem}

Now we give some definitions which we will later use in our main construction:
\begin{dfn}\label{definition of tX}
    Let \((X,\mu)\) be a cancellative presheaf of partial magmas.
    \begin{enumerate}
        \item Define the presheaf \(\Xoneinfty := \bigsqcup_{n \geq 1} (\Xone)^n\).

        \item Given $(x,e) \in (\Xone)(T)$, we write $\lambda_x^e$ to mean \(\lambda_{x}\) if \(e= 1\) and \(\lambda_{x}^\dagger\) if \(e = -1\).

        \item Define the morphism of presheaves
              \begin{align*}
                  \epsilon_\mu\colon & \Xoneinfty \to \parhom(X,X)                                                                    \\
                                     & ((x_1,e_1),\dots,(x_n,e_n)) \mapsto \lambda_{x_1}^{e_1} \circ \dots \circ \lambda_{x_n}^{e_n}.
              \end{align*}

        \item We denote by \(\tX\) the image of \(\Xone\) under \(\epsilon_\mu\). We denote by \(\tXinfty\) the image of \(\Xoneinfty\) under \(\epsilon_\mu\). Observe that \(\tX\) is naturally a subpresheaf of \(\tXinfty\).


    \end{enumerate}

\end{dfn}

\begin{rem}\label{comments on definition of tX}
    \begin{enumerate}
        \item Observe that the operation of concatenation of sequences endows \(\Xoneinfty\) with the structure of a presheaf of semigroups.

        \item Since \(\tXinfty\) is closed under composition, it follows from \cref{parhom is a presheaf of semigroups} that \(\tXinfty\) is a presheaf of semigroups. Moreover, the map \(\epsilon_\mu\colon \Xoneinfty \to \tXinfty\) is a morphism of presheaves of semigroups (certainly concatenation commutes with restriction).

        \item If \(\mu\) is associative (\cref{associativity definition}) then \(\lambda(\mu)\) is a morphism of presheaves of partial magmas.

        \item We will later show that, assuming the domain of \(\mu\) is large enough, the morphisms \(\epsilon_\mu\) (restricted to \(X\times \{\pm1\}\)) and \(\lambda(\mu)\) are embeddings.
    \end{enumerate}
\end{rem}

\begin{rem}\label{group operations and epsilon_mu}
    Let \(\phi\colon (X,\mu) \to (G,\cdot)\) be a morphism of presheaves of partial magmas on \(\scrC\), where \((G,\cdot)\) is a presheaf of groups. It follows immediately from \cref{dagger maps to inverse} that for any \((x,e)\in (\Xone)(T)\) and \(z \in \dom \lambda_{x}^{e}\) we have \(\phi(\epsilon_{\mu}(x,e)(z)) = \phi(x)^e \cdot \phi(z)\).

    In particular, for any sequence \(\sigma = ((x_1,e_1),\dots,(x_n,e_n)) \in \Xoneinfty\) and \(z \in \dom \epsilon_\mu(\sigma)\), we have
    \begin{align*}
        \phi(\epsilon_\mu(\sigma)(z)) = \phi(x_1)^{e_1}\dots\phi(x_n)^e_n \cdot \phi(z).
    \end{align*}
\end{rem}

\begin{dfn}\label{formaldagger}
    Let \((X,\mu)\) be a cancellative presheaf of partial magmas. Define the involution \((-)^\dagger\) on \(X \times \{\pm 1\}\) by
    \begin{align*}
        (x,e)^\dagger := (x, -e).
    \end{align*}
    We extend \((-)^\dagger\) to \((X\times \{\pm 1\})^{(\infty)}\) by acting coordinatewise and reversing the order; i.e.,
    \begin{align*}
        (\alpha_1,\dots,\alpha_n)^\dagger := (\alpha_n^\dagger,\dots,\alpha_1^\dagger).
    \end{align*}
\end{dfn}
\begin{rem}
    We will later show that, if \((X,\mu)\) is a group chunk, then \((-)^\dagger\) descends to a well-defined morphism on \(\tXinfty/\ssim\).
\end{rem}

\begin{prop}\label{formaldagger looks like inverse}
    Let \((X,\mu)\) be a cancellative presheaf of partial magmas, and \(\sigma \in \Xoneinfty(T)\). Let \(\phi = \epsilon_\mu(\sigma)\) and \(\phi' = \epsilon_\mu(\sigma^\dagger)\). Then for any \(y \in \dom \phi(T)\) we have \(y \in (\dom \phi' \circ \phi)(T)\), and \(\phi'\circ \phi(y) = y\).

    In particular, \(\phi' \circ \phi \ssim \id_{X_T}\).
\end{prop}
\begin{proof}
    We proceed by induction on the length of \(\sigma\).

    First suppose \(\sigma\) has length 1. If \(\sigma = ((x,1))\) with \(x \in X(T)\) then, following the notation of \cref{lambda notation}, we have \(\phi = \lambda_x\). Let \(y \in \dom x(T)\), and suppose \(x(y) = z\). Denote by \(\Gamma_\mu\) the graph of \(\mu\). Then by definition \((x,y,z)\in \Gamma_\mu(T')\), and in particular (by \cref{dagger behaves linke an inverse}) we have \(z\in \dom x^\dagger\) and \(x^\dagger(z) = \mu_{13}(x,z) = y\). The case \(\sigma = ((x,-1))\) is similar.

    For the general case, denote \(\psi = \epsilon_\mu(\alpha_2,\dots,\alpha_n)\), \(\psi' =  \epsilon_\mu((\alpha_2,\dots,\alpha_n)^\dagger)\). Let \(y \in \dom \phi\), and \(z = \psi(y)\). By assumption we have \(z\in \dom \epsilon_\mu(\alpha_1)\) and \(\epsilon_\mu(\alpha_1)(z) = \phi(y)\). By the base case we have \(\epsilon_\mu(\alpha_1)(z)\in \dom \epsilon_\mu(\alpha_1^\dagger)\) and \(\epsilon_\mu(\alpha_1^\dagger) \circ \epsilon_\mu(\alpha_1)(z) = z\). Therefore \(\phi(y)\in \dom \epsilon_\mu(\alpha_1^\dagger)\) and \(\epsilon_\mu(\alpha_1^\dagger) \circ \phi(y) = z\).

    By induction, we have \(z \in \dom \psi'\) and \(\psi'(z) = y\); therefore \(\epsilon_\mu(\alpha_1^\dagger) \circ \phi(y) \in \dom \psi'\) and \(\psi'\circ \epsilon_\mu(\alpha_1^\dagger) \circ \phi(y) = \phi'\circ\phi(y) = y\). Since \(\phi'= \psi'\circ \epsilon_\mu(\alpha_1^\dagger)\) we conclude.
\end{proof}

%% file: Chapters/Partial_Morphisms_and_Partial_Magmas/Universal_Morphisms_to_Groups.tex
\section{Universal Morphisms to Groups 1} \label{section: first universal maps to groups}
In this section we use Freyd's adjoint functor theorem (\cref{Freyd's adjoint functor theorem}) to show that, on a small category $\scrC$, the forgetful functor from presheaves of groups on $\scrC$ to presheaves of partial magmas on $\scrC$ has a left adjoint. This is equivalent to saying that every presheaf of partial magmas admits a universal morphism into a presheaf of groups (see \cref{universal property characterization of adjoint functors}).


We proceed to quickly verify the conditions for Freyd's adjoint functor theorem. Recall the definitions from \cref{definitions from category theory}.
\begin{lem}
    The category of presheaves of groups on $\scrC$ is complete.
\end{lem}
\begin{proof}
    Immediate since the category of groups is complete, and limits in functor categories exist and can be taken pointwise (\cite[Chapter~5]{Saunders_categories}).
\end{proof}

\begin{lem}
    The forgetful functor from the category of presheaves of groups to the category presheaves of partial magmas is continuous.
\end{lem}
\begin{proof}
    We have that the forgetful functor from groups to sets creates limits (as in \cite[Chapter~5]{Saunders_categories}) and similarly for the forgetful functor from magmas to sets.
    Therefore the forgetful functor from groups to partial magmas creates limits, and in particular preserves them. The result follows since limits on functor categories can be computed pointwise (as in \cite[Chapter~5]{Saunders_categories}).
\end{proof}


We will use the following notion in order to verify the solution set condition in the next proof:
\begin{dfn}
    Let $\kappa$ be a cardinal. We say a presheaf $X$ on $\scrC$ is \emph{bounded} by $\kappa$ if $\abs{X(T)}<\kappa$ for every $T\in \scrC$.
\end{dfn}

\begin{rem}
    If $\scrC$ is small, then every presheaf on $\scrC$ is bounded by some cardinal $\kappa$.
\end{rem}

\begin{prop}\label{Universal maps to groups - general case}
    Suppose $\scrC$ is a small category. Then the forgetful functor $\scrF$ from the category of presheaves of groups on $\scrC$ to the category of presheaves of partial magmas on $\scrC$ has a left adjoint.
\end{prop}
\begin{proof}
    First observe that, since $\scrC$ is small, we have that the category of presheaves of groups on $\scrC$ is locally small. Since the functor $\scrF$ is continuous, and the source is complete, in order to apply Freyd's adjoint functor theorem we just need to verify the solution set condition (\cref{solution set condition}).

    Let $X$ be a presheaf of partial magmas on $\scrC$. Since $\scrC$ is small, there exists a cardinal $\kappa$ such that $X$ is bounded by $\kappa$. Then for any presheaf of groups $Y$, and any morphism of partial magmas $\phi\colon X \to \scrF(Y)$, it follows that the image $\phi(X)$ is bounded by $\kappa$. Since $\scrC$ is small it follows that the collection of such images up to isomorphism is small; therefore this set satisfies the solution set condition.
\end{proof}

\begin{cor}\label{from presheaf adjoint to sheaf adjoint}
    Let $\scrC$ be a site, and let $\scrG$ the left adjoint to the forgetful functor from presheaves of groups to presheaves of partial magmas on $\scrC$. Denote by $\scrG^{++}$ the composition of $\scrG$ with sheafification. Then $\scrG^{++}$ defines a functor from presheaves of partial magmas on $\scrC$ to sheaves of groups on $\scrC$ which is left adjoint to the forgetful functor.
\end{cor}
\begin{proof}
    By \cref{sheafification and groups}, we have that the sheafification defines a left adjoint to the forgetful functor from sheaves of groups to presheaves of groups. The result follows immediately since the adjoint of a composition is the composition of the adjoints (\cref{The adjoint of a composition is the composition of the adjoints}).
\end{proof}

%% file: Chapters/Group_Chunks_On_Presheaves/Main_Group_Chunks_on_Presheaves.tex
\chapter{Group Chunks on Presheaves}\label{chapter: group chunks on presheaves}
In this chapter we prove an abstract ``Group Chunk'' theorem which we will later show generalizes previously known results in model theory and algebraic geometry. See \cref{section: overview of main results} for an overview of the group chunk theorem and many incarnations.

In \cref{section: cancellative magmas on sites} we define local nontriviality and show that sections of separated cancellative presheaves of partial magmas are determined locally by their associated left actions. We also record two lemmas which will be used to verify strong associativity of the model-theoretic group chunk in \cref{section: group chunks on definable types}. Group chunks are defined in \cref{section: group chunks}, and here we also prove their key properties, namely that they allow for products of arbitrary depth to be defined (\cref{Long products lemma}) and that partial morphisms generated by left translations are determined locally by their action on a single point (\cref{Check locally at one point}). From these results it quickly follows that compatibility gives an equivalence relation; we study the quotient in \cref{section: second universal morphisms to groups} and show that it satisfies the desired universal property. In \cref{section: third universal maps to groups} we show that the presheaf quotient is locally given by pairs of elements; and therefore the sheaf quotient $(\tX^\infty/\ssim)^+$, which we know from the previous section satisfies the universal property, can be realized as a quotient of $X^2$.

\input{Chapters/Group_Chunks_On_Presheaves/Cancellative_Magmas_on_Sites.tex}

\input{Chapters/Group_Chunks_On_Presheaves/Group_Chunk_definition.tex}

\input{Chapters/Group_Chunks_On_Presheaves/Universal_Morphisms_to_Groups_2.tex}

\input{Chapters/Group_Chunks_On_Presheaves/Universal_Morphisms_to_Groups_3.tex}

%% file: Chapters/Group_Chunks_On_Presheaves/Cancellative_Magmas_on_Sites.tex
\section{Cancellative Magmas on Sites} \label{section: cancellative magmas on sites}
In this section we discuss results specific to cancellative presheaves of magmas on sites. We begin with some definitions. Then we show that, in good situations, a section of a separated cancellative presheaf of partial magmas $(X,\mu)$ is determined by its local action on $X$. Lastly we prove two technical lemmas for verifying strong associativity which will be useful in later applications.

First we discuss local nontriviality.
\begin{dfn}[Local nontriviality]\label{local nontriviality definition}
    Let $\scrC$ be a site.
    \begin{enumerate}
        \item    Let $T \in \scrC$. A presheaf $X$ on $\scrC$ is \emph{locally nontrivial over $T$} if there exists a covering sieve $S$ of $T$ such that for every $T' \to T$ in $S$ we have $X(T')\ne \emptyset$.

        \item   We say $X$ is \emph{locally nontrivial} if it is locally nontrivial over $T$ for every $T\in \scrC$.
    \end{enumerate}
\end{dfn}

\begin{rem}\label{local nontriviality over T remark}
    Let $X_T$ be a presheaf on $\scrC_T$. Then $X_T$ is locally nontrivial if and only if it is locally nontrivial over $T$. Indeed, given a cover $S$ which exhibits local nontriviality over $T$, and any morphism $f\colon T' \to T$, the pullback $f^*S$ exhibits local nontriviality over $T'$.
\end{rem}

\begin{dfn}
    We say that a presheaf of partial magmas $(X,\mu)$ on a site $\scrC$ is \emph{separated} if $X$ is separated. We say it is \emph{locally nontrivial} if $X$ is locally nontrivial.
\end{dfn}

Sections of separated cancellative presheaves of partial magmas are determined locally by their associated left actions:
\begin{lem}\label{cancellation on covers}
    Let $(X,\mu)$ be a separated, cancellative presheaf of magmas. Let $x,y \in X(T)$, and $e \in \{\pm 1\}$. Suppose there exists a cover $S$ of $T$ such that for every $T' \in S$ there exists $z \in (\dom\alpha \cap \dom\beta)(T')$ with $\lambda_{x}^{e}(z) = \lambda_{y}^{e}(z)$. Then $x = y$.
\end{lem}
\begin{proof}
    We show the case $e = 1$; the other case is similar.

    Suppose that over some $T' \in S$ we have $x(z) = y(z) = w$. Then by (right) cancellation we have $x = y$ over $T'$. Then since $x = y$ over every $T' \in S$, and $X$ is separated, this implies $x = y$ over $T$.
\end{proof}

%

If we assume the domain of the multiplication is sufficiently large, we can make the above statement stronger:
\begin{lem}\label{agreement is an equivalence relation on tX for separated cancellative magmas with large domain}
    Let $(X,\mu)$ be a presheaf of partial magmas which is separated, cancellative, and such that for any $T$ and $\alpha_1,\alpha_2 \in \tX(T)$ we have that $\dom \alpha_1 \cap \dom \alpha_2$ is locally nontrivial.

    Then for any $x,y \in X(T)$, and $e \in \{\pm 1\}$, we have
    \begin{align*}
        \lambda_{x}^e \ssim  \lambda_y^e \iff x = y.
    \end{align*}
\end{lem}
\begin{proof}
    Right to left is immediate. Suppose $\lambda_{x}^e \ssim  \lambda_{y}^e$. Then by assumption there exists a cover $S$ of $T$ such that for every $T' \in S$ there exists $z \in (\dom\alpha \cap \dom\beta)(T')$, and then the result follows from \cref{cancellation on covers}.
\end{proof}

\begin{cor}\label{iota is an embedding with moderate domain}
    Let $(X,\mu)$ be as in \cref{agreement is an equivalence relation on tX for separated cancellative magmas with large domain}. Then the morphism of presheaves $\lambda(\mu)\colon X \to \tXinfty$ of \cref{definition of tX} is an embedding.
\end{cor}

Finally we record two technical lemmas which will be useful in applications for proving strong associativity from ostensibly weaker properties:
\begin{lem}\label{associativity follows from weak property}
    Let $\mu$ be a cancellative partial binary operation on a separated presheaf $X$. Take any $\alpha \in \tX(T)$ and $b,c,y \in X(T)$ such that $b \in (\dom \alpha)(T)$, $\alpha(b) = c$ and $y \in \dom \alpha \circ \lambda_b \cap \dom c$.

    Suppose there exists a cover $S$ of $T$ such that for every $T' \in S$ there exists $z\in X(T')$ such that
    \begin{align*}
        z \in (\dom \alpha \circ \lambda_b \circ \lambda_y \cap \dom \alpha \circ \lambda_{b(y)} \cap \dom \lambda_{\alpha(b(y))} \cap \dom \lambda_{c}\circ \lambda_y \cap \dom \lambda_{c(y)})(T')
    \end{align*}
    and
    \begin{align*}
        \alpha(b(y))(z)  = \alpha(b(y)(z)), \quad b(y)(z) = b(y(z)), \quad \alpha(b(y(z))) = c(y(z)), \quad c(y(z)) = c(y)(z).
    \end{align*}

    Then $\alpha(b(y)) = c(y)$.
\end{lem}

\begin{proof}
    Choose $z$ locally on $T$ as in the statement of the lemma. Then for each $T'$ and $z$, we have
    \begin{align*}
        \alpha(b(y))(z) & = \alpha(b(y)(z))                                  \\
                        & = \alpha(b(y(z)))  \tag{since $b(y)(z) = b(y(z))$} \\
                        & = c(y(z))                                          \\
                        & = c(y)(z)
    \end{align*}
    as required. Since $X$ is separated, we conclude that $\alpha(b(y)) = c(y)$ by \cref{cancellation on covers}.
\end{proof}



\begin{lem}\label{strong associativity follows form weak property}
    Let $\mu$ be a cancellative partial binary operation on a separated presheaf $X$, and $a,b,y \in X(T)$ such that
    \begin{align*}
        b\in \dom \lambda_a^\dagger, \quad y \in \dom \lambda_{a^\dagger (b)} \cap \dom \lambda_a^\dagger \circ \lambda_b, \quad a^\dagger(b)(y) \in \dom \lambda_a
    \end{align*}
    and
    \begin{align*}
        a(a^\dagger(b)(y)) = a(a^\dagger(b))(y),
    \end{align*}
    Then $a^\dagger (b(y)) = a^\dagger(b)(y)$.
\end{lem}

\begin{proof}
    Observe
    \begin{align*}
        a(a^\dagger(b)(y)) & = a(a^\dagger(b))(y) \tag{By assumption}               \\
                           & = b(y) \tag{by \cref{formaldagger looks like inverse}}
    \end{align*}
    Then again by \cref{formaldagger looks like inverse} we have
    \begin{align*}
        a^\dagger (b(y)) = a^\dagger (a(a^\dagger(b)(y))) = a^\dagger(b)(y)
    \end{align*}
\end{proof}

%% file: Chapters/Group_Chunks_On_Presheaves/Group_Chunk_definition.tex
\section{Group Chunks}\label{section: group chunks}

In this section we define group chunks, which are separated, cancellative, strongly associative partial magmas where the domain of the multiplication is sufficiently large. The goal, which we achieve in the next section, is to show that the universal group associated to a group chunk $(X,\mu)$ has a particularly nice form; namely that it is generated by the left action of $X$ on itself.

Roughly, the domain will be sufficiently large if any finite collection of products of arbitrary depth can be defined simultaneously. The first result in this section shows that it is enough to verify this for products of depth at most 2. The most important result in this section is \cref{Check locally at one point}, which states that compatibility of compositions of left translations can be checked locally at one point. This together with the large domain assumption implies that compatibility is a particularly well-behaved equivalence relation on compositions of left translations. At the end of this section we derive pleasant properties of this equivalence relation as we prepare to take the quotient in the next section.

\begin{dfn}\label{large domain definition}
    Let $(X,\mu)$ be a cancellative presheaf of partial magmas. We say $\mu$ has \emph{large domain} if the following holds:

    For any $N \in \Z_{>0}$ and partial morphisms $\alpha_1, \ldots, \alpha_N, \beta_1, \ldots, \beta_N$ in  $\tX(T)$, the intersection
    \begin{align*}
        \bigcap_{i = 1}^N \dom \alpha_i\circ \beta_i
    \end{align*}
    is locally nontrivial.

    We say $(X,\mu)$ has large domain if $\mu$ does.
\end{dfn}

\begin{dfn}[Group Chunk]\label{Group Chunk definition}
    A \emph{group chunk} on a site $\mathscr{C}$ is a presheaf of partial magmas $(X,\mu)$ on $\scrC$ which is separated, cancellative (\cref{cancellativity definition}), strongly associative (\cref{strong associativity definition}) and with large domain (defined above).
    A \emph{morphism of group chunks} is a morphism of the underlying partial magmas.
\end{dfn}
A separated presheaf of groups is certainly a group chunk.


\begin{prop}\label{Long products lemma}
    Let $(X,\mu)$ be a group chunk. Then for any $n \in \Z_{>0}$ and partial morphisms $\phi_{1},\cdots,\phi_{n} \in \tXinfty(T)$, the intersection $ \bigcap_{i = 1}^n \dom\phi_{i}$
    is locally nontrivial.
\end{prop}

\begin{proof}
    By definition, for each $i$ we have $\phi_i =  \alpha_{i1} \circ \cdots\circ\alpha_{ir_i}$ with $\alpha_{ij} \in \tX(T)$ for each $i,j$.

    We may assume without loss of generality that $r_i \geq 2$ for each $i$. Indeed, if not, for each $i$ with $r_i = 1$ we may replace $\phi_i = \alpha_{i1}$ with $\phi'_i = \alpha_{i1}\circ\alpha_{i1}$ and observe that $\dom\phi'_i \subset \dom\phi_i$.

    Now we proceed by induction on $m := \max_i r_i$. The case $m = 2$ is covered by \cref{large domain definition}. So suppose $m >2$. Then by induction there exists a cover $S$ of $T$ such that for each $T'$ in $S$ there exists $x$ with
    \begin{align*}
        x \in \bigcap_{i = 1}^n(\dom \alpha_{i2} \circ \dots \circ \alpha_{ir_i})(T').
    \end{align*}
    For each such $T'$ and $i$ define $b_{ir_i} = x$. For $1 < j\leq r_i$, recursively define
    \begin{equation*}
        b_{i(j - 1)} = \alpha_{ij} \circ \dots \circ\alpha_{ir_i}(x) \in X(T')
    \end{equation*}
    and observe
    \begin{align*}
        b_{i(j - 1)} = \alpha_{ij}(b_{ij}) \text{ for $1<j \leq r_i$}.
    \end{align*}
    Then by strong associativity we have
    \begin{align}\label{ai bi ssim}
        b_{i(j - 1)} \ssim \alpha_{ij}\circ b_{ij} \text{ for $1<j\leq r_i$.}
    \end{align}
    Moreover, since $\mu$ has large domain, for each $T'$ we may find a cover $S'$ of $T'$ such that for each $T''$ in $S'$ there exists $y$ with
    \begin{align}\label{y in intersection of domains}
        y \in \bigcap_{\substack{i,j \\ 1\leq j \leq r_i}} (\dom \alpha_{ij}\circ b_{ij})(T'').
    \end{align}
    \Cref{y in intersection of domains} implies
    \begin{align}\label{bijindomalphaihj}
        b_{ij}(y) \in \dom \alpha_{ij} \text{ for $1 \leq j\leq r_i$ }
    \end{align}
    and \ref{ai bi ssim} implies
    \begin{align}\label{alphaijbijequation}
        \alpha_{i(j)}\circ b_{i(j)}(y) =  b_{i(j - 1)}(y)\text{ for $1<j\leq r_i$.}
    \end{align}
    \begin{claim}
        \begin{enumerate}
            \item  For all $i$ and $j$ with $1 \leq j \leq r_i$,
                  \begin{align*}
                      \alpha_{i(j+1)} \circ \alpha_{i(j+2)} \circ \dots \circ \alpha_{ir_i} \circ x(y) \in \dom \alpha_{ij},
                  \end{align*}
            \item   For $1<j \leq r_i$ we have
                  \begin{align*}
                      \alpha_{ij} \circ \alpha_{i(j+1)} \circ \dots \circ \alpha_{ir_i} \circ x(y) = b_{i(j - 1)}(y).
                  \end{align*}
        \end{enumerate}
    \end{claim}
    \begin{proof}
        We prove both statements together by downward induction on $j$. We get the base case $j = r_i$ by setting $j = r_i$ in \cref{bijindomalphaihj,alphaijbijequation}.

        For $j<r_i$, assume that we have
        \begin{align*}
            \alpha_{i(j+1)} \circ \alpha_{i(j+2)} \circ \dots \circ \alpha_{ir_i} \circ x(y) = b_{i(j - 1)}(y).
        \end{align*}
        Then by \cref{bijindomalphaihj} we get
        \begin{align*}
            \alpha_{i(j+1)} \circ \alpha_{i(j+2)} \circ \dots \circ \alpha_{ir_i} \circ x(y) \in \dom \alpha_{ij},
        \end{align*}
        and, for $j>1$, by \cref{alphaijbijequation} we get
        \begin{align*}
            \alpha_{ij} \circ \alpha_{i(j+1)} \circ \dots \circ \alpha_{ir_i} \circ x(y) = b_{i(j - 1)}(y).
        \end{align*}
    \end{proof}

    From the claim we get that for each $i$ we have
    \begin{align*}
        \alpha_{i2} \circ \alpha_{i3} \circ \dots \circ \alpha_{ir_i} \circ x(y) \in \dom \alpha_{i1},
    \end{align*}
    so we conclude
    \begin{align*}
        x(y) \in \dom \phi_i (T'').
    \end{align*}
    Therefore, collating the covers $S'$, we get a cover of $T$ which witnesses local nontriviality over $T$. The result then follows from \cref{local nontriviality over T remark}.
\end{proof}

The following result shows that $\ssim$ is particularly nicely behaved on $\tXinfty$.
\begin{prop}\label{Check locally at one point}
    Let $(X,\mu)$ be a group chunk. Let $\phi_1,\phi_2 \in \tXinfty(T)$. Suppose there exists a cover $S$ of $T$ such that for every $T' \in S$ there exists $z \in (\dom\phi_1 \cap \dom\phi_2)(T')$ such that ${\phi_1(z) = \phi_2(z})$. Then $\phi_1\ssim \phi_2$.
\end{prop}

\begin{proof}

    Let $T' \in \scrC_T$ and $y \in (\dom \phi_1 \cap \dom \phi_2)(T')$. We need to show ${\phi_1(y) = \phi_2(y)}$.

    Let $S'$ be the pullback of the cover $S$ to $T'$. Fix any $T'' \in S'$ and $z \in (\dom\phi_1 \cap \dom\phi_2)(T'')$ with $\phi_1(z) = \phi_2(z)$. For each $i$ let us denote $c_{i(r_i+1)} = z$, $d_{i(r_i+1)} = y$ and for $j\leq r_i$ define
    \begin{align*}
         & c_{ij} = \alpha_{ij}(c_{i(j+1)}) = \alpha_{ij} \circ \cdots\circ\alpha_{ir_i}(z)  \\
         & d_{ij} = \alpha_{ij}(d_{i(j+1)}) = \alpha_{ij} \circ \cdots\circ\alpha_{ir_i}(y).
    \end{align*}
    In particular $c_{i1} = \phi_i(z)$ and $d_{i1} = \phi_i(y)$. We have $c_{11} = c_{21}$ by assumption, and our goal is to show $d_{11} = d_{21}$.

    By strong associativity, for each $i,j$ we have
    \begin{align}\label{alpha c d compatibility}
        c_{ij} \ssim \alpha_{ij}\circ c_{i(j+1)}, &  & d_{ij} \ssim \alpha_{ij}\circ d_{i(j+1)}.
    \end{align}

    By \cref{Long products lemma}, there exists a cover $S''$ of $T''$ such that for every $T''' \in S''$ we may find $x \in X(T''')$ with
    \begin{align*}
        x \in \left(\dom \phi_1 \cap \dom \phi_2 \cap \left( \bigcap_{i,j}\dom c_{ij}\circ z^\dagger \right) \cap \left( \bigcap_{i,j}\dom d_{ij}\circ y^\dagger \right)\right)(T''').
    \end{align*}

    Since $x \in \dom z^\dagger$, we have by \cref{formaldagger looks like inverse} that $z\circ z^\dagger(x) = x$.
    \begin{claim}
        For each $i,j$ we have
        \begin{align*}
            c_{ij} \circ z^\dagger(x) = \alpha_{ij} \circ \cdots\circ\alpha_{ir_i} \circ z \circ z^\dagger (x), \\
            d_{ij} \circ y^\dagger(x) = \alpha_{ij} \circ \cdots\circ\alpha_{ir_i} \circ y \circ y^\dagger (x).
        \end{align*}
    \end{claim}
    \begin{proof}
        We proceed by descending induction on $j$. If $j = r_i + 1$ the result is immediate. Let $j \leq r_i$. Then by induction
        \begin{align*}
            \alpha_{ij} \circ \cdots\circ\alpha_{ir_i} \circ z \circ z^\dagger (x) = \alpha_{ij}\circ c_{i(j+1)} \circ z^\dagger(x)
        \end{align*}
        and in particular $z^\dagger(x) \in \dom \alpha_{ij} \circ c_{i(j+1)}$. Then by \cref{alpha c d compatibility} we have
        \begin{align*}
            c_{ij} \circ z^\dagger(x) = \alpha_{ij} \circ c_{i(j+1)} \circ z^\dagger(x)
        \end{align*}
        and the result follows. Similarly for the $d_{ij}$.
    \end{proof}
    It follows from the claim that\begin{align*}
        \phi_i(x) = \phi_i \circ z \circ z^\dagger(x) =  c_{i1} \circ z^\dagger(x)
    \end{align*}
    and similarly $\phi_i(x) = d_{i1}\circ y^\dagger(x)$. Therefore,
    \begin{align*}
        \phi_1(x) = c_{11} \circ z^\dagger(x) = c_{21} \circ z^\dagger(x) = \phi_2(x),
    \end{align*}
    and then
    \begin{align*}
        d_{11}\circ y^\dagger(x) = \phi_1(x) = \phi_2(x) = d_{21} \circ y^\dagger(x).
    \end{align*}
    Since $X$ is separated, we conclude $d_{11} = d_{21}$ by \cref{cancellation on covers}.
\end{proof}

The above result has several useful corollaries.

\begin{cor}\label{compatibility is an equivalence relation in the cases we care about}
    Let $(X,\mu)$ be a group chunk. Then the relation $\ssim$ is an equivalence relation on $\tXinfty$.
\end{cor}
\begin{proof}
    Identity and symmetry are immediate. For transitivity, let $\phi_1,\phi_2,\phi_3 \in X(T)$ and suppose $\phi_1\ssim \phi_2$ and $\phi_2 \ssim \phi_3$. Then by \cref{Long products lemma} there exists a cover $S$ of $T$ such that for every $T' \in S$ there exists $z \in (\cap_i \dom \phi_i)(T')$. Then $\phi_1(z) = \phi_2(z) = \phi_3(z)$ and we conclude by \cref{Check locally at one point} that $\phi_1 \ssim \phi_3$ as required.
\end{proof}

\begin{cor}\label{tXinfty composition compatibility}
    Let $(X,\mu)$ be a group chunk, $\phi \in \tXinfty(T)$ and $z \in (\dom \phi)(T)$. Then
    \begin{align*}
        \lambda_{\phi(z)} \ssim \phi \circ \lambda_z.
    \end{align*}
\end{cor}
\begin{proof}
    Write $\phi = \alpha_1\circ \dots \alpha_r$ as before. We proceed by induction on $r$. The case $r = 1$ follows from strong associativity. So suppose $r>1$, and let $\psi = \alpha_2\circ\dots\circ\alpha_r$. By induction, we have $ \psi \circ \lambda_z \ssim \lambda_{\psi(z)}$. Since compatibility respects composition (\cref{compatibility respects composition}) this implies $\alpha_1\circ \psi \circ \lambda_z \ssim \alpha_1\circ \lambda_{\psi(z)}$. Then by the $r = 1$ case we have $\alpha_1\circ \lambda_{\psi(z)} \ssim \lambda_{\alpha_1(\psi(z))} = \lambda_{\phi(z)}$. We conclude by transitivity of $\ssim$ (\cref{compatibility is an equivalence relation in the cases we care about}).
\end{proof}

\begin{cor}\label{dagger descends}
    Let $(X,\mu)$ be a group chunk, and let $\sigma_1, \sigma_2 \in \Xoneinfty(T)$. Suppose $\epsilon_\mu(\sigma_1) \ssim \epsilon_\mu(\sigma_2)$. Then $\epsilon_\mu(\sigma_1^\dagger) \ssim \epsilon_\mu(\sigma_2^\dagger)$.
\end{cor}
\begin{proof}
    By \cref{Long products lemma}, there exists a cover $S$ of $T$ so that for each $T' \in S$ there exists $z \in (\dom\epsilon_\mu(\sigma_1) \cap \dom \epsilon_\mu(\sigma_2))(T').$ For any such $z$, let $y = \epsilon_\mu(\sigma_1)(z) = \epsilon_\mu(\sigma_2)(z)$. It follows immediately from \cref{formaldagger looks like inverse} that $\epsilon_\mu(\sigma_1^\dagger)(y) = \epsilon_\mu(\sigma_2^\dagger)(y) = z$. Then the result follows from \cref{Check locally at one point}.
\end{proof}

%% file: Chapters/Group_Chunks_On_Presheaves/Universal_Morphisms_to_Groups_2.tex
\section{Universal Morphisms to Groups 2}\label{section: second universal morphisms to groups}
In this section we give an explicit construction for the universal presheaf of groups associated to a group chunk; this will be roughly the quotient $\tXinfty/\ssim$. The only hiccup occurs where $X(T) = \emptyset$, and there is nothing to generate a group with; instead of excluding these cases, (which would probably be the most sensible thing to do, and which would make this section significantly shorter), we manually throw in an identity element over these values - this is the purpose of \cref{Gmu definition}.

We show first that $\tXinfty/\ssim$ is (mostly) a presheaf of groups, and that it comes with an embedding of $X$. This follows quickly from the results of the previous section. Next we fill in identity elements over objects $T$ with $X(T) = \emptyset$. Finally we verify the universal property. The sheafification of the quotient is studied in the next section.

\begin{lem}\label{the quotient of tXinfty is separated}
    Let $(X,\mu)$ be a group chunk. Then the presheaf quotient $\tXinfty/\ssim$ is separated.
\end{lem}
\begin{proof}
    Let $\phi,\psi \in \tXinfty(T)$, and suppose there exists a cover $S$ of $T$ such that over every $T' \in S$ we have $\phi \ssim \psi$. By \cref{Long products lemma}, we may refine our cover $S$ so that for each $T'\in S$ there exists $z \in (\dom\phi \cap \dom \psi)(T')$. For every such $z$ we have $\phi(z) = \psi(z)$, so by \cref{Check locally at one point} we conclude $\phi \ssim \psi$.

\end{proof}

\begin{lem}\label{tXinfty/ssim is a presheaf of semigroups}
    Composition of partial morphisms endows $\tXinfty/\ssim$ with the structure of a presheaf of semigroups.
\end{lem}
\begin{proof}
    Immediate since the same is true of $\tXinfty$ (\cref{comments on definition of tX}) and composition respects compatibility (\cref{compatibility respects composition}).
\end{proof}

\begin{dfn}\label{iotadefinition}
    Let $(X,\mu)$ be a group chunk. We define $\iota_\mu\colon X \to \tXinfty/\ssim$ to be the composition of $\lambda(\mu)$ (\cref{definition of tX}) with the projection to $\tXinfty/\ssim$, sending sections $x$ to the class of $\lambda_x$.
\end{dfn}

\begin{lem}\label{iota gives an embedding into tXinfty}
    Let $(X,\mu)$ be a group chunk. Then the morphism of presheaves $\iota_\mu\colon X \to \tXinfty/\ssim$ is an embedding of presheaves of partial magmas.
\end{lem}
\begin{proof}
    The fact that $\iota_\mu$ is injective follows immediately from \cref{agreement is an equivalence relation on tX for separated cancellative magmas with large domain}. Moreover $\iota_\mu$ respects the partial binary operation since $\lambda(\mu)$ does (by associativity; \cref{comments on definition of tX}) and since the semigroup structure on $\tXinfty/\ssim$ is induced by composition.
\end{proof}

\begin{dfn}\label{longdaggerdefinition}
    Let $(X,\mu)$ be a group chunk. We denote by $(-)^\dagger$ the involution on the quotient $\tXinfty/\ssim$ induced by the involution $(-)^\dagger$ on $\Xoneinfty$ and \cref{dagger descends}; i.e., for any partial morphism $\phi \in \tXinfty(T)$, fix any $\sigma$ with $\phi =  \epsilon(\sigma)$ we define $\phi^\dagger$ to be the class of $\epsilon(\sigma^\dagger)$ in $(\tXinfty/\ssim)(T)$.
\end{dfn}

\begin{rem}
    Note that the above is well-defined by \cref{dagger descends}.
\end{rem}


\begin{cor}\label{tXinfty/ssim is a group if nonempty}
    Let $(X,\mu)$ be a group chunk. Then for any $T$ such that $X(T)\ne \emptyset$, the semigroup $(\tXinfty/\ssim)(T)$ is in fact a group, with inverse given by $(-)^\dagger$, and the restriction morphisms are group homomorphisms.
\end{cor}
\begin{proof}
    Suppose $X(T) \ne \emptyset$. It follows from \cref{formaldagger looks like inverse} that for any $\phi \in \tXinfty(T)$ we have $\phi^\dagger \circ \phi \ssim \phi\circ\phi^\dagger \ssim \id_{X_T}$. In particular $(\tXinfty/\ssim)(T)$ contains the identity morphism and is closed under inverse. It is therefore a group. The final statement follows since any morphism of semigroups between two groups is a group homomorphism.
\end{proof}

\begin{dfn}\label{Gmu definition}
    Let $(X,\mu)$ be a group chunk. Define the presheaf $G_\mu$ as follows: for any $T \in \scrC$ define
    \begin{equation*}
        G_\mu(T) = \begin{cases}
            (\tXinfty/\ssim)(T) & X(T) \ne \emptyset \\
            1                   & X(T) = \emptyset
        \end{cases}
    \end{equation*}
    and for any morphism $f\colon T' \to T$ in $\scrC$ define
    \begin{equation}\label{Gmu definition equation}
        G_\mu(f) = \begin{cases}
            (\tXinfty/\ssim)(f)                  & X(T) \ne \emptyset \text{ and } X(T') \ne \emptyset \\
            1 \mapsto \id_{(\tXinfty/\ssim)(T')} & X(T) = \emptyset \text{ and } X(T') \ne \emptyset   \\
            1 \mapsto 1                          & X(T) = \emptyset \text{ and } X(T') = \emptyset
        \end{cases}
    \end{equation}
\end{dfn}
\begin{lem}\label{G_mu is a separated presheaf of groups}
    The mapping $G_\mu$ defined above is a presheaf, and moreover it is a separated presheaf of groups.
\end{lem}
\begin{proof}
    First note that for any $T' \to T$, if $X(T) \ne \emptyset$ then $X(T') \ne \emptyset$, so \cref{Gmu definition equation} above covers all possible cases.

    To show $G_\mu$ is a presheaf we need to verify that it respects composition. So let $T'' \to T' \to T$ be morphisms in $\scrC$; we need to verify that the maps $G_\mu(T) \to G_\mu(T'')$ and $G_\mu(T) \to G_\mu(T') \to G_\mu(T'')$ are the same.

    If $X(T)$ and $X(T')$ are both empty, the result follows immediately from the definition.

    If $X(T)$ is empty and $X(T')$ is not, then $X(T'')$ is nonempty, and by definition both $G_\mu(T) \to G_\mu(T')$ and $G_\mu(T) \to G_\mu(T'')$ send $1$ to the respective identity elements. The result follows since the restriction maps of $\tXinfty/\ssim$ are group homomorphisms and therefore respect the identity elements.

    If $X(T)$ is nonempty, then so are $X(T')$ and $X(T'')$, and the result follows since $\tXinfty/\ssim$ is a presheaf.

    Next note that $G_\mu(T)$ is a group for every $T$ (by \cref{tXinfty/ssim is a group if nonempty}) and the restriction morphisms $G_{\mu}(T) \to G_{\mu}(T')$ are group homomorphisms. Indeed, this holds when $X(T)$ is nonempty by \cref{tXinfty/ssim is a group if nonempty}, and if $X(T)$ is empty then the statement is trivial. In particular $G_\mu$ is a presheaf of groups.

    The fact that $G_\mu$ is separated follows since $\tXinfty/\ssim$ is separated (\cref{the quotient of tXinfty is separated}) and if $X(T) = \emptyset$ then $G_\mu(T)$ is just one point so we get the separation property at $T$ automatically.
\end{proof}
\begin{notation}
    We will use $(\cdot)$ to denote the group operation on $G_\mu$.
\end{notation}

\begin{lem}\label{extending from tXinfty to G_mu}
    Let $(X,\mu)$ be a group chunk. Let $\phi\colon \tXinfty/\ssim  \to G$ be a morphism of presheaves of semigroups, where $G$ is a presheaf of groups. Then $\phi$ extends uniquely to a morphism of sheaves of groups $\phi'\colon G_\mu \to G$. The morphism $\phi'$ is defined by\begin{equation*}
        \phi'_T = \begin{cases}
            \phi_T              & X(T) \ne \emptyset \\
            1 \mapsto id_{G(T)} & X(T) = \emptyset
        \end{cases}
    \end{equation*}
    for $T \in \scrC$.
\end{lem}
\begin{proof}
    Uniqueness is immediate since $G_\mu(T) = (\tXinfty/\ssim)(T)$ if $X(T)\ne \emptyset$, and if $X(T) = \emptyset$ then any morphism of presheaves of groups must send $1$ to the identity.

    Therefore we just need to verify that the mapping $\phi'$ gives a well-defined morphism of presheaves of groups. First $\phi'$ is a well-defined morphism of presheaves since for any $T' \to T$ if $X(T) \ne \emptyset$ then the corresponding naturality square is the same as for $\phi$, so it commutes by assumption, and if $X(T) = \emptyset$ then following the definitions we see that going both directions around the naturality square sends $1$ to the $\id_{G_{\mu}(T')}$. Lastly, $\phi'$ respects the binary operation since $\phi$ does.
\end{proof}

\begin{rem}
    The embedding of presheaves of partial magmas $\iota_\mu\colon X \to \tXinfty/\ssim$ sending sections $x$ to $\lambda_x$ (see \cref{iota gives an embedding into tXinfty}) induces an embedding of presheaves of partial magmas from $X$ into $G_\mu$. We will also call this $\iota_\mu$.
\end{rem}

\begin{prop}\label{existence of a unit}
    Let $\phi\colon (X,\mu) \to (G,\cdot)$ be a morphism of group chunks, where $(G,\cdot)$ is a presheaf of groups. Then there exists a unique morphism of presheaves of groups $\phi'\colon G_\mu \to G$ which makes the following diagram commute:
    \[\begin{tikzcd}
            X && G \\
            & {G_\mu}
            \arrow["\phi"', from=1-1, to=1-3]
            \arrow["{\iota_\mu}", from=1-1, to=2-2]
            \arrow["\phi'", dashed, from=2-2, to=1-3]
        \end{tikzcd}\]
\end{prop}
\begin{proof}
    We first show uniqueness. Indeed for any such $\phi'$ and $x \in X(T)$ we must have $\phi'(\lambda_x) = \phi(x)$. Since $G_\mu(T)$ is generated as a group by the $\lambda_x$, this determines $\phi'$ uniquely.

    Therefore it is enough to show that such a $\phi'$ exists. For this, consider first the induced morphism
    \begin{align*}
        \psi\colon \Xoneinfty \to G, \quad ((x_1,e_1),\dots,(x_n,e_n)) \mapsto \phi(x_1)^{e_1}\cdot \dots \cdot \phi(x_n)^{e_n}.
    \end{align*}
    Observe that $\psi$ is a morphism of presheaves of semigroups with respect to the concatenation operation on $\Xoneinfty$.
    \begin{claim}
        Let $\sigma_1,\sigma_2 \in \Xoneinfty(T)$, and recall $\epsilon_\mu$ from \cref{definition of tX}. Suppose $\epsilon_\mu (\sigma_1) \ssim \epsilon_\mu(\sigma_2)$. Then $\psi(\sigma_1) = \psi(\sigma_2)$.
    \end{claim}
    \begin{proof}
        Let $\alpha_1 = \epsilon_\mu (\sigma_1)$, $\alpha_2 = \epsilon_\mu(\sigma_2)$. By \cref{Long products lemma}, locally on $T$ we may find $z \in \dom \alpha_1 \cap \dom \alpha_2$. Then $\alpha_1(z) = \alpha_2(z)$ by assumption. But by \cref{group operations and epsilon_mu} we have $\phi(\alpha_i(z)) = \psi(\sigma_i)\phi(z)$ for each $i$. Therefore $\psi(\sigma_1)\phi(z) = \psi(\sigma_2)\phi(z)$ and we conclude.
    \end{proof}
    It follows from the claim that $\psi$ descends to a morphism $\phi' \colon \tXinfty/\ssim \to G$. Moreover since $\psi$, $\epsilon_\mu$ and the projection $\tXinfty \to \tXinfty/\ssim$ are morphisms of presheaves of semigroups, it follows that $\phi'$ is a morphism of presheaves of semigroups. We conclude by \cref{extending from tXinfty to G_mu}.
\end{proof}
\begin{dfn}
    We define the functor $\scrG$ from group chunks on $\scrC$ to separated presheaves of groups on $\scrC$ sending a group chunk $(X,\mu)$ to $G_\mu$ and sending any morphism of group chunks $(X,\mu) \to (X',\mu')$ to the unique morphism $G_{\mu}\to G_{\mu'}$ induced by \cref{existence of a unit} and the composition $\iota_{\mu'} \circ \phi$.
\end{dfn}

\begin{prop}\label{left adjoint for group chunks}
    The mapping $\scrG$ is functorial and gives a left adjoint to the forgetful functor from the category of separated presheaves of groups to group chunks on $\scrC$.
\end{prop}
\begin{proof}
    Given morphisms of group chunks $\phi\colon (X,\mu) \to (X',\mu')$ and $\psi\colon (X',\mu) \to (X'',\mu'')$, then we have $\scrG(\psi) \circ \scrG(\phi) \circ \iota_{\mu} = \iota_{\mu''}\circ \psi \circ \phi$ so by above we have $\scrG(\psi) \circ \scrG(\phi) = \scrG(\psi\circ \phi)$. This shows functoriality.

    The fact that $\scrG$ gives a left adjoint to the forgetful functor follows immediately from the characterization of left adjoints by universal morphisms (\cref{universal property characterization of adjoint functors}).
\end{proof}

\begin{lem}\label{sub group chunks can produce the same group}
    Let $(X',\mu') \hookrightarrow (X,\mu)$ be an embedding of group chunks such that the induced map $\mu|_{\dom\mu'}\colon \dom\mu' \to X$ is surjective. Then the induced morphism of presheaves of groups $\scrG(X') \to \scrG(X)$ is an isomorphism.
\end{lem}
\begin{proof}
    For any $c \in X(T)$, we have $a,b \in X'(T)$ with $\mu'(a,b) = \mu(a,b) = x$. By strong associativity, this implies $\lambda_a \circ \lambda_b \ssim \lambda_c$, and therefore the induced embedding $\widetilde{X'}^{(\infty)}/\ssim \hookrightarrow \tXinfty/\ssim$ is an isomorphism. The result then follows from the construction of $\scrG$.
\end{proof}

%% file: Chapters/Group_Chunks_On_Presheaves/Universal_Morphisms_to_Groups_3.tex
\section{Universal Morphisms to Groups 3} \label{section: third universal maps to groups}
In this section we construct a left adjoint to the forgetful functor from sheaves of groups to group chunks. We show that the presheaf constructed in the previous section is locally generated in 2 steps, and this allows for a simpler description of the sheafification. The 2-step generation is crucial in applications for proving representability, as we will see in later chapters.

\begin{dfn}
    We define the functor $\scrG^+$ from group chunks on $\scrC$ to sheaves of groups on $\scrC$ to be the composition of $\scrG$ with sheafification.
\end{dfn}
\begin{rem}
    Since for any group chunk $(X,\mu)$ we have that $G_\mu$ is separated, we have that the sheafification is given by $G_\mu^+$. In particular there is no conflict here with \cref{from presheaf adjoint to sheaf adjoint}.
\end{rem}

\begin{prop}\label{group chunks to sheaves left adjoint}
    The functor $\scrG^+$ gives a left adjoint to the forgetful functor from sheaves of groups to group chunks.
\end{prop}
\begin{proof}
    Immediate, as in \cref{from presheaf adjoint to sheaf adjoint}.
\end{proof}

Recall the definition of $\epsilon_\mu$ from \cref{definition of tX}.
\begin{dfn}\label{X2 ssim definition}
    Let $(X,\mu)$ be a group chunk. Denote by $\zeta$ the composition
    \begin{align*}
         & X^2 \hookrightarrow \Xoneinfty \xrightarrow{\epsilon_\mu} \tXinfty \to \tXinfty/\ssim; \\
         & (x_1,x_2)  \mapsto \lambda_{x_1} \circ \lambda_{x_2}^\dagger.
    \end{align*}

    We denote by $\ssim$ the equivalence relation on $X^2$ given by the kernel pair of $\zeta$; i.e., for $x_1,x_2,y_1,y_2 \in X(T)$ we have $(x_1,x_2)\ssim (y_1,y_2)$ if $\lambda_{x_1}\circ \lambda_{x_2}^\dagger \ssim \lambda_{y_1}\circ \lambda_{y_2}^\dagger$.
\end{dfn}

\begin{rem}\label{quotient of X2 is separated}
    Observe that $X^2/\ssim$ is a subpresheaf of $\tXinfty/\ssim$ (since by construction, for each $T$, the induced map $X^2/\ssim \to \tXinfty$ is injective) and it follows from \cref{the quotient of tXinfty is separated} that $X^{(2)}/\ssim$ is separated. In particular the sheafification is $(X^{(2)}/\ssim)^+$.


\end{rem}


\begin{lem}
    Let $(X,\mu)$ be a group chunk. Then the map $\zeta\colon X^2 \to \tXinfty/\ssim$ is locally surjective; i.e., for any $\phi \in (\tXinfty/\ssim)(T)$ there exists a cover $S$ of $T$ such that for every $T'\in S$ there exists a pair $(x_1,x_2) \in X(T')^2$ such that (using the notation of \cref{lambda notation}) $\lambda_{x_1} \circ \lambda_{x_2}^\dagger \ssim \phi$.
\end{lem}
\begin{proof}
    For any $T$ with $(\tXinfty/\ssim)(T)\ne \emptyset$, we have by \cref{tXinfty/ssim is a group if nonempty} that $(\tXinfty/\ssim)(T)$ is a group generated by the elements $\lambda_x$ with inverse given by $(-)^\dagger$. So it suffices to show that, locally, the image of $\zeta$ contains the $\lambda_x$ and is closed under composition and inverse. More precisely, we need:
    \begin{enumerate}
        \item For any $a \in X(T)$, there exists a cover $S$ of $T$ such that for every $T' \in S$ there exists $(x_1,x_2) \in X^2$ with $\lambda_{x_1}\circ \lambda_{x_2}^\dagger \ssim \lambda_x$
        \item For any $a,b,c,d \in X(T)$, there exists a cover $S$ of $T$ such that for every $T' \in S$ there exists $(x_1,x_2) \in X^2$ with $\lambda_{x_1}\circ \lambda_{x_2}^\dagger \ssim \lambda_a \circ \lambda_b^\dagger \circ \lambda_c \circ \lambda_d^\dagger$
        \item For any $a,b\in X(T)$, there exists a cover $S$ of $T$ such that for every $T' \in S$ there exists $(x_1,x_2) \in X^2$ with $\lambda_{x_1}\circ \lambda_{x_2}^\dagger \ssim (\lambda_a \circ \lambda_b^\dagger)\inv$.
    \end{enumerate}

    First let us show that the image is locally closed under composition. So take any $a,b,c,d \in X(T)$. By \cref{Long products lemma}, we may find a cover $S$ of $T$ so that for every $T' \in S$ there exists $z \in \dom\lambda_a\circ \lambda_{b}^\dagger \circ \lambda_{c}\circ \lambda_d^\dagger$. Then by \cref{tXinfty composition compatibility} we have
    \begin{align*}
        \lambda_{a(b^\dagger(c(d^\dagger(z))))} \ssim\lambda_a\circ \lambda_b^\dagger \circ \lambda_c\circ \lambda_d^\dagger \circ \lambda_z.
    \end{align*}
    Since compatibility respects composition (\cref{compatibility respects composition}), and $(-)^\dagger$ gives the inverse in $\tXinfty$, it follows that
    \begin{align*}
        \lambda_{a(b^\dagger(c(d^\dagger(z))))}\circ \lambda_{z}^\dagger & \ssim \lambda_a\circ \lambda_b^\dagger \circ \lambda_c  \circ \lambda_d \circ \lambda_z\circ   \lambda_z^\dagger \\
                                                                         & \ssim \lambda_a\circ \lambda_b^\dagger \circ \lambda_c \circ \lambda_d^\dagger
    \end{align*}
    so setting $x_1 = a(b^\dagger(c(d^\dagger(z))))$ and $x_2 = z$ we conclude.

    The fact that the image is closed under inverse is clear since by \cref{dagger maps to inverse} for any $a,b \in X(T)$ we have $(\lambda_a\circ \lambda_b^\dagger)\inv = \lambda_b \circ \lambda_a^\dagger$.


    Finally we show that the image locally contains the generators; indeed for any $a\in X(T)$, arguing as above we may take $(x_1,x_2) = (a(z),z)$ for $z$ chosen locally on $T$.
\end{proof}

\begin{dfn}\label{G_mu prime definintion}
    Let $(X,\mu)$ be a group chunk. Define the subpresheaf $G'_\mu$ of $G_\mu$ by
    \begin{equation}
        G'_\mu(T) = \begin{cases}
            (X^2/\ssim)(T) & X(T) \ne \emptyset, \\
            1              & X(T) = \emptyset,
        \end{cases}
    \end{equation}
    where $\ssim$ is the equivalence relation defined in \cref{X2 ssim definition}.
\end{dfn}

\begin{rem}\label{local nontrivial X sheafification remark}
    \begin{enumerate}
        \item Observe that $G'_\mu$ is separated since $G_\mu$ is (\cref{G_mu is a separated presheaf of groups}). In particular the sheafification is $(G'_\mu)^+$.

        \item Observe that if $X(T)$ is locally nontrivial then $(G'_\mu)^+ = (X^2/\ssim)^+$.
    \end{enumerate}
\end{rem}

\begin{cor}\label{g_mu prime vs g_mu sheafification}
    Let $(X,\mu)$ be a group chunk. Then the embedding $G'_\mu \hookrightarrow G_\mu$ is locally surjective, and therefore induces an isomorphism of sheaves $(G'_\mu)^+ \simeq G_\mu^+$.
\end{cor}
\begin{proof}
    Local surjectivity follows from the previous result, and the second statement then follows from \cref{locally surjective morphism of presheaves leades to isomorphism of sheaves}.
\end{proof}




\begin{prop}
    \label{simplified group chunks to sheaves left adjoint}
    The functor $(\scrG')^+$ from group chunks to sheaves of groups which sends a group $(X,\mu)$ to the sheaf of groups defined in \cref{G_mu prime definintion} gives a left adjoint to the forgetful functor from sheaves of groups to group chunks. If $(X,\mu)$ is locally nontrivial, then $(\scrG')^+ = (X^2/\ssim)^+$.
\end{prop}
\begin{proof}
    Immediate from \cref{group chunks to sheaves left adjoint}, \cref{g_mu prime vs g_mu sheafification}, \cref{adjoint functors and full subcategories} and \cref{local nontrivial X sheafification remark}.
\end{proof}

\begin{lem}\label{locally surjective sub group chunks can produce the same sheaf of groups}
    Let $(X',\mu') \hookrightarrow (X,\mu)$ be an embedding of group chunks such that the induced map $\mu|_{\dom \mu'}\colon \dom\mu' \to X$ is locally surjective. Then the induced morphism of presheaves of groups $\scrG(X')^+ \to \scrG(X)^+$ is an isomorphism.
\end{lem}
\begin{proof}
    Arguing as in \cref{sub group chunks can produce the same group}, we see that the induced embedding $\widetilde{X'}^{(infty)}/\ssim \hookrightarrow \tXinfty/\ssim$ is locally surjective. The result then follows from \cref{locally surjective morphism of presheaves leades to isomorphism of sheaves}.
\end{proof}

%% file: Chapters/First-order_structures/Main_First_order_structures.tex
\chapter{Model Theory}\label{chapter: model theory}
In this section we show that model-theoretic quotients can be realized as sheaf-theoretic quotients, and use this to show how the group chunk result of Hrushovski and Rideau-Kikuchi (\cref{statement of model theoretic group chunk theorem}) can be viewed as a special case of the construction of the previous chapter.

We restrict ourselves to considering quotients of definable/type-definable sets by definable equivalence relations, since this is what we need for our intended application. However statements analogous to \cref{equivalence of interpretable sets and presheaf quotients} should hold for other kinds of model-theoretic quotients; for example quotients of type-definable sets by type-definable equivalence relations (hyperimaginaries) or quotients of pro-definable and $\infty$-definable sets.

In \cref{section: definable sets and model-theoretic quotients}, we begin with a review of definable sets, the category $\Def(\calT)$, and model-theoretic quotients - i.e., elimination of imaginaries. In \cref{section: interpretable sets} we give a careful review of Shelah's $(-)^{eq}$ construction in the context of a multi-sorted structure. The definitions are extended to type-definable and interpretable sets in \cref{section: type-definable and interpretable sets}. In \cref{section: Sheaf Quotients and type-interpretable sets}, we define the Grothendieck topology on the category of (type-)definable sets, and show that (type-)interpretable sets correspond to sheaf quotients. In \cref{section: group chunks on definable types}, we show that the model-theoretic group chunk of Hrushovski and Rideau-Kikuchi (\cite[Proposition~3.15, Proposition~3.16]{Hrushovski2019}) can be seen as an instance of the group chunk on presheaves we defined in the previous chapter.


For this chapter we fix a complete first-order theory $\mathcal{T}$ in a (possibly many-sorted) language $\calL$. We assume $\calL$ contains equality symbols $=_S$ for each sort-symbol $S$ of $\calL$.

\input{Chapters/First-order_structures/Definable_sets_and_Model-theoretic_quotients.tex}
\input{Chapters/First-order_structures/Interpretable_sets.tex}
\input{Chapters/First-order_structures/Type-Definable_and_Type-Interpretable_Sets.tex}
\input{Chapters/First-order_structures/Sheaf_quotients_and_type-interpretable_sets}

\input{Chapters/First-order_structures/Group_Chunks_on_Definable_types.tex}

%% file: Chapters/First-order_structures/Definable_sets_and_Model-theoretic_quotients.tex
\section{Definable Sets and Model-Theoretic Quotients}\label{section: definable sets and model-theoretic quotients}

We refer to any of the standard references \cite{Tent-Ziegler}, \cite{marker2013model} for definitions of the basic notions of languages, theories, models, definable sets etc. The interested reader can refer to \cite{Mac_Lane_Moerdijk_1994} for a more general construction of models valued in topoi.


Let us set down some conventions.
\begin{notation}
    Let $M$ be an $\calL$-structure.
    \begin{enumerate}
        \item For each sort symbol $S$ of $\calL$ we denote by $S(M)$ the corresponding sort of $M$.
        \item If $\phi$ is a relation, constant symbol, or formula from $\calL$, we denote by $\phi(M)$ the interpretation or the set realizations of $\phi$ in $M$; this will be a subset of a product of sorts of $M$. Note in particular that this applies if $\phi$ is a formula with no free variables; in this case $\phi(M)$ is either $1$ (i.e., $\{\emptyset\}$) for true or $0$ (i.e., $\emptyset$) for false.
        \item If $f$ is a function symbol of $\calL$, we denote by $\Gamma_f(M)$ the graph of the interpretation of $f$ in $M$.
        \item By a subset $A$ of $M$ we mean a subset of the union of the sorts of $M$.
    \end{enumerate}
\end{notation}

\begin{dfn}[Definable sets]
    Let $M$ be a model of $\calT$, and $A$ a subset of $M$.
    \begin{enumerate}
        \item   An \emph{$A$-definable set} of $M$ is a subset $X\subset M_1\times \dots \times M_k$, where $M_1,\dots,M_k$ are sorts of the structure, defined by an $\calL$-formula with parameters in $A$.
        \item   A \emph{definable set} of $M$ is an $A'$-definable set of $M$ for some subset $A'\subset M$.
        \item   An $A$-\emph{definable subset} of an $A$-definable set $X$ of $M$ is an $A$-definable set $Y$ of $M$ such that $Y$ is a subset of $X$.
        \item   An $A$-\emph{definable equivalence relation} on an $A$-definable set $X$ of $M$ is an $A$-definable subset $E$ of $X\times X$ such that $E$ is an equivalence relation on $X$.
        \item   An $A$-\emph{definable map} between $A$-definable sets $X$ and $Y$ of $M$ is an $A$-definable subset of $X\times Y$ which is a function from $X$ to $Y$.
    \end{enumerate}
\end{dfn}
\begin{rem}
    There are alternative interpretations of the notion of a definable set; for example one may wish to consider a definable set as a functor from structures to sets. In the end, however, both constructions yield equivalent categories, so for most purposes the differences are unimportant.
\end{rem}

\begin{notation}
    Following standard conventions, we write $0$-definable in place of $\emptyset$-definable.
\end{notation}

\begin{rem}
    From a category-theoretic perspective it is important to note that a definable set comes with a given embedding into some $M_1\times \dots \times M_k$; in particular they are subobjects in the category of sets (see \cref{section: subobjects}). In applications of model theory to algebraic geometry, for example, it is useful to bear in mind that a definable set in a model of ACF corresponds to a finite union of varieties with a given embedding into affine space. In the more general setup of \cite{Mac_Lane_Moerdijk_1994}, a definable object is a pair $(A, X)$, where $X$ is a list of sorts and $A$ is a subobject of the corresponding product.
\end{rem}


\begin{dfn}[The Category of $0$-Definable Sets]
    Denote by $\Def(\calT)$ the category of $0$-definable sets of $M$, for some model $M$ of $\calT$, and $0$-definable maps between them.
\end{dfn}
\begin{lem}\label{DefT is independent of the choice of model}
    The category $\Def(\calT)$ is independent of the choice of $M$ up to equivalence.
\end{lem}
\begin{proof}
    Indeed, given models $M$ and $N$ of $\calT$, there is a functor $F$ from $0$-definable sets of $M$ to $0$-definable sets of $N$ defined as follows: for any formula $\psi$ one sends $\psi(M)$ to $\psi(N)$. This map $F$ is well-defined; indeed if $\psi$ and $\phi$ are formulas and $\psi(M) = \phi(M)$ then this is witnessed by a sentence in $\calT$, so $\psi(N) = \phi(N)$.
    For similar reasons (since the relevant properties are witnessed by sentences of $\calT$) we have that
    \begin{enumerate}
        \item $F$ sends $0$-definable maps to $0$-definable maps.
        \item $F$ respects composition of $0$-definable maps and sends the identity to the identity - in particular $F$ is a well-defined functor.
        \item $F$ is full and faithful.
    \end{enumerate}
    Since $F$ is surjective by definition, this shows that $F$ is an equivalence of categories.
\end{proof}
\begin{rem}\label{models give embeddings of definable sets}
    Observe that, by above, for every model $M$ of $T$ we get an embedding of $\Def(\calT)$ into the category of sets. Different models give different embeddings. In fact one can define a model of $\calT$ to be an embedding of $\Def(\calT)$ into the category of sets which satisfies certain good properties; see \cite{mrg_categorical_logic}.
\end{rem}



\begin{dfn}[Model-Theoretic Quotients]\label{model-theoretic quotients definition}
    Let $M$ be a model of $\calT$, $X$ be a $0$-definable set of $M$ and $E \subset X^2$ be a $0$-definable equivalence relation on $X$. We define a \emph{model-theoretic quotient} of $X$ by $E$ in $M$ to be a $0$-definable surjection $\rho_E\colon X \twoheadrightarrow D$ such that the classes of $E$ are given by the fibers of $\rho_E$.
\end{dfn}

\begin{lem}\label{model-theoretic quotients are category-theoretic quotients}
    Model-theoretic quotients are quotients in the sense of category theory; i.e., if $\rho_E\colon X \to D$ is a model-theoretic quotient of $X$ by $E$, then we have a co-equalizer diagram in $\Def(\calT)$:
    \[\begin{tikzcd}
            E \arrow[r, shift left, "p_1"] \arrow[r,"p_2"'] & X \arrow[r, "\rho_E"] & D
        \end{tikzcd}\]
    where $p_1$, $p_2$ are the two projections of $E$ to $X$.
\end{lem}
\begin{proof}
    Indeed, given any $0$-definable map $g\colon X \to G$ such that $g$ is constant on equivalence classes of $E$, the $g$ factors uniquely through $D$ in the category of sets, and the resulting map is $0$-definable.
\end{proof}
\begin{rem}
    Observe that the converse to the above statement is false. Consider for example the theory of an infinite set $X$ in the 1-sorted language with an equivalence relation $E$ with two infinite classes. Then any model-theoretic quotient map $\rho_E$ as in the diagram above must have as its image two distinct definable elements. However the only definable 1-element set in this structure is $X^0$, so there is no model-theoretic quotient. On the other hand, since any definable map from $X$ which is constant on both equivalence classes is constant; it follows that the constant map $X \to X^0$ is a quotient in the category $\Def(\calT)$.
\end{rem}

\begin{dfn}[Uniform elimination of imaginaries]\label{uniform elimination of imaginaries definition}
    We say $\mathcal{T}$ has \emph{uniform elimination of imaginaries} (this terminology comes from {\cite[4.4.4]{Hodges}}) if for every model $M$ of $\calT$, any $0$-definable set $X$ of $M$, and any $0$-definable equivalence relation $E$ on $X$, $M$ admits a model-theoretic quotient of $X$ by $E$.
\end{dfn}

\begin{rem}\label{uniform elimination of imaginaries - some or any M}
    Observe that it is equivalent to require that the definition holds for \emph{some} model $M$ since given a model-theoretic quotient $\rho_E$ of $X$ by $E$ one can write down a sentence in $\calL$ witnessing the property that $\rho_E$ is a model-theoretic quotient.
\end{rem}
\begin{rem}\label{uniform elimination of imaginaries and category theory}
    By \cref{model-theoretic quotients are category-theoretic quotients}, if $\calT$ has uniform elimination of imaginaries then the category $\Def(\calT)$ admits quotients by equivalence relations.
\end{rem}

\begin{lem}\label{criterion for uniform elimination of imaginaries}
    Let $M$ be a model of $\calT$ with sorts $(S_i(M))_{i \in I}$. Suppose that $M$ admits model theoretic quotients by equivalence relations $E$ on products of sorts $S_1(M)\times\dots\times S_k(M)$. Then $M$ has uniform elimination of imaginaries.
\end{lem}
\begin{proof}
    Indeed, one can extend any $0$-definable equivalence relation $E$ on a definable set $X\subset S_1(M)\times \dots S_k(M)$ to all of $S_1(M)\times\dots\times S_k(M)$ by declaring the complement of $X$ to be a new equivalence class, and then taking the restriction of the induced quotient map to $X$.
\end{proof}


\begin{lem}
    Uniform elimination of imaginaries is invariant under adding parameters to the language; i.e., if $\calT$ has uniform elimination of imaginaries, then for any subset $A$ of a model $M$ of $\calT$ we have that $\calT_A$ has uniform elimination of imaginaries.
\end{lem}
\begin{proof}
    See \cite[1.1.8]{PillayGST} for a proof in the single sort case which easily generalizes. For the reader's convenience we give here a proof adapted from the lecture notes of a course of Scanlon (\cite{Speirs2013Lec15}):

    Let $E$ be an $A$-definable equivalence relation on an $A$-definable subset $X$ of $M$. By \cref{criterion for uniform elimination of imaginaries} we may assume $X$ is a product of sorts of $M$. Suppose $E$ is defined by an $\calL_A$-formula $\psi(x_1,x_2,a)$ (where $a$ represents a tuple from $A$ and each $x_i$ a tuple of variables from the product $X$). Suppose the tuple $a$ is contained in a product of sorts $Y$.  Consider the formula $\psi'(x_1,y_1,x_2,y_2)$, defined to be a disjunction of the expressions:
    \begin{enumerate}
        \item $\psi(x_1,x_2,y_1) \land y_1 = y_2 \land $ ``$\psi(X,X,y_1)$ is an equivalence relation on $X$''
        \item $y_1 \ne y_2$
        \item ``$\psi(X,X,y_1)$ is not an equivalence relation on $X$''
    \end{enumerate}
    Let $E' = (\psi')^M$. Then $E'$ is a $0$-definable equivalence relation on the product of sorts $X\times Y$. Let $\rho_{E'}$ be the corresponding $0$-definable quotient map; then the map $x \mapsto \rho_{E'}(x,a)$ gives a quotient for $X$ by $E$.
\end{proof}

\begin{rem}
    An alternative notion, which we will not use, is that of (non-uniform) elimination of imaginaries. A theory $\calT$ is said to have \emph{elimination of imaginaries} if every class of a $0$-definable equivalence relation has a canonical parameter - see for example \cite[8.4]{Tent-Ziegler}. A theory $\calT$ has uniform elimination of imaginaries if and only if it has elimination of imaginaries and two distinct definable elements (see for example \cite[Lemma~8.4.7]{Tent-Ziegler} for the proof in the single-sorted case which easily generalizes to the many sorted case).
\end{rem}



%% file: Chapters/First-order_structures/Interpretable_sets.tex
\section{Interpretable Sets}\label{section: interpretable sets}
The $(-)^{eq}$ expansion of a multi-sorted structure is defined in, for example, \cite{Categorical_Meq}, or \cite{haskell2008stable}. The single-sorted variant is discussed in \cite[8.4]{Tent-Ziegler}, \cite{poizat2011course} or \cite{Hodges}.

\begin{dfn}[{$\calLeqT$}]
    Let $M$ be a model of $\calT$. We define an expansion $\calL^{eq}(T)$ of the language $\calL$ as follows:
    \begin{enumerate}
        \item $\calLeqT$ contains $\calL$. The sort symbols of $\calL$ will be called the \emph{home sort symbols} of $\calLeqT$.
        \item For every $0$-definable equivalence relation $E$ on a product of sorts $S_1(M)\times\dots\times S_n(M)$ of $M$, $\calLeqT$ contains a sort symbol $S_E$ and a function symbol $\pi_E\colon S_1\times \dots \times S_n \to S_E$.
    \end{enumerate}
\end{dfn}

\begin{lem}
    The language $\calLeqT$ is independent of the choice of $M$.
\end{lem}
\begin{proof}
    Immediate, since for any formula $\phi$, if $\phi(M)$ is an equivalence relation on the product of sorts $S_1(M)\times \dots \times S_n(M)$ then this is witnessed by a sentence in $\calT$ and therefore $\phi(N)$ is an equivalence relation on $S_1(N)\times \dots \times S_n(N)$.
\end{proof}

\begin{rem}\label{induced equivalences on subcategories of sets}
    Given a collection of sets $Z = (Z_i)_{i\in I}$, let us denote by $\calC(Z)$ the smallest full subcategory of the category of sets which contains the collection $Z$, the sets $1$ and $0$ (i.e., $\{\emptyset\}$ and $\emptyset$), and is closed under taking finite Cartesian products and subsets.

    Observe that given another family $Y = (Y_i)_{i\in I}$ and a family of bijections $(\iota_i\colon Y_i \xrightarrow{\sim} Z_i)_{i\in I}$, we get an equivalence of categories $\iota\colon \calC(Y) \xrightarrow{\sim} \calC(Z)$ which preserves the sets $0$, $1$, and the intersections, complements, projections of subsets and graphs of functions. For each set $A\in \calC(Y)$ we get a bijection $A \xrightarrow{\sim} \iota(A)$, which we also denote by $\iota$.

    In particular, if $M$ is a model of $\calT$, and $\Sigma$ a collection of sort symbols, and $(N_s)_{S\in \Sigma}$ a collection of sets, then given a family of bijections  $(\iota_S\colon S(M) \to N_S)_{S\in \sigma}$, the corresponding equivalence $\iota\colon \calC((S(M)_{S\in \Sigma})) \to \calC((N_S)_{S\in \Sigma})$ induces an equivalence of categories of $\Def(\calT)$ with its image under $\iota$.
\end{rem}

\begin{dfn}[{$M^{eq}$.}]\label{Meq definition}
    Let $M$ be a model of $T$ with sorts $(M_i)_{i \in I}$. Let $\Sigma$ denote the set of sort symbols of $\calL$. We define an $\calLeqT$-structure $\Meq$ as follows:

    \begin{enumerate}
        \item   For each $0$-definable equivalence relation $E$ on a product of sorts $S_1(M)\times\dots\times S_n(M)$ of $M$, choose a pair $(M_E,\rho_E)$, where $M_E$ is a choice of representative of the quotient of $S_1(M)\times\dots\times S_n(M)$ by $E$ in the category of sets and $\rho_E\colon S_1(M)\times\dots\times S_n(M) \to M_E$ is the quotient map.

        \item For each $E$ as above, define $S_E(\Meq) = M_E$. Note that for each $S\in \Sigma$, we have a sort $S(\Meq) := M_{=_S}$. These will be called the \emph{home sorts} of $\Meq$. Observe that the quotient map $\rho_{=_S}$ is a bijection.


        \item For each $E$, interpret $\pi_E$ as the quotient map $M_{=_{S_1}}\times\dots\times M_{=_{S_n}}\to M_E$ induced by $\rho_E$ and the isomorphisms $\rho_{=_{S_i}}$ (so for example the maps $\pi_{=_{S_i}}$ are interpreted as the identity maps).



        \item  Denote by $\iota$ the induced equivalence $\calC((S(M)_{S\in \Sigma})) \xrightarrow{\sim} \calC((S(\Meq)_{S\in \Sigma}))$ induced by the bijections $\rho_{=_S}$ and \cref{induced equivalences on subcategories of sets}.



        \item For any relation or constant symbol $R$ in $\calL$, define $R(\Meq) := \iota(R(M))$, and for function symbols $f$ in $\calL$ define $\Gamma_f(\Meq) := \iota(\Gamma_f(M))$.

    \end{enumerate}
\end{dfn}
\begin{rem}
    Set-theoretically, there is a ``best'' representative of the quotient $(M_1\times\dots\times M_n)/E$ whose elements are the $E$-equivalence classes - one does not need the axiom of choice here. However in practice the choice of representative is unimportant, and we think it is clarifying to make these choices explicit.
\end{rem}

\begin{rem}
    By above, for every model $M$ of $\calT$ we get an embedding of $\Int(\calT)$ into the category of sets.
\end{rem}





\begin{lem}\label{L formulas preserved in Meq.}
    Let $\iota$ be as in \cref{Meq definition}. Then for any $\calL$-formula $\phi$, we have $\iota(\phi(M))= \phi(\Meq)$.
\end{lem}
\begin{proof}
    Indeed, since $\iota$ respects constants and function symbols, it follows by induction on complexity of terms that for any $\calL$-term $t$ we have $\Gamma_t(\Meq) = \iota(\Gamma_t(M))$, where $\Gamma_t$ denotes the graph of the function corresponding to $t$. Since $\iota$ also respects relations it follows that the result holds for any atomic formula, and then the full statement follows by \cref{induced equivalences on subcategories of sets} and induction on formula complexity since $\iota$ gives isomorphisms of the sorts of $M$ with the home sorts of $\Meq$.
\end{proof}

\begin{lem}\label{0 interpretable in home sorts are 0 definable}

    Let $\phi$ be an $\calL^{eq}$-formula with free variables in the home sorts (or no free variables). Then there exists an $\calL$-formula $\psi$ with $\psi(\Meq) = \phi(\Meq)$.

\end{lem}

\begin{proof}
    The corresponding statement for single-sorted languages is proved in \cite[Proposition~8.4.5]{Tent-Ziegler}. The proof for the multi-sorted case is essentially the same.




    We proceed by induction on formula complexity. Indeed, all atomic formulas with variables from the home sorts are either atomic $\calL$-formulas, in which case the result follows by definition, or of the form $\pi_E(x_1) = \pi_E(x_2)$ for some $0$-definable equivalence relation $E$; this gives the $0$-interpretable set $\iota(E)$. The induction steps for negation and conjunction are immediate. For existential quantification, the definable set given by $(\exists y \in M_E)\ \phi(x,y)$ in $\Meq$ is the image of that given by $\exists y\ \phi(x,\pi_E(y))$ in $M$, so we conclude by induction.


\end{proof}

\begin{cor}\label{embedding from def into int}
    The embedding $\iota$ from \cref{Meq definition} induces a fully faithful embedding from $\Def(\calT)$ to $\Int(\calT)$.
\end{cor}
\begin{proof}
    Indeed, the by \cref{L formulas preserved in Meq.} and \cref{0 interpretable in home sorts are 0 definable} we have that the image of $\Def(\calT)$ under $\iota$ is precisely the collection $0$-interpretable subsets of the home sorts and $0$-interpretable maps between them. Then the result follows from \cref{induced equivalences on subcategories of sets}.
\end{proof}

\begin{dfn}
    Let $M$ be a model of $\calT$. We denote by $\calTeq$ the complete theory of the $\calLeqT$-structure $\Meq$.
\end{dfn}
\begin{lem}
    The theory $\calTeq$ is independent of the choice of $M$.
\end{lem}
\begin{proof}
    Indeed, it follows from the previous two results that for any $\calLeqT$-sentence $\sigma$ there is an $\calL$-sentence $\sigma'$ such that for any model $M$ of $T$ we have $\sigma(\M^{eq}) = \iota(\sigma(M))$.
\end{proof}

\begin{dfn}
    Fix a model $M$ of $\calT$ and a subset $A \subset M$. We define an $A$-\emph{interpretable set} of $M$ to be an $A$-definable set of $\Meq$.
\end{dfn}
\begin{notation}\label{zero-interpretable}
    As above, we write $0$-interpretable in place of $\emptyset$-interpretable.
\end{notation}

\begin{dfn}
    We denote by $\Int(\calT)$ the category of $0$-interpretable sets of $M$, for some model $M$ of $\calT$, and $0$-interpretable maps between them; i.e., $\Int(\calT) := \Def(\calT^{eq})$.
\end{dfn}
\begin{lem}
    The category $\Int(\calT)$ is independent of the choice of $M$.
\end{lem}
\begin{proof}
    Immediate, since different models of $\calT$ give rise to models of $\calT^{eq}$, so we conclude by \cref{DefT is independent of the choice of model}
\end{proof}

\begin{notation}\label{definable sets are interpretable notation}
    In light of \cref{embedding from def into int}, we identify $\Def(\calT)$ with a full subcategory of $\Int(\calT)$.
\end{notation}

\begin{lem}\label{Every interpretable set is a quotient}
    Let $M$ be a model of $T$. Then every $0$-interpretable set $D$ of $M$ is a model-theoretic quotient (in $\Meq$) of a $0$-definable set $X$ (in $M$) by a $0$-definable equivalence relation $E$. If $D$ is a product of sorts of $\Meq$, then $X$ is a product of home sorts.
\end{lem}
\begin{proof}
    Suppose $D$ is contained in the sorts $M_{E_1} \times \cdots \times M_{E_k}$, where each $E_i$ is a definable equivalence relation on a product of sorts $S_{i1}(M) \times \dots \times S_{in_i}(M)$. We define $X$ to be the pre-image of $X$ in $S_{11}(M) \times \dots \times S_{kn_k}(M)$ under the map $\pi_{E_1}\times \dots \times \pi_{E_k}$; note that if $D$ is a product of sorts of $M$, then by construction $X$ is a product of home sorts. The $E$-equivalence classes are defined to be the fibers of this map. Both $X$ and $E$ are $0$-interpretable, so by \cref{0 interpretable in home sorts are 0 definable} they are $0$-definable (in the sense of \cref{definable sets are interpretable notation}) and the quotient map is given by $\pi_{E_1}\times \dots \times \pi_{E_k}$.
\end{proof}

\begin{lem}\label{Meq has elimination of imaginaries}
    Let $M$ be an $\calL$-structure. Then $M^{eq}$ has uniform elimination of imaginaries. In particular the category $\Int(\calT)$ admits quotients by equivalence relations.
\end{lem}
\begin{proof}
    See \cite[Lemma~16.13]{poizat2011course} for the result in the single-sort case, again this proof easily generalizes.

    Indeed, let $E$ be a $0$-interpretable equivalence relation on a $0$-interpretable set $X$ of $M$. We need to find a model-theoretic quotient of $X$ by $E$ in $\Meq$. By \cref{criterion for uniform elimination of imaginaries} we may assume $X$ is a product of sorts of $\Meq$. By \cref{Every interpretable set is a quotient}, we may realize $X$ as a model-theoretic quotient of a product of home sorts $X'$ by a $0$-definable equivalence relation $E'$. Let $\rho_{E'}$ be the quotient map and let $E''$ denote the preimage $(\rho_{E'}\times \rho_{E'})\inv(E)$ in $X' \times X'$. Then $E''$ is an equivalence relation on $X'$ which is refined by $E'$, so by \cref{model-theoretic quotients are category-theoretic quotients} the map $\pi_{E''}\colon X' \to M_{E''}$ factors definably (in $\Meq$) through $X$. Let $\rho_E\colon X \to M_{E''}$ be the induced $0$-interpretable map. Then for any $x_1,x_2\in X$ we have $\rho_{E}(x_1) = \rho_{E}(x_2)$ if and only if $x_1'\sim_{E''} x_2'$, where $x_1'$ and $x_2'$ are pre-images of $x_1$, $x_2$ respectively under $\rho_{E'}$. By definition of $E''$, this holds if and only if $x_1\sim_E x_2$, as required.

    The final statement follows by \cref{uniform elimination of imaginaries and category theory}.
\end{proof}
\begin{cor}
    The theory $\calT$ has uniform elimination of imaginaries if and only if the embedding $\iota\colon \Def(\calT) \hookrightarrow \Int(\calT)$ is an equivalence of categories.
\end{cor}
\begin{proof}
    Indeed, left to right follows from \cref{Every interpretable set is a quotient}, and right to left follows from \cref{Meq has elimination of imaginaries}.
\end{proof}

\begin{cor}
    The embedding $\Int(\calT)\hookrightarrow\Int(\calT^{eq})$ is an equivalence of categories.
\end{cor}
\begin{proof}
    Immediate by the previous two results.
\end{proof}

%% file: Chapters/First-order_structures/Type-Definable_and_Type-Interpretable_Sets.tex






\section{Type-Definable and Type-Interpretable Sets}\label{section: type-definable and interpretable sets}
In this short section we bootstrap the definitions from the previous section to type-definable and type-interpretable sets.

\begin{dfn}[Type-Definable sets]\label{type-definable sets definition}
    Let $M$ be a model of $\calT$, and $A$ a subset of $M$.
    \begin{enumerate}
        \item   An $A$-\emph{type-definable set} of $M$ is a subset $X\subset M_1\times \dots \times M_k$, where $M_1,\dots,M_k$ are sorts of the structure, defined by a \emph{partial} $\calL$-type over $A$.
        \item   A \emph{type-definable set} of $M$ is an $A'$-type-definable set of $M$ for some subset $A'\subset M$.
        \item   An $A$-\emph{type-definable subset} of an $A$-type-definable set $X$ of $M$ is an $A$-type-definable set $Y$ of $M$ such that $Y$ is a subset of $X$.
        \item   An $A$-\emph{definable equivalence relation} on an $A'$-type-definable set $X$ of $M$ is an equivalence relation $E$ on $X$ which is the intersection of an $A$-definable set with $X\times X$.
        \item   An $A$-\emph{definable map} between $A'$-type-definable sets $X$ and $Y$ of $M$ is a map from $X$ to $Y$ which is the intersection of an $A$-definable set with $X\times Y$.
        \item  An $A$-\emph{type-interpretable set} of $M$ is an $A$-type definable set of $\Meq$.
        \item Similarly we have $A$-interpretable maps and equivalence relations on $A$-type-interpretable sets of $\Meq$.
    \end{enumerate}
\end{dfn}
\begin{notation}
    As above, we use $0$-type-definable/interpretable in place of $\emptyset$-type-definable/interpretable
\end{notation}

\begin{dfn}\label{type-definable categories definition}
    \begin{enumerate}
        \item We denote by $\typDef(\calT)$ the category of $0$-type-definable sets of $M$, for any $\abs{T}^+$-saturated model $M$ of $T$, with $0$-definable maps between them.
        \item We denote by $\typInt(\calT)$ the category of $0$-type-interpretable sets of $M$, for any $\abs{T}^+$-saturated model $M$ of $T$, with $0$-interpretable maps between them.
    \end{enumerate}
\end{dfn}

\begin{lem}
    The categories $\typDef(\calT)$, $\typInt(\calT)$ are independent of the choice of $M$.
\end{lem}
\begin{proof}
    We show the result for $\typDef(\calT)$, the proof for $\typInt(\calT)$ is nearly identical.

    Let $N$ be another $\abs{\calT}^+$-saturated model. As in \cref{DefT is independent of the choice of model}, there is a functor $F$ sending $0$-type-definable sets of $M$ to $0$-type-definable sets of $N$ as follows: for any partial type $p$, one sends $p(M)$ to $p(N)$. This map $F$ is well-defined; indeed if $p$ and $q$ are partial types and $p(M) = q(M)$ then every complete extension of $p$ contains $q$ (otherwise, since $M$ is chosen $\abs{T}^+$-saturated, we could find an element in $p(M)$ but not in $q(M)$) and vice versa. It follows that $p(N) = q(N)$.

    For similar reasons (since the relevant properties are witnessed by sentences in complete extensions of the types) we have that
    \begin{enumerate}
        \item $F$ sends $0$-definable maps to $0$-definable maps.
        \item $F$ respects composition of $0$-definable maps and sends the identity to the identity - in particular $F$ is a well-defined functor.
        \item $F$ is full and faithful.
    \end{enumerate}
    Since $F$ is surjective by definition, this shows that $F$ is an equivalence of categories.
\end{proof}

\begin{rem}
    For every $\abs{\calT}^+$-saturated model $M$ of $\calT$ we get an embedding of $\typDef(\calT)$ and $\typInt(\calT)$ into the category of sets.
\end{rem}

\begin{rem}\label{type definable sets are definable}
    Observe from the definitions that we have fully faithful embeddings of $\Def(\calT)$ into $\typDef(\calT)$ and $Int(\calT)$ into $\typInt(\calT)$.
\end{rem}

\begin{rem}\label{type-definable sets are type-interpretable}
    As in \cref{embedding from def into int}, it follows from \cref{L formulas preserved in Meq.} and \cref{0 interpretable in home sorts are 0 definable} that the embedding $\iota$ of \cref{Meq definition} induces an isomorphism of categories (which we also call $\iota$) between $\typDef(\calT)$ and the full subcategory of $\typInt(\calT)$ consisting of $0$-type-interpretable subsets of products of the home sorts. In particular, we can (and do) identify $\typDef(\calT)$ with a full subcategory of $\typInt(\calT)$.
\end{rem}

\begin{dfn}\label{definition: model theoretic quotients of type-definable sets}
    We define a model-theoretic quotient of a $0$-type-definable set by a $0$-definable equivalence relation as in \cref{model-theoretic quotients definition}.
\end{dfn}
\begin{rem}\label{remarks: model-theoretic quotients of type-definable sets}
    \begin{enumerate}
        \item A model-theoretic quotient of a $0$-type-definable set by a $0$-definable equivalence relation is a categorical quotient in $\typDef(\calT)$.
              The proof is the same as for \cref{model-theoretic quotients are category-theoretic quotients}
        \item Every $0$-type-interpretable set $D \in \typInt(\calT)$ is a model-theoretic quotient of a $0$-type-definable set $X$ by a $0$-definable equivalence relation. The proof is the same as for \cref{0 interpretable in home sorts are 0 definable}.
    \end{enumerate}
\end{rem}




\begin{lem}
    Let $X$ be a $0$-type-interpretable set of $M$ and $E$ a $0$-interpretable equivalence relation on $X$. Then there exists a model-theoretic quotient of $X$ by $E$ in $\typInt(\calT)$.
\end{lem}
\begin{proof}
    Choose a formula $\phi$ such that $E = (\phi(M)\cap X\times X)$. Then by compactness we may choose a definable superset $Y \supset X$ such that $E_1 := \phi(M)\cap Y^2$ is a definable equivalence relation on $Y$. Then since $M^{eq}$ eliminates imaginaries, there exists an $M^{eq}$-definable set $Y'$ and an $M^{eq}$-definable map $f_{E_1}\colon Y \to Y'$ whose fibers are exactly the $E_1$-equivalence classes. In particular the fibers of the restriction of $f_{E_1}$ to $X$ are exactly the $E$-equivalence classes, so we may define the quotient to be the restriction of $f_{E_1}$ to $X$.
\end{proof}

%% file: Chapters/First-order_structures/Sheaf_quotients_and_type-interpretable_sets.tex
\section{Sheaf Quotients and Type-Interpretable Sets}\label{section: Sheaf Quotients and type-interpretable sets}

Recall the notion of a congruence object or internal equivalence relation $E$ on an object $X$ in a category $\scrC$ (see \cref{equivalence relation definition}), and that the Yoneda embedding takes congruences to congruences (\cref{Yoneda preserves equivalence relations}). Recall also \cref{presheaf notation} and \cref{Yoneda Notation}.
\begin{dfn}
    Let $X\in \typDef(\calT)$ be a definable/type-definable set and $E$ a definable equivalence relation on $X$. By the presheaf quotient of $X$ by $E$ we mean the coequalizer of the diagram $
        \begin{tikzcd}
            h_E \arrow[r, shift left] \arrow[r] & h_X
        \end{tikzcd}$ in the presheaf category $\Presh(\typDef(\calT))$. Note this is defined up to natural isomorphism.
\end{dfn}

\begin{rem}\label{pointwise quotients for presheaves}
    Recall that quotients in the category of presheaves are taken pointwise, so for any diagram $
        \begin{tikzcd}
            h_E \arrow[r, shift left] \arrow[r] & h_X,
        \end{tikzcd}$
    as above we have for any $T \in \typDef(\calT)$ that $(h_X/h_E)(T)$ is naturally in bijection with the quotient set $h_X(T)/h_E(T)$.

    In particular, the quotient map $h_X \to h_X/h_E$ is a surjection of presheaves, and is therefore pointwise surjective.
\end{rem}

\begin{notation}\label{equivlaence classes notation}
    Given a presheaf $F$ on a category $\scrC$, an equivalence relation $E$ on $F$, and a fixed choice of presheaf quotient $F/E$, for any $T\in \scrC$ and $t \in F(T)$ we denote by $[t]_{E(T)}$ the image of $t$ in $(F/E)(T)$ under the quotient map.
\end{notation}

One might guess that there is an equivalence of categories between presheaf quotients and model-theoretic quotients; in fact this fails because maps of presheaf quotients correspond to $0$-definable maps, and in general there exist morphisms of interpretable sets which do not lift to $0$-definable maps. The next result and the examples which follow make this precise.
\begin{lem}\label{morphisms of presheaf quotients correspond to definable maps}
    Let $E_1,E_2$ be $0$-definable equivalence relations on $0$-type-definable sets $X_1,X_2$ respectively, and let $F_1,\ F_2$ denote the respective presheaf quotients. Then for any morphism of the presheaf quotients $\sigma\colon F_1 \to F_2$ there is a $0$-definable map $\phi\colon X_1 \to X_2$ which descends to $\sigma$.
\end{lem}
\begin{proof}
    First, precomposition with the quotient map $h_{X_1} \to F_1$ induces a map $\sigma'\colon h_{X_1} \to F_2$. By the Yoneda lemma, this corresponds to an element in the set $F_2(X_1)$. Then since the quotient map $h_{X_2} \to F_2$ is a surjection of presheaves, $\sigma'$ lifts to a map from $h_{X_1}$ to $h_{X_2}$, which corresponds to a $0$-definable map from $X_1$ to $X_2$ by the Yoneda lemma.
\end{proof}

\begin{ex}
    Consider the theory of a (one-sorted) structure $M$ with an equivalence relation $E$ dividing $M$ into two infinite classes. Observe that there is an automorphism of the corresponding interpretable set $M_E$ which sends a representative of one class to the representative of the other. However we will show that the presheaf quotient of $M$ by $E$ has no nontrivial endomorphisms.

    First observe that there are no $0$-definable maps from $M$ to $M$ other than the identity. Indeed, given any $0$-definable map $f$ form $M$ to $M$, suppose for some $x \in M$ that $f(x) \ne x$. Then given any other $y$ in the same equivalence class as $f(x)$, we may find an automorphism of $M$ which fixes $x$ and sends $f(x)$ to $y$, giving a contradiction. Now the statement follows from the previous lemma.
\end{ex}

One might hope that the above example could be fixed by adding parameters (and then defining a map which is constant on each equivalence class), but the next example shows that in general there are $0$-interpretable endomorphisms of interpretable sets which do not correspond to definable endomorphisms, even if we consider endomorphisms definable with parameters.
\begin{ex}
    Consider a structure $M$ with two equivalence relations $E_1$, $E_2$, where $E_2$ has infinitely many infinite classes and $E_1$ has two infinite classes, each of which divide each $E_2$ class into two infinite partitions. Let $E_3$ be the refinement of $E_1$ and $E_2$. Then there is a $0$-interpretable endomorphism $\sigma$ of $M_{E_3}$ sending any $E_3$-class to the unique disjoint $E_3$-class with which it is identified under $E_2$. However, arguing as above, we see that there are no parameter-definable maps from $M$ to $M$ which descend to $\sigma$.
\end{ex}

\begin{dfn}\label{Grothendieck topology on type-definable sets definition}
    We define a Grothendieck topology on $\Def(\calT)$ and $\typDef(\calT)$ by declaring that a cover of a $0$(-type)-definable set $X$ is a single $0$-definable surjection $X' \to X$.
\end{dfn}

\begin{rem}
    The Grothendieck topology defined above is a special case of the regular coverage of a regular category; see \cite{nlab:regular_coverage} or \cite{sketches_of_an_elephant}. One could alternatively consider the coherent topology, where one takes covers to be finite collections of jointly surjective definable maps. This is often used to study the category of definable sets of a theory; see for example \cite[Chapter~10]{Mac_Lane_Moerdijk_1994}. We believe that the results of this section should still go through using this topology.
\end{rem}

One can verify that
\begin{lem}
    In both $\Def(\calT)$ and $\typDef(\calT)$, the above definition satisfies the axioms of a Grothendieck topology. Moreover, representable presheaves are sheaves.
\end{lem}
\begin{proof}
    Indeed, identity maps are $0$-definable surjections, and a composition of $0$-definable surjections is a $0$-definable surjection. This gives the identity and local character axioms. For stability under base change, note that for any $0$-definable surjection $Y \to X$ in either category, and any $0$-definable map $T \to X$, the base change map $Y\times_X T \to T$ is definable and surjective (indeed, given any $t \in T$, let $y$ be an element of the preimage in $Y$ of the image of $t$ in $X$. Then $(y,t) \in Y\times_X T$.
\end{proof}
\begin{rem}
    Applying the above definition to the theory $\calT^{eq}$ we get Grothendieck topologies on the categories $\typInt(\calT) = \typDef(\calTeq)$ and  $\Int(\calT)=\Def(\calTeq)$.
\end{rem}

\begin{rem}\label{sheaves on typInt are sheaves everywhere}
    It follows immediately from the definition that the embedding $\Def(\calT) \hookrightarrow \typDef(\calT)$ is cover-preserving. It follows from \cref{type-definable sets are type-interpretable} and \cref{0 interpretable in home sorts are 0 definable} that the embeddings $\Def(\calT) \hookrightarrow \Int(\calT)$ and $\typDef(\calT) \hookrightarrow \typInt(\calT)$ are also cover-preserving.

    In particular, sheaves on $\typInt(\calT)$ induce sheaves on $\Int(\calT)$, $\typDef(\calT)$, and $\Def(\calT)$.
\end{rem}

\begin{lem}\label{representable presheaves are sheaves in typDef}
    Representable presheaves are sheaves in $\typInt(\calT)$, $\Int(\calT)$, $\typDef(\calT)$, and $\Def(\calT)$
\end{lem}
\begin{proof}

    By \cref{sheaves on typInt are sheaves everywhere}, it is enough to verify this on $\typInt(\calT)$. Fix a $\abs{T}^+$-saturated model $M$ of $T$ and a corresponding embedding of $\typInt(\calT)$ into the category of sets.

    Given a $0$-type-interpretable set $X$ and a $0$-interpretable surjection $T' \to T$, any $0$-interpretable map $t' \in X(T')$ which equalizes the projections $\begin{tikzcd}
            T'\times_T T' \arrow[r, shift left] \arrow[r] & T
        \end{tikzcd}$
    descends uniquely to a function (in the category of sets) from $T$ to $X$ which sends any $t\in T$ to the image in $X$ of any element of its preimage in $T'$. This function is $0$-interpretable.
\end{proof}

\begin{lem}\label{sheafification identification lemma}
    For any type-definable set $X$, and definable equivalence relation $E$ on $X$, the presheaf quotient $(h_X/h_E)$ is separated. In particular, the sheafification is $(h_X/h_E)^+$.
\end{lem}
\begin{proof}
    Denote $F = h_X/h_E$. Let $T' \to T$ be a definable surjection. Take any $s_1,s_2 \in F(T)$. Suppose $s_1$ and $s_2$ become equal after restriction to $T'$; we need to show $s_1 = s_2$. Let $t_1, t_2 \in X(T)$ be lifts of $s_1$ and $s_2$ respectively, and let $t_1'$, $t_2'$ be the restrictions to $T'$. We need to show $t_1 \sim_{E(T)} t_2$, or equivalently $t_1(b) \sim_E t_2(b)$ for any $b \in T$. By assumption we have $t_1'\sim_{E(T')}t_2'$, or equivalently for every element $a \in T'$ we have $t_1'(a) \sim_E t_2'(a)$.  So take any $b \in T$, and lift it to an element $a \in T'$. Then following the definitions we have $t_1(b) = t_1'(a) \sim_E t_2'(a) = t_2(b)$, as required.

    The statement about the sheafification then follows from \cref{properties of the plus operation}.
\end{proof}





\begin{dfn}
    We define the functor
    \begin{align*}
        \th\colon \typInt(\calT) \to \Sh(\typDef(\calT)); \quad D \to \th_D
    \end{align*}
    to be the composition of the Yoneda embedding on $\typInt(\calT)$ with the pullback to sheaves on $\typDef(\calT)$.

    In particular, for $D \in \typInt(\calT)$ and $T\in \typDef(\calT)$ we have that $\th_D(T)$ is the set of $0$-interpretable maps from $T$ to $D$; i.e., following \cref{Yoneda Notation}, $\th_D(T) = D(T)$. For any $0$-interpretable map $f\colon D_1 \to D_2$, the restriction $\th_{f}\colon \th_{D_1} \to \th_{D_2}$ is given by postcomposition with $f$, i.e., for any $\phi \in D_1(T)$, $\th_{f}(\phi) = f(\phi)$.

\end{dfn}

\begin{rem}\label{interpretation of th on definable sets}
    Observe that for $X \in \typDef(\calT)$ we have by \cref{type-definable sets are type-interpretable} that $\th_{\iota(X)} = h_X$.
\end{rem}

\begin{prop}
    The functor $\th$ is fully faithful.
\end{prop}
\begin{proof}
    Let us first show that $\th$ is faithful. Suppose $\phi_1,$ $\phi_2$ are two distinct morphisms between $0$-type-interpretable sets $D_1$, $D_2$. By \cref{remarks: model-theoretic quotients of type-definable sets}, we may find a $0$-interpretable morphism $\pi\in D_1(X)$ exhibiting $D_1$ as a model-theoretic quotient of a $0$-type-definable set $X$ by a $0$-definable equivalence relation $E$. By definition, for each $i$ we have $\th_{\phi_i}(\pi) = \phi_i\circ \pi$. Since $\phi_1 \ne \phi_2$ and $\pi$ is surjective, it follows that $\phi_1 \circ \pi \ne \phi_2 \circ \pi$, so $\th_{\phi_1} \ne \th_{\phi_2}$ as required.

    Now let us show that $\th$ is full. Let $D_1$, $D_2$, $X$, and $E$ be as above and suppose we have a morphism of presheaves $f\colon \th_{D_1} \to \th_{D_2}$. We need to show that there exists a $0$-interpretable morphism $\phi\colon D_1 \to D_2$ with $\th_{\phi} = f$.

    First we show that $f$ commutes with precomposition:
    \begin{claim}
        For any $0$-definable map $\psi\colon S \to T$ and $t \in D_1(T)$ we have $f(t \circ \psi) = f(t) \circ \psi$ as elements of $D_1(S)$.
    \end{claim}
    \begin{proof}
        By definition, for each $i$ and $t' \in D_i(T)$ we have $\th_{D_i}(\psi)(t') = D_i(\psi)(t') = t'\circ \psi$. Then, since $f$ is a morphism of presheaves, we have
        \begin{align*}
            f(t) \circ \psi = \th_{D_2}(\psi)(f(t)) = f(\th_{D_1}(\psi)(t)) = f(t \circ \psi).
        \end{align*}
    \end{proof}

    We use the above lemma to produce a candidate for $\phi$. Consider the induced diagram shown below, where $p_1,p_2$ denote the projections from $E$ to $X$. Observe that $\pi\circ p_1 = \pi\circ p_2$ by definition, so by the claim we have
    \begin{align*}
        f(\pi)\circ p_1 = f(\pi\circ p_1 ) = f(\pi \circ p_2) = f(\pi)\circ p_2.
    \end{align*}
    Since model-theoretic quotients are category-theoretic quotients (\cref{remarks: model-theoretic quotients of type-definable sets}), it follows that there is a unique $0$-interpretable map $\phi\colon D_1 \to D_2$ such that $f(\pi) = \phi \circ \pi$ as shown.
    \[
        \begin{tikzcd}
            E \arrow[r, "p_1", shift left] \arrow[r, "p_2"'] & X \arrow[r, "\pi"] \arrow[dr,"f(\pi)"'] & D_1 \ar[d,dashed, "\phi"]\\
            &&D_2.
        \end{tikzcd}
    \]

    Finally we show that $f = \th_\phi$, i.e., that for any $0$-type-definable set $T$ and $0$-interpretable map $t\colon T \to D_1$ we have $f(t) = \phi\circ t$.

    Consider the fiber product $ X \times_{\pi,D_1,t} T$. This is a type-interpretable subset of the home sorts, and therefore type-definable by \cref{0 interpretable in home sorts are 0 definable}. Extending the diagram above, we have the diagram below where the square is Cartesian and the lower triangle commutes:
    \[ 
        \begin{tikzcd}
            & {X\times_{\pi,D_1,t}T} \arrow[r, "q_T"] \arrow[d, "q_X"'] & T \arrow[d, "t"] \arrow[dd, "f(t)", bend left] \\
            E \arrow[r, "p_2"'] \arrow[r, "p_1", shift left] & X \arrow[r, "\pi"] \arrow[rd, "f(\pi)"']                   & D_1 \arrow[d, "\phi"']                          \\
            &                                                           & D_2
        \end{tikzcd}\]
    Again, our goal is to show $f(t) = \phi\circ t$. Since $\pi$ is surjective we have that $q_T$ is surjective, so it is enough to show
    \begin{align*}
        f(t) \circ q_T = \phi \circ t \circ q_T.
    \end{align*}
    But by the claim we have
    \begin{align*}
        f(t) \circ q_T = f(t \circ q_T) = f(\pi \circ q_X) = f(\pi) \circ q_X = \phi \circ \pi \circ q_X = \phi \circ t \circ q_T
    \end{align*}
    as required.
\end{proof}

Next we show that the embedding $\th$ essentially surjects onto the subcategory of sheaves which are sheaf quotients of type-definable sets with definable equivalence relations.

\begin{prop}
    Let $X \in \typDef(\calT)$ and $E$ be a $0$-definable equivalence relation on $X$. Let $D \in \typInt(\calT)$ and suppose $\pi_E\colon X \to D$ exhibits $D$ as a model-theoretic quotient of $X$ by $E$.
    Then $\th_D \simeq (h_X/h_E)^+$.
\end{prop}
\begin{proof}
    By \cref{locally surjective morphism of presheaves leades to isomorphism of sheaves}, it is enough to show that we have a locally surjective embedding of presheaves $h_X/h_E \hookrightarrow \th_D$.

    Recall that $\th_{\iota(X)} = h_X$ (\cref{interpretation of th on definable sets}) and consider the morphism $\th_{\pi_E}\colon h_X \to \th_D$. By definition, for any $t_1,t_2 \in X(T)$ we have $\pi_E \circ t_1 = \pi_E \circ t_2$ if and only if we have $(t_1,t_2) \in E(T)$. In particular, the morphism $\th_{\pi_E}$ descends to an embedding $\eta\colon h_X/h_E \hookrightarrow \th_D$. We claim that $\eta$ is locally surjective. Indeed, given any $t \in D(T)$ we can form a Cartesian diagram
    \[
        \begin{tikzcd}
            {X\times_{\pi_E,D,t}T} \arrow[r, "q_T", two heads] \arrow[d, "q_X"] & T \arrow[d, "t"] \\
            X \arrow[r, "\pi_E", two heads]                                        & D
        \end{tikzcd}
    \]
    and by construction (recalling \cref{equivlaence classes notation}) we have $\eta([q_X]_{E(X\times_D T)}) = \pi_E\circ q_X$. The latter is equal to the restriction of $t$ to $X\times_D T$.

\end{proof}

\begin{cor}\label{equivalence of interpretable sets and presheaf quotients}
    \begin{enumerate}
        \item     There is an equivalence of categories $j$ from the category of $0$-type-interpretable sets $\typInt(\calT)$ to the category of sheaves on $\typDef(\calT)$ which are sheaf quotients of $0$-type-definable sets with $0$-definable equivalence relations, such that if a $0$-type-interpretable set $D$ is a model-theoretic quotient of a $0$-type-definable set $X$ by a $0$-definable equivalence relation $E$ on $X$ then $j(D)$ is naturally isomorphic to the sheaf quotient $(h_X/h_E)^+$.
        \item  There is an equivalence of categories $j$ from the category of $0$-interpretable sets $\Int(\calT)$ to the category of sheaves on $\Def(\calT)$ which are sheaf quotients of $0$-definable sets with $0$-definable equivalence relations, such that if a $0$-interpretable set $D$ is a model-theoretic quotient of a $0$-definable set $X$ by a $0$-definable equivalence relation $E$ on $X$ then $j(D)$ is naturally isomorphic to the sheaf quotient $(h_X/h_E)^+$.
    \end{enumerate}
\end{cor}
\begin{proof}
    Both statements follow from the previous two results on setting $j = \th$, since every $0$-(type-)interpretable set is a model-theoretic quotient by \cref{Every interpretable set is a quotient} and \cref{remarks: model-theoretic quotients of type-definable sets}.
\end{proof}

%% file: Chapters/First-order_structures/Group_Chunks_on_Definable_types.tex
\section{Group Chunks on Definable Types}\label{section: group chunks on definable types}
In this section we show how the Group Chunk Theorem of Hrushovski and Rideau-Kikuchi \cite{Hrushovski2019} can be derived from our work in previous sections.

As in \cite{Hrushovski2019}, we fix a large cardinal $\kappa$ and a $\kappa$-saturated strongly $\kappa$-homogeneous model $\bU$ of $\calT$. By a \emph{global partial type} we mean a partial type over $\bU$ (or a \emph{filter}, in the terminology of \cite{Hrushovski2019}). A \emph{small} set will be a subset of $\bU$ of size less than $\kappa$.

Recall the notions of a definable partial type (\cite[2.1]{Hrushovski2019}), the push-forward $f_*(p)$ of a definable partial type $p$ by a definable function $f$ (\cite[Definition~2.1]{Hrushovski2019}), and the Morley product of definable types (\cite[Definition~2.2]{Hrushovski2019}).

Note the following easy lemma:
\begin{lem}\label{restriction can be taken over definably closed}
    Let $p$ be a global partial type definable over a small set of parameters $C$, $A$ a small set, and $d\in \bU$. Suppose $d \models p|_{CA}$. Then $d \models p|_{\dcl(CA)}$.
\end{lem}
\begin{proof}
    Indeed, take any tuple $b \in \dcl(C,A)$, and suppose $b = \psi(c,a)$ for some definable map $\psi$ with $c\in C$ and $a \in A$. Then for any formula $\phi(x,b) \in p|_{Cb}$, we must have $\phi(x,\psi(c,a)) \in p|_{Cac}$ (otherwise $p$ is inconsistent) so $d \models p|_{Cac}$ implies $d \models p|_{Cb}$ as required.
\end{proof}

\begin{thm}[{\cite[Proposition~3.15, Proposition~3.16]{Hrushovski2019}}]\label{statement of model theoretic group chunk theorem}
    Let $p$ be a global partial type definable over a small set of parameters $C$. Let $(\cdot)$ be a $C$-definable map from $p^{\otimes 2}$ to $p$. Assume $(\cdot)$ is:
    \begin{enumerate}
        \item \textbf{Generically Cancellative}: For $a,b\models p^{\otimes 2}|_C$ we have $b\in \dcl(a,a\cdot b, C)$ and $a \in \dcl(b,a\cdot b, C)$.
        \item \textbf{Generically Associative}: For $a,b,c \models p^{\otimes 3}|_C$, we have $(a\cdot b)\cdot c = a\cdot (b\cdot c)$.
        \item \textbf{Generically Transitive}: For $a \models p|_C$, we have $(\lambda_a)_* p = p$, where $\lambda_a$ denotes the left multiplication by $a$.
    \end{enumerate}
    Let $X = p|_C(\bU)$. Then there exists a universal $C$-definable embedding from $X$ into a $C$-type-interpretable group $G$ which respects the multiplication operation.
\end{thm}

Before we begin the proof, we give an easy model-theoretic lemma which shows the purpose of the generic transitivity assumption:
\begin{lem}\label{consequences of transitivity}
    Let $B$ be a small set, and suppose $(a,c) \models p^{\otimes 2}|_{BC}$. Then
    \begin{enumerate}
        \item Suppose $a\cdot c = d$. Then $d \models p|_{BCa}$.

        \item There exists $b$ such that $b \models p|_{BCa}$ and $a\cdot b = c$.
    \end{enumerate}
\end{lem}
\begin{proof}
    For the first statement, we have $c \models p|_{CBa}$ by assumption, so $d \models ((\lambda_a)_*p)|_{CBa}$. But $p|_{CBa} = ((\lambda_a)_*p)|_{CBa}$ by generic transitivity.

    For the second statement, suppose $a\cdot c = d$. Then since $c \models p|_{BCa}$ we have $d \models ((\lambda_a)_*p)|_{BCa}$. Since $p = (\lambda_a)_*(p)$  we have that $c$ and $d$ have the same type over $CBa$. Therefore there exists an automorphism $f$ of $\bU$ over $CBa$ which sends $d$ to $c$, and we can take $b = f(c)$.
\end{proof}



Next we define the induced presheaf of partial magmas and state its basic properties:
\begin{dfn}
    Let $p,(\cdot), X$ be as in the statement of \cref{statement of model theoretic group chunk theorem}.

    \begin{enumerate}
        \item  Denote by $\calT_C$ the theory obtained by labeling the elements of $C$ in $\bU$. Note that $C$-type-interpretable subsets of $\bU$, where $\bU$ is considered as a model of $\calT$, correspond to $0$-type-interpretable subsets of $\bU$ considered as a model of $\calT_C$.

        \item We endow $\typDef(\calT_C)$ with the structure of a site as in \cref{Grothendieck topology on type-definable sets definition}. Note that by \cref{representable presheaves are sheaves in typDef} the presheaf $h_X$ is a sheaf on $\typDef(\calT_C)$.

        \item  For each $n$, let $X^{\otimes n}$ denote $p^{\otimes n}|_C(\bU)$.

        \item   Observe that $(\cdot)$ induces a partial morphism of sheaves $\mu\colon h_X\times h_X \to h_X$ with domain equal to $h_{X^{\otimes 2}}$ (by \cref{Yoneda induces an embedding of partial morphisms}) such that for any $(x_1,x_2) \in X^{\otimes 2}(T)$ and $t\in T$ we have $\mu(x_1,x_2)(t) = x_1(t)\cdot x_2(t)$.
    \end{enumerate}
\end{dfn}
We will show that the sheaf of partial magmas $(h_X,\mu)$ defined above is a group chunk in the sense of \cref{section: group chunks}, and that the universal sheaf of groups given by \cref{group chunks to sheaves left adjoint} is representable as a $C$-type-interpretable set.

\begin{lem}\label{model theory local nontriviality result}
    The sheaf $h_X$ is locally nontrivial.
\end{lem}
\begin{proof}
    Indeed, for any type-definable set $T$ we have for example the definable projection $T\times X \to X$.
\end{proof}

\begin{lem}\label{model theory cancellation axiom}
    The presheaf of partial magmas $(h_X, \mu)$ defined above is cancellative.
\end{lem}
\begin{proof}
    Indeed, for any $0$-type-definable $T$ and $x_1,x_2, x_3 \in X^3(T)$ with $x_1\cdot x_2 = x_3$ we have by generic cancellativity that for each $t \in T$ that there exists a unique $t'$ such that $x_1(t) \cdot t' = x_3(t)$, so the map $x_2$ is uniquely determined by $x_1$ and $x_3$; similarly $x_1$ is determined by $x_2, x_3$.
\end{proof}

Now let us consider the induced left translation maps. Recall \cref{lambda notation}, \cref{lambda notation} and \cref{definition of tX}.
\begin{rem}\label{model theory lambda domain}
    For any $x\in X(T)$ we have by definition of $\lambda_x$ that
    \begin{align*}
        \dom \lambda_x = h_{X^{\otimes 2}} \times_{p_1,h_X, x} h_T,
    \end{align*}
    and in particular for any $0$-definable map $T' \to T$ we have
    \begin{align*}
        \dom \lambda_x(T' \to T) = \{y\in X(T') : \text{ for all } t \in T' \text{ we have } y(t) \models p|_{Cx|_{T'}(t)} \}.
    \end{align*}
\end{rem}

\begin{lem}
    Let $x,z \in X(T)$. Then there exists $y \in \dom \lambda_x(T)$ with $x(y) = z$ if and only if for every $t \in T$ we have $z(t) \models p|_{Cx(t)}$.
\end{lem}
\begin{proof}
    Indeed, if there exists such a $y$, then for every $t \in T$ we have $y(t) \models p|_{Cx(t)}$ and then $z(t) \models p|_{Cx(t)}$ by \cref{consequences of transitivity}. Conversely, one can give a definable map $y$ from $T$ to $X$ by declaring for each $t \in T$ that $y(t)$ is the unique $b \in X$ such that $x(t) \cdot b = z(t)$; the existence and uniqueness of such a $b$ is guaranteed by \cref{consequences of transitivity} and generic cancellativity.
\end{proof}

From the above lemma and by definition of $\lambda_x^\dagger$ we immediately get:
\begin{cor}\label{model theory lambda dagger domain}
    For any $x\in X(T)$ and $0$-definable map $T' \to T$ we have
    \begin{align*}
        \dom \lambda_x^\dagger(T' \to T) = \dom \lambda_x(T' \to T) = \{y\in X(T') : \text{ for all } t \in T' \text{ we have } y(t) \models p|_{Cx|_{T'}(t)} \}.
    \end{align*}
\end{cor}

\begin{lem}\label{independence implies products are defined}
    Let $(x_1,e_1),\dots,(x_{n},e_{n})\in (X \times \{\pm1\})(T)$, and $z \in X(T)$. Suppose that for every $t\in T$ we have $z(t) \models p|_{Cx_1(t),\dots,x_n(t)}$. Then we have
    \begin{align*}
        z \in (\dom \lambda_{x_1}^{e_1} \circ \dots \circ \lambda_{x_n}^{e_n})(T).
    \end{align*}
\end{lem}
\begin{proof}
    Indeed, for every $t\in T$ we have by \cref{model theory lambda domain} and \cref{model theory lambda dagger domain} that $z \in \dom \lambda_{x_n}^{e_n}$, and for each $t \in T$ we have by \cref{consequences of transitivity} that
    $\lambda_{x_n}^{e_n}(z)(t) \models p|_{Cx_1(t),\dots,x_{n - 1}(t)}$, so we get the result by induction.
\end{proof}

\begin{lem}\label{independent elements can be multiplied before}
    Let $(x,y,z) \in X^{\otimes 3}(T)$. Then for $e \in \{\pm1\}$ if we denote $\alpha = \lambda_{x}^{e}$ we have
    \begin{align*}
        z \in (\dom \lambda_{\alpha(y)})(T).
    \end{align*}
\end{lem}
\begin{proof}
    Indeed, it follows immediately from \cref{restriction can be taken over definably closed} that for $t\in T$ that we have $z(t) \models p|_{Cx(t)y(t)}$ implies $z(t) \models p|_{C\alpha(y)(t)}$.
\end{proof}

\begin{lem}\label{model-theoretic generic strong associativity}
    For any $a,b,y \in X^{\otimes 3}(T)$ and $e\in \{\pm1\}$, let $\alpha = \lambda_{a}^{e}$. Then $y\in \dom \alpha \circ \lambda_b \cap \dom \lambda_{\alpha(b)}$ and $\alpha(b(y)) = \alpha(b)(y)$.
\end{lem}
\begin{proof}
    The fact that $y$ lies in the relevant domains follows from \cref{independence implies products are defined} and \cref{independent elements can be multiplied before}.
    For the case $e = 1$, the generic associativity assumption immediately gives $a(b(y)) = a(b)(y)$.

    For the case $e = -1$ we will use \cref{strong associativity follows form weak property}. By \cref{consequences of transitivity} we have $(a,a^\dagger(b)) \in X^{\otimes 2}(T)$ and then by \cref{restriction can be taken over definably closed} it follows that $(a,a^\dagger(b),y) \in X^{\otimes 3}$; it follows by the $e = 1$ case that $a^\dagger(b)(y) \in \dom \lambda_a$ and $a(a^\dagger(b)(y)) = a(a^\dagger(b))(y)$, so we conclude by \cref{strong associativity follows form weak property}.
\end{proof}

With the above results in hand we can easily prove \cref{statement of model theoretic group chunk theorem}:
\begin{proof}[Proof of \cref{statement of model theoretic group chunk theorem}]

    We know that $h_X$ is separated (\cref{representable presheaves are sheaves in typDef}), locally nontrivial (\cref{model theory local nontriviality result}), and cancellative (\cref{model theory cancellation axiom}).

    Let us verify that the domain is large (\cref{large domain definition}). Indeed, given any $T\in \typDef(\calT)$ and pairs $(x_1,e_1),\dots,(x_{n},e_{n})\in (X \times \{\pm1\})(T)$, consider the type-definable set
    \begin{align*}
        S := \{(t,s)\in T\times X : s \models p|_{x_1(t)\dots x_n(t)}\},
    \end{align*}
    together with the projection $q_X\colon S \to X$. Observe that $S$ is a cover of $T$ (since for any $t_1,\dots t_n$, the type $ p|_{x_1(t)\dots x_n(t)}$ is consistent), and by \cref{independence implies products are defined} we have
    \begin{align*}
        q_X \in \bigcap_{i,j \in\{1,\dots,n\}} \dom \lambda_{x_i}^{e_i} \circ \lambda_{x_j}^{e_j}.
    \end{align*}

    Next we show that $(X,\mu)$ is strongly associative (\cref{strong associativity definition}). Let $(a,b) \in X^{\otimes 2}(T)$, $e \in \{\pm1\}$, $\alpha = \lambda_{x}^{e}$ and $y \in (\dom \alpha \circ \lambda_b \cap \dom \lambda_{\alpha(b)})(T)$, we need to show $\alpha(b(y)) = \alpha(b)(y)$. Consider the type-definable set
    \begin{align*}
        S := \{(t,s)\in T\times X : s \models p|_{a(t)b(t)c(t)}\},
    \end{align*}
    together with the projection $z\colon S \to X$. We verify that $z$ witnesses the required properties from \cref{associativity follows from weak property}. Indeed, writing $a,b,y$ for $a|_{S}, b|_S, c|_S$ and checking for each $(t,s) \in S$ we see that we have $(b,y,z) \in X^{\otimes 3}(S)$. It follows from \cref{restriction can be taken over definably closed} that also $(a,b(y),z), (a(b),y,z) \in X^{\otimes 3}(S)$ and $(a(b(y)),z), (a(b)(y),z) \in X^{\otimes 2}(T)$. By \cref{consequences of transitivity} we also have that $(a,b,y(z)) \in X^{\otimes 3}(S)$. Then it follows from \cref{independence implies products are defined} that $z$ is in the domain of all the relevant partial functions, and from \cref{model-theoretic generic strong associativity} that all the relevant identities hold, so we conclude by \cref{associativity follows from weak property} that $\alpha(b(y)) = \alpha(b)(y)$ as required.

    Finally, observe that the equivalence relation $\ssim$ on $(h_X)^2$ defined in \cref{X2 ssim definition} is representable as a definable equivalence relation $X$ (which we also denote by $\ssim$). Since $h_X$ is locally nontrivial, it follows (\cref{local nontrivial X sheafification remark}) that the associated sheaf of groups is representable as an interpretable set by a model-theoretic quotient $X^2/\ssim$. The result follows.
\end{proof}

%% file: Chapters/Scheme-Theoretic_Preliminaries/Main_Schemes.tex
\chapter{Scheme-Theoretic Preliminaries}\label{Scheme-Theoretic Preliminaries chapter}
In this chapter we collect some results from scheme theory which we will need later. This chapter is probably best skipped on first reading and referred back to as needed, though readers unfamiliar with the notion of $S$-density (or universal schematic density relative to $S$, in the parlance of \cite{EGA}) may find it beneficial to look through \cref{section: S-dominance and density} and subsequent sections.

Our basic reference is the Stacks Project which the reader should consult for missing definitions and derivations.

Let us recall the following result for later reference:
\begin{prop}[{\cite[\href{https://stacks.math.columbia.edu/tag/02FV}{Tag 02FV}]{stacks-project}}]\label{composition criteria for being locally of finite presentation}
    Let $f\colon X \to Y$ be a morphism of schemes over a scheme $S$. If $X$ is locally of finite presentation over $S$, and $Y$ is locally of finite type over $S$, then $f$ is locally of finite presentation.
\end{prop}

\input{Chapters/Scheme-Theoretic_Preliminaries/Flat_Morphisms.tex}

\input{Chapters/Scheme-Theoretic_Preliminaries/Schematic_Images.tex}
\input{Chapters/Scheme-Theoretic_Preliminaries/Hilbert_Schemes.tex}
\input{Chapters/Scheme-Theoretic_Preliminaries/S-Dominance_and_S-Density.tex}

\input{Chapters/Scheme-Theoretic_Preliminaries/Schematic_Dominance_and_Schematic_Image.tex}

\input{Chapters/Scheme-Theoretic_Preliminaries/S-Dominance_and_Flatness.tex}
\input{Chapters/Scheme-Theoretic_Preliminaries/Dominance_vs_Schematic_Dominance_vs_S-Dominance.tex}

%% file: Chapters/Scheme-Theoretic_Preliminaries/Flat_Morphisms.tex
\section{Flat Morphisms}
The following two results are immediate from the definition of flatness and faithful flatness. Recall \cref{restrictions of morphisms}.
\begin{lem}\label{restrictions of flat morphisms to opens are flat}
    Let $\phi\colon X \to Y$ be a flat morphism. Then for any opens $V \subset Y$ and $U \subset \phi\inv(V)$, the restriction $\phi|_{[U \to V]}$ is flat.
\end{lem}

\begin{lem}\label{restrictions of faithfully flat morphisms to opens are faithfully flat}
    Let $\phi\colon X \to Y$ be a faithfully flat morphism. Then for any open $V \subset Y$ the restriction $\phi|_{[\phi\inv(V) \to V]}$ is faithfully flat.
\end{lem}

\begin{lem}[{\cite[\href{https://stacks.math.columbia.edu/tag/01UA}{Tag 01UA}]{stacks-project}}]\label{a flat morphism of schemes which is locally of finite presentation is universally open}
    A flat morphism of schemes which is locally of finite presentation is universally open.
\end{lem}

\begin{lem}[{\cite[\href{https://stacks.math.columbia.edu/tag/02VW}{Tag 02VW}]{stacks-project}}]\label{A faithfully flat morphism is an epimorphism of schemes}
    A faithfully flat morphism is an epimorphism of schemes
\end{lem}

\begin{lem}[{\cite[\href{https://stacks.math.columbia.edu/tag/02JZ}{Tag 02JZ}]{stacks-project}}]\label{if f is faithfully flat and g circ f is flat then g is flat}
    Let $f\colon X \to Y$ and $g\colon Y \to Z$ be morphisms of schemes. Suppose $f$ is faithfully flat and $g\circ f$ is flat. Then $g$ is flat.
\end{lem}

Recall that schemes are sheaves for various flat topologies:
\begin{lem}[{\cite[\href{https://stacks.math.columbia.edu/tag/023P}{Tag 023P}]{stacks-project}}]
    Schemes are sheaves for the fpqc (faithfully flat and quasicompact) and fppf (faithfully flat and finite presentation) topologies.
\end{lem}
Moreover,
\begin{lem}\label{fppf morphisms are epimorphisms on the fppf topology}
    Let $f\colon X \to Y$ be an fpqc/fppf morphism of schemes. Then $f$ is an epimorphism of fpqc/fppf sheaves, respectively.
\end{lem}
\begin{proof}
    Indeed, given any morphism $y\colon T \to Y$, we have an fpqc/fppf cover $X\times_Y T \to T$ and the pullback $X\times_Y T \to X$ lifts the restriction $X\times_Y T \to Y$.
\end{proof}

The following result will be important:
\begin{prop}\label{construction of morphism from flat descent}
    Let $\phi\colon X \to Z$ be a morphism of schemes with graph $\Gamma_\phi$. Let $\zeta\colon X \to Y$ be an fpqc morphism of schemes. Suppose there exists a closed subscheme $D\subset Y\times Z$ such that $\zeta^*(D) = \Gamma_\phi$. Then there is a unique morphism $\tilde{\phi}\colon Y \to Z$ with graph $D$ such that $\tilde{\phi}\circ \zeta = \phi$.
\end{prop}

\begin{proof}
    Since morphisms are determined uniquely by their graphs, it is enough to show that $D$ is the graph of a morphism, or equivalently that the projection $p_1\colon D \to Y$ is an isomorphism (\cref{graph of morphis criterion}). Since schemes are fpqc sheaves this can be checked fpqc locally (see for example \cite[\href{https://stacks.math.columbia.edu/tag/02L4}{Tag 02L4}]{stacks-project}) Since by assumption the projection $q_1\colon \zeta^*(D) \to X$ is an isomorphism, and $\zeta$ is fpqc, we conclude.
\end{proof}

\begin{prop}[{Generic Flatness, \cite[\href{https://stacks.math.columbia.edu/tag/052B}{Tag 052B}]{stacks-project}}] \label{Generic Flatness theorem}
    Let $\pi\colon X \to Y$ be a finite type morphism of schemes with $Y$ reduced, and let $\F$ be a finite type quasicoherent sheaf on X. Then there is a dense open subscheme $U \subset Y$ such that $\F|_{\pi\inv(U)}$ is of finite presentation and flat over $U$.
\end{prop}

%% file: Chapters/Scheme-Theoretic_Preliminaries/Schematic_Images.tex
\section{Schematic Images}
If $X$ is a scheme, we denote by $\abs{X}$ the underlying topological space.

Let $\phi\colon X \to Y$ be a morphism of schemes. We denote the set-theoretic image of $\phi$ by $\Ima(\phi)$ or $\phi(X)$. We denote the scheme-theoretic or schematic image of a morphism $\phi$ by $\Schim_Y(\phi)$ or $\Schim(\phi)$, or sometimes $\Schim_Y(X)$ if $\phi$ is clear from context. Recall that $\Schim(\phi)$ is the smallest closed subscheme of $Y$ through which $\phi$ factors; this corresponds to the largest quasicoherent subsheaf of ideals of $\ker(\O_Y \to \phi_* \O_x)$.

Set-theoretically we have $\overline{\phi(X)} \subset \abs{\Schim(\phi)}$, but we do not have equality in general.

\begin{lem}\label{Graphs and schematic images}
    Let $\phi\colon X \to Y$ be a morphism of schemes. Then the graph $\Gamma_\phi$ of $\phi$ is a locally closed subscheme of the schematic image of the induced morphism $(id_X, \phi)\colon X \to X\times Y$. If $Y$ is separated, we have $\Gamma_\phi = \Schim(id_X, \phi)$.
\end{lem}
\begin{proof}
    By \cref{Diagonal cartesian square for graphs} we have a commutative diagram
    \[
        \begin{tikzcd}
            X \arrow[rd, dashed] \arrow[rdd, "{(\id_X,\phi)}", bend right] \arrow[rrd, "\phi", bend left] &                         &   \\ & \Gamma_\phi \arrow[r] \arrow[d, hook]     & Y \arrow[d, "\Delta_Y", hook] \\
            & X \times Y \arrow[r, "\phi \times \id_Y"] & Y \times Y
        \end{tikzcd}
    \]
    where the square is Cartesian and the dashed arrow is an isomorphism. It follows that $\Schim(id_X,\phi) = \Schim(\Gamma_\phi \hookrightarrow X \times Y)$. Since $\Delta_Y$ is a locally closed immersion, it follows that
    $\Gamma_\phi \hookrightarrow X \times Y$ is also, and in particular that $\Gamma_\phi$ is a locally closed subscheme of $\Schim(\id_X,\phi)$.

    If $Y$ is separated (so $\Delta_Y$ is a closed immersion), it follows that  $\Gamma_\phi \hookrightarrow X \times Y$ is a closed immersion, and in particular $\Gamma_\phi= \Schim(\id_X,\phi)$.
\end{proof}

\begin{prop}\label{quasicoherent kernel lemma}
    Let $\phi\colon X \to Y$ be a morphism of schemes with $X$ reduced or $\phi$ quasicompact. Then the kernel $\calI := \ker(\O_Y \to \phi_* \O_X)$ is quasicoherent.
\end{prop}
See the chapter on scheme-theoretic images in \cite{FOAG}.

\begin{lem}\label{schematic image of morphism with reduced source}
    Let $\phi\colon X \to Y$ be a morphism of schemes with $X$ reduced; then $\Schim(\phi)$ is reduced.
\end{lem}
See \cite[\href{https://stacks.math.columbia.edu/tag/056B}{Tag 056B}]{stacks-project}.

\begin{prop}\label{Schematic image of well-behaved morphisms}
    Let $\phi\colon X \to Y$ be a morphism of schemes with $\ker(\O_Y \to \phi_* \O_X)$ quasicoherent (e.g., by \cref{quasicoherent kernel lemma}, if $\phi$ is quasicompact or $X$ reduced). Then
    \begin{enumerate}
        \item We have $\overline{\phi(X)} = \abs{\Schim(\phi)}$
        \item For any open subscheme $U \subset Y$, the schematic image of $\phi|_{\phi\inv(U)}\colon \phi\inv(U) \to U$ is equal to $U \cap \Schim(\phi)$.
    \end{enumerate}
\end{prop}
See the proof of \cite[\href{https://stacks.math.columbia.edu/tag/01R8}{Tag 01R8}]{stacks-project}.

Taking the schematic image of a quasicompact morphism commutes with flat base change:
\begin{lem}\label{schematic image of a quasicompact morphism commutes with flat base change on target}
    Let $f\colon X \to Y$ be a flat morphism of schemes, and $g\colon V \to Y$ be a quasicompact morphism of schemes. Then
    \begin{align*}
        \Schim(f^*(g)) = f^*(\Schim(g))
    \end{align*}
\end{lem}
See \cite[\href{https://stacks.math.columbia.edu/tag/081I}{Tag 081I}]{stacks-project}.
\begin{cor}\label{schematic image of a quasicompact morphism commutes with flat base change on base}
    Let $f\colon S' \to S$ be a flat morphism of schemes, and $g\colon X \to Y$ a quasicompact morphism of schemes over $S$. Then
    \begin{align*}
        \Schim(f^*(g)) = f^*(\Schim(g)).
    \end{align*}
\end{cor}
\begin{proof}
    Note that the pullback $Y_{S'} \to Y_S$ is flat, and then apply \cref{Cartesian rectangles} and the previous result.
\end{proof}

%% file: Chapters/Scheme-Theoretic_Preliminaries/Hilbert_Schemes.tex
\section{Hilbert Schemes}

\begin{dfn}\label{Hilbert schemes definition}
    Let $S$ be a Noetherian scheme, and let $B$ be a projective $S$-scheme. We denote by $\Hilb_B$ the Hilbert scheme of $B$ over $S$ (see \cite{FGAHilbertSchemes}). This represents the Hilbert functor, which is the presheaf on the category of locally Noetherian schemes over $S$ sending any $T \to S$ to the set of closed subschemes of $B_T$ which are flat over $T$.
\end{dfn}

\begin{lem}\label{Hilbert schemes are disjoint unions of projective schemes}
    Fix a very ample line bundle on $B$. Then $\Hilb_B$ is a disjoint union of projective $S$-schemes $\Hilb^p_B$ where $\Hilb^p_B$ represents the presheaf on the category of locally Noetherian schemes over $S$ sending any $T \to S$ to the set of closed subschemes $Z$ of $B_T$ which are flat over $T$ such that, over every $t \in T$, the fiber $Z_t$ has Hilbert polynomial $p$.
\end{lem}
See \cite{FGAHilbertSchemes} for the proof.

In particular, it follows that $\Hilb_B$ is separated and locally of finite type over $S$.

Recall (\cref{products and Yoneda combine}) that we use concatenation to denote products.
\begin{prop}\label{tau_YX definition}
    Let $X,Y$ be projective schemes over a Noetherian scheme $S$, and consider the permutation $\rho_{YX}\colon XY \to YX$.

    Then the pullback $\rho_{YX}^*$ of closed subschemes of $YX$ induces a morphism of Hilbert Schemes $\tau_{YX}\colon \Hilb_{YX} \to \Hilb_{XY}$ such that, for each locally Noetherian $T \to S$, and $t \in \Hilb_{YX}(T)$ corresponding to a closed subscheme $Z \subset TYX$, the morphism $\tau_{YX}(t)$ corresponds to the closed subscheme $\rho_{YX}^*(Z)$. In particular, we have $\rho_{YX}^*(Z) = t^*(\tau_{YX}^*(\Omega_{YX}))$, where $\Omega_{YX} \subset \Hilb_{YX} \cdot YX$ is the universal closed subscheme.

\end{prop}
\begin{proof}
    To produce a morphism $\tau_{YX}$ it is enough to show that the pullback $\rho_{YX}^*$ induces a morphism of the Hilbert functors; in particular we need to show that pullback of that for any morphism $f\colon T' \to T$ of locally Noetherian schemes over $S$, and any $Z \subset TYX$, we have $f^*(\rho_{YX}^*(Z)) = \rho_{YX}^*(f^*(Z))$. But this follows quickly from the Cartesian cube \cref{cartesian cubes}.
\end{proof}

%% file: Chapters/Scheme-Theoretic_Preliminaries/S-Dominance_and_S-Density.tex
\section{S-Dominance and Density}\label{section: S-dominance and density}
The notion of $S$-dominance will be important for us. The definition can be found in \cite{EGA}, but the terminology ``$S$-dense'' is borrowed from \cite{ER15}.


\begin{dfn}(cf. \cite*[IV, 11.10]{EGA})
    \label{Schematic and S-dominance and density definition}
    Let $X,Y$ be schemes over a base $S$, and $\phi\colon X \to Y$ a morphism over $S$.
    \begin{enumerate}
        \item We say $\phi$ is \emph{schematically dominant} if the corresponding morphism of structure sheaves $\O_Y \to \phi_*\O_X$ is injective.
        \item We say $\phi$ is \emph{$S$-dominant} if it is schematically dominant and remains so after arbitrary base change $S' \to S$ (in \cite[IV, 11.10.8]{EGA}, this property is called ``universal schematic dominance relative to $S$'')
        \item Let $U\hookrightarrow X$ be an open immersion. We say $U$ is \textsl{schematically} or \emph{$S$-dense} in $X$ if the immersion $U\hookrightarrow X$ is schematically or $S$-dominant, respectively.
    \end{enumerate}
    If $S = \Spec(A)$ for a ring $A$, we will say $A$-dominant/dense in place of $\Spec(A)$-dominant/dense.
\end{dfn}
\begin{rem}
    See \cref{section: dominance vs schematic dominance} for a comparison of the notions of dominance and schematic dominance.
\end{rem}


\begin{notation}\label{S-general notation}
    We say that a property $P$ of morphisms $\phi\colon X \to Y$ is stable under restriction to $S$-dense opens if for any $S$-dense $U \subset X$ we have that $P$ holds of $\phi$ if and only if $P$ holds of $\phi|_U$.
\end{notation}
In \cref{chapter: S-rational morphisms} we will define the notion of an $S$-rational morphism. Properties of morphisms which are stable under restriction to $S$-dense opens induce properties of $S$-rational morphisms.

Schematically dominant morphisms are epimorphisms in the category of separated schemes:
\begin{prop}
    Let $f\colon X \to Y$ be a schematically dominant morphism of schemes over $S$. Let $g_1,g_2\colon Y \to Z$ be two morphisms over $S$, and suppose $Z$ is separated over $S$. Suppose $g_1\circ f = g_2 \circ f$. Then $g_1 = g_2$.
\end{prop}
See \cite[VI.11.10.1]{EGA}.

\begin{prop}\label{S-dominant morphisms epimorphism property}
    Let $f\colon X \to Y$ be an $S$-dominant morphism of schemes over $S$. Let $g_1,g_2\colon Y \to Z$ be two morphisms over $S$, and suppose $Z$ has separated fibers over $S$. Suppose $g_1\circ f = g_2 \circ f$. Then $g_1 = g_2$.
\end{prop}
See \cite[Proposition 1.4]{ArtinGroupChunk}

\section{Composition of \texorpdfstring{$S$}{S}-Dominant Morphisms}
\begin{lem}\label{A composition of schematically-dominant morphisms is schematically-dominant}
    A composition of schematically or $S$-dominant morphisms is schematically or $S$-dominant, respectively.
\end{lem}
\begin{proof}
    The statement for schematic dominance is contained in \cite[VI, 11.10.3]{EGA}, and the statement for $S$-dominance follows from it since composition commutes with base change.
\end{proof}

\begin{lem}\label{composition is schematically dominant implies f is schematically dominant}
    Let $g\colon X\to Y$ and $f\colon Y \to Z$ be morphisms of schemes over $S$. Suppose $f\circ g$ is schematically or $S$-dominant. Then $f$ is schematically or $S$-dominant.
\end{lem}
\begin{proof}
    Again we just need to check the statement for schematic dominance, but if the composition $\O_Z \to f_*\O_Y \to f_*g_*\O_X$ is injective then certainly so is $\O_Z \to f_*\O_Y$.
\end{proof}

\begin{cor} \label{S-dominance is stable under restriciton to S-dense opens}
    Let $\phi\colon X \to Y$ be a morphism of schemes over $S$, and $U\subset X$ $S$-dense. Then $\phi$ is $S$-dominant if and only if $\phi|_U\colon U \to Y$ is $S$-dominant.
\end{cor}
\begin{proof}
    Indeed, left to right is by \cref{A composition of schematically-dominant morphisms is schematically-dominant} and left to right is by \cref{composition is schematically dominant implies f is schematically dominant}.
\end{proof}

\begin{lem}\label{Preservation of schematic dominance on restrictions to opens on the target}
    Let $\phi\colon X \to Y$ be a schematically or $S$-dominant morphism of schemes. Then for any $V \subset Y$, the restriction $\phi|_{\phi\inv(V)}\colon \phi\inv(V) \to V$ is schematically or $S$-dominant.
\end{lem}
\begin{proof}
    Since taking restrictions to opens commutes with base change, we just need to verify the statement for schematic dominance. If the morphism $\O_Y \to \phi_*\O_X$ is injective, then so is its restriction to $V$.
\end{proof}

\begin{cor}\label{schematic density and restriction to opens}
    Let $U\subset X$ be a schematically or $S$-dense open subscheme. Then for any open subscheme $V \subset X$, we have that $U\cap V$ is schematically or $S$-dense in $V$.
\end{cor}

An intersection of schematically or $S$-dense opens is schematically or $S$-dense. In fact, more is true:
\begin{cor}\label{Intersection and schematically/S-dense opens}
    Let $U\subset Y$ be a schematically or $S$-dense open, and $V\subset Y$ be an open subscheme. The following are equivalent:
    \begin{enumerate}
        \item $V$ is schematically or $S$-dense in $Y$ (respectively),
        \item $U \cap V$ is schematically or $S$-dense in $Y$,
        \item $U \cap V$ is schematically or $S$-dense in $U$.
    \end{enumerate}
\end{cor}
\begin{proof}
    Indeed, we have $1 \iff 2$ by \cref{S-dominance is stable under restriciton to S-dense opens}, $1 \implies 3$ by \cref{schematic density and restriction to opens}, and $3 \implies 1$ by \cref{composition is schematically dominant implies f is schematically dominant}.
\end{proof}

\section{Products of \texorpdfstring{$S$}{S}-Dominant Morphisms}
One key advantage of $S$-dominance, as opposed to schematic dominance, is:
\begin{prop}\label{A product of S-dominant morphisms is S-dominant}
    A product of $S$-dominant morphisms is $S$-dominant. I.e., if $\phi_1\colon X_1 \to Y_1$ and $\phi_2\colon X_2 \to Y_2$ are $S$-dominant morphisms of schemes over $S$, then the morphism $\phi_1\times_S\phi_2\colon X_1\times_S X_2 \to Y_1 \times_S Y_2$ is $S$-dominant.
\end{prop}
\begin{proof}
    Immediate, since $S$-dominance is stable under composition and base change, and $\phi_1\times_S \phi_2$ is the same as the composition
    \begin{align*}
        X_1\times_S X_2 \to X_1 \times_S Y_2 \to  Y_1 \times_S Y_2.
    \end{align*}
\end{proof}
\begin{cor}\label{products of S-dense opens are S-dense.}
    Let $U \subset X$ and $V\subset Y$ be $S$-dense opens. Then $U \times_S V$ is $S$-dense in $X\times Y$.
\end{cor}

\begin{prop}\label{dominant open morphisms}
    Let $\phi\colon X \to Y$ be an open morphism of schemes.
    \begin{enumerate}
        \item If $\phi$ is schematically dominant, then for any schematically dense $U\subset X$, $\phi(U)\subset Y$ is a schematically dense open.
        \item If $\phi$ is $S$-dominant, then for any $S$-dense $U \subset X$, $\phi(U)\subset Y$ is an $S$-dense open.
    \end{enumerate}
\end{prop}
\begin{proof}
    Since a composition of open schematically or $S$-dominant morphisms is open and schematically or $S$-dominant (\cref{composition is schematically dominant implies f is schematically dominant}), it is enough to show in each case that $\phi(X)$ is schematically or $S$-dense respectively.

    \begin{enumerate}
        \item Let $V = \phi(X)$, and $j\colon V\hookrightarrow Y$ be the corresponding open immersion. Then $\O_Y \to \phi_*\O_X$ factors as $\O_Y \to j_*\O_V \to \phi_*\O_X$; thus if $\O_Y \to \phi_*\O_X$ is injective then certainly so is $\O_Y \to j_*\O_V$.

        \item Let $V = \phi(X)$. We need to show $V_{S'}$ is schematically dense for any $S' \to S$. But we have $\phi(X_{S'}) \subset V_{S'}$, and $\phi(X_{S'})$ is schematically dense by above, so we conclude by \cref{composition is schematically dominant implies f is schematically dominant}.
    \end{enumerate}
\end{proof}


%% file: Chapters/Scheme-Theoretic_Preliminaries/Schematic_Dominance_and_Schematic_Image.tex
\section{Schematic Dominance and Schematic Image}

\begin{rem}\label{schematic image of schematically dominant morphism}
    Observe that if $\phi\colon X \to Y$ is schematically dominant, then by definition it follows that $\Schim(\phi) = Y$; in general the converse is false.
\end{rem}

Precomposition by schematically dominant morphisms preserves schematic image:
\begin{prop}\label{Precomposition by schematically dominant morphisms preserves schematic image}
    Let $\phi\colon X \to Y$ and $\psi\colon Y \to Z$ be morphisms of schemes, and suppose $\phi$ is schematically dominant. Then $\Schim(\psi \circ \phi) = \Schim(\psi)$.
\end{prop}
\begin{proof}
    If $\O_Y \to \phi_*(\O_X)$ is injective, then so is $\psi_*\O_Y \to \psi_*(\phi_*(\O_X))$, so the morphism $\O_Z \to \psi_*(\O_Y)$ and the composition $\O_Z \to \psi_*(\O_Y) \to \psi_*(\phi_*(\O_X))$ have the same kernel.
\end{proof}
\begin{lem}\label{Schematically dominant maps to schematic images}
    Let $\phi\colon X \to Y$ be a morphism such that $\ker(\O_Y \to \phi_* \O_x)$ is quasicoherent (e.g. if $\phi$ is quasicompact or $X$ reduced). Then $\phi$ maps schematically dominantly onto its schematic image.
\end{lem}
\begin{proof}
    Let $\calI = \ker(\O_Y \to \phi_* \O_X)$, $X'  = \Schim(\phi)$, and $X \xrightarrow{\phi'} X' \xrightarrow{\iota} Y$ be the corresponding factorization. Then $\O_Y \to \phi_*(\O_X)$ factors as $\O_Y \to \iota_*(\O_{X'}) \to \phi_*(\O_X)$. Since $\calI$ is quasicoherent, we have $\O_Y \to \iota_*(\O_{X'})$ is just the quotient by $\calI$; therefore $\iota_*(\O_{X'}) \to \iota_*(\phi'_*(\O_X))$ is injective. It follows that $\O_{X'} \to \phi'_*(\O_X)$ is injective, since for any open $U \subset X'$ we have that the map $\O_{X'}(U) \to \phi'_*(\O_X)(U)$ is the same as the map $\iota_*(\O_{X'})(V) \to \iota_*(\phi'_*(\O_X))(V)$ for some open $V \subset Y$ with $U = V\cap X'$.
\end{proof}
\begin{cor}\label{Schematic image criteria for schematic dominance}
    Let $\phi\colon X \to Y$ be a morphism such that $\ker(\O_Y \to \phi_* \O_x)$ is quasicoherent (e.g. if $\phi$ is quasicompact or $X$ reduced). Then $\phi$ is schematically dominant if and only if $\Schim(\phi) = Y$.
\end{cor}
\begin{proof}
    Left to right is \cref{schematic image of schematically dominant morphism}, and right to left follows from the above result.
\end{proof}

%% file: Chapters/Scheme-Theoretic_Preliminaries/S-Dominance_and_Flatness.tex
\section{\texorpdfstring{$S$}{S}-Dominance and Flatness}

Note that a surjective morphism of schemes over $S$ will not be $S$-dominant in general. However:
\begin{prop}\label{a faithfully flat morphism is S-dominant}
    Let $\phi\colon X \to Y$ be a faithfully flat morphism of schemes over $S$. Then $\phi$ is $S$-dominant.
\end{prop}
\begin{proof}
    Since faithful flatness is stable under base change, it is enough to show that $\phi$ is schematically dominant, or equivalently that the morphism $\O_Y \to \phi_*\O_X$ is injective. This can be checked on stalks. By assumption the maps on stalks are faithfully flat ring maps; such maps are injective.
\end{proof}

\begin{cor}\label{Precomposition by faithfully flat morphisms preserves schematic image}
    Let $\phi\colon X \to Y$ and $\psi\colon Y \to Z$ be morphisms of schemes, and suppose $\phi$ is faithfully flat. Then $\Schim(\psi \circ \phi) = \Schim(\psi)$.
\end{cor}
\begin{proof}
    Immediate by the above and \cref{Precomposition by schematically dominant morphisms preserves schematic image}.
\end{proof}

\begin{cor}\label{image of S-dense open under faithfully flat morphism locally of finite presentation}
    Let $\phi\colon X\to Y$ be a morphism of schemes over $S$ which is faithfully flat and locally of finite presentation. Then for any $S$-dense open $U\subset X$ we have that $f(U)$ is an $S$-dense open subscheme of $Y$, and $\phi|_{[U \to f(U)]}$ is faithfully flat.
\end{cor}
\begin{proof}
    Since a flat morphism of schemes which is locally of finite presentation is universally open (\cref{a flat morphism of schemes which is locally of finite presentation is universally open}) and a faithfully flat morphism is $S$-dominant (\cref{a faithfully flat morphism is S-dominant}), it follows from \cref{dominant open morphisms} that $f(U)$ is an $S$-dense open. The morphism $\phi|_{[U \to f(U)]}$  is certainly surjective, and flatness follows from \cref{restrictions of flat morphisms to opens are flat}.
\end{proof}

The pullback of a quasicompact schematically dominant morphism by a flat morphism is schematically dominant:
\begin{lem}
    Let $f\colon X \to Y$ be a flat morphism of schemes, and let $g\colon Z \to Y$ be a quasicompact schematically dominant morphism. Then $f^*(g)$ is schematically dominant.
\end{lem}
\begin{proof}
    Note that $f^*(g)$ is quasicompact, since $g$ is. Therefore by \cref{Schematic image criteria for schematic dominance} it is enough to show that $\Schim(f^*(g)) = X$. By \cref{schematic image of schematically dominant morphism} we have $\Schim(g) = Y$. Then by \cref{schematic image of a quasicompact morphism commutes with flat base change on target}, we have $\Schim(f^*(g)) = f^*(\Schim(g)) =  f^*(Y) = X$.

    See \cite[IV.3,~11.10.5]{EGA} for an alternative proof.
\end{proof}
It follows that pullback of a quasicompact $S$-dominant morphism by a flat morphism is $S$-dominant:
\begin{cor}\label{the pullback of a quasicompact S-dominant morphism by a flat morphism is S-dominant}
    Let $f\colon X \to Y$ be a flat morphism of schemes over $S$, and let $g\colon Z \to Y$ be a quasicompact $S$-dominant morphism. Then $f^*(g)$ is $S$-dominant.
\end{cor}
\begin{proof}
    For any $S' \to S$, we need to show that $(f^*(g))_{S'}$ is schematically dominant. Since $g$ is $S$-dominant and quasicompact, we have that $g_{S'}$ is schematically dominant and quasicompact. By \cref{pullback preserves limits in categories with fiber products}, we have $(f^*(g))_{S'} = f_{S'}^*(g_{S'})$. Since $f_{S'}$ is flat, we conclude by the previous result.
\end{proof}

%% file: Chapters/Scheme-Theoretic_Preliminaries/Dominance_vs_Schematic_Dominance_vs_S-Dominance.tex
\section{Dominance vs Schematic Dominance vs \texorpdfstring{$S$}{S}-Dominance}\label{section: dominance vs schematic dominance}

Recall
\begin{dfn}
    A morphism of schemes $\phi\colon X \to Y$ is said to be \emph{dominant} if its set-theoretic image is dense in $\abs{Y}$.
\end{dfn}
\begin{lem}[{\cite[\href{https://stacks.math.columbia.edu/tag/01RL}{Tag 01RL}]{stacks-project}}]\label{dominant quasicompact morphisms}
    Let $f\colon X\to S$ be a quasicompact morphism of schemes. Then $f$ is dominant if and only if for every irreducible component $Z\subset S$ the generic point of $Z$ is in the image of $f$.
\end{lem}

In general, dominance is neither necessary nor sufficient for schematic dominance. But the two notions agree in good situations:

\begin{prop}\label{schematic dominance vs dominance}
    Let $\phi\colon X \to Y$ be a morphism of schemes. Suppose $Y$ is reduced and $\ker(\O_Y \to \phi_* \O_X)$ is quasicoherent. Then $\phi$ is schematically dominant if and only if it is dominant.
\end{prop}
\begin{proof}
    By \cref{schematic image of schematically dominant morphism}, we have $\abs{\Schim{\phi}} = \overline{\phi(X)}$.
    Suppose $\phi$ is dominant. Then $\abs{\Schim{\phi}} = \overline{\phi(X)} = Y$. Since $Y$ is reduced this implies $\Schim(\phi) = Y$, and we conclude by \cref{Schematic image criteria for schematic dominance}.

    Conversely, if $\phi$ is schematically dominant, then $\Schim(\phi) = Y$, so $\overline{\phi(X)} = \abs{Y}$ and $\phi$ is dominant.
\end{proof}

\begin{cor}\label{schematically dense opens of a reduced scheme}
    Let $X$ be a reduced scheme, and $U \subset X$ an open subscheme. Then $U$ is dense if and only if it is schematically dense.
\end{cor}
\begin{proof}
    Apply \cref{schematic dominance vs dominance} and \cref{quasicoherent kernel lemma}.
\end{proof}

\begin{cor}[Generic faithful flatness]\label{Generic faithful flatness proposition}
    Let $\pi\colon X \to Y$ be a dominant finite type morphism of schemes with $Y$ reduced. Then there exists a dense open $U \subset Y$ such that $\pi|_{[\pi\inv(U) \to U]}$ is faithfully flat and of finite presentation (fppf).
\end{cor}
\begin{proof}
    First note that $\pi$ is schematically dominant by \cref{schematic dominance vs dominance}. By \cref{Generic Flatness theorem}, there exists a dense open $V \subset Y$ such that the restriction $\pi|_{\pi\inv(V)}$ is flat and of finite presentation. It is therefore open (\cref{a flat morphism of schemes which is locally of finite presentation is universally open}) and schematically dominant (\cref{Preservation of schematic dominance on restrictions to opens on the target}) so by the above and \cref{dominant open morphisms} we get that $U = \pi(\pi\inv(V))$ is a dense open subscheme of $Y$, and $\pi|_{[\pi\inv(U) \to U]}$ is fppf.
\end{proof}

\begin{cor}\label{schematic image of a geometrically integral variety}
    Let $\phi\colon A \to B$ be a morphism of varieties over a field $k$, and suppose $A$ is geometrically integral. Then $\Schim(\phi)$ is geometrically integral.
\end{cor}
\begin{proof}
    Since the schematic image of a quasicompact morphism commutes with flat base change (\cref{schematic image of a quasicompact morphism commutes with flat base change on base}), it is enough to show that $\Schim(\phi)$ is integral. It is reduced by \cref{schematic image of morphism with reduced source}. It is irreducible since $A$ is irreducible, and by above $\phi$ is dominant onto $\Schim(\phi)$ by \cref{Schematically dominant maps to schematic images} and \cref{schematic dominance vs dominance}. Indeed, suppose $\Schim(\phi)$ were a union of proper closed subschemes $Z_1$, $Z_2$; then since $\phi(A)$ is dense in $\Schim(\phi)$ we have that $\phi(A)$ is not contained in either $Z_1$ or $Z_2$, so taking the pre-images of $Z_i$ in $A$ we get a contradiction.
\end{proof}

$S$-dominance/density is often equivalent to schematic dominance/density in every fiber:

\begin{lem}(cf \cite[VI, 11.10.9]{EGA})    Let $\phi\colon X \to Y$ be a morphism of schemes over $S$. Suppose $X$ is flat over $S$, and $Y$ is locally Noetherian. Then $\phi$ is $S$-dominant if and only if for every $s \in S$ we have that $\phi_s\colon X_s \to Y_s$ is schematically dominant.
\end{lem}

\begin{lem}(cf \cite[VI, 11.10.10]{EGA})
    Let $X$ be a scheme flat and locally of finite presentation over $S$. An open subscheme $U \subset X$ is $S$-dense if and only if for every $s\in S$, $U_s \subset X_s$ is schematically dense.
\end{lem}

\begin{lem}(cf \cite[VI, 11.10.6]{EGA})\label{S-dominance over fields}
    Suppose $S = \Spec(k)$ for some field $k$, and $\phi\colon X \to Y$ is a morphism over $S$. Then $\phi$ is schematically dominant if and only if it is $S$-dominant. In particular an open $U\subset X$ is schematically dense if and only if it is $S$-dense.
\end{lem}

%% file: Chapters/Group_Chunks_on_Schemes.tex
\chapter{Group Chunks on Schemes} \label{Group Chunks on Schemes}
In this short chapter we discuss the relationship between our abstract group chunk result \cref{main abstract group chunk result} and Artin's construction of a group scheme from a birational law.

The next definition is different from Artin's (\cref{Artin Group Chunk definition})---we discuss the relationship below. Recall \cref{products and Yoneda combine}.
\begin{dfn}\label{scheme-theoretic group chunk definition}
    Let $S$ be a scheme. A \emph{scheme-theoretic group chunk} over $S$ is an $S$-scheme $X$, faithfully flat and locally of finite presentation over $S$, together with a subscheme $W \subset X \times X \times X$ such that, if we denote $X_1 = X_2 = X_3 = X$, then for all permutations $(i,j,k)$ of  $\{1,2,3\}$ with $i<j$, we have that
    \begin{enumerate}
        \item The projection $\rho_{ij}\colon W \to X_iX_j$ is an open immersion with image $U_{ij}$ which is both $X_i$-dense and $X_j$-dense. In particular, $\rho$ exhibits $W$ as the graph of a morphism $m_{ij}$ from an open $U_{ij} \subset X_iX_j$ to $X_k$.
        \item If we view each $m_{ij}$ as a partial morphism from $X_iX_j$ to $X_k$, and denote $(x_1,x_2,x_3) = \id_{X^3}$ (so $x_i\colon X^3 \to X$ is the $i^{\text{th}}$ coordinate projection), then
              \begin{align*}
                  m_{12}(m_{12}(x_1,x_2),x_3) & \ssim m_{12}(x_1,m_{12}(x_2,x_3))  \\
                  m_{12}(m_{13}(x_1,x_2),x_3) & \ssim m_{13}(x_1,m_{12}(x_2,x_3)).
              \end{align*}
    \end{enumerate}
\end{dfn}

\begin{lem} (\cite[Proposition~1.17]{ArtinGroupChunk}).\label{Schemes are fppf locally nontrivial}
    Let $X$ be a scheme faithfully flat and locally of finite presentation over $S$. Then $X$ is fppf-locally nontrivial as a sheaf on the fppf site of schemes over $S$.
\end{lem}
\begin{proof}
    Indeed, for any scheme $T \to S$, let $X_T$ be the base change of $X$ to $T$. Then $X_T$ is an fppf cover of $T$, and we have a projection morphism $X_T \to T$.
\end{proof}

\begin{prop}
    Let $(X,W)$ be a scheme-theoretic group chunk over $S$, and define $m_{12}$ as above. Then $(h_X, m_{12})$ is a locally nontrivial group chunk on the fppf site of schemes over $S$.
\end{prop}
\begin{proof}
    First note that $X$ is fppf-locally nontrivial by above. It is separated since schemes are sheaves for the fppf topology. The cancellation and strong associativity axioms are immediate; we just need to verify the existence of enough products.

    For this let us introduce some notation.

    \begin{notation}\label{left right fibres notation}
        Let $X_1 = X_2 = X$. Given a pair of morphisms $(a,b)\in (X_1\times_S X_2)(T)$ and a subobject $U \subset X_1\times_S X_2$, let us denote
        \begin{align*}
            \leftindex_a U :=  T \times_{a,X,\rho_1}U \\
            U_b :=  U \times_{\rho_2,X,b}T            \\
        \end{align*}
        where $\rho_i\colon U \to X_i$ denotes $i$'th coordinate projection on $U$.
    \end{notation}

    We denote by
    \begin{align*}
        \lambda_a\colon \leftindex_a(U_{12})\xrightarrow{\sim} \leftindex_a(U_{13})
    \end{align*}
    the isomorphism induced by pulling back the isomorphism
    \begin{align*}
        (\rho_1,m_{12})\colon U_{12} \xrightarrow{\sim} U_{13}
    \end{align*}
    by $a \in X_1(T)$.

    \begin{rem}
        Observe that, if we view $m_{12}$ as a partial morphism of presheaves from $X^2$ to $X$, the definition of $\lambda_a$ here agrees with \cref{Group Chunk definition}.
    \end{rem}

    Observe that, viewing $\lambda_a$ and $\lambda_b$ as partial morphisms from $X_T$ to $X_T$, we have
    \begin{align*}
        \dom \lambda_b \circ \lambda_a = \lambda_a\inv(\leftindex_b(U_{12})\cap \leftindex_a(U_{13}))
    \end{align*}
    We claim this is a $T$-dense open of $\leftindex_a(U_{12})$ and in particular it is nonempty. Indeed, since the $U_{ij}$ are $X_1$-dense, we have that $\leftindex_b(U_{12})$ and $\leftindex_a(U_{13})$ are both $T$-dense opens of $\leftindex_a(X^2) = X_T$. In particular, since an intersection of $T$-dense opens is $T$-dense (by \cref{Intersection and schematically/S-dense opens}), the intersection  $\leftindex_b(U_{12}) \cap \leftindex_a(U_{13})$ is a $T$-dense open of $\leftindex_a(U_{13})$. Since $\lambda_a$ is an isomorphism, the claim follows.

    Therefore, for any finite collection of points $a_i,b_i \in X(T)$, we have that the intersection $\cap_i\dom \lambda_{b_i} \circ \lambda_{a_i}$ is a $T$-dense open of $X_T$. Since $X_T \to T$ is faithfully flat and locally of finite presentation over $T$, it follows that the same is true of $\cap_i\dom \lambda_{b_i} \circ \lambda_{a_i}$. Therefore by \cref{Schemes are fppf locally nontrivial} the intersection of the domains is locally nontrivial as required.

\end{proof}

The following is a slight improvement on Artin's original definition by \cite{ER15}; in the original paper, they use the term ``strict $S$-birational group law" instead of ``group chunk".
\begin{dfn}[Artin Group Chunk]\label{Artin Group Chunk definition}
    Let $S$ be a scheme. An \emph{Artin group chunk} over $S$ is an $S$-scheme $X$, faithfully flat and locally of finite presentation over $S$, whose fibers over $S$ are separated and have no embedded components, together with a subscheme $W \subset X \times X \times X$ such that, if we denote $X_1 = X_2 = X_3 = X$, then for all permutations $(i,j,k)$ of $\{1,2,3\}$ with $i<j$, we have that
    \begin{enumerate}
        \item \label{Artin group chunk open immersion condition} The projection $\rho_{ij}\colon W \to X_iX_j$ is an open immersion whose image $U_{ij}$ is both $X_i$-dense and $X_j$-dense. In particular, $\rho_{ij}$ exhibits $W$ as the graph of a morphism $m_{ij}$ from $U_{ij} \subset X_iX_j$ to $X_k$.
        \item \label{Artin group chunk associtivity condition} If we view each $m_{ij}$ as a partial morphism from $X_iX_j$ to $X_k$, and denote $(x_1,x_2,x_3) = \id_{X^3}$, then (recalling the notion of compatibility from \cref{partial morphism in a category definition}, and \cref{products and Yoneda combine}) we have
              \begin{align*}
                  m_{12}(m_{12}(x_1,x_2),x_3) & \ssim m_{12}(x_1,m_{12}(x_2,x_3)).
              \end{align*}
    \end{enumerate}
\end{dfn}

\begin{rem}
    Comparing the above to \cref{scheme-theoretic group chunk definition}, one may notice that we seem to be missing one of the associativity conditions; however it is shown in \cite[Section~3.2.1]{ArtinGroupChunk} that this follows from the other conditions, using the fact that the fibers of $X$ over $S$ are separated and have no embedded components. In particular, an Artin group chunk is a scheme-theoretic group chunk; and one can use it to construct a sheaf of groups as described in previous sections. In fact the sheaf of groups constructed in both \cite{ArtinGroupChunk} or \cite{ER15} agrees with our construction.

    We conjecture that, if one were to assume the extra associativity condition, one should be able to follow the construction of \cite{ArtinGroupChunk} or \cite{ER15} without making any assumptions on the fibers of $X$ over $S$ and get representability results very similar to those below.
\end{rem}

In \cite[Proposition~3.2]{ArtinGroupChunk}, Artin shows that, provided the fibers of $X$ over $S$ are geometrically integral, then one can take the $U_{ij}$ to be $S$-dense instead of $X_1$ dense and $X_2$ dense:
\begin{prop}(cf \cite[Proposition~3.2]{ArtinGroupChunk})\label{geometrically integral fibers gives X dense open}
    Let $S$ be a scheme. Let $X$ be a scheme faithfully flat and locally of finite presentation over $S$, whose fibers over $S$ are separated and geometrically integral.
    Let $W \subset X \times X \times X$ be a subscheme such that, if we denote $X_1 = X_2 = X_3 = X$, then for all permutations $(i,j,k)$ of $\{1,2,3\}$ with \(i<j\) we have that
    \begin{enumerate}
        \item The projection $\rho_{ij}\colon W \to X_iX_j$ is an open immersion whose image $U_{ij}$ is $S$-dense.
        \item $W$ satisfies condition 2 from \cref{Artin Group Chunk definition}.
    \end{enumerate}
    Then there exist $S$-dense opens $X' \subset X$ and $W' \subset W\cap (X')^3 $ such that $(X',W')$ is an Artin group chunk over $S$.
\end{prop}

\begin{rem}\label{multiple components remark}
    We believe it should be possible to prove the above result, using essentially the same argument, if $X$ is a finite union of schemes $X_i$ where each $X_i$ is faithfully flat and locally of finite presentation over $S$ with separated and geometrically integral fibers.
\end{rem}

The following theorem is a slight improvement of Artin's original result (\cite{ArtinGroupChunk}) due to Edixhoven, and Romagny (\cite[Theorem~3.18]{ER15}):
\begin{prop}\label{Artin Group chunk theorem}
    Let $S$ be a scheme and $(X,W)$ an Artin group chunk over $S$, and let $G = \mathscr{G}(X,W)$ be the associated group sheaf. Then
    \begin{itemize}
        \item $G$ is a group algebraic space faithfully flat and locally of finite presentation over $S$, and locally separated over $S$.
        \item The embedding $X \hookrightarrow G$ is representable by $S$-dense open immersions.
        \item If $X \to S$ is smooth, then so is $G \to S$
        \item If $X \to S$ is quasi-separated or locally Noetherian, than $G\to S$ is quasi-separated. If $X \to S$ is separated than so is $G$.
        \item If $X \to S$ is of finite presentation, then $G$ is a scheme fppf-locally on $S$.
    \end{itemize}
\end{prop}
Moreover, in good cases, $G$ is representable by a scheme. We cite \cite[Theorem~3.19]{ER15}, which is an amalgamation of work of Artin \cite{ArtinGroupChunk}, Anantharaman, and Bosch, Lutkebohmert and Raynaud:
\begin{prop}\label{representability criterion for Artin group chunks}
    Assume in the above theorem that $S$ is locally Noetherian and either
    \begin{enumerate}
        \item $S$ has dimension 0,
        \item $S$ has dimension 1 and $X \to S$ is separated, or
        \item $X \to S$ is smooth.
    \end{enumerate}
    Then $G$ is a scheme.
\end{prop}
In the case where $X$ is a variety over a field; one gets that $G$ is a variety. See \cite{ER15} and \cite{WeilAndre1955OAGo}.

\begin{lem}\label{S-dense open of Artin group chunk gives group chunk}
    Let $(X,W)$ be an Artin group chunk, and $X' \subset X$ an $S$-dense open subscheme. Then $(X', W \cap (X')^3)$ is an Artin group chunk.
\end{lem}
\begin{proof}
    Certainly $X'$ is faithfully flat and locally of finite presentation over $S$, with separated and geometrically integral fibers.

    Let $X'_i = X_i$ for $i = 1,2,3$, and $W' = W \cap (X')^3$. Then for each permutation $(i,j,k)$ of $\{1,2,3\}$ with $i < j$ the projection $\rho'_{ij}\colon W' \to X_iX_j$ is an open immersion; indeed it is the composition of the open immersion $W' \hookrightarrow W \cap X_i'X_j'X_k$ and the projection $W \cap X_i'X_j'X_k \to X_i'X_j'$, and the latter is an open immersion since it is the pullback of $W \hookrightarrow X_iX_j$. The associativity property is automatic, so it just remains to show that the image $U'_{ij}$ of each $\rho'_{ij}$ is $X'_i$-dense and $X'_j$-dense in $X_i'X_j'$.

    Indeed, since $X_j' \subset X_j$ is $S$-dense, we have that $X_i'X_j' \subset X_i'X_j$ is $X_i'$-dense; and since $U_{ij} \subset X_iX_j$ is $X_i$-dense we have $U_{ij} \cap X_i'X_j$ is $X_i'$ dense. Since $U_{ij}' = (U_{ij} \cap X_i'X_j) \cap X_i'X_j'$ we conclude $U_{ij}'$ is $X_i'$-dense. Similarly it is $X_j'$-dense.
\end{proof}

\begin{lem}\label{S-dense sub group chunk gives same group.}
    Let $(X,W)$ be an Artin group chunk, and let $W' \subset W$ be an open and $X' \subset X$ be an $S$-dense open such that $(X',W')$ is an Artin group chunk. Then the group algebraic space associated with both group chunks is the same.
\end{lem}
\begin{proof}
    Denote by $m_{ij}$, $m_{ij}'$ the induced partial binary operations on $X$, $X'$, and let $U_{ij}$, $U'_{ij}$ be the respective domains. Let $X_i = X$ and $X'_i = X'$ for $i = 1,2,3$. To show that the associated groups are the same, by \cref{locally surjective sub group chunks can produce the same sheaf of groups} it is enough to show that $m_{12}|_{U'_{12}}\colon U'_{12} \to X_3$ is faithfully flat and locally of finite presentation (hence surjective as a morphism of fppf sheaves by \cref{fppf morphisms are epimorphisms on the fppf topology}).

    First note that pulling back the projection $W \to U_{12}$ we get a commutative triangle
    \[
        \begin{tikzcd}
            W \cap X_1'X_2'X_3 \arrow[rr, "\rho_{12}"] \arrow[rd, "\rho_3"] &   & U'_{12} \arrow[ld, "\mu_{12}"] \\
            & X_3 &
        \end{tikzcd}
    \]
    where the top arrow is an isomorphism. So it is enough to show $W\cap X_1'X_2'X_3 \to X_3$ is faithfully flat and locally of finite presentation. Note that since $X \to S$ is faithfully flat and locally of finite presentation the same is true of $X_1X_2X_3 \to X_3$. Therefore $W\cap X_1'X_2'X_3 \to X_3$ is flat and locally of finite presentation; it just remains to show surjectivity. Note that, since $U_{13}$ is $X_3$-dense in $X_1X_3$, we have that $W \to X_3$ is surjective; therefore it is enough to show that $W\cap X_1'X_2'X_3$ is $X_3$-dense in $W$.

    Since $X' \subset X$ is $S$-dense, it follows that $X_1'X_3 \subset X_1X_3$ is $X_3$-dense. By \cref{Intersection and schematically/S-dense opens}, we have that $(U_{13} \cap X_1'X_3) \subset U_{13}$ is $X_3$-dense, or equivalently $W \cap X_1'X_2X_3 \subset W$ is $X_3$-dense. Similarly $W \cap X_1X_2'X_3 \subset W$ is $X_3$-dense, and taking the intersection we conclude.
\end{proof}

%% file: Chapters/S-Rational_Morphisms/Main_S-Rational_Morphisms.tex
\chapter{\texorpdfstring{$S$}{S}-Rational Morphisms}\label{chapter: S-rational morphisms}
In this chapter we introduce $S$-rational morphisms and study their properties, which we will need in \cref{chapter: S-rational families}.

We begin by introducing the notion of an $S$-diffuse morphism. Roughly speaking, an $S$-diffuse morphism (\cref{diffuse definition}) will be a representative of an $S$-rational morphism which is amenable to composition. In \cref{section: S-rational morphisms} we define $S$-rational morphisms and discuss the basic consequences of the definition. In \cref{section: faithfully flat S-rational morphisms} we introduce faithfully flat $S$-rational morphisms; these will be play a large role in the next two chapters because they satisfy the descent properties which we will need to construct canonical $S$-rational families. Importantly, in a locally Noetherian setting, a faithfully flat $S$-rational morphism is automatically $S$-diffuse.

In the next two sections we define composition and products of $S$-rational morphisms. We show that the class of $S$-diffuse $S$-rational morphisms is stable under composition, and that the class of faithfully flat $S$-rational morphisms is closed under taking composition and products. In \cref{section: S-birational morphisms} we study $S$-birational morphisms and their inverses, and we show that if two diffuse $S$-rational morphisms are mutually inverse with respect to composition then they are $S$-birational and mutually inverse. In the final section we study closed graphs of $S$-rational morphisms, and we show in good situations that $S$-rational morphisms (and their inverses) are determined by their closed graphs.


We next define $S$-rational morphisms, and  weakly $S$-generically faithfully flat (wff) morphisms, and show that they are stable under taking products. Finally we study wfd/wffd morphisms and show that they are relatively well-behaved; at least for schemes locally of finite presentation over a base. In particular, wfd/wffd morphisms are stable under composition \emph{and} product.

\input{Chapters/S-Rational_Morphisms/S-Diffuse_Morphisms.tex}

\input{Chapters/S-Rational_Morphisms/S-rational_morphisms.tex}

\input{Chapters/S-Rational_Morphisms/Faithfully_Flat_S-Rational_Morphisms.tex}

\input{Chapters/S-Rational_Morphisms/Composition_of_S-Rational_Morphisms.tex}
\input{Chapters/S-Rational_Morphisms/Products_of_S-rational_morphisms.tex}

\input{Chapters/S-Rational_Morphisms/S-Birational_Morphisms.tex}

\input{Chapters/S-Rational_Morphisms/Graphs_of_S-Rational_Morphisms.tex}

%% file: Chapters/S-Rational_Morphisms/S-Diffuse_Morphisms.tex
\section{\texorpdfstring{$S$}{S}-Diffuse Morphisms}\label{section: diffuse morphisms}
\begin{dfn}\label{diffuse definition}
    We say a morphism $\phi\colon X \to Y$ of schemes over $S$ is $S$-diffuse (or just diffuse, if $S$ is clear from context) if for every $S$-dense open $V\subset Y$, $\phi\inv(V)$ is $S$-dense in $X$.
\end{dfn}

\begin{lem}\label{a dominant morphism of integral schemes over a field is diffuse}
    A dominant morphism $\phi\colon X \to Y$ of integral schemes over a field $k$ is diffuse.
\end{lem}
\begin{proof}
    Since $S$-density is the same as density for reduced schemes when $S = \Spec(k)$ (\cref{S-dominance over fields}, \cref{schematically dense opens of a reduced scheme}), and since the schemes are irreducible (so an open is dense if and only if it is nonempty), we just need to check that the preimage of a nonempty open is nonempty. But this is immediate, since any nonempty open of $Y$ contains the generic point which is in the image of $\phi$ by (\cref{dominant quasicompact morphisms}).
\end{proof}

$S$-diffuseness is stable under restriction to $S$-dense opens:

\begin{prop}\label{preimage being S-dense is stable under restriction to S-dense opens}
    Let $\phi\colon X \to Y$ be a morphism of $S$-schemes and $U\subset X$, $V\subset Y$ $S$-dense opens. Then $\phi\inv(V) \subset X$ is $S$-dense if and only if $\phi|_U \inv(V) \subset U$ is $S$-dense.
\end{prop}
\begin{proof}
    Right to left follows from the fact that a composition of $S$-dominant morphisms is $S$-dominant (\cref{A composition of schematically-dominant morphisms is schematically-dominant}). For left to right, note $\phi|_U \inv(V) = \phi\inv(V) \cap U$ and recall that an intersection of $S$-dense opens is $S$-dense (\cref{Intersection and schematically/S-dense opens}).
\end{proof}
\begin{cor}\label{S-diffusity is stable under restriction to S-dense opens}
    Let $\phi\colon X \to Y$ be a morphism and $U\subset X$ an $S$-dense open. Then $\phi$ is $S$-diffuse iff $\phi|_U$ is $S$-diffuse.
\end{cor}

$S$-diffuseness is also stable under restricting the target to $S$-dense opens with $S$-dense preimages:
\begin{lem}\label{S-diffusity and restriction to the target}
    Let $\phi\colon X \to Y$ be a morphism of $S$-schemes, $V\subset Y$ and $U\subset X$ be $S$-dense open subschemes with $U \subset \phi\inv(V)$. Then $\phi$ is diffuse if and only if $\phi|_{[U \to V]}$ is diffuse.
\end{lem}
\begin{proof}
    By \cref{S-diffusity is stable under restriction to S-dense opens}, we have that $\phi$ is $S$-diffuse if and only if $\phi|_U$ is $S$-diffuse.

    If $\phi|_U$ is diffuse then for any $S$-dense open $W \subset V$ we have that $W$ is $S$-dense in $Y$ and $\phi|_{[U \to V]}\inv(W) = \phi|_U\inv(W)$, and the latter is $S$-dense.

    Conversely if $\phi|_{[ U \to V ]}$ is $S$-diffuse then for any $S$-dense open $W \subset Y$ we have that $V \cap W$ is $S$-dense in $V$, and then $\phi|_{[ U \to V ]}\inv(W\cap V)$ is $S$-dense in $U$, hence in $X$. Since $\phi|_{[ U \to V ]}\inv(W\cap V) \subset \phi|_{U}\inv(W)$ we conclude by \cref{composition is schematically dominant implies f is schematically dominant}.
\end{proof}

\begin{prop}\label{flat morphisms of locally Noetherian schemes are S-diffuse}
    A flat morphism of locally Noetherian schemes over $S$ is $S$-diffuse.
\end{prop}
\begin{proof}
    Immediate, since the pullback of a quasicompact $S$-dominant morphism by a flat morphism is $S$-dominant (\cref{the pullback of a quasicompact S-dominant morphism by a flat morphism is S-dominant}), and an open immersion to a locally Noetherian scheme is quasicompact.
\end{proof}

\begin{cor}\label{generic flatness criterion for S-diffuseness in the locally Noetherian case}
    Let $\phi\colon X \to Y$ be a morphism of locally Noetherian schemes over $S$, $V\subset Y$ be $S$-dense, and suppose there exists a representative $\phi|_U$ of $\phi$ which is flat over $V$. Then $\phi$ is $S$-diffuse if and only if $\phi\inv(V)$ is $S$-dense in $X$.
\end{cor}
\begin{proof}
    If $\phi$ is $S$-diffuse then so is $\phi|_U$ (\cref{S-diffusity is stable under restriction to S-dense opens}), so $\phi|_U\inv(V)$ is $S$-dense in $U$, hence in $X$.

    Conversely, suppose $\phi\inv(V)$ is $S$-dense in $X$. Since $\phi|_U$ is flat over $V$ we have that $\phi|_{U\cap \phi\inv(V)}$ is flat, hence diffuse by \cref{flat morphisms of locally Noetherian schemes are S-diffuse} since everything is locally Noetherian. Since $\phi\inv(V) \cap U$ is $S$-dense in $U$, hence in $X$, we conclude $\phi$ is diffuse by \cref{S-diffusity is stable under restriction to S-dense opens}.
\end{proof}

$S$-diffuseness is preserved by composition:
\begin{prop}\label{A composition of S-diffuse morphisms is S-diffuse}
    A composition of $S$-diffuse morphisms is $S$-diffuse.
\end{prop}
\begin{proof}
    Let $f\colon X \to Y$ and $g\colon Y \to Z$ be $S$-diffuse, and $W \subset Z$ an $S$-dense open. Then $g\inv(W)$ is $S$-dense, and therefore $f\inv g\inv(W) = (g\circ f)\inv(W)$ is $S$-dense.
\end{proof}

\begin{lem}\label{Projections are diffuse}
    Let $X$, $A$ be schemes over $S$. Then the projection $\rho_A\colon A \times_S X \to A$ is $S$-diffuse.
\end{lem}
\begin{proof}
    Indeed, for any $S$-dense $U \subset A$ we have $\rho_{A}\inv(U) = UX$ which is $S$-dense in $AX$ by \cref{A product of S-dominant morphisms is S-dominant}.
\end{proof}


%% file: Chapters/S-Rational_Morphisms/S-rational_morphisms.tex
\section{\texorpdfstring{$S$}{S}-Rational Morphisms}\label{section: S-rational morphisms}

\begin{dfn}\label{S-rational morphisms definition}
    An  $S$-rational morphism $\phi\colon X \dashrightarrow Y$ is an equivalence class of pairs $(U \subset X,\phi_U\colon U \to Y)$ where $U \subset X$ is an $S$-dense open and $\phi\colon U \to Y$ is a morphism, and two pairs $(U,\phi_U)$, $(U',\phi_{U'})$ are equivalent if there exists an $S$-dense open $U'' \subset U \cap U'$ such that $\phi_U|_{U''} = \phi_{U'}|_{U''}$.

    If $S = \Spec(A)$ for a ring $A$, we will say $A$-rational in place of $\Spec(A)$-rational.
\end{dfn}

We will frequently refer to representatives $(U,\phi_U)$ simply by $\phi_U$, and say that $\phi_U$ is a representative over $U$. We also say that $\phi$ is the rational morphism from $X$ to $Y$ induced by $\phi_U$.

Note that over any $S$-dense open $U$, an $S$-rational morphism may have more than one representative over $U$.

\begin{rem}\label{embedding of morphisms to S-rational morphisms}
    If $X,Y$ are $S$-schemes, note that there is a natural map from $\Hom_S(X,Y)$ to the set of $S$-rational morphisms from $X$ to $Y$ sending each morphism $f$ to the class of the pair $(X,f)$. If $Y$ has separated fibers over $S$ (and in particular if $Y$ is separated), then it follows from \cref{S-dominant morphisms epimorphism property} that this map is injective.
\end{rem}

\begin{dfn}[Domain of an $S$-rational morphism]\label{domain of an S-rational morphism definition}
    We define the domain of an $S$-rational morphism $\phi\colon X \dashrightarrow Y$, denoted $\dom(\phi)$, to be the union over all $S$-dense opens $U$ such that $\phi$ has a representative $(U,\phi_U)$.
\end{dfn}

Under good circumstances one can find a representative of a rational morphism over the domain:
\begin{prop}\label{domain for S-rational morphisms to separated target}
    Let $\phi\colon X \dashrightarrow Y$ be an $S$-rational morphism, and suppose $Y$ has separated fibers over $S$. Then
    \begin{enumerate}
        \item For any representatives $\phi_U, \phi_V$, we have $\phi_U|_{U \cap V} = \phi_V|_{U \cap V}$.
        \item There exists a representative $(\dom(\phi),\phi_{\dom(\phi)})$ of $\phi$
        \item For any $S$-dense open $U \subset \dom(\phi)$, there exists a unique representative $\phi_U$ over $U$.
    \end{enumerate}
\end{prop}
\begin{proof}
    The first statement follows from \cref{S-dominant morphisms epimorphism property}, and the other statements follow immediately from the first.
\end{proof}

In light of the above, we adopt the following conventions:
\begin{notation}\label{notation for S-rational morphisms with separated target}
    Let $\phi\colon X \dashrightarrow Y$ be an $S$-rational morphism, and suppose $Y$ has separated fibers over $S$.
    \begin{itemize}
        \item For every $U \subset \dom(\phi)$, we denote by $\phi_U$ the unique representative over $U$.
        \item For every $T \to S$ and $t \in \dom(\phi)(T)$, we denote by $\phi(t)$ the morphism $\phi_{\dom(\phi)}(t)$. Observe this is equal to $\phi_U(t)$ for any $U \subset \dom(\phi)$ with $t \in U(T)$.
    \end{itemize}

    Let $f\colon X \to Y$ be a morphism, and suppose $Y$ has separated fibers over $S$. Then for any $S$-dense open $U \subset X$ we write $f_U := f|_U$, and note that $f_U$ is the unique representative of the $S$-rational morphism corresponding to $f$ over $U$.
\end{notation}

The following allows us to define the schematic image of an $S$-rational morphism:
\begin{lem}
    Let $\phi\colon X \to Y$ be an $S$-rational morphism. Then any two representatives of $\phi$ have the same schematic image.
\end{lem}
\begin{proof}
    Let $\phi_U, \phi_V$ be two representatives. Then there exists an $S$-dense open $W$ such that $\phi_U, \phi_V$ agree on $W$. It follows from \cref{Precomposition by schematically dominant morphisms preserves schematic image} that $\Schim{\phi_U} = \Schim(\phi_U|_W) = \Schim(\phi_V)$.
\end{proof}
\begin{dfn}
    Let $\phi\colon X \dashrightarrow Y$ be an $S$-rational morphism. We define the schematic image of $\phi$, denoted $\Schim(\phi)$ to be the schematic image of any of its representatives.
\end{dfn}

\begin{dfn}
    Let $\phi\colon X \dashrightarrow Y$ be an $S$-rational morphism. We say $\phi$ is $S$-dominant/diffuse if one (or equivalently, by \cref{S-dominance is stable under restriciton to S-dense opens} and \cref{S-diffusity is stable under restriction to S-dense opens}, any) representative of $\phi$ is $S$-dominant/diffuse.
\end{dfn}

Finally, let us note the following for use in the next section:
\begin{lem}\label{representatives of rational morphisms being locally of finite presentation}
    Let $X$ and $Y$ be schemes locally of finite presentation over $S$ and let $\phi\colon X \dashrightarrow Y$ be an $S$-rational morphism. Then every representative $\phi_U$ of $\phi$ is locally of finite presentation.
\end{lem}
\begin{proof}
    If $X$ is locally of finite presentation over $S$ then so is every $U \subset X$, and the result follows by \cref{composition criteria for being locally of finite presentation}.
\end{proof}

%% file: Chapters/S-Rational_Morphisms/Faithfully_Flat_S-Rational_Morphisms.tex
\section{Faithfully Flat \texorpdfstring{$S$}{S}-Rational Morphisms}\label{section: faithfully flat S-rational morphisms}


\begin{dfn}\label{flatness and faithful flatness for rational morphisms definition}
    Let $\phi\colon X \dashrightarrow Y$ be an $S$-rational morphism.
    \begin{enumerate}
        \item We say $\phi$ is flat if one of its representatives is flat.
        \item We say $\phi$ is faithfully flat if there exists an $S$-dense open $U\subset X$, and a representative $\phi_U$ of $\phi$ such that $V := \phi_U(U)$ is an $S$-dense open of $Y$ and $\phi_U|_{[U \to V]}$ is faithfully flat.
    \end{enumerate}
\end{dfn}
The above definition is uncomfortable at first, since it is certainly not true that every representative of a flat/faithfully flat $S$-rational morphism is flat/faithfully flat. One would prefer to characterize properties of $S$-rational morphisms by properties of their representatives which are stable under $S$-dense opens. The next result shows that this is possible, for flat $S$-rational morphisms.
\begin{lem}\label{characterization of flat S-rational morphisms}
    Let $\phi\colon X\dashrightarrow Y$ be an $S$-rational morphism. The following are equivalent:
    \begin{enumerate}
        \item $\phi$ is flat.
        \item For some representative $\phi_U$ of $\phi$, there exists an $S$-dense open $U' \subset U$ such that $\phi_U|_{U'}$ is flat.
        \item For any representative $\phi_U$ of $\phi$, there exists an $S$-dense open $U' \subset U$ such that $\phi_U|_{U'}$ is flat.
    \end{enumerate}
\end{lem}
\begin{proof}
    Indeed, $1 \implies 2$ by definition, and $2 \implies 1$ since the morphism $\phi_U|_{U'}$ in the statement is a representative of $\phi$. $3 \implies 2$ is trivial.

    For $2 \implies 3$, with $\phi_U$ given as in 2, take any other representative $\phi_W\colon W \to Y$ of $\phi$. Then $U' \cap W \subset W$ is $S$-dense. In particular there exists an $S$-dense open $W'\subset U' \cap W$ such that $\phi_W|_{W'} = \phi_U|_{W'}$, and the latter is flat.
\end{proof}

Faithfully flat $S$-rational morphisms is equally well-behaved, provided that the schemes are locally of finite presentation:
\begin{lem}\label{characterization of faithfully flat S-rational morphisms}
    Let $X$ and $Y$ be schemes locally of finite presentation over $S$ and let $\phi\colon X\dashrightarrow Y$ be an $S$-rational morphism. The following are equivalent:
    \begin{enumerate}
        \item $\phi$ is faithfully flat.
        \item For some representative $\phi_U$ of $\phi$, there exists an $S$-dense open $U' \subset U$ such that $V := \phi_U(U')$ is an $S$-dense open of $Y$ and  $\phi_U|_{[U'\to V]}$ is faithfully flat.
        \item For any representative $\phi_U$ of $\phi$, there exists an $S$-dense open $U' \subset U$ such that, for any $S$-dense open $U'' \subset U'$, the image $V := \phi_U(U'')$ is an $S$-dense open of $Y$ and $\phi_U|_{[U''\to V]}$ is faithfully flat.
    \end{enumerate}
\end{lem}
\begin{proof}
    We have $1 \implies 2$ by definition, and $2 \implies 1$ since we have a representative $\phi_{U'} := \phi_U|_{U'}$ such that $\phi_{U'}|_{[U'\to V']}$ is the same as $\phi_{U}|_{[U'\to V']}$ which is faithfully flat. $3\implies 2$ is trivial.

    For $2 \implies 3$, with $U, U', V$ and $\phi|_U$ given as in $2$, take any other representative $\phi_W$ of $\phi$. Then there exists an $S$-dense open $W' \subset W \cap U'$ such that $\phi_W|_{W'} = \phi_U|_{W'}$. Take any $S$-dense open $W'' \subset W'$. Since everything is locally of finite presentation (\cref{representatives of rational morphisms being locally of finite presentation}) we can apply \cref{image of S-dense open under faithfully flat morphism locally of finite presentation}, and we see that the image $V' \colon= \phi_U|_{[U'\to V]}(W')$ is an $S$-dense open in $V$, hence in $Y$. Then $\phi_W|_{[W''\to V']}$ is the same as $\phi_U|_{[W''\to V']}$, which is faithfully flat, as required.
\end{proof}

A key example of a wffd morphism is a dominant morphism of integral varieties over a field:
\begin{lem}\label{a dominant moprhism of integral schemes over a field is faithfully flat}
    Let $\phi\colon X \to Y$ be a dominant rational morphism of integral varieties over a field $k$. Then $\phi$ is faithfully flat.
\end{lem}
\begin{proof}
    Since dense opens agree with $k$-dense opens (\cref{S-dominance over fields}) and the varieties are integral (so every nonempty open is dense), we know that $\phi$ is generically faithfully flat by \cref{Generic faithful flatness proposition}.
\end{proof}

\begin{lem}\label{faithfully flat S-rational morphisms are dominant}
    A faithfully flat $S$-rational morphism is $S$-dominant.
\end{lem}
\begin{proof}
    Immediate from \cref{a faithfully flat morphism is S-dominant}
\end{proof}

\begin{lem}\label{A flat S-rational morphism of locally Noetherian schemes is S-diffuse.}
    A flat $S$-rational morphism of locally Noetherian schemes is $S$-diffuse.
\end{lem}
\begin{proof}
    Immediate from \cref{flat morphisms of locally Noetherian schemes are S-diffuse}.
\end{proof}

%% file: Chapters/S-Rational_Morphisms/Composition_of_S-Rational_Morphisms.tex
\section{Composition of \texorpdfstring{$S$}{S}-Rational Morphisms}
The following allows us to define the postcomposition of an $S$-rational morphism with an arbitrary morphism:
\begin{lem}
    Let $\phi\colon X \dashrightarrow Y$ be an $S$-rational morphism, and $g\colon Y \to Z$ be a morphism. Let $\phi_U, \phi_V$ be representatives of $\phi$. Then $g\circ\phi_U$ is equivalent to $g\circ \phi_V$.
\end{lem}
\begin{proof}
    Indeed, let $W \subset U \cap V$ be an $S$-dense open such that $\phi_U|_W = \phi_V|_W$. Then $(g\circ\phi_U)|_W = (g\circ\phi_V)|_W$.
\end{proof}
\begin{dfn}\label{postcomposition by morphisms}
    Let $\phi\colon X \dashrightarrow Y$ be an $S$-rational morphism, and $g\colon Y \to Z$ be a morphism over $S$. We define the $S$-rational morphism $g\circ \phi$ to be the $S$-rational morphism induced by $g\circ\phi_U$ for any representative $\phi_U$ of $\phi$.
\end{dfn}

We can also precompose with diffuse $S$-rational morphisms:

\begin{prop}\label{composition of diffuse S-rational morphisms is well defined}
    Let $\phi\colon X \dashrightarrow Y$ and $\psi\colon Y \dashrightarrow Z$ be $S$-rational morphisms, and suppose $\phi$ is $S$-diffuse. Then
    \begin{enumerate}
        \item There exist representatives $\phi_U$, $\psi_V$ of $\phi$ and $\psi$ such that $\phi_U(U) \subset V$.
        \item The rational morphism from $X$ to $Z$ induced by $\psi_V \circ \phi_U$ is independent of the choice of representatives $\phi_U$, $\psi_V$ with the above property.
    \end{enumerate}
\end{prop}
\begin{proof}
    For the first statement, any representatives $\phi_{U'}$, $\psi_V$, and let $U = \phi_{U'}\inv( V )$. Then since $\phi$ is $S$-diffuse, $U$ is $S$-dense in $U'$, hence in $X$, and we have a representative $\phi_U = \phi_{U'}|_U$.

    For the second statement, given any other representatives $\phi_{U'}$, $\psi_{V'}$ such that $\phi_{U'}(U') \subset V'$, fix an $S$-dense open $V'' \subset V \cap V'$ such that $\psi_V|_{V''} = \psi_{V'}|_{V''}$. Since $\phi$ is diffuse, we may find an $S$-dense open $U'' \subset U' \cap U'' \cap \phi_{U}\inv(V'')$ such that $\phi_U|_{U''} = \phi_{U'}|_{U''}$. Then $(\psi_{V'} \circ \phi_{U'})|_{U''} = (\psi_V \circ \phi_U)|_{U''}$, as required.
\end{proof}

\begin{dfn}\label{composition of S-rational morphisms definition}
    Let $\phi\colon X \dashrightarrow Y$, $\psi\colon Y \dashrightarrow Z$ be $S$-rational morphisms, and suppose $\phi$ is $S$-diffuse. We define the \emph{composition} of $\psi$ and $\phi$, which we denote by $\psi \circ \phi$, to be the rational morphism from $X$ to $Z$ induced by $\psi_V \circ \phi_U$ for any representatives $\phi_U$, $\psi_V$ such that $\phi_U(U) \subset V$.
\end{dfn}
\begin{notation}
    Following \cref{Yoneda Notation}, we will frequently denote $\psi \circ \phi$ by $\psi(\phi)$.
\end{notation}
\begin{rem}
    If $\phi\colon X \dashrightarrow Y$ is an $S$-diffuse $S$-rational morphism and $g\colon Y \to Z$ is a morphism with corresponding $S$-rational morphism $\psi$  (see \cref{embedding of morphisms to S-rational morphisms}), it follows immediately from the definitions that $g\circ \phi$ to be $\psi \circ \phi$, so there is no conflict with our previous definition.
\end{rem}
\begin{dfn}
    Let $f\colon X \to Y$ be a diffuse morphism of schemes over $S$, and $\phi\colon X \dashrightarrow Y$ be the corresponding $S$-rational morphism. For any $S$-rational morphism $\psi\colon Y \to Z$, denote by $\psi\circ f$ the $S$-rational morphism $\psi\circ \phi$.
\end{dfn}
\begin{rem}
    One could also define the composition $\psi\circ \phi$ assuming that $Z$ has separated fibers over $S$ and $\phi_U\inv(\dom(\psi))$ is $S$-dense in $X$ for some (equivalently, by \cref{preimage being S-dense is stable under restriction to S-dense opens}, for every) representative $\phi_U$ of $\phi$. But we will not need this for our purposes.
\end{rem}

The following observation will be surprisingly helpful:
\begin{rem}\label{useful remark}
    Let $\phi\colon X \to Y$ be an $S$-rational morphisms, and let $x = \id_X \in X(X)$. It follows immediately from the definitions that $\phi(x) = \phi$.
\end{rem}

\begin{lem}\label{composition of S-rational morphisms preserves many properties}
    Let $\phi\colon X \to Y$ and $\psi\colon Y \to Z$ be $S$-rational morphisms with $\phi$ $S$-diffuse. Then
    \begin{itemize}
        \item If $\psi$ is also $S$-diffuse, then  $\psi\circ \phi$ is $S$-diffuse
        \item If $\phi$ and $\psi$ are $S$-dominant, then so is $\psi\circ \phi$.
        \item If $\phi$ and $\psi$ are flat, then $\psi\circ \phi$ is flat.
        \item If $\phi$ and $\psi$ are faithfully flat, and $Y$ and $Z$ are locally of finite presentation over $S$, then $\psi\circ \phi$ is faithfully flat.
    \end{itemize}
    Note that in particular we get a category of schemes over $S$ and diffuse $S$-rational morphisms.
\end{lem}
\begin{proof}
    \begin{enumerate}
        \item Immediate from \cref{A composition of S-diffuse morphisms is S-diffuse}.

        \item Immediate from \cref{composition is schematically dominant implies f is schematically dominant}.

        \item Let $\phi_U$ and $\psi_V$ be flat representatives, and let $U' = \phi_U\inv(V)$. Since $\phi$ is $S$-diffuse, we have that $U'$ is $S$-dense in $U$, so $\phi_{U'} \colon= \phi_U|_{U'}$ is a representative of $\phi$. Then $\psi_V \circ \phi_{U'}$ is flat.

        \item Suppose we have $S$ dense opens $U \subset X$, $V,V' \subset Y$ and  $W \subset Z$ and representatives $\phi_U$, $\psi_V$ with $\phi_U|_{[U \to V]}$ and $\psi_{V'}|_{V' \to W}$ faithfully flat. Let $V'' = V \cap V'$, $U'' = \phi_U\inv(V'')$ and $W'' = \psi_{V'}(V'')$. Then $V''$ is $S$-dense in $Y$. Since $\phi$ is diffuse, $U''$ is also $S$-dense in $X$. In particular we have representatives $\phi_{U''} := \phi_{U}|_{U''}$ and $\psi_{V''} := \psi_{V'}|_{V''}$, and $\psi_{V''} \circ \phi_{U''}$ is a representative of the composition.

              We have that $\psi_{V'}$ is locally of finite presentation by \cref{representatives of rational morphisms being locally of finite presentation}, so by \cref{image of S-dense open under faithfully flat morphism locally of finite presentation} we get that $W''$ is an $S$-dense open of $Z$. Since $\phi_{U''}|_{[U'' \to V'']}$ and $\psi_{V''}|_{[V'' \to W]}$ are faithfully flat we conclude that $(\psi_{V''} \circ \phi_{U''})|_{[U'' \to W'']}$ is faithfully flat, so we conclude.
    \end{enumerate}
\end{proof}

\begin{lem}\label{composition faithfully flat criterion for S-rational morphisms}
    Let $X,Y,Z$ be schemes locally of finite presentation over $S$, and let $\phi\colon X \dashrightarrow Y$ and $\psi\colon Y \dashrightarrow Z$ be $S$-rational morphisms. Suppose $\phi$ is faithfully flat and diffuse, and $\psi \circ \phi$ is faithfully flat. Then $\psi$ is faithfully flat.
\end{lem}
\begin{proof}
    Since $\phi$ is diffuse, we may find representatives $\psi_V$ of $\psi$ and $\phi_U$ of $\phi$ with $\phi_U(U) \subset V$. Applying \cref{characterization of faithfully flat S-rational morphisms} to $\phi$ and $\psi\circ \phi$, we may find a dense open $U' \subset U$ such that $V' := \phi_{U}(U')$ and $W' := \psi_V \circ \phi_{U'}(U')$ are $S$-dense opens and both morphisms $\phi_U|_{[U' \to V']}$ and $\psi_V \circ \phi_{U'}|_{[U' \to W']}$ are faithfully flat. In particular we have representatives $\phi_{U'} := \phi_U|_{U'}$ and $\psi_{V'} := \psi_V|_{V'}$ with $\psi_{V'}(V') = W'$ and $\phi_{U'}|_{[U' \to W']}$ and $\psi_{V'}|_{[ V' \to W' ]} \circ \phi_{U'}|_{[U' \to V']}$ both faithfully flat. By \cref{if f is faithfully flat and g circ f is flat then g is flat} we have that $\psi_{V'}|_{{V' \to W'}}$ is flat, since it is also surjective we conclude.
\end{proof}

\begin{lem}[Composition of $S$-rational morphisms is associative]\label{composotion of S-rational morphisms is associative}
    Let $f\colon X \dashrightarrow Y$, $g\colon  Y \dashrightarrow Z$, and $h\colon Z \dashrightarrow A$ be $S$-rational morphisms, with $f$ and $g$ diffuse. Then $h\circ(g\circ f)$ (which is defined by the previous result) is equal to $(h\circ g)\circ f$.
\end{lem}
\begin{proof}
    Since $f$ and $g$ are diffuse, we may find representatives $f_U$, $g_V$, $h_W$ such that $f_U(U) \subset V$ and $g_V(V)\subset W$. Then by \cref{composition of diffuse S-rational morphisms is well defined}, $h\circ(g\circ f)$ and $(h\circ g)\circ f$ are both represented by $h_W\circ g_V \circ f_U$.
\end{proof}

In the category of schemes locally of finite presentation over $S$ and diffuse $S$-rational morphisms, faithfully flat morphisms are epimorphisms:
\begin{lem}\label{faithfully flat morphisms are S-rational epimorphisms}
    Let $f\colon X \dashrightarrow Y$ be a faithfully flat and diffuse $S$-rational morphism, and $g_1,g_2\colon Y \dashrightarrow Z$ be $S$-rational morphisms. Suppose $X$ and $Y$ are locally of finite presentation over $S$, and $g_1\circ f = g_2 \circ f$. Then $g_1 = g_2$.
\end{lem}
\begin{proof}
    Since $f$ is diffuse we may choose $S$-dense opens $V \subset Y$ and $U \subset X$ and representatives $g_{1V}, g_{2V}, f_U$ with $f_{U}(U) \subset V$. Then $g_1\circ f$ and $g_2 \circ f$ are represented by $g_{1V} \circ f_U$ and $g_{2V}\circ f_U$ respectively. Since $g_1\circ f = g_2 \circ f$, there exists an $S$-dense open $U' \subset U$ such that the restrictions to $U'$ agree; i.e., if $f_{U'} = f_U|_{U'}$, we have $g_{1V}\circ f_{U'} = g_{2V}\circ f_{U'}$.

    Since $f$ is faithfully flat, by \cref{characterization of faithfully flat S-rational morphisms} there exist $S$-dense opens $U'' \subset U'$ and $V' \subset V$ such that $f_{U'}(U'') = V'$ and $f_{U'}|_{[U'' \to V']}$ is faithfully flat. Then we have representatives $f_{U''}:= f_{U'}|_{U''}$, $g_{1V'} := g_{1V}|_{V'}$ and $g_{2V'} := g_{2V}|_{V'}$, and $g_{1,V'}\circ f_{U''} = g_{2,V'}\circ f_{U''}$. Since $f_{U''}|_{[U'' \to V']}$ is faithfully flat, hence an epimorphism (by \cref{A faithfully flat morphism is an epimorphism of schemes}), we conclude that $g_{1V'} = g_{2V'}$ and therefore $g_1 = g_2$.
\end{proof}

Precomposition by faithfully flat $S$-rational morphisms preserves schematic image:
\begin{lem}\label{Precomposition by faithfully flat S-rational morphisms preserves schematic image}
    Let $\phi\colon X \dashrightarrow Y$ and $\psi\colon Y \dashrightarrow Z$ be $S$-rational morphisms, and suppose $\phi$ is faithfully flat. Then $\Schim(\psi \circ \phi) = \Schim(\psi)$
\end{lem}
\begin{proof}
    Since $\phi$ is faithfully flat, we may find representatives $\psi_V$ and $\phi_U$ such that $\phi_U(U) = V$ and $\phi_U|_{[U \to V]}$ is faithfully flat. Then by definition and \cref{Precomposition by faithfully flat morphisms preserves schematic image} we have
    \begin{align*}
        \Schim(\psi \circ \phi) = \Schim(\psi_V \circ \phi_U) = \Schim(\psi_V) = \Schim(\psi).
    \end{align*}
\end{proof}

%% file: Chapters/S-Rational_Morphisms/Products_of_S-rational_morphisms.tex
\section{Products of \texorpdfstring{$S$}{S}-Rational Morphisms}
First we discuss ``pairs'' of $S$-rational morphisms; i.e., $S$-rational morphisms to products.

\begin{lem}
    Let $f\colon X\dashrightarrow Y$, $g\colon X\dashrightarrow Z$ be $S$-rational morphisms. Choose an $S$-dense open $U$ such that there exist representatives $f_U$, $g_U$ of $f$ and $g$. Then the $S$-rational morphism $X \dashrightarrow Y \times_S Z$ induced by $(f_U,g_U)\colon U \to Y \times_S Z$ is independent of the choice of $U, f_U, g_U$.
\end{lem}
\begin{proof}
    Fix any other $U', f_{U'}, g_{U'}$. Then there exists $S$-dense $U'' \subset U \cap U'$ such that $f_U|_{U''} = f_{U'}|_{U''}$ and $g_U|_{U''} = g_{U'}|_{U''}$. Then
    \begin{align*}
        (f_U,g_U)|_{U''} = (f_U|_{U''},g_U|_{U''}) = ((f_{U'},g_{U'}))|_{U''}.
    \end{align*}
    hence the result.
\end{proof}

In light of the above we make the following definition:
\begin{dfn}[Pairs of $S$-Rational Morphisms]\label{pairs definition}
    Let $\phi\colon X\dashrightarrow Y$, $\psi\colon X\dashrightarrow Z$ be $S$-rational morphisms. We denote by $(\phi,\psi)$ the $S$-rational morphism from $X$ to $Y\times_S Z$ induced by $(\phi_U, \psi_U)$ for any representatives $\phi_U$, $\psi_U$ of $\phi$ and $\psi$ over the same $S$-dense open $U$.
\end{dfn}
\begin{rem}\label{pairs of S-rational morphisms and morphisms}
    If $f\colon X \to Y$ is a morphism, with associated $S$-rational morphism $\phi$, and $\psi\colon Y \to Z$ is an $S$-rational morphism, we can define the pair $(f,\psi)$ to be equal to $(\phi,\psi)$. Similarly we define $(\phi,g)$ for morphisms $g\colon Y \to Z$.
\end{rem}

Similarly we can define products of $S$-rational morphisms:
\begin{lem}
    Let $f_1\colon X_1 \dashrightarrow Y_1$, $f_2\colon X_2 \dashrightarrow Y_2$ be $S$-rational morphisms with representatives $f_{1U}$, $f_{2V}$. Then $U \times_S V$ is an $S$-dense open of $X\times_S Y$, and the $S$-rational morphism from $X_1\times_S X_2$ to $Y_1 \times_S Y_2$ induced by $f_{1,U}\times_S f_{2,V}$ is independent of the choice of representatives $f_{1U}$, $f_{2V}$.
\end{lem}
\begin{proof}
    The $S$-density of $U\times V$ follows from \cref{products of S-dense opens are S-dense.}.

    Given any other representatives $f_{1U'}$, $f_{2V'}$, fix $S$-dense opens $U'', V''$ such that
    $f_{1U'}|_{U''} = f_{1U}|_{U''} $and $f_{2V'}|_{V''} = f_{2V}|_{V''}$. Then
    \begin{align*}
        (f_{1U}\times_S f_{2V})|_{U'' \times_S V''} = (f_{1U}|_{U''}\times_S f_{2V}|_{V''}) = (f_{1,U'}\times_S f_{2V'})|_{U'' \times_S V''}.
    \end{align*}
\end{proof}

\begin{dfn}
    Let $f_1\colon X_1 \dashrightarrow Y_1$, $f_2\colon X_2 \dashrightarrow Y_2$ be $S$-rational morphisms. We denote by $f_1\times_S f_2$ the $S$ rational morphism from $X_1\times_S X_2$ to $Y_1 \times_S Y_2$ induced by $f_{1U}\times_S f_{2V}$ for any representatives $f_{1U}, f_{2V}$.
\end{dfn}
\begin{rem}
    We can define products of morphisms with $S$-rational morphisms analogously to \cref{pairs of S-rational morphisms and morphisms}.
\end{rem}

\begin{lem}\label{properties preserved by products of $S$-rational morphisms}
    \begin{enumerate}
        \item Products of $S$-dominant $S$-rational morphisms are $S$-dominant.
        \item Products of flat $S$-rational morphisms are flat.
        \item Products of faithfully flat $S$-rational morphisms are faithfully flat.
    \end{enumerate}
\end{lem}
\begin{proof}
    \begin{enumerate}
        \item Follows form \cref{A product of S-dominant morphisms is S-dominant}.
        \item Immediate since products of flat morphisms are flat, and by \cref{products of S-dense opens are S-dense.}.
        \item Immediate since products of faithfully flat morphisms are faithfully flat, and by \cref{products of S-dense opens are S-dense.}.
    \end{enumerate}
\end{proof}

%% file: Chapters/S-Rational_Morphisms/S-Birational_Morphisms.tex
\section{ \texorpdfstring{$S$}{S}-Birational Morphisms}\label{section: S-birational morphisms}
\begin{dfn}\label{S-birational morphism definition}
    Let $\phi\colon X \to Y$ be an $S$-rational morphism. We say $\phi$ is an \emph{$S$-birational morphism} if there exist $S$-dense opens $U \subset X$ and $V \subset Y$ and a representative $\phi_U$ with $\phi_U(U) = V$ and such that $\phi_U|_{[U \to V]}$ is an isomorphism.

    If $\phi$ is $S$-birational, and $U$, $V$ are as above, we define the inverse of $\phi$, denoted $\phi\inv$, to be the $S$-rational morphism induced by $(\phi_U|_{[U \to V]})\inv$. Note this is independent of the choice of $U,V$ and $\phi_U$.
\end{dfn}
\begin{rem}
    Observe that if $\phi$ is $S$-birational than so is $\phi\inv$, and $(\phi\inv)\inv = \phi$.
\end{rem}

We now make a definition which we will not see again until the next section:
\begin{dfn}[$S$-Birationally Projective Schemes]\label{$S$-birationally projective definition}
    Let $X$ be a scheme over $S$. We say that $X$ is $S$-birationally projective if there exists an $S$-birational morphism from $X$ to a scheme which is projective over $S$.
\end{dfn}

\begin{lem}\label{S-birational S-rational morphisms are faithfully flat}
    An $S$-birational morphism is faithfully flat and diffuse.
\end{lem}
\begin{proof}
    If $\phi_U|_{[U \to V]}$ is an isomorphism then it is certainly faithfully flat, and $\phi_U$ is diffuse.
\end{proof}

\begin{rem}\label{composing S-rational inverses gives the identity}
    Observe that if $\phi$ is $S$-birational, then by above $\phi \circ \phi\inv$ and $\phi\inv\circ \phi$ are both defined and equal as $S$-rational morphisms to the identity on $Y$, $X$ respectively.
\end{rem}

\begin{cor}\label{faithful flatness of S-rational morphisms preserved by composing with S-birational morphisms}
    Let $f\colon X \dashrightarrow Y$ be an $S$-rational morphism, and $h\colon W \dashrightarrow X$ be an $S$-birational morphism. Then $f$ is diffuse/faithfully flat if and only if $f\circ h$ is diffuse/faithfully flat, respectively.

    Similarly, if $f$ is diffuse and $g\colon Y \dashrightarrow Z$ is an $S$-birational morphism, then $f$ is faithfully flat if and only if $g\circ f$ is faithfully flat.
\end{cor}
\begin{proof}
    Immediate from the above remark and \cref{composition of S-rational morphisms preserves many properties}.
\end{proof}

The next lemma is more difficult than one would like it to be.
\begin{lem}\label{criteria for S-birationality}
    Let $f\colon X \dashrightarrow Y$ and $g\colon X \dashrightarrow Y$ be diffuse $S$-rational morphisms. Suppose $f\circ g = \id_Y$, and $g \circ f = \id_X$. Then $f$ is $S$-birational, with inverse $g$.
\end{lem}
\begin{proof}
    Choose representatives $g_{V_1}, f_{U_1}$ of $f$ and $g$. Since $f$ and $g$ are diffuse, we may choose representatives $g_{V_2}, f_{U_2}$ such that $f_{U_2}(U_2) \subset V_1$, $g_{V_2}(V_2) \subset U_1$. Shrinking $U_2, V_2$ we may assume
    \begin{align*}
        U_2 \subset U_1, &  & f_{U_2} = f_{U_1}|_{U_2}, \\
        V_2 \subset V_1, &  & g_{V_2} = g_{V_1}|_{V_2}.
    \end{align*}

    Then $f_{U_1} \circ g_{V_2}$ represents  $f\circ g = \id_Y$, and $g_{V_1} \circ f_{U_2}$ represents $g\circ f = \id_X$. Therefore, shrinking $U_2, V_2$, we may assume $f_{U_1} \circ g_{V_2} = \id_Y|_{V_2}$ and $g_{V_1} \circ f_{U_2} \id_X|_{U_2}$.

    Define
    \begin{align*}
        U = f_{U_2}\inv(V_2), &  & f_{U} = f_{U_2}|_{U}  \\
        V = g_{V_2}\inv(U_2), &  & g_{V} = g_{V_2}|_{V}.
    \end{align*}
    Then $U$ and $V$ are $S$-dense in $X$ and $Y$ respectively since $f$ and $g$ are diffuse, and
    \begin{align*}
        g_{V_2} \circ f_{U} = (g_{V_1} \circ f_{U_2})|_{U} = \id_X|_{U}, \\
        f_{U_2} \circ g_{V} = (f_{U_1} \circ g_{V_2})|_{V} = \id_Y|_{V}.
    \end{align*}
    In particular,
    \begin{align*}
        f_U(U) \subset g_{V_2}\inv(U) \subset g_{V_2}\inv(U_2) = V, \\
        g_V(V) \subset f_{U_2}\inv(V) \subset f_{U_2}\inv(V_2) = U,
    \end{align*}
    and we have $f_U|_{[U \to V]} \circ g_V|_{[V \to U]} = \id_V$ and $g_V|_{[V \to U]} \circ f_U|_{[U \to V]} = \id_U$, as required.
\end{proof}

\begin{prop}\label{composition of S-birational morphisms}
    Let $f\colon X\dashrightarrow Y$ and $g\colon Y \dashrightarrow Z$ be $S$-birational morphisms. Then $g\circ f$ is $S$-birational with inverse $f\inv\circ g\inv$.
\end{prop}
\begin{proof}
    Suppose $f_{U_1}$ induces an isomorphism of $S$-dense opens $U_1 \xrightarrow{\sim} V_1$ and $g_{V_2}$ induces an isomorphism $V_2 \xrightarrow{\sim} W_2$. Define
    \begin{align*}
        V = V_1 \cap V_2,   &  & g_V = g_{V_2}|_{V}, \\
        U = f_{U_1}\inv(V), &  & f_U = f_{U_1}|_{U}, \\
        W = g_{V_2}(V) =  g_V(V).
    \end{align*}
    Since $f_{U_1}|_{[U_1 \to V_1]}$ and $g_{V_2}|_{[V_2 \to W_2]}$ are isomorphisms, it follows that $U$ and $W$ are $S$-dense opens of $U_1$ and $W_1$ respectively, and then that $(g_{V} \circ f_U)|_{[U \to W]}$ is an isomorphism with inverse $(f_U\inv \circ g_V\inv)|_{[W \to U]}$; the result follows.
\end{proof}

%% file: Chapters/S-Rational_Morphisms/Graphs_of_S-Rational_Morphisms.tex
\section{Graphs of \texorpdfstring{$S$}{S}-Rational Morphisms}
Recall \cref{Graphs and schematic images}.
\begin{dfn}
    Let $\phi\colon X \dashrightarrow Y$ be an $S$-rational morphism. We define the closed graph of $\phi$, denoted $\Lambda_\phi$, to be the schematic image of the induced $S$-rational morphism $(id_X,\phi)\colon X \dashrightarrow X \times_S Y$.
\end{dfn}

\begin{rem}\label{closed graph is schematic image of restriction graph.}
    Let $\phi$ be as above. Suppose that, for some representative $\phi_U$, we have that $\phi_U$ factors through a subscheme $V \subset Y$. Then $(\id_U,\phi_U)$ factors through  $U \times_S V$, and if $\Gamma_{\phi_U|_{[U \to V]}}$ is the graph of $\phi_U|_{[U \to V]}$ we have a commutative diagram
    \[
        \begin{tikzcd}
            U \arrow[rrd, "{(id_X|_U, \phi_U)}"] \arrow[d] &                              &              \\
            {\Gamma_{\phi_U|_{[U \to V]}}} \arrow[r, hook] & U \times_S V \arrow[r, hook] & X \times_S Y
        \end{tikzcd}
    \]
    where the vertical arrow is an isomorphism. Thus $\Lambda_\phi = \Schim_{X\times_S Y}(\Gamma_{\phi_U|_{[U \to V]}}).$
\end{rem}

The following lemma (and in particular the corollary which follows) will be crucial.
\begin{prop}\label{Intersection of grap hs with domains}
    Let $\phi\colon X \dashrightarrow Y$ be an $S$-rational morphism of schemes with closed graph $\Lambda_\phi$. Suppose $Y$ is separated.

    Let $\phi_U$ be a representative of $\phi$, with graph $\Gamma_{\phi_U}$. Let $j_U\colon U \hookrightarrow X$ be the corresponding open immersion. Consider the morphism $(j_U, \phi_U)\colon U \to X\times_S Y$. Suppose that $\ker(\O_{X \times_S Y} \to (j_U, \phi_U)_*{\O_U})$ is quasicoherent.

    Then $\Lambda_{\phi} \cap (U \times_S Y) = \Gamma_{\phi_U}$.
\end{prop}
\begin{proof}
    By definition, since $(j_U, \phi_U)$ represents $(\id_X,\phi)$, we have $\Lambda_\phi = \Schim((j_U, \phi_U))$. Since $Y$ is separated, we have $\Gamma_{\phi_U} = \Schim(\id_U,\phi_U)$ by \cref{Graphs and schematic images}. We have a Cartesian diagram
    \[
        \begin{tikzcd}
            U \arrow[r, "\id_U"] \arrow[d, "{(\id_U,\phi_U)}"] & U \arrow[d, "{(j_U,\phi_U)}", hook'] \\
            U\times_S Y \arrow[r, hook]                         & X\times_S Y
        \end{tikzcd}
    \]
    Since $\ker(\O_{X \times_S Y} \to (j_U, \phi_U)_*{\O_U})$ is quasicoherent, we have that taking schematic image commutes with restriction to opens by \cref{Schematic image of well-behaved morphisms}, so we conclude.
\end{proof}
\begin{cor}\label{Intersection of closed graph with product from domain in Noetherian to separated case}
    Let $\phi\colon X \dashrightarrow Y$ be an $S$-rational morphism of schemes with closed graph $\Lambda_\phi$. Suppose $X$ is Noetherian, and $Y$ is separated.

    Let $\phi_U$ be a representative of $\phi$, and let $j_U\colon U \hookrightarrow X$ be the corresponding open immersion. Then $\Lambda_{\phi} \cap (U \times_S Y) = \Gamma_{\phi_U}$.
\end{cor}
\begin{proof}
    Since $X$ is Noetherian, so is $U$, and it follows that $(j_U, \phi_U)$ is quasicompact, and therefore $\ker(\O_{X \times_S Y} \to (j_U, \phi_U)_*{\O_U})$ is quasicoherent by \cref{quasicoherent kernel lemma}. We conclude by the previous result.
\end{proof}

The above result implies that $S$-rational morphisms from Noetherian schemes to separated schemes are determined by their closed graphs:
\begin{prop}\label{S-rational morphisms from Noetherian schemes to separated schemes are determined by their closed graphs:}
    Let $X$, $Y$ be $S$-schemes with $X$ Noetherian and $Y$ separated. Let $\phi_1,\phi_2$ be $S$-rational morphisms from $X$ to $Y$. Suppose $\Lambda_{\phi_1}= \Lambda_{\phi_2}$. Then $\phi_1 = \phi_2$.
\end{prop}
\begin{proof}
    Fix any $S$-dense open $U \subset \dom(\phi_1) \cap \dom(\phi_2).$ By above we have
    \begin{align*}
        \Gamma_{\phi_{1,U}} = \Lambda_{\phi_1}\cap(U \times_S Y) = \Gamma_{\phi_{2,U}}
    \end{align*}
    hence $\phi_{1,U} = \phi_{2,U}$ and $\phi_1 =\phi_2$.
\end{proof}

\begin{lem}\label{graphs of representatives of inverse S-rational morphisms}
    Let $\phi\colon X \dashrightarrow Y$ be an $S$-birational morphism with inverse $\psi$. Take representatives $\phi_U$, $\psi_V$ inducing mutually inverse isomorphisms on $S$-dense opens $U\subset X$, $V \subset Y$. Let $\phi' = \phi_U|_{[U \to V]}$ and $\psi' = \psi_V|_{[V \to U]}$.

    Then we have a Cartesian diagram
    \[
        \begin{tikzcd}
            \Gamma_{\phi'} \arrow[r] \arrow[d, hook] & \Gamma_{\psi'} \arrow[d, hook] \\
            X \times_S Y \arrow[r, "\rho_{YX}"]    & Y\times_S X,
        \end{tikzcd}
    \]
    where $\rho_{YX}$ denotes the coordinate permutation (as in \cref{products and Yoneda combine}) and the vertical arrows are the usual embeddings.
\end{lem}
\begin{proof}
    By \cref{Inverse and morphism graph lemma}, we have a Cartesian diagram
    \[
        \begin{tikzcd}
            \Gamma_{\phi'} \arrow[r] \arrow[d, hook] & \Gamma_{\psi'} \arrow[d, hook] \\
            U \times_S V \arrow[r, "\rho_{VU}"]    & V\times_S U
        \end{tikzcd}
    \]
    Then the result follows by the Cartesian rectangle lemma (\cref{Cartesian rectangles}).
\end{proof}

The closed graph of an $S$-birational morphism is the permutation of the closed graph of its inverse:
\begin{prop}\label{Graphs of S-birational S-rational morphisms}
    Let $\phi\colon X \dashrightarrow Y$ be an $S$-birational morphism with inverse $\psi$. Then we have a Cartesian diagram
    \[
        \begin{tikzcd}
            \Lambda_\phi \arrow[r] \arrow[d, hook] & \Lambda_{\psi} \arrow[d, hook] \\
            X\times_S Y \arrow[r, "\rho_{YX}"]   & Y \times_S X
        \end{tikzcd}
    \]

\end{prop}
\begin{proof}
    Suppose we have representatives $\phi_U, \psi_V$ inducing mutually inverse isomorphisms on $S$-dense opens $U\subset X$, $V \subset Y$. As above, let $\phi' = \phi_U|_{[U \to V]}$ and $\psi' = \psi_V|_{[V \to U]}$. By \cref{closed graph is schematic image of restriction graph.}, we have that $\Lambda_{\phi} = \Schim_{X \times_S Y}(\Gamma_{\phi'})$ and $\Lambda_{\psi} = \Schim_{Y \times_S X}(\Gamma_{\psi'})$.

    By the previous result, we have a Cartesian diagram
    \[
        \begin{tikzcd}
            \Gamma_{\phi'} \arrow[r] \arrow[d, hook] & \Gamma_{\psi'} \arrow[d, hook] \\
            X \times_S Y \arrow[r, "\rho_{YX}"]    & Y\times_S X
        \end{tikzcd}
    \]
    Since $\rho_{YX}$ is an isomorphism, and pullback by isomorphisms preserves schematic images, the result follows.
\end{proof}

The above result has a partial converse:
\begin{prop}\label{detecting inverse $S$-rational morphisms by looking at graphs}
    Let $X$, $Y$ be schemes over $S$ with $Y$ Noetherian and $X$ separated. Let $\phi\colon X \dashrightarrow Y$, $\theta\colon Y \dashrightarrow X$ be $S$-rational morphisms. Suppose $\phi$ is $S$-birational, and we have a Cartesian diagram
    \begin{equation}\label{proposed inverse cartesian diagram}
        \begin{tikzcd}
            \Lambda_\phi \arrow[r] \arrow[d, hook] & \Lambda_{\theta} \arrow[d, hook] \\
            X\times_S Y \arrow[r, "\rho_{YX}"]   & Y \times_S X
        \end{tikzcd}
    \end{equation}
    Then $\theta = \phi\inv$.
\end{prop}
\begin{proof}
    Let $\psi = \phi\inv$. By above, we have a Cartesian diagram
    \[
        \begin{tikzcd}
            \Lambda_\phi \arrow[r] \arrow[d, hook] & \Lambda_{\psi} \arrow[d, hook] \\
            X\times_S Y \arrow[r, "\rho_{YX}"]   & Y \times_S X.
        \end{tikzcd}
    \]
    Observe that the horizontal arrows in the above and \cref{proposed inverse cartesian diagram} are isomorphisms. Inverting them we get Cartesian diagrams
    \[
        \begin{tikzcd}
            \Lambda_\phi  \arrow[d, hook] & \Lambda_{\theta} \arrow[l] \arrow[d, hook] && \Lambda_\phi  \arrow[d, hook] & \Lambda_{\psi} \arrow[l] \arrow[d, hook] \\
            X\times_S Y    & Y \times_S X  \arrow[l, "\rho_{XY}"'] &&  X\times_S Y    & Y \times_S X   \arrow[l, "\rho_{XY}"']
        \end{tikzcd}
    \]
    it follows by uniqueness of the fiber product that $\Lambda_\psi = \Lambda_\theta$, and then by \cref{S-rational morphisms from Noetherian schemes to separated schemes are determined by their closed graphs:} that $\psi = \theta$.
\end{proof}

%% file: Chapters/Rational_Families/Main_Rational_Families.tex
\chapter{\texorpdfstring{$S$}{S}-Rational Families}\label{chapter: S-rational families}

In this chapter we study rational families of $S$-rational morphisms and their parametrizations. The basic notions of composition and pullback are defined in \cref{section: S-rational families}. In \cref{section: invertible families} we define invertible $S$-rational families, which will be important in the next chapter, and study them using their graphs. In \cref{section: canonical S-rational families} we define the notion of a tame family, and show that tame $S$-rational families admit canonical parametrizations. In \cref{section: transformations of canonical families} we show that taking compositions and inverses of rational families induce corresponding morphisms of the canonical parameter spaces.

\input{Chapters/Rational_Families/S-Rational_Families}

\input{Chapters/Rational_Families/Invertible_Families.tex}
\input{Chapters/Rational_Families/Canonical_S-rational_families}
\input{Chapters/Rational_Families/Transformations_of_Canonical_Families.tex}

%% file: Chapters/Rational_Families/S-Rational_Families.tex
\section{\texorpdfstring{$S$}{S}-Rational Families}\label{section: S-rational families}

For the following sections, we work in the category of schemes over a fixed base scheme $S$. We make heavy use of \cref{products and Yoneda combine}.

\begin{dfn}\label{S-rational families definition}
    Let $A, X, Y$ be $S$-schemes. We define an $S$-rational family of morphisms (or $S$-rational family, for short) from $X$ to $Y$ over $A$ to be an $S$-rational morphism $\phi\colon AX \dashrightarrow Y$.
\end{dfn}

\begin{rem}
    We use the above terminology for the following reason: if $S = \Spec(k)$, and $A,X,Y$ are varieties over $k$, with $X$ geometrically integral, then the (set-theoretic) image of the projection of $\dom \phi$ to $A$ is a dense open subscheme $U \subset A$, and for every point $a \in U$ the pullback $(\dom \phi)_a$ is a dense open subscheme of $X_a$, so we get a rational map from $X_a \to Y$. So in this case $\phi$ really does give a family of rational maps from $X$ to $Y$ parametrized by a dense open subset of $A$.
\end{rem}

\begin{dfn}[Composition of $S$-rational families]\label{composition of S-rational families}
    Let $A, B, X, Y, Z$ be $S$-schemes, and $\phi\colon AX \dashrightarrow Y$, $\psi\colon BY\dashrightarrow Z$ be $S$-rational families over $A$ and $B$ respectively.

    Suppose that the induced $S$-rational morphism
    \begin{align*}
        (\id_B \times \phi)\colon BAX \dashrightarrow BY
    \end{align*}
    is diffuse. Then we define the composition of $\phi$ and $\psi$, denoted $\psi * \phi$, to be the $S$-rational family over $BA$ induced by the composition
    \[
        \begin{tikzcd}
            BAX \arrow[r, "\id_B\times \phi", dashed] & BY \arrow[r, "\psi",dashed] & Z.
        \end{tikzcd}
    \]
    Equivalently, using \cref{products and Yoneda combine} if $(b,a,x) = \id_{BAX}$, we have
    \begin{align*}
        \psi * \phi = \psi(b,\phi(a,x)).
    \end{align*}
\end{dfn}

In light of the above, we make the following convention:
\begin{notation}\label{composition is defined notation}
    Let $\phi$, $\psi$ be as above. We say that $\psi * \phi$ \emph{is defined} if $\id_B \times \phi$ is diffuse.
\end{notation}

Composition of $S$-rational families is associative:
\begin{prop}\label{Composition of S-rational families is associative.}
    Let $\phi\colon AX \dashrightarrow Y$, $\psi\colon BY\colon \dashrightarrow Z$, $\theta\colon CZ \dashrightarrow W$ be $S$-rational families over $A,B,C$ respectively.

    Suppose $\psi*\phi$, $\theta * \psi$, $\theta*(\psi * \phi)$ and $(\theta*\psi)*\phi$ are all defined.
    Then $\theta*(\psi * \phi) = (\theta*\psi)*\phi$.
\end{prop}
\begin{proof}
    Immediate since both $S$-rational families are given by the composition
    \[
        \begin{tikzcd}
            CBAX \arrow[r, "\id_{CB}\times \phi",dashed] & CBY \arrow[r, "\id_C \times \psi",dashed] & CZ \arrow[r, "\theta",dashed] & W.
        \end{tikzcd}
    \]
\end{proof}

\begin{dfn}[Pullbacks]
    Let $A,B,X,Y$ be $S$-schemes, $\phi\colon AX \dashrightarrow Y$ an $S$-rational family over $A$, and $f\colon B \dashrightarrow A$ an $S$-rational morphism. Suppose that the product
    \begin{align*}
        f\times \id_X \colon BX \dashrightarrow AX
    \end{align*}
    is diffuse. Then we define the pullback of $\phi$ by $f$, denoted $f^*(\phi)$, to be the composition
    \[
        \begin{tikzcd}
            BX \arrow[r, "f \times \id_X", dashed] & AX \arrow[r, "\phi", dashed] & Y.
        \end{tikzcd}
    \]
    Equivalently, if $(b,a,x) = \id_{BX}$, we have
    \begin{align*}
        f^*(\phi) = \phi(f(b),x).
    \end{align*}
\end{dfn}
In light of the above, we adopt the convention:
\begin{notation}
    Let $\phi$, $f$ be as above. We say $f^*(\phi)$ \emph{is defined} if $f\times \id_X$ is diffuse.
\end{notation}

\begin{prop}\label{pullbacks and products of faithfully flat morphisms are defined}
    Let $\phi\colon AX \dashrightarrow Y$, $\psi\colon BY \dashrightarrow Z$, and $f\colon D \dashrightarrow B$ be $S$-rational morphisms of locally Noetherian schemes.

    If $\phi$ is flat, then $\psi*\phi$ is defined.

    If $f$ is flat, then $f^*(\psi)$ is defined.
\end{prop}
\begin{proof}
    Immediate, since $\id_Y$ and $\id_B$ correspond to flat $S$-rational morphisms, and products of flat $S$-rational morphisms are flat (see \cref{properties preserved by products of $S$-rational morphisms}) and in particular $S$-diffuse (\cref{A flat S-rational morphism of locally Noetherian schemes is S-diffuse.}).
\end{proof}

Pullback commutes with composition of $S$-rational families:
\begin{prop}\label{Pullback commutes with composition of S-rational families:}
    Let $\phi\colon AX \dashrightarrow Y$, $\psi\colon BY \dashrightarrow Z$, $f\colon C \dashrightarrow A$ and $g\colon D \dashrightarrow B$ be $S$-rational families over $A, B$ respectively. Suppose
    \begin{align*}
        f^*(\phi), &  & g^*(\psi), &  & \psi*\phi, &  & (g\times_Sf)^*(\psi*\phi), &  & g^*(\psi)* f^*(\phi)
    \end{align*}
    are all defined. Then $(g\times_S f)^*(\psi*\phi) = g^*(\psi)* f^*(\phi)$.
\end{prop}
\begin{proof}
    Indeed, if $(d,c,x) = \id_{DCX}$, both expressions are equal to
    \begin{align*}
        \psi(g(d),\phi(f(c),x)).
    \end{align*}
\end{proof}

In good situations, pullback by flat $S$-rational morphisms maps respects graphs of rational families, in the sense of the next result. This will be very important in \cref{section: canonical S-rational families}.
\begin{prop}\label{technical graph pullback result.}
    Let $S$ be a Noetherian scheme, and $A, B, X, Y$ be schemes of finite type over $S$. Let $\phi\colon AX \dashrightarrow Y$ be an $S$-rational family over $A$, and let $f\colon B \dashrightarrow A$  be a flat $S$-rational morphism (so $f^*(\phi)$ is defined by \cref{pullbacks and products of faithfully flat morphisms are defined}).

    Let $U_1 \subset B$, $U_2 \subset A$ be $S$-dense opens such that there exists a representative $f_{U_1}$ of $f$ where $f_{U_1}(U_1) \subset U_2$ and $f_{U_1}$ is flat. Then we have
    \begin{align}
        (f_{U_1}|_{[U_1 \to U_2]})^*(\Lambda_\phi \cap U_2XY) = f_{U_1}^*(\Lambda_\phi \cap U_2XY) = \Lambda_{f^*(\phi)} \cap U_1XY.
    \end{align}
\end{prop}
\begin{proof}
    Note that the first equality follows immediately from \cref{pullback is the same as pullback by restriction to subobject of the target}.

    Fix a representative $\phi_{W_2}$ of $\phi$ with $W_2 \subset U_2X$. Let $W_1 = (f_{U_1} \times \id_X)\inv(W_2)$. Since $f_{U_1}\times \id_X$ is flat and in particular $S$-diffuse (\cref{flat morphisms of locally Noetherian schemes are S-diffuse}) we have that $W_1 \subset U_1X$ is an $S$-dense, and we have a representative $(f\times \id_X)_{W_1} := (f_U\times \id_X)|_{W_1}$.

    Since the closed graph is the schematic image of the graph of a representative (\cref{closed graph is schematic image of restriction graph.}), we have $\Lambda_\phi = \Schim_{AXY}{\Gamma_{\phi_{W_2}}}$. Since everything is Noetherian, all the maps are quasicompact, so taking schematic image commutes with restriction to opens in the target (\cref{Schematic image of well-behaved morphisms}), and therefore
    \begin{align*}
        \Lambda_\phi \cap U_2XY = \Schim_{U_2XY}(\Gamma_{\phi_{W_2}}).
    \end{align*}
    Since $f_{U_1}$ is flat, and the schematic image of a quasicompact morphism commutes with flat base change (\cref{schematic image of a quasicompact morphism commutes with flat base change on base}), we then have
    \begin{align*}
        f_{U_1}^*(\Lambda_\phi \cap U_2XY ) = \Schim_{U_1XY}(f_{U_1}^*(\Gamma_{\phi_{W_2}})).
    \end{align*}
    By the Cartesian rectangle lemma (\cref{Cartesian rectangles}), since $W_1 = (f_{U_1} \times \id_X)\inv(W_2)$, we have
    \begin{align*}
        f_{U_1}^*(\Gamma_{\phi_{W_2}}) = (f_{U_1}\times \id_X)^*(\Gamma_{\phi_{W_2}}) = (f\times \id_X)_{W_1}^*(\Gamma_{\phi_{W_2}}),
    \end{align*}
    and therefore


    \begin{align*}
        f_{U_1}^*(\Lambda_\phi \cap U_2XY ) = \Schim_{U_1XY}((f\times \id_X)_{W_1}^*(\Gamma_{\phi_{W_2}})).
    \end{align*}
    Since taking the graph of a morphism commutes with pullback (\cref{Taking the graph of a morphism commutes with pullback:}), we have
    \begin{align*}
        (f\times \id_X)_{W_1}^*(\Gamma_{\phi_{W_2}})= \Gamma_{\phi_{W_2}\circ (f\times\id_X)_{W_1}}.
    \end{align*}
    Note that $\phi_{W_2}\circ (f\times\id_X)_{W_1}$ is a representative of $f^*(\phi)$; denote it by $f^*(\phi)_{W_1}$. Thus we have
    \begin{align*}
        f_{U_1}^*(\Lambda_\phi \cap U_2XY ) = \Schim_{U_1XY}(\Gamma_{f^*(\phi)_{W_1}}).
    \end{align*}
    But by \cref{closed graph is schematic image of restriction graph.} we also have
    \begin{align*}
        \Lambda_{f^*(\phi)} \cap U_1XY = \Schim_{U_1XY}(\Gamma_{f^*(\phi)_{W_1}}),
    \end{align*}
    so we conclude.
\end{proof}

\begin{dfn}[$S$-Birational Equivalence for $S$-Rational Families]\label{S-birational equivalence for rational families definition}
    We say that two $S$-rational families $\phi\colon AX \dashrightarrow Y$, $\phi'\colon A'X \dashrightarrow Y$ are $S$-birationally equivalent if there is an $S$-birational map $f\colon A\dashrightarrow A'$ such that $\phi = f^*(\phi')$.
\end{dfn}

%% file: Chapters/Rational_Families/Invertible_Families.tex
\section{Invertible Families}\label{section: invertible families}
\begin{dfn}\label{associated $S$-rational morphism}
    Let $\phi\colon AX \dashrightarrow Y$ be an $S$-rational family over $A$. We define the \emph{associated $S$-rational morphism} $\overline{\phi}$ to be the induced $S$-rational morphism $(\id_A,\phi)\colon AX \dashrightarrow AY$.
\end{dfn}

\begin{lem}\label{S-rational families are determined by their associated S-rational morphisms}
    Let $\phi$, $\overline{\phi}$ be as above. Then for any representative $\overline{\phi}_U$ of $\overline{\phi}$, we have that $\rho_Y \circ \overline{\phi}_U$ (see \cref{postcomposition by morphisms}) is a representative of $\phi$.

    In particular, $\phi$ is determined by $\overline{\phi}$.
\end{lem}
\begin{proof}
    Let $\phi_{U'}$ be a representative of $\phi$, and denote by $j_{U'}\colon U' \hookrightarrow X$ the corresponding open immersion. Then (see \cref{pairs definition}) $\overline{\phi}$ is represented by $(j_{U'}, \phi_{U'})$. Therefore there exists an $S$-dense open $U'' \subset U \cap U'$ such that $\overline{\phi}_U|_{U''} = (j_{U'}, )\phi_{U'}|_{U''}$. Postcomposing by $\rho_Y$, we see $\rho_Y \circ \overline{\phi}_U|_{U''} = \phi_{U'}|_{U''}$ and we conclude.
\end{proof}

\begin{dfn}[Invertible $S$-rational families]\label{invertible S-rational families}
    We say an $S$-rational family $\phi\colon AX \dashrightarrow Y$ over $A$ is invertible if the associated $S$-rational morphism $\overline{\phi}$ is $S$-birational.

    If $\overline{\phi}\inv$ is the inverse of $\overline{\phi}$, we define the inverse of $\phi$, denoted $\phi^\dagger$, to be the $S$-rational family over $A$ defined by
    \begin{align*}
        \rho_X \circ \overline{\phi}\inv \colon AY \dashrightarrow X.
    \end{align*}
\end{dfn}
\begin{rem}
    Observe that in the above situation we have $(\overline{\phi})\inv = \overline{(\phi^\dagger)}$.
\end{rem}

\begin{lem}
    Let $\phi\colon AX \dashrightarrow Y$ be an invertible $S$-rational family over $A$. Then $\phi$ is $S$-diffuse.
\end{lem}
\begin{proof}
    Indeed, we have that $\overline{\phi}$ is $S$-birational and therefore $S$-diffuse by \cref{S-birational S-rational morphisms are faithfully flat}, and for any representative $\phi_U$ and any $S$-dense $V \subset Y$ we have $\phi_U\inv(V) = \overline{\phi}_U\inv(AV)$.
\end{proof}

\begin{prop}\label{invertible S-rational families are faithfully flat in good situations}
    Let $A,X,Y$ be $S$-schemes, and $\phi\colon AX \dashrightarrow Y$ an invertible $S$-rational family over $A$. Suppose $A$ is  flat/faithfully flat over $S$. Then $\phi$ is flat/faithfully flat, respectively.
\end{prop}
\begin{proof}
    We have that $\overline{\phi}$ is faithfully flat and diffuse by \cref{S-birational S-rational morphisms are faithfully flat}, and $\rho_Y\colon AY \to Y$ is diffuse (\cref{Projections are diffuse}) and flat/faithfully flat since $A$ is flat/faithfully flat over $Y$. Since $\phi = \rho_Y \circ \overline{\phi}$ (\cref{S-rational families are determined by their associated S-rational morphisms}), we conclude by \cref{composition of S-rational morphisms preserves many properties}.
\end{proof}

\begin{cor}
    Let $A,B,X,Y,Z$ be locally Noetherian $S$-schemes.  Let $\phi\colon AX \dashrightarrow Y$, $\psi\colon BY\dashrightarrow Z$ be $S$-rational families over $A$ and $B$. Suppose $\phi$ is invertible. Then $\psi * \phi$ is defined.
\end{cor}
\begin{proof}
    We have that $\phi$ is faithfully flat by above; we conclude by \cref{pullbacks and products of faithfully flat morphisms are defined}.
\end{proof}

A composition of invertible families is invertible:
\begin{prop}\label{A composition of invertible families is invertible}
    Let $\phi\colon AX \dashrightarrow Y$, $\psi\colon BY\dashrightarrow Z$ be invertible $S$-rational families over $A$ and $B$. Suppose that $A$ and $B$ are faithfully flat and locally of finite presentation over $S$, and $Y$ is locally of finite presentation over $S$.

    Then $\psi * \phi$ (which is defined by the previous proposition) is invertible, with inverse $(\phi^\dagger *\psi^\dagger) \circ \rho_{ABZ}\colon BAZ \dashrightarrow X$, where $\rho_{ABZ}\colon BAZ \to ABZ$ is the coordinate permutation (as in \cref{products and Yoneda combine}).
\end{prop}
\begin{proof}
    Indeed, we can write $\overline{\psi*\phi}$ as the composition
    \[
        \begin{tikzcd}[column sep = huge]
            BAX \arrow[r, "\id_B \times \overline{\phi}",dashed] & BAY \arrow[rr, "{\rho_{BAZ} \circ (\id_A \times \overline{\psi}) \circ \rho_{ABY}}",dashed] && BAZ.
        \end{tikzcd}
    \]
    Since the inverse of the composition is the composition of the inverses (\cref{composition of S-birational morphisms}),  this is invertible with inverse
    \[
        \begin{tikzcd}[column sep = huge]
            BAX & \arrow[l, "\id_B \times \overline{\phi} \inv"',dashed]  BAY  && \arrow[ll, "{\rho_{BAY} \circ (\id_A \times \overline{\psi}\inv) \circ \rho_{ABZ}}"',dashed] BAZ.
        \end{tikzcd}
    \]
    Recall that $(\psi*\phi)^\dagger = \rho_X \circ (\overline{\psi*\phi})\inv$ by definition. Comparing with the above we see this is equal to the composition
    \[
        \begin{tikzcd}[column sep = large]
            X & \arrow[l, "\phi^\dagger"', dashed]  AY  && \arrow[ll, "{(\id_A \times \psi^\dagger) \circ \rho_{ABZ}}"', dashed] BAZ
        \end{tikzcd}
    \]
    which is precisely $(\phi^\dagger*\psi^\dagger) \circ \rho_{ABZ}$.
\end{proof}

\begin{prop}[Invertible families and closed graphs]\label{invertible families and closed graphs}
    Let $\phi\colon AX\dashrightarrow Y$ be an invertible $S$-rational family over $A$, and let $\theta = \phi^\dagger$.
    Then we have a Cartesian square
    \[
        \begin{tikzcd}
            \Lambda_\phi \arrow[r] \arrow[d, hook] & \Lambda_\theta \arrow[d, hook] \\
            AXY \arrow[r, "\rho_{AYX}"]            & AYX
        \end{tikzcd}
    \]
\end{prop}

\begin{proof}
    Suppose that we have representatives $\overline{\phi}_U$ and $\overline{\theta}_V$ inducing mutually inverse isomorphisms $\phi' = \overline{\phi}|_{[U \to V]}$ and $\theta' = \overline{\theta}|_{[V \to U]}$ on the $S$-dense opens $U \subset X$, $V \subset Y$. By \cref{graphs of representatives of inverse S-rational morphisms}, we have a Cartesian diagram
    \[
        \begin{tikzcd}
            \Gamma_{\phi'} \arrow[r] \arrow[d, hook] & \Gamma_{\theta'} \arrow[d, hook] \\
            AXAY \arrow[r, "\rho_{3412}"]    & AYAX.
        \end{tikzcd}
    \]
    Set $\phi_U = \rho_Y\circ\overline{\phi}_U$, $\theta_V = \rho_X \circ \overline{\theta}_V$. By \cref{S-rational families are determined by their associated S-rational morphisms} these give representatives of $\phi$ and $\theta$. By \cref{pullbakcs of monomorphisms lemma}, we see that we also have Cartesian diagrams
    \[
        \begin{tikzcd}[column sep = large]
            \Gamma_{\phi_U} \arrow[r] \arrow[d, hook]     & \Gamma_{\phi'} \arrow[d, hook]   & \Gamma_{\theta_{V}} \arrow[r] \arrow[d, hook] & \Gamma_{\theta'} \arrow[d, hook] \\
            AXY \arrow[r, "{(a,x,y) \mapsto (a,x,a,y)}"'{yshift = -4pt}] & AXAY                             & AYX \arrow[r, "{(a,x,y) \mapsto (a,x,a,y)}"'{yshift = -4pt}]                                 & AYAX
        \end{tikzcd}
    \]
    where, for any $T \to S$, the lower horizontal arrows act on points $(a,x,y) \in AXY(T)$, $(a,y, x) \in AYX(T)$ as indicated.

    By the Cartesian cube lemma \cref{cartesian cubes}, we conclude we have a Cartesian diagram
    \[
        \begin{tikzcd}
            \Gamma_{\phi_U} \arrow[r] \arrow[d, hook] & \Gamma_{\theta_V} \arrow[d, hook] \\
            AXY \arrow[r, "\rho_{AYX}"]               & AYX.
        \end{tikzcd}
    \]

    Since $\rho_{AYX}$ is an isomorphism, and pullback by isomorphisms preserves schematic image, we conclude.
\end{proof}

\begin{prop}\label{detecting inverse $S$-rational families by looking at graphs}
    Let $A$, $X$, $Y$ be schemes over $S$ with $A, Y$ Noetherian and $X$ separated. Let $\phi\colon AX \dashrightarrow Y$, $\theta\colon AY \dashrightarrow X$ be $S$-rational families over $A$. Suppose $\phi$ is invertible, and we have a Cartesian diagram
    \begin{equation}
        \begin{tikzcd}
            \Lambda_\phi \arrow[r] \arrow[d, hook] & \Lambda_{\theta} \arrow[d, hook] \\
            AXY \arrow[r, "\rho_{AYX}"]   & AYX .
        \end{tikzcd}
    \end{equation}
    Then $\theta = \phi^\dagger$.
\end{prop}
\begin{proof}
    Completely analogous to the proof of \cref{detecting inverse $S$-rational morphisms by looking at graphs}
\end{proof}

%% file: Chapters/Rational_Families/Canonical_S-rational_families.tex
\section{Canonical \texorpdfstring{$S$}{S}-Rational Families}\label{section: canonical S-rational families}

\begin{lem}\label{xi existence result}
    Let $S$ be a Noetherian scheme, $A$ a locally Noetherian $S$-scheme, $B$ a projective $S$-scheme, and $Z \hookrightarrow AB$ a closed scheme. Suppose $Z$ is $S$-generically flat over $A$ (i.e., suppose there exists an $S$-dense open $U \subset A$ over which $Z$ is flat.) Then there is a unique $S$-rational morphism $\xi\colon A \dashrightarrow \Hilb_B$ such that, for any $S$-dense $U \subset A$ over which $Z$ is flat, we have $U \subset \dom(\xi)$ and $Z\cap UB = \xi_U^*(\Omega_{B})$, where $\Omega_B\subset\Hilb_B \cdot B$ (recall \cref{products and Yoneda combine}) is the universal closed subscheme.
\end{lem}
\begin{proof}
    Given such a $U$, the existence of a unique morphism $\xi_U$ with the required properties follows immediately from the definition of the Hilbert Scheme. Given any other $U'$, let $W = U \cap U'$. Then
    \begin{align*}
        (W\hookrightarrow U)^*(Z \cap UB) = Z \cap WB = (W\hookrightarrow U')^*(Z \cap U'B)
    \end{align*}
    so again by definition of the Hilbert scheme it follows that $\xi_U|_W = \xi_{U'}|_W$. Therefore the $\xi_U$ cohere to induce a unique $S$-rational morphism from $A$ to $\Hilb_B$ with the required property.
\end{proof}
\begin{rem}
    Suppose $S = \Spec k$ for some field $k$, and $A$ is reduced and of finite type over $k$. Then, for any $S$-scheme $B$ and closed subscheme $Z\subset AB$, since an $S$-dense open of $A$ is the same as a dense open of $A$ (\cref{S-dominance over fields}), and by generic flatness (\cref{Generic Flatness theorem}) we have that $Z$ is $S$-generically flat over $A$.
\end{rem}

\begin{dfn}\label{Z(phi) definition}
    Let $S$ be a Noetherian scheme, $\phi\colon AX \dashrightarrow Y$ be an $S$-rational family, with $A$ locally Noetherian and $X,Y$ projective over $S$. Suppose the closed graph $\Lambda_\phi \subset AXY$ is $S$-generically flat over $A$.
    \begin{enumerate}
        \item We denote by $\xi_\phi$ the $S$-rational morphism from $A$ to $\Hilb_{XY}$ induced by $\Lambda_\phi$ and the previous proposition.
        \item We define the closed subscheme $Z(\phi) \hookrightarrow \Hilb_{XY}$ to be the schematic image of $\xi_\phi$.
        \item We denote by $\zeta_\phi$ the $S$-rational morphism from $A$ to $Z(\phi)$ induced by $\xi_\phi$.
    \end{enumerate}
\end{dfn}

\begin{lem}
    Suppose in the above definition that $A$ is reduced. Then $Z(\phi)$ is reduced.
\end{lem}
\begin{proof}
    Immediate from \cref{schematic image of morphism with reduced source}.
\end{proof}

\begin{prop}\label{Z(phi) is projective in good circumstances}
    Suppose in the above definition that $A$ is of finite type over $S$. Then $Z(\phi)$ is projective over $S$.
\end{prop}
\begin{proof}
    Let $\xi_{\phi,U}$ be any representative of $\xi_\phi$. Observe that $U$ is quasicompact (since $A$, being of finite type over $S$, is Noetherian). By \cref{Hilbert schemes are disjoint unions of projective schemes}, we have an open cover of $U$ by $\cup_{p \in \Q[x]}\xi_U\inv(\Hilb_{XY}^p)$. Since $U$ is quasicompact, it admits a finite subcover, and therefore there exist polynomials $p_1,\dots,p_n$ such that $\xi_U(U) \subset \cup_{i = 1}^n\Hilb_{XY}^{p_i}$ and then $Z(\phi) \subset \cup_{i = 1}^n\Hilb_{XY}^{p_i}$. Since $Z(\phi)$ is closed and each $\Hilb_{XY}^{p_i}$ is projective, we conclude.
\end{proof}

\begin{lem}\label{Graphs being S-generically flat is preserved under ff pullback}
    Let $S$ be a Noetherian scheme, $A, B$ be schemes of finite type over $S$ and $X,Y$ be projective over $S$. Let $\phi\colon AX \dashrightarrow Y$ be an $S$-rational family over $A$ and $f\colon B \dashrightarrow A$  be a faithfully flat $S$-rational morphism. Then $\Lambda_\phi$ is $S$-generically flat over $A$ if and only if $\Lambda_{f^*(\phi)}$ is $S$-generically flat over $B$.
\end{lem}
\begin{proof}
    Similar to the proof of the previous proposition, we find $S$-dense opens $U_1 \subset B$, $U_2 \subset A$ and a representative $f_{U_1}$ such that $f_{U_1}|_{[U_1 \to U_2]}$ is faithfully flat. By \cref{technical graph pullback result.} we have
    \begin{align*}
        f_{U_1}|_{[U_1 \to U_2]}^*(\Lambda_\phi \cap U_2XY) = \Lambda_{f^*(\phi)} \cap U_1XY.
    \end{align*}
    Then since $f_{U_1}|_{[U_1 \to U_2]}$ is faithfully flat, we have that $f_{U_1}|_{[U_1 \to U_2]}^*(\Lambda_\phi \cap U_2XY)$ is flat over $U_1$ if and only if $\Lambda_\phi \cap U_2XY$ is flat over $U_2$.
\end{proof}


\begin{prop}[Existence of canonical families]\label{Existence of canonical families}
    Let $S$ be a Noetherian scheme. Let $A$ be a scheme of finite type over $S$ and $X,Y$ be schemes projective over $S$. Let $\phi\colon AX \dashrightarrow Y$ be an $S$-rational family over $A$ with $\Lambda_\phi$ $S$-generically flat over $A$, and suppose $\zeta_\phi$ is faithfully flat.

    Let $U \subset A$ be the maximal $S$-dense open such that $\Lambda_\phi$ is flat over $U$. Note $U \subset \dom(\zeta_\phi)$ by \cref{Z(phi) definition} and \cref{xi existence result}. Then

    \begin{enumerate}
        \item The $S$-rational morphism $\zeta_\phi \times \id_X$ is faithfully flat and $S$-diffuse, and there exist $S$-dense opens $W \subset UX \cap \dom(\phi)$ and $\widetilde{W}\subset Z(\phi)X$ such that $(\zeta_\phi \times \id_X)_W|_{[W \to \widetilde{W}]}$ is faithfully flat.
        \item \label{cartesian cube part} For any such $W, \tW$, there exists a unique morphism $\tphi_\tW\colon \tW \to Y$, independent of the choice of $W$, such that $\phi_W = \tphi_\tW \circ (\zeta_\phi \times \id_X)_W$.
        \item The rational morphism $\tphi\colon Z(\phi)X \dashrightarrow Y$ induced by $\tphi_\tW$ is independent of the choice of $\tW$.
        \item \label{The part about graphs} The closed graph $\Lambda_\phi$ is well-approximated by the universal closed subscheme $\Omega \subset \Hilb_{XY} \cdot XY$ in the following sense: for any $S$-dense opens $U' \subset \dom \zeta_{\phi}$, $Z' \subset Z(\phi)$ such that $(\zeta_{\phi})_{[U' \to Z']}$ is faithfully flat, and $\Lambda_\mu$ is flat over $U'$, we have $\Lambda_{\tphi} \cap Z'XY = \Omega \cap Z'XY$.
        \item We view $\tphi$ as an $S$-rational family over $Z(\phi)$, and call it the \emph{canonical family associated to $\phi$}. Then $\tphi$ satisfies $\zeta_\phi^*(\tphi) = \phi$, and moreover $\tphi$ is the unique $S$-rational family over $Z(\phi)$ with this property.
        \item Finally, $\zeta_\phi$ is the unique flat $S$-rational morphism from $A$ to $Z(\phi)$ with the property that $\zeta_\phi^*(\tphi) = \phi$.
    \end{enumerate}
\end{prop}

\begin{proof}
    \begin{enumerate}
        \item By \cref{properties preserved by products of $S$-rational morphisms}, we have that $\zeta_\phi \times \id_X$ is faithfully flat. Then since $A,X,Y$ are of finite type over $S$ and $S$ is Noetherian, the remainder of the statement follows from \cref{characterization of faithfully flat S-rational morphisms}.
        \item  By definition of $U$ and $\zeta_\phi$ we have a Cartesian square
              \[
                  \begin{tikzcd}
                      \Lambda_\phi \cap UXY \arrow[d, hook] \arrow[r] & \Omega \arrow[d, hook] \\
                      UXY \arrow[r]                                   & \Hilb_{XY}\cdot XY,
                  \end{tikzcd}
              \]
              where $\Omega$ denotes the universal closed subscheme. Applying the Cartesian rectangle \cref{Cartesian rectangles} we see that
              \[
                  \begin{tikzcd}
                      \Lambda_\phi \cap UXY \arrow[d, hook] \arrow[r] & \Omega \cap Z(\phi)XY \arrow[d, hook] \\
                      UXY \arrow[r]                                   & Z(\phi)XY,
                  \end{tikzcd}
              \]

              is Cartesian, and then that
              \[
                  \begin{tikzcd}
                      \Lambda_\phi \cap UXY \arrow[d, hook] \arrow[r] & \Omega \cap Z(\phi)XY \arrow[d, hook] \\
                      UX \arrow[r]                                   & Z(\phi) X,
                  \end{tikzcd}
              \]
              is Cartesian.

              Now take $W$, $\tW$ from the first part. We then have a commutative cube
              \[
                  \begin{tikzcd}
                      \Lambda_\phi \cap WY \arrow[rd] \arrow[rr] \arrow[dd] &                                             & \Omega \cap \tW Y \arrow[rd] \arrow[dd] &                                  \\
                      & \Lambda_\phi \cap UXY \arrow[dd] \arrow[rr] &                                         & \Omega \cap Z(\phi)XY \arrow[dd] \\
                      W \arrow[rd] \arrow[rr]                               &                                             & \tW \arrow[rd]                          &                                  \\
                      & UX \arrow[rr]                               &                                         & Z(\phi)X
                  \end{tikzcd}
              \]
              where the front, left, and right faces are Cartesian; it follows that the back face is Cartesian.

              Since everything is Noetherian and $Y$ is separated we have by \cref{Intersection of closed graph with product from domain in Noetherian to separated case} that $\Lambda_\phi \cap WY = \Gamma_{\phi_W}$, where the latter is the graph of $\phi_W$. Since $(\zeta_\phi \times \id_X)|_{[W \to \tW]}$ is faithfully flat and quasicompact (as everything is Noetherian) we have by \cref{construction of morphism from flat descent} that there exists a unique morphism $\tphi_\tW^W\colon \tW \to Y$ such that $\phi_W = \tphi_\tW^W \circ (\zeta_\phi \times \id_X)_W$, and $\Gamma_{\tphi_\tW^W} = \tW Y \cap \Omega$.

              To verify independence on $W$, fix any other $S$-dense $W' \subset UX \cap \dom(\phi)$ such that
              $(\zeta_\phi \times \id_X)_{W'}|_{[W' \to \widetilde{W}]}$ is faithfully flat. Let $W'' = W \cup W'$, and observe that $(\zeta_\phi \times \id_X)_{W''}|_{[W'' \to \widetilde{W}]}$ is faithfully flat. Then we have
              \begin{align*}
                  \phi_W = \phi_{W''}|_W = (\tphi_\tW^{W''} \circ (\zeta_\phi \times \id_X)_{W''})|_{W} = \tphi_\tW^{W''} \circ (\zeta_\phi \times \id_X)_{W}
              \end{align*}
              and therefore $\tphi_\tW^{W} = \tphi_\tW^{W''}$ by the uniqueness property of $\tphi_\tW^{W}$. Similarly $\tphi_\tW^{W'} = \tphi_\tW^{W''}$.  Thus we may set $\tphi_\tW = \tphi_\tW^W$.

        \item Fix any other pair of $S$-dense opens $V \subset UX \cap \dom(\phi)$ and $\widetilde{V}\subset Z(\phi)X$ such that $(\zeta_\phi \times \id_X)_{V}|_{[V \to \widetilde{V}]}$ is faithfully flat. Let $\widetilde{V'} = \widetilde{V} \cap \widetilde{W}$, and define
              \begin{align*}
                  V' = (\zeta_\phi \times \id_X)_V\inv(\widetilde{V'}) \cup (\zeta_\phi \times \id_X)_W\inv(\widetilde{V'}).
              \end{align*}
              Then observe that $(\zeta_\phi \times \id_X)_{V'}|_{[V' \to \widetilde{V'}]}$ is faithfully flat, and we have
              \begin{align*}
                  \phi_{V'} = \phi_V|_{V'} = (\tphi_{\widetilde{V}} \circ (\zeta_\phi \times \id_X)_{V})|_{V'} = \tphi_{\widetilde{V}} \circ (\zeta_\phi \times \id_X)_{V'}
              \end{align*}
              and thus $\tphi_{\widetilde{V}} = \tphi_{\widetilde{V'}}$. Similarly $\tphi_{\widetilde{W}} = \tphi_{\widetilde{V'}}$.

        \item Indeed, for any such $U'$, $Z'$, repeating the argument of part \ref{cartesian cube part} with $U, Z(\phi)$ replaced by $U', Z'$, we get a representative $\tphi_\tW$ together with a map $\Gamma_{\tphi_\tW} \to \Omega\cap Z'XY$. In particular, we get a closed immersion $\Schim_{\Omega \cap Z'XY}(\Gamma_{\tphi_\tW}) \hookrightarrow \Omega \cap Z'XY$. By \cref{closed graph is schematic image of restriction graph.} and \cref{Schematic image of well-behaved morphisms}, we have $\Schim_{\Omega \cap Z'XY}(\Gamma_{\tphi_\tW})  = \Lambda_{\tphi} \cap Z'XY$, so we get an embedding $\Lambda_{\tphi} \cap Z'XY \hookrightarrow \Omega \cap Z'XY$. But by \cref{technical graph pullback result.}, it follows that we have
              \begin{align*}
                  (\zeta_{\phi})_{[U' \to Z']}^*(\Lambda_{\tphi} \cap Z'XY) = \Lambda_{\phi} \cap U'XY = (\zeta_{\phi})_{[U' \to Z']}^*(\Omega \cap Z'XY),
              \end{align*}
              where the second equality follows by definition of $\zeta_\phi$. In particular the embedding $\Lambda_{\tphi} \cap Z'XY \hookrightarrow \Omega \cap Z'XY$ becomes an isomorphism after pullback by $(\zeta_{\phi})_{[U' \to Z']}$. Since $(\zeta_{\phi})_{[U' \to Z']}$ is faithfully flat and quasicompact (and therefore an fpqc cover), we conclude that $\Lambda_{\tphi} \cap Z'XY = \Omega \cap Z'XY$ as desired.

        \item By definition, $\zeta_\phi^*(\tphi) = \tphi \circ (\zeta_\phi \times \id_X)$, and $\zeta_\phi \times \id_X$ is faithfully flat, so the uniqueness of $\tphi$ follows immediately from \cref{faithfully flat morphisms are S-rational epimorphisms}.

        \item Given any other flat $S$-rational morphism $f\colon A\dashrightarrow Z(\phi)$ such that $f^*(\tphi) = \phi$, fix $S$-dense opens $U' \subset \dom \zeta_\phi \cap \dom f$ and $Z' \subset Z(\phi)$ such that
              \begin{enumerate}
                  \item $f_{U'}$, and $(\zeta_\phi)_{U'}$ are both flat
                  \item $f_{U'}(U') \cup (\zeta_\phi)_{U_1'}(U') \subset Z'$
                  \item We have $\Lambda_{\tphi} \cap Z'XY = \Omega \cap Z'XY$
              \end{enumerate}
              (Note that such $U'$, $Z'$ can be obtained by suitably shrinking one of the $U'$ from part \ref{The part about graphs}). By the universal property of the Hilbert scheme, it is enough to show that $f_{U'}^*(\Omega \cap Z'XY) = (\zeta_{\phi})_{U'}^*(\Omega \cap Z'XY)$. But by \cref{technical graph pullback result.} we have
              \begin{align*}
                  f^*_{U'}(\Lambda_{\tphi} \cap Z'XY ) = \Lambda_{\phi} \cap U'XY=  (\zeta_{\phi})_{U'}^*(\Lambda_{\tphi} \cap Z'XY)
              \end{align*}
              so we conclude.
    \end{enumerate}
\end{proof}

\begin{prop}[Canonical $S$-Rational Families and Pullback]\label{Canonical S-rational families and pullback.}
    Let $S$ be a Noetherian scheme, $A, B$ be schemes of finite type over $S$, and let $X,Y$ be projective over $S$. Let $\phi\colon AX \dashrightarrow Y$ be an $S$-rational family over $A$ with $\Lambda_\phi$ $S$-generically flat over $A$. Let $f\colon B \dashrightarrow A$  be a flat $S$-rational morphism, (so $f^*(\phi)$ is defined by \cref{pullbacks and products of faithfully flat morphisms are defined}) and suppose $\Lambda_{f^*(\phi)}$ is $S$-generically flat over $B$.

    Then we have a commutative diagram of $S$-rational morphisms:
    \[
        \begin{tikzcd}
            B \arrow[rd, "\xi_{f^*(\phi)}", dashed] \arrow[d, "f"', dashed] &            \\
            A \arrow[r, "\xi_\phi", dashed]                                  & \Hilb_{XY}
        \end{tikzcd}
    \]
    If $f$ is faithfully flat, we then have $Z(f^*(\phi)) = Z(\phi)$ and $\zeta_{f^*(\phi)} = \zeta_\phi \circ f$.

    If further one of $\zeta_\phi$, $\zeta_{f^*(\phi)}$ is faithfully flat, then so is the other, and $\widetilde{f^*(\phi)} = \widetilde{\phi}$. In particular, $\zeta_\phi$ is the unique faithfully flat $S$-rational morphism from $A$ to $Z(f^*(\phi))$ such that $\zeta_\phi^*(\widetilde{f^*(\phi)}) = \phi$.

\end{prop}

\begin{proof}
    Fix an $S$-dense open $U_2 \subset A$ such that $\Lambda_\phi$ is flat over $U_2$. Then we have $U_2 \subset \dom \xi_\phi$ and $\xi_{U_2}^*(\Omega_{XY}) = \Lambda_\phi \cap UXY$, where $\Omega_{XY}$ denotes the universal closed subscheme (see \cref{Z(phi) definition}).

    Fix any representative $f_{U_0}$ of $f$. Since $f$ is flat, we may find an $S$-dense open subscheme $U_1 \subset U_0$ and a representative $f_{U_1} = f_{U_0}|_{U_1}$ such that:
    \begin{enumerate}
        \item $f_{U_1}(U_1) \subset U_2$
        \item $f_{U_1}$ is flat (by \cref{characterization of flat S-rational morphisms})
        \item $\Lambda_{f^*(\phi)}$ is flat over $U_1$ (so $U_1 \subset \dom(\xi)$, and $\xi^*_{f^*(\phi),U_1}(\Omega_{XY}) = \Lambda_\phi \cap U_1XY$).
    \end{enumerate}
    We first show that $\xi_\phi\circ f = \xi_{f^*{\phi}}$. For this it is enough to show
    \begin{align}
        \xi_{\phi,U_2} \circ f_{U_1} = \xi_{f^*(\phi), U_1}.
    \end{align}
    Note that, by construction of the Hilbert scheme and definition of $U_1, U_2$, the above holds if and only if
    \begin{align}
        f_{U_1}|_{[U_1 \to U_2]}^*(\Lambda_\phi \cap U_2XY) = \Lambda_{f^*(\phi)} \cap U_1XY,
    \end{align}
    so we can conclude by \cref{technical graph pullback result.}.

    If $f$ is faithfully flat, then $Z(f^*(\phi)) = Z(\phi)$ follows from \cref{Precomposition by faithfully flat S-rational morphisms preserves schematic image}, and then $\zeta_{f^*(\phi)} = \zeta_\phi \circ f$ by above.

    If further $\zeta_\phi$ is faithfully flat then, since a composition of faithfully flat morphisms is faithfully flat, so is $\zeta_{f^*\phi}$. If $\zeta_{f^*\phi}$ is faithfully flat, we get that $\zeta_{\phi}$ is faithfully flat by \cref{composition faithfully flat criterion for S-rational morphisms}, and we see
    \begin{align*}
        \zeta_{f^*\phi}(\widetilde{\phi}) = f^*(\zeta_\phi^*(\widetilde{\phi})) = f^*(\phi),
    \end{align*}
    so by the uniqueness property of $\widetilde{\phi}$ from \cref{Existence of canonical families} it follows that $\widetilde{\phi} = \widetilde{f^*(\phi)}$. The uniqueness of $\zeta_\phi$ then follows from \cref{Existence of canonical families}.

\end{proof}

The above results allow us to construct canonical families for a broader class of $S$-rational families:
\begin{dfn}\label{tame definition}
    Let $\phi\colon AX \dashrightarrow Y$ be an $S$-rational family. We say that $\phi$ is \emph{tame} if the following hold: $S$ is Noetherian, $A$ is of finite type over $S$, $X$ and $Y$ are $S$-birationally projective (\cref{$S$-birationally projective definition}), and there exists $S$-birational models $X'$, $Y'$ of $X$ and $Y$ respectively which are projective over $S$ such that if $\phi'\colon AX' \dashrightarrow Y'$ denotes the induced rational family from $X'$ to $Y'$ we have
    \begin{enumerate}
        \item The graph $\Gamma_{\phi'}$ is $S$-generically flat over $A$,
        \item The induced map $\zeta_{\phi'}$ is faithfully flat.
    \end{enumerate}
\end{dfn}

\begin{lem}\label{tameness criterion over fields}
    Let $A$ be a geometrically integral variety over a field $k$, and $X,Y$ be varieties over $k$. Let $\phi\colon AX \dashrightarrow Y$ be a $k$-rational family. Then $\phi$ is tame, and $Z(\phi)$ is geometrically integral.
\end{lem}
\begin{proof}
    First, recall that a morphism of varieties is $k$-dominant if and only if it is dominant (\cref{S-dominance over fields}). In particular density agrees with $k$-density. Also recall that every variety is birational to a projective variety.

    Since $A$ is reduced, we have by generic flatness (\cref{Generic Flatness theorem}) that $\Lambda_\phi$ is ($\Spec(k)$)-generically flat over $A^2$.

    We have that $A^2$ is geometrically integral, and the map $\zeta_\phi$ is dominant by \cref{Schematically dominant maps to schematic images}; therefore also $Z(\phi)$ is geometrically integral by \cref{schematic image of a geometrically integral variety}, and $\zeta_\phi$ is faithfully flat by \cref{a dominant moprhism of integral schemes over a field is faithfully flat}.
\end{proof}

\begin{thm}\label{canonical families for families on S-birationally projective schemes}
    Let $\phi\colon AX \dashrightarrow Y$ be a tame $S$-rational family. Then
    \begin{itemize}
        \item  There exists a ``canonical'' rational family $\tphi\colon Z(\phi)X \dashrightarrow Y$, together with a faithfully flat $S$-rational morphism $\zeta_\phi\colon A \to Z(\phi)$, which satisfies the following universal property: for any $S$-rational family $\psi\colon BX \dashrightarrow Y$ with $B$ Noetherian, and faithfully flat $S$-rational morphism $f\colon A \to B$ such that $\phi = f^*(\psi)$, there exists a unique faithfully flat $S$-rational morphism $\eta\colon B \dashrightarrow Z(\phi)$ with $\psi = \eta^*(\tphi)$.
        \item      The family $\tphi$ is constructed by taking a choice of projective models of $X$ and $Y$, applying the construction of \cref{Existence of canonical families} on these models, and considering the induced family $\tphi\colon Z(\phi)X \dashrightarrow Y$.
        \item  The family $\tphi$ is unique up to $S$-birational equivalence (\cref{S-birational equivalence for rational families definition}).
    \end{itemize}
\end{thm}
\begin{proof}
    The fact that the family constructed in the statement satisfies the universal property follows from \cref{Graphs being S-generically flat is preserved under ff pullback} and \cref{Canonical S-rational families and pullback.}. The fact that $\tphi$ is unique up to birational equivalence follows from the universal property; indeed given any other family $\theta\colon ZX \dashrightarrow Y$ with this property we would have faithfully flat $S$-rational morphisms $Z(\phi) \dashrightarrow Z$, $Z \dashrightarrow Z(\phi)$ which by the universal property must compose to give the identity; then it follows from \cref{criteria for S-birationality} these morphisms are $S$-birational.
\end{proof}

%% file: Chapters/Rational_Families/Transformations_of_Canonical_Families.tex
\section{Transformations of Canonical Families}\label{section: transformations of canonical families}

\begin{prop}\label{Canonical families and composition}
    Let $S$ be a Noetherian scheme. Let $A,B$ be schemes of finite type over $S$ and $X,Y,Z$ be schemes projective over $S$.

    Let $\phi\colon X \dashrightarrow Y$ and $\psi\colon BY \dashrightarrow Z$ be $S$-rational families over $A, B$ respectively, satisfying
    \begin{enumerate}
        \item The composition $\psi * \phi$ is defined
        \item $\Lambda_\phi$, $\Lambda_\psi$, and $\Lambda_{\psi * \phi}$ are $S$-generically flat over $A,B$, and $BA$ respectively (so $\zeta_\phi$, $\zeta_\psi$, and $\zeta_{\psi*\phi}$ are defined as in \cref{Z(phi) definition})
        \item $\zeta_\phi$, $\zeta_\psi$ are faithfully flat (so $\widetilde{\phi}$, $\widetilde{\psi}$ are defined by \cref{Existence of canonical families})
        \item The composition $\widetilde{\psi} * \widetilde{\phi}$ is defined
        \item $\Lambda_{\widetilde{\psi} * \widetilde{\phi}}$ is $S$-generically flat over $Z(\psi)Z(\phi)$.
    \end{enumerate}
    Then
    \begin{align*}
        \xi_{\widetilde{\psi} * \widetilde{\phi}} \circ (\zeta_\psi \times \zeta_\phi) = \xi_{\psi * \phi},
    \end{align*}
    where the $\xi$ are defined as in \cref{xi existence result}.
    In particular, we have $Z(\widetilde{\psi} * \widetilde{\phi}) = Z(\psi*\phi)$, and
    \begin{align*}
        {\zeta_{\widetilde{\psi} * \widetilde{\phi}}\ \circ (\zeta_\psi \times \zeta_\phi) = \zeta_{\psi * \phi}}.
    \end{align*}

    If further one of $\zeta_{\widetilde{\psi} * \widetilde{\phi}}$ and $\zeta_{\psi * \phi}$ is faithfully flat, then so is the other, and we have
    \begin{align*}
        \widetilde{ \widetilde{\psi} * \widetilde{\phi} } = \widetilde{ \psi * \phi }.
    \end{align*}

\end{prop}
\begin{proof}
    First note that, by \cref{properties preserved by products of $S$-rational morphisms} and \cref{A flat S-rational morphism of locally Noetherian schemes is S-diffuse.}, $\zeta_\psi \times \zeta_\phi$ is faithfully flat and $S$-diffuse.
    \begin{claim}
        We have a commutative diagram of $S$-rational morphisms:
    \end{claim}
    \[
        \begin{tikzcd}
            & BA \arrow[rd, "\xi_{\psi * \phi}", dashed] \arrow[ld, "\zeta_\psi \times \zeta_\phi"', dashed] &              \\
            Z(\psi)Z(\phi) \arrow[rr, "\xi_{\widetilde{\psi} * \widetilde{\phi}}"', dashed] & & \Hilb_{XZ}.
        \end{tikzcd}
    \]
    \begin{proof}
        Since $\zeta_\psi \times \zeta_\phi$ is faithfully flat, by \cref{Canonical S-rational families and pullback.}, it is enough to show that $(\zeta_\psi \times \zeta_\phi)^*(\widetilde{\psi} * \widetilde{\phi}) = \psi * \phi$. Indeed, we verify this using \cref{useful remark} and \cref{products and Yoneda combine}. If $(b,a,x) = \id_{BAX}$, we have
        \begin{align*}
            (\zeta_\psi \times \zeta_\phi)^*(\widetilde{\psi} * \widetilde{\phi})(b,a,x)
             & = (\widetilde{\psi} * \widetilde{\phi})(\zeta_\psi(b), \zeta_\phi(a), x) \\
             & = \widetilde{\psi}(\zeta_\psi(b), \widetilde{\phi}(\zeta_\phi(a),x))     \\
             & = \widetilde{\psi}(\zeta_\psi(b),\phi(a,x))                              \\
             & = \psi(b, \phi(a,x))                                                     \\
             & = (\psi*\phi)(b,a,x).
        \end{align*}
    \end{proof}

    It follows from the claim that the schematic image of $\xi_{\widetilde{\psi} * \widetilde{\phi}} \circ (\zeta_\psi \times \zeta_\phi)$ is equal to $Z(\psi * \phi)$. Since $\zeta_\psi \times \zeta_\phi$ is faithfully flat, and precomposition by faithfully flat $S$-rational morphisms preserves schematic image (\cref{Precomposition by faithfully flat morphisms preserves schematic image}), we conclude that $Z(\widetilde{\psi} * \widetilde{\phi})$ (which by definition is the schematic image of $\xi_{\widetilde{\psi} * \widetilde{\phi}}$) is equal to $Z(\psi * \phi)$, and the result follows.
\end{proof}

\begin{prop}\label{descent morphism for inverses}
    Let $S$ be a Noetherian scheme, $A$ a Noetherian scheme over $S$, and $X,Y$ be projective over $S$. Let $\phi\colon AX \dashrightarrow Y$ be an invertible $S$-rational family over $A$ with $\Lambda_\phi$ $S$-generically flat over $A$. Recall the canonical morphism $\tau_{XY}\colon \Hilb_{XY} \to \Hilb_{YX}$ from \cref{tau_YX definition}.

    Then $\Lambda_{\phi^\dagger}$ is $S$-generically flat over $A$. Moreover,
    we have $\xi_{\phi^\dagger} = \tau_{XY} \circ \xi_{\phi}$. In particular, the isomorphism $\tau_{XY}$ induces an isomorphism (which we will also denote by $\tau_{XY}$) from $Z(\phi)$ to $Z(\phi^\dagger)$.
\end{prop}

\begin{proof}
    First, by \cref{invertible families and closed graphs}, we have a Cartesian square:
    \begin{equation} \label{inverse cartesian diagram a}
        \begin{tikzcd}
            \Lambda_\phi \arrow[r,"\sim"] \arrow[d, hook] & \Lambda_{\phi\dagger} \arrow[d, hook] \\
            AXY \arrow[r, "\rho_{AYX}"]           & AYX.  \\
        \end{tikzcd}
    \end{equation}

    Projecting the bottom line down to $A$, we get a commutative diagram
    \[
        \begin{tikzcd}
            \Lambda_\phi \arrow[rr,"\sim"] \arrow[dr] && \Lambda_{\phi\dagger} \arrow[dl] \\
            & A &
        \end{tikzcd}
    \]
    In particular, if $\Lambda_\phi$ is flat over $U \subset A$, then so is $\Lambda_{\phi\dagger}$.

    Now fix any $S$-dense $U \subset A$ such that $\Lambda_\phi$ and $\Lambda_{\phi^\dagger}$ are flat over $U$. Then, by construction of $\xi_\phi$, $\xi_{\phi^\dagger}$ we have $U \subset \dom(\xi_\phi) \cap \dom(\xi_{\phi^\dagger})$ and moreover
    \begin{align*}
        \Lambda_\phi \cap UXY = \xi_{\phi,U}^*(\Omega_{XY}), &  & \Lambda_{\phi^\dagger} \cap UYX = \xi_{\phi^\dagger,U}^*(\Omega_{YX})
    \end{align*}
    where $\Omega_{XY}$ and $\Omega_{YX}$ denote the respective universal subschemes. We will show that $\xi_{\phi^\dagger,U} = \tau_{XY} \circ \xi_{\phi,U}$. By definition of $\tau_{XY}$ this is equivalent to showing
    \begin{align*}
        \Lambda_{\phi^\dagger} \cap UYX = \rho_{UXY}^*(\Lambda_\phi \cap UXY).
    \end{align*}
    But this follows on pulling back \cref{inverse cartesian diagram a} along ${U\hookrightarrow A}$, so we get the desired equality. The fact that $\tau_{XY}$ induces an isomorphism of the schematic images follows immediately.
\end{proof}

\begin{prop}\label{Canonical families of inverses}
    Let $S$ be a Noetherian scheme, $A$ a scheme of finite type over $S$, and $X,Y$ be $S$-birationally projective over $S$. Let $\phi\colon AX \dashrightarrow Y$ be an invertible $S$-rational family over $A$ with closed graph $\Lambda_\phi$ $S$-generically flat over $A$, and $\zeta_\phi$ faithfully flat.

    Then $\zeta_{\phi^\dagger}$ is faithfully flat, and $\widetilde{\phi}$ is invertible with inverse $(\widetilde{\phi})^\dagger = \tau_{YX}^*(\widetilde{\phi^\dagger})$.
\end{prop}

\begin{proof}
    Denote $\psi = \phi^\dagger$. Note that by \cref{descent morphism for inverses} we have $\zeta_\psi = \tau_{XY} \circ \zeta_{\phi}$; in particular, since $\tau_{XY}$ is an isomorphism, $\dom(\zeta_\psi) = \dom(\zeta_\phi)$ and $\zeta_\psi$ is faithfully flat. Therefore we have a canonical family $\widetilde{\psi}$.

    Let $\mathcal{A} \subset A$ denote the largest $S$-dense open such that $\Lambda_\phi$ (and hence $\Lambda_{\psi}$) is flat over $\mathcal{A}$; then $\mathcal{A} \subset \dom(\zeta_\phi)$ by \cref{Z(phi) definition}. Since products of faithfully flat $S$-rational morphisms are faithfully flat (\cref{properties preserved by products of $S$-rational morphisms}), we have that $\zeta_\phi\times \id_X$, $\zeta_\phi \times \id_Y$, $\zeta_\psi \times \id_X$ and $\zeta_\psi \times \id_Y$ are all faithfully flat. Since everything is of finite presentation, and $\overline{\phi}, \overline{\psi}$ are mutually inverse, by \cref{characterization of faithfully flat S-rational morphisms} we may find $S$-dense opens $U \subset \mathcal{A}X\cap \dom(\phi)$ and $V \subset \mathcal{A} Y \cap \dom(\psi)$ such that:

    \begin{enumerate}
        \item $\overline{\phi}$ and $\overline{\psi}$ induce mutually inverse isomorphisms $\overline{\phi}_U|_{[U \to V]}$, $\overline{\psi}_V|_{[V \to U]}$ between $U$ and $V$,
        \item The images
              \begin{align*}
                  U':= (\zeta_\phi \times \id_X)_U(U) \subset Z(\phi)X,  &  & U'' := (\zeta_\psi \times \id_X)_U(U)\subset Z(\psi)X  \\
                  V' := (\zeta_\phi \times \id_Y)_V(V) \subset Z(\phi)Y, &  & V'' := (\zeta_\psi \times \id_Y)_V(V) \subset Z(\psi)Y
              \end{align*}
              are $S$-dense opens, and
        \item The morphisms
              \begin{align*}
                  (\zeta_\phi \times \id_X)_U|_{[U \to U']}, &  & (\zeta_\psi \times \id_X)_U|_{[U \to U'']} \\
                  (\zeta_\phi \times \id_Y)_V|_{[V \to V']}, &  & (\zeta_\psi \times \id_Y)_V|_{[V \to V'']}
              \end{align*}
              are faithfully flat.
    \end{enumerate}

    Observe that, by \cref{Existence of canonical families}, we have $U' \subset \dom(\widetilde{\phi})$ and $V'' \subset \dom(\widetilde{\psi})$.

    \begin{claim}
        Consider the associated $S$-rational morphisms $\overline{\widetilde{\phi}}$, $\overline{\widetilde{\psi}}$ (\cref{associated $S$-rational morphism}). We have that $\overline{\widetilde{\phi}}$ factors through $V'$ and $\overline{\widetilde{\psi}}$ factors through $U''$, and we have commutative squares
        \[
            \begin{tikzcd}[column sep = huge]
                U \arrow[d, "(\zeta_\phi \times \id_X)_U"'] \arrow[r, "\overline{\phi}_U"] & V \arrow[d, "(\zeta_\phi \times \id_Y)_V"] &&
                V \arrow[d, "(\zeta_\psi \times \id_Y)_V"'] \arrow[r, "\overline{\psi}_V"] & U \arrow[d, "(\zeta_\psi \times \id_X)_U"]
                \\
                U' \arrow[r, "{\overline{\widetilde{\phi}}_{U'}|_{[U' \to V']}}"]      & V'
                &&
                V'' \arrow[r, "{\overline{\widetilde{\psi}}_{V''}|_{[V'' \to U'']}}"]         & U''
            \end{tikzcd}
        \]
    \end{claim}
    \begin{proof}
        We verify the result for $\overline{\widetilde{\phi}}$ and the square on the left; the proof for the square on the right is identical.

        We verify this pointwise. Indeed, for any locally Noetherian $T \to S$ and $(a,x) \in U(T) \subset AX(T)$, we have $(a,\phi(a,x)) \in V(T)$, and following the definitions we have:
        \begin{align*}
            \overline{\widetilde{\phi}}_{U'} \circ (\zeta_{\phi} \times \id_X)_U(a,x) & = \overline{\widetilde{\phi}}_{U'}(\zeta_\phi(a), x)                                \\
                                                                                      & =(\zeta_{\phi}(a), \widetilde{\phi}_{U'}(\zeta_\phi,(a), x) )                       \\
                                                                                      & = (\zeta_\phi(a), \phi(a, x)) \tag{\text{By definition of $\widetilde{\phi}_{U'}$}} \\
                                                                                      & = (\zeta_\phi \times \id_Y)_V(a,\phi(a,x))                                          \\
                                                                                      & = (\zeta_\phi \times \id_Y)_V \circ \overline{\phi}_U(a,x)
        \end{align*}
        (note that in the above, following \cref{notation for S-rational morphisms with separated target}, we write $\zeta_{\phi}(a)$ for $\zeta_{\phi,\mathcal{A}}(a)$).

        In particular, we have that $  \overline{\widetilde{\phi}}_{U'} \circ (\zeta_{\phi} \times \id_X)_U$ factors through $(\zeta_\phi \times \id_Y)_V(V) = V'$. Since $(\zeta_{\phi} \times \id_X)_U|_{[U \to U']}$ is surjective, it follows that $\overline{\widetilde{\phi}}_{U'}$ factors through $V'$, and we get the result.
    \end{proof}
    Again, by \cref{descent morphism for inverses} we have $\zeta_\phi = \tau_{YX} \circ \zeta_{\psi}$, and in particular we have commutative diagrams
    \[
        \begin{tikzcd}
            & U \arrow[ld, "(\zeta_\phi \times \id_X)_U"'] \arrow[rd, "(\zeta_\psi \times \id_X)_U"] &     & &   & V \arrow[ld, "(\zeta_\phi \times \id_Y)_V"'] \arrow[rd, "(\zeta_\psi \times \id_Y)_V"] &                \\
            U'  &    &\arrow[ll, "(\tau_{YX} \times \id_X)_{U''}"] U'' && V'\arrow[rr,"(\tau_{XY} \times \id_Y)_{V'}"'] &       & V''.
        \end{tikzcd}
    \]
    Putting everything together we get a commutative diagram
    \[
        \begin{tikzcd}[row sep = large]
            & U \arrow[ldd, "(\zeta_\phi \times \id_X)_U"'] \arrow[rd, "(\zeta_\psi \times \id_X)_U"] \arrow[rrrr, "\overline{\phi}_U", shift left]
            &   &   & &
            V \arrow[ldd, "(\zeta_\phi \times \id_Y)_V"' near start] \arrow[rd, "(\zeta_\psi \times \id_Y)_V"] \arrow[llll, "\overline{\psi}_V"] &
            \\
            &   &
            U'' \arrow[lld, "(\tau_{YX} \times \id_X)_{U''}"]
            &   &  & &
            V'' \arrow[llll,"\overline{\widetilde{\psi}}_V" near end]
            \\
            U' \arrow[rrrr,"\overline{\widetilde{\phi}}_U"']
            && & &
            V' \arrow[rru, "(\tau_{XY} \times \id_Y)_{V'}"']
            &  &
        \end{tikzcd}
    \]
    Now since a faithfully flat morphism is an epimorphism of schemes (\cref{A faithfully flat morphism is an epimorphism of schemes}), and $(\zeta_\phi \times \id_X)_U$ is faithfully flat, it follows that $\overline{\widetilde{\phi}}$ is $S$-birational, with inverse
    \begin{align*}
        (\tau_{YX} \times \id_X) \circ \overline{\widetilde{\psi}} \circ (\tau_{XY} \times \id_Y)
    \end{align*}
    and thus
    \begin{align*}
        (\widetilde{\phi})^\dagger = \rho_X \circ (\overline{\widetilde{\phi}})\inv = \widetilde{\psi} \circ (\tau_{XY} \times \id_Y) = \tau_{XY}^*(\widetilde{\psi})
    \end{align*}
    as required.
\end{proof}

%% file: Chapters/Group_Chunks_from_S-Rational_Families/Main_Group_Chunks_from_S-Rational_Families.tex
\chapter{Group Chunks from \texorpdfstring{$S$}{S}-Rational Families}
In this chapter we give our final main result. In the first section we state the result, make some preliminary remarks, and show how the result simplifies over fields. In the next section we show how to construct the partial magma; most of the technical work happens here. In the final section we show how to shrink the partial magma to produce an Artin group chunk.


\input{Chapters/Group_Chunks_from_S-Rational_Families/Statement_of_the_Result.tex}

\input{Chapters/Group_Chunks_from_S-Rational_Families/Constructing_the_Partial_Magma.tex}
\input{Chapters/Group_Chunks_from_S-Rational_Families/Final_group_chunk_construction.tex}

%% file: Chapters/Group_Chunks_from_S-Rational_Families/Statement_of_the_Result.tex
\section{Statement of the Main Result}
\begin{thm}[see \cref{main result: most general case}]\label{statement of the main result at start of section}
    Let $S$ be a Noetherian scheme, and let $A,X,Y$ be schemes over $S$, with $A$ fppf and with geometrically integral fibers over $S$ and such that $Y$ is locally Noetherian and $X$ is $S$-birationally projective.

    Let $\phi\colon AX \dashrightarrow Y$ be an invertible $S$-rational family over $A$, with inverse $\phi^\dagger$. Let $\psi = \phi^\dagger * \phi$. Suppose
    \begin{enumerate}[label = (\roman*).]
        \item The family $\psi$ is tame (\cref{tame definition})
        \item The closed graphs $\Lambda_{\widetilde{\psi} * \widetilde{\psi}}$, $\Lambda_{(\widetilde{\psi})^\dagger * \widetilde{\psi}}$ and $\Lambda_{\widetilde{\psi} * (\widetilde{\psi})^\dagger}$ are $S$-generically flat over $Z(\psi)^2$, for some choice of $S$-birational projective model of $X$ witnessing tameness of $\psi$.
        \item \label{main statement key condition} The $S$-rational morphisms
              \begin{align*}
                  q_1\colon A^2 \dashrightarrow AZ(\psi); &  & (a_1,a_2) \in A^2(T) \mapsto (a_1, \zeta_\psi(a_1,a_2)) \\
                  q_2\colon A^2 \dashrightarrow AZ(\psi); &  & (a_1,a_2) \in A^2(T) \mapsto (a_2, \zeta_\psi(a_1,a_2))
              \end{align*}
              are both faithfully flat.
    \end{enumerate}

    Then composition of $\psi$ with itself induces an $S$-rational morphism $m_{12}\colon Z(\psi)^2 \dashrightarrow Z(\psi)$, and there exist $S$-dense opens $U\subset \dom(m_{12})$, $D \subset Z(\psi)$ and $W\subset D^3 \cap \Gamma_{m_{12U}}$ (where $\Gamma_{m_{12U}}$ denotes the graph of the representative of $m_{12}$ with domain $U$) such that the pair $(D,W)$ is an Artin group chunk (\cref{Artin Group Chunk definition}).
\end{thm}

\begin{rem}
    Note that the fact that $\psi$ is tame implies $\zeta_\psi$ and $\zeta_{\psi^\dagger}$ exist (for appropriate choice of projective model of $X$) by \cref{xi existence result} and \cref{descent morphism for inverses}, and also that $\zeta_{\psi^\dagger}$ is faithfully flat by \cref{descent morphism for inverses}. Then $\widetilde{\psi}$, $\widetilde{\psi^\dagger}$ both exist by \cref{Existence of canonical families} and are invertible by \cref{Canonical families of inverses}.

    The second condition implies that $\zeta_{\widetilde{\psi} * \widetilde{\psi}}$, $\zeta_{(\widetilde{\psi})^\dagger * \widetilde{\psi}}$, and $\zeta_{\widetilde{\psi} * (\widetilde{\psi})^\dagger}$ exist.
\end{rem}

\begin{rem}
    By \cref{canonical families for families on S-birationally projective schemes}, different choices of $S$-birational projective models for $X$ give $S$-birationally equivalent canonical parameter spaces $Z(\psi)$ and canonical families $\widetilde{\psi}$. In particular
    \begin{enumerate}
        \item By \cref{faithful flatness of S-rational morphisms preserved by composing with S-birational morphisms}, condition \ref{main statement key condition} holds independently of the choice of $S$-birational projective model for $X$.

        \item By \cref{S-dense open of Artin group chunk gives group chunk}, the conclusion of the theorem holds independently of the choice of $S$-birational projective model for $X$.
    \end{enumerate}
\end{rem}

\begin{rem}
    By \cref{S-dense sub group chunk gives same group.}, the group algebraic space associated to the group chunk produced in \cref{statement of the main result at start of section} by \cref{Artin Group chunk theorem} is independent of the choice of $D, U, W,$ and $S$-birational projective model for $X$.
\end{rem}

See \cref{summary of constructing group chunks from rational families} for a discussion of how \cref{statement of the main result at start of section} relates to results of Hrushovski and Weil.

Before we continue, let us remark that, in the case where $A,X,Y$ are varieties over a field, with $A$ geometrically integral, we get many of the hypotheses of \cref{statement of the main result at start of section} for free:
\begin{prop}\label{the hypotheses for the main situation are not so bad over fields}
    Let $A$ be a geometrically integral variety over a field $k$, and $X,Y$ be varieties over $k$.

    Let $\phi\colon AX \dashrightarrow Y$ be an invertible $k$-rational family over $A$, with inverse $\phi^\dagger$. Let $\psi = \phi^\dagger * \phi$. Suppose the rational morphisms
    \begin{align*}
        q_1\colon A^2 \dashrightarrow AZ(\psi); &  & (a_1,a_2) \in A^2(T) \mapsto (a_1, \zeta_\psi(a_1,a_2)) \\
        q_2\colon A^2 \dashrightarrow AZ(\psi); &  & (a_1,a_2) \in A^2(T) \mapsto (a_2, \zeta_\psi(a_1,a_2))
    \end{align*}
    are both dominant.

    Then $\phi$ satisfies the assumptions, and hence the conclusion, of \cref{statement of the main result at start of section}.

    Moreover, in this case the associated group algebraic space is a variety over $k$.
\end{prop}
\begin{proof}
    Tameness follows from \cref{tameness criterion over fields}, and we get $S$-generic flatness of the closed graphs by generic flatness. The third condition holds by \cref{a dominant moprhism of integral schemes over a field is faithfully flat}.

    The fact that the associated group is a scheme follows from \cref{representability criterion for Artin group chunks}. For the proof that $X$ is a variety, see \cite{ER15}, or \cite{WeilAndre1955OAGo} for an explicit proof in the case where $k$ is not algebraically closed.
\end{proof}

%% file: Chapters/Group_Chunks_from_S-Rational_Families/Constructing_the_Partial_Magma.tex
\section{Constructing the Partial Magma}\label{section: constructing the partial magma}
We begin the proof of \cref{statement of the main result at start of section}. First, we fix an $S$-birational projective model for $X$. For most of the rest of this section, we work with the following:

\begin{situation}\label{Hypothesis for constructing the map}
    Let $S$ be a Noetherian scheme, and let $A,X,Y$ be schemes over $S$, with $A$ fppf over $S$, $Y$ locally Noetherian, and $X$ projective over $S$.

    Let $\phi\colon AX \dashrightarrow Y$ be an invertible (hence faithfully flat and $S$-diffuse, by \cref{invertible S-rational families are faithfully flat in good situations} and \cref{A flat S-rational morphism of locally Noetherian schemes is S-diffuse.}) $S$-rational family over $A$, with inverse $\phi^\dagger$. Let $\psi = \phi^\dagger * \phi$. We assume that:
    \begin{enumerate}[label = (\roman*).]
        \item The family $\psi$ is tame (\cref{tame definition}).
        \item \label{closed graph condition} The closed graphs $\Lambda_{\widetilde{\psi} * \widetilde{\psi}}$, $\Lambda_{(\widetilde{\psi})^\dagger * \widetilde{\psi}}$ and $\Lambda_{\widetilde{\psi} * (\widetilde{\psi})^\dagger}$ are $S$-generically flat over $Z(\psi)^2$.
        \item \label{key condition} The $S$-rational morphisms
              \begin{align*}
                  q_1\colon A^2 \dashrightarrow AZ(\psi); &  & (a_1,a_2) \in A^2(T) \mapsto (a_1, \zeta_\psi(a_1,a_2)) \\
                  q_2\colon A^2 \dashrightarrow AZ(\psi); &  & (a_1,a_2) \in A^2(T) \mapsto (a_2, \zeta_\psi(a_1,a_2))
              \end{align*}
              are both faithfully flat.
    \end{enumerate}
\end{situation}
\begin{rem}
    Observe that in the above we do not require $A$ to have geometrically integral fibers over $S$.
\end{rem}

The main result of this section is:
\begin{prop}\label{construction of the partial magma}
    In \cref{Hypothesis for constructing the map}, we have $ Z(\widetilde{\psi}*\widetilde{\psi}) = Z(\psi)$ (\cref{definition of the m_ij lemma}), and if we define the $S$-rational morphism
    \begin{align*}
        m_{12} = \zeta_{\widetilde{\psi}*\widetilde{\psi}}\colon Z(\psi)Z(\psi) \dashrightarrow Z(\psi),
    \end{align*}
    then there exists an $S$-dense open $U \subset \dom m_{12}$ such that $(Z(\psi), m_{12}|_{U})$ has the structure of a cancellative and strongly associative partial magma, and the images of the isomorphisms $(\rho_1,m_{12})_U$ and $(\rho_2,m_{12})_U$ are $S$-dense opens in $Z(\phi)^2$.
\end{prop}

We will break the proof into several steps, with the final proof at the end of this section.


\begin{lem}\label{Psi inverse identities}
    In \cref{Hypothesis for constructing the map}, we have $\psi^\dagger = \psi \circ \rho_{213}$, where $\rho_{213}\colon A^2X \to A^2X$ swaps the first two coordinates, and if we denote by $\tau$ the involution $\tau_{XX}\colon \Hilb_{X^2} \to \Hilb_{X^2}$ given by \cref{tau_YX definition}, we have $\zeta_{\psi} = \tau \circ\zeta_{\psi^\dagger}$, and $(\widetilde{\psi})^\dagger = \tau^*(\widetilde{\psi^\dagger})$.
\end{lem}
\begin{proof}
    The identity $\psi^\dagger = \psi \circ \rho_{213}$ follows immediately from \cref{A composition of invertible families is invertible}, and the other two claims follow from \cref{descent morphism for inverses} and \cref{Canonical families of inverses}.
\end{proof}

\begin{lem}\label{existence of the fi lemma}
    In \cref{Hypothesis for constructing the map}, there exist faithfully flat $S$-rational morphisms $f_i\colon A^3 \dashrightarrow Z(\psi)^2$, $i = 1,2,3,$ such that, if we denote $(a_1,a_2,a_3) = \id_{A^3}$, we have
    \begin{align*}
         & f_{1}(a_1,a_2,a_3) = (\zeta_\psi(a_1,a_3), \zeta_\psi(a_3,a_2))  \\
         & f_{2}(a_1,a_2,a_3) = (\zeta_\psi(a_2,a_3), \zeta_\psi(a_1,a_3))  \\
         & f_{3}(a_1,a_2,a_3) = (\zeta_\psi(a_3,a_1), \zeta_\psi(a_3,a_2)).
    \end{align*}
\end{lem}
\begin{proof}
    We show how to construct $f_1$; the constructions for $f_2$ and $f_3$ are similar.

    We will construct $f_1$ as a composition $h_2\circ h_1$ of $S$-rational morphisms $h_1, h_2$ which we show to be faithfully flat.

    If $(a_1,a_2,a_3) = \id_{A^3}$, and $(a_1,a_2,z) = \id_{A_1A_2Z(\psi)}$, the $S$-rational morphisms in question are defined by
    \begin{align*}
         & h_1 = (a_1,a_3,\zeta_\psi(a_3,a_2)), \\
         & h_2 = (\zeta_\psi(a_1,a_2),z).
    \end{align*}
    Observe that $h_1$ is a permutation of $q_1 \times \id_{A_1}$. Since $q_1$ is faithfully flat by assumption, and products of faithfully flat $S$-rational morphisms are faithfully flat (\cref{properties preserved by products of $S$-rational morphisms}) we have that $h_1$ is faithfully flat. The $S$-rational morphism $h_2$ is also a product of faithfully flat $S$-rational morphisms, hence faithfully flat (and hence $S$-diffuse). Therefore the composition is well-defined (by \cref{composition of diffuse S-rational morphisms is well defined}) and faithfully flat (by \cref{properties preserved by products of $S$-rational morphisms}), and we see immediately that

    \begin{align*}
        h_2 \circ h_1(a_1,a_2,a_3) = (\zeta_\psi(a_1,a_3), \zeta_\psi(a_3,a_2))
    \end{align*}
    as required.
\end{proof}

\begin{lem}\label{pullback equality result}
    In \cref{Hypothesis for constructing the map}, the projection $\rho_{12}\colon A^3 \to A^2$ is faithfully flat, and with the $f_i$ defined as in \cref{existence of the fi lemma} we have
    \begin{align*}
        f_1^*(\widetilde{\psi}*\widetilde{\psi}) = f_2^*(\rho_{21}^*(\widetilde{\psi}*(\widetilde{\psi})^\dagger)) = f_3^*((\widetilde{\psi})^\dagger*\widetilde{\psi})  = \rho_{12}^*(\psi),
    \end{align*}
\end{lem}
\begin{proof}
    The projection $\rho_{12}$ is faithfully flat since $A$ is faithfully flat over $S$. Note in particular that $\rho_{12}^*(\psi)$ is defined (\cref{pullbacks and products of faithfully flat morphisms are defined}).

    For the next statement, let us prove
    \begin{align*}
        f_2^*(\rho_{21}^*(\widetilde{\psi}*(\widetilde{\psi})^\dagger)) = \rho_{12}^*(\psi);
    \end{align*}
    the other equalities are similar (indeed easier). We can verify the equality using \cref{useful remark} and \cref{products and Yoneda combine}. First note that if $(a_2,a_3,x) = \id_{A^2X}$, we have using the identities from \cref{Psi inverse identities} that
    \begin{align*}
        (\widetilde{\psi})^\dagger(\zeta_\psi(a_2,a_3,x)) & = (\tau^*(\widetilde{\psi^\dagger}))(\tau \circ \zeta_{\psi^\dagger}(a_2,a_3),x) \\
                                                          & = (\widetilde{\psi^\dagger})(\zeta_{\psi^\dagger}(a_2,a_3),x)                    \\
                                                          & = \psi^\dagger(a_2,a_3,x) \tag{By \cref{Existence of canonical families}}        \\
                                                          & = \psi(a_3,a_2,x)
    \end{align*}
    Then if $(a_1,a_2,a_3,x) = \id_{A^3X}$ we have that
    \begin{align*}
        f_2^*(\rho_{21}^*(\widetilde{\psi}*(\widetilde{\psi})^\dagger))(a_1,a_2,a_3,x) & = \widetilde{\psi}*(\widetilde{\psi})^\dagger(\zeta_\psi(a_1,a_3), \zeta_\psi(a_2,a_3),x) \\
                                                                                       & = \widetilde{\psi}(\zeta_\psi(a_1,a_3),(\widetilde{\psi})^\dagger(\zeta_\psi(a_2,a_3),x)) \\
                                                                                       & = \widetilde{\psi}(\zeta_\psi(a_1,a_3),\psi(a_3,a_2,x))\tag{By above}                     \\
                                                                                       & =\psi(a_1,a_3,\psi(a_3,a_2,x))                                                            \\
                                                                                       & = \phi^\dagger(a_1,\phi(a_3,\phi^\dagger(a_3, \phi(a_2,x))))                              \\
                                                                                       & = \phi^\dagger(a_1,\phi(a_2,x))                                                           \\
                                                                                       & = \psi(a_1,a_2,x)                                                                         \\
                                                                                       & = \rho_{12}^*(\psi)(a_1,a_2,a_3,x)
    \end{align*}
    as required.
\end{proof}

\begin{lem}\label{definition of the m_ij lemma}
    In \cref{Hypothesis for constructing the map}, define the $f_i$ as in \cref{existence of the fi lemma}. Then we have equalities of $S$-rational morphisms from $A^3$ to $\Hilb_{X^2}$:
    \begin{align*}     \xi_{\widetilde{\psi}*\widetilde{\psi}}\circ f_1
        =
        \xi_{\widetilde{\psi}*(\widetilde{\psi})^\dagger} \circ \rho_{21} \circ f_2
        =
        \xi_{(\widetilde{\psi})^\dagger*\widetilde{\psi}}\circ f_3
        =
        \xi_\psi \circ \rho_{12}.
    \end{align*}
    In particular,
    \begin{align*}
        Z(\widetilde{\psi}*\widetilde{\psi}) = Z(\widetilde{\psi}*(\widetilde{\psi})^\dagger) = Z((\widetilde{\psi})^\dagger*\widetilde{\psi}) = Z(\psi),
    \end{align*}
    and we can define the $S$-rational morphisms
    \begin{align*}
         & m_{12} = \zeta_{\widetilde{\psi}*\widetilde{\psi}}\colon Z(\psi)Z(\psi) \dashrightarrow Z(\psi)                           \\
         & m_{23} = \zeta_{\widetilde{\psi}*(\widetilde{\psi})^\dagger} \circ \rho_{21}\colon Z(\psi)Z(\psi) \dashrightarrow Z(\psi) \\
         & m_{13} = \zeta_{(\widetilde{\psi})^\dagger*\widetilde{\psi}}\colon Z(\psi)Z(\psi) \dashrightarrow Z(\psi).
    \end{align*}
    In particular, if we identify $m_{12}$ with the partial morphism given by its representative $(m_{12})_{\dom m_{12}}$ (see \cref{domain for S-rational morphisms to separated target}), then $(Z(\psi),m_{12})$ is a partial magma of schemes.

    Moreover, we have
    \begin{align*}
        m_{12}\circ f_1 = m_{23} \circ f_2 = m_{13}\circ f_3 = \zeta_\psi \circ \rho_{12},
    \end{align*}
    and the $m_{ij}$ are faithfully flat. In particular we have canonical families $\widetilde{\widetilde{\psi}*\widetilde{\psi}}$, $\widetilde{\widetilde{\psi}*(\widetilde{\psi})^\dagger}$ and $\widetilde{(\widetilde{\psi})^\dagger*\widetilde{\psi}}$.
\end{lem}
\begin{proof}
    Note that, by \cref{Canonical S-rational families and pullback.}, we have that $\xi_{\phi}\circ \rho_{12} = \xi_{\rho_{12}^*(\phi)}$, and $\xi_{\widetilde{\psi}*(\widetilde{\psi})^\dagger} \circ \rho_{21} = \xi_{\rho_{21}^*(\widetilde{\psi}*(\widetilde{\psi})^\dagger)}$. Then the first part of the result follows immediately from the previous part and \cref{Canonical S-rational families and pullback.}.

    Since $f_1, \rho_{21} \circ f_2, f_3$ and $\rho_{21}$ are faithfully flat, we get the equality of the schematic images by \cref{Precomposition by faithfully flat S-rational morphisms preserves schematic image}, and then by definition
    \begin{align*}
        m_{12}\circ f_1 = m_{23} \circ f_2 = m_{13}\circ f_3 = \zeta_\psi \circ \rho_{12}.
    \end{align*}

    Now since we have firstly that $\rho_{12}\colon A^3 \to A^2$ is faithfully flat, secondly that $\zeta_\psi$ is faithfully flat (by assumption) and thirdly that a composition of faithfully flat $S$-rational morphisms is faithfully flat (\cref{composition of S-rational morphisms preserves many properties}), it follows that $m_{12}\circ f_1$ is faithfully flat. Then $m_{12}$ is faithfully flat by \cref{composition faithfully flat criterion for S-rational morphisms}. Similarly $m_{23}$, $m_{13}$ are faithfully flat.

    The existence of the canonical families then follows from \cref{Existence of canonical families}.
\end{proof}

\begin{lem} \label{equality of canonical families result}
    In \cref{Hypothesis for constructing the map}, the $S$-rational morphisms $\zeta_{\psi * \psi}$, $\zeta_{\psi^\dagger * \psi}$, and $\zeta_{\psi * \psi^\dagger}$ are faithfully flat, and we have equalities of the canonical families:
    \begin{align*}
        \widetilde{\psi * \psi} = \widetilde{\psi^\dagger * \psi} = \widetilde{\psi * \psi^\dagger} =     \widetilde{\psi}.
    \end{align*}

\end{lem}
\begin{proof}
    Note that, since pullback by faithfully flat morphisms preserves the canonical families (\cref{Canonical S-rational families and pullback.}), and by \cref{pullback equality result}, we have
    \begin{align*}
        \widetilde{\widetilde{\psi}*\widetilde{\psi}} = \widetilde{f_1^*(\widetilde{\psi}*\widetilde{\psi})}
        = \widetilde{\rho_{12}^*(\psi)} = \widetilde{\psi}.
    \end{align*}
    By \cref{Canonical families and composition} and \cref{definition of the m_ij lemma} we have that $\zeta_{\psi * \psi}$, $\zeta_{\psi^\dagger * \psi}$, and $\zeta_{\psi * \psi^\dagger}$ are faithfully flat. \cref{Canonical families and composition} also gives
    \begin{align*}
        \widetilde{\widetilde{\psi}*\widetilde{\psi}} = \widetilde{\psi*\psi}.
    \end{align*}
    Then, since $\psi^\dagger = \rho_{21}^*(\psi)$, and pullback commutes with composition by \cref{Pullback commutes with composition of S-rational families:}, we have
    \begin{align*}
        \psi^\dagger * \psi = (\rho_{12}\times\id_{A^2})^*(\psi * \psi)
    \end{align*}
    and since $\rho_{12}\times\id_{A^2}$ is faithfully flat, again by \cref{Canonical S-rational families and pullback.} we have $\widetilde{\psi^\dagger*\psi} = \widetilde{\psi}$. Arguing similarly for $\widetilde{\psi*\psi^\dagger}$ we conclude, as required, that
    \begin{align*}
        \widetilde{\psi * \psi} = \widetilde{\psi^\dagger * \psi} = \widetilde{\psi * \psi^\dagger} =     \widetilde{\psi}.
    \end{align*}
\end{proof}

\begin{lem}\label{mij inverses statement}
    In \cref{Hypothesis for constructing the map}, the $S$-rational morphisms $(\rho_1,m_{12})$ and $(\rho_2,m_{12})$ from $Z(\psi)^2$ to $Z(\psi)^2$ are $S$-birational, with $S$-rational inverses $(\rho_1,m_{13})$ and $(m_{23},\rho_1)$ respectively.

\end{lem}
\begin{proof}
    The plan here is to apply \cref{criteria for S-birationality}.

    We first show that $(\rho_1,m_{12})$, $(\rho_2,m_{12})$, $(\rho_1,m_{13})$ and $(m_{23},\rho_1)$ are $S$-diffuse (in fact faithfully flat). We give the argument for $(\rho_1,m_{12})$, the argument for the others is similar.

    First, let $g\colon A^4 \to Z(\psi)^2$ denote the composition $(\rho_1,m_{12}) \circ f_1$. Since $f_1$ is faithfully flat, by \cref{if f is faithfully flat and g circ f is flat then g is flat} it is enough to show that $g$ is faithfully flat. Note first that $g = (\rho_1\circ f_1,\ m_{12}\circ f_1)$. By the previous part we have $m_{12}\circ f_1 = \zeta_\psi \circ \rho_{12}$, so in particular if $(a_1,a_2,a_3) = \id_{A^3}$, we have
    \begin{align*}
        g(a_1,a_2,a_3) = (\zeta_\psi(a_1,a_3), \zeta_\psi(a_1,a_2)).
    \end{align*}
    In particular, we can write $g$ as the composition
    \begin{gather*}
        A_1A_2A_3 \dashrightarrow A_1A_3Z(\psi) \dashrightarrow Z(\psi)Z(\psi)\\
        (a_1,a_2,a_3) = \id_{A^3} \mapsto (a_1,a_3, \zeta_{\psi}(a_1,a_2)) \mapsto (\zeta_\psi(a_1,a_3), \zeta_\psi(a_1,a_2)).
    \end{gather*}
    Note that the first morphism is faithfully flat by assumption \ref{key condition}, and the second morphism is a product of faithfully flat  morphisms and therefore faithfully flat (\cref{properties preserved by products of $S$-rational morphisms}). Since compositions of faithfully flat morphisms are faithfully flat (\cref{composition of S-rational morphisms preserves many properties}), it follows that $g$ (and hence $(\rho_1,m_{12})$) is faithfully flat.

    It remains to show that the morphisms $(\rho_1,m_{13}),\ (\rho_1,m_{12})$ and $(\rho_2,m_{12})$, $(m_{23},\rho_1)$ compose to give the identity. Let us verify that $(\rho_1,m_{13}) \circ (\rho_1,m_{12}) = \id_{Z(\psi)^2}$. Note that the composition is equal to $(\rho_1, m_{13}\circ (\rho_1,m_{12}))$, so it is enough to show that $m_{13}\circ (\rho_1,m_{12}) = \rho_2$. Since $f_1$ is faithfully flat, it is enough to verify that $m_{13}\circ (\rho_1,m_{12}) \circ f_1 = \rho_2\circ f_1$.

    By definition, $\rho_2\circ f_1 = \zeta_\psi \circ \rho_{32}$.

    For the left-hand side, first note that $(\rho_1,m_{12}) \circ f_1 = (\rho_1\circ f_1, m_{12}\circ f_1)$. Again, since $m_{12}\circ f_1 = \zeta_\psi \circ \rho_{12}$ (\cref{definition of the m_ij lemma}), if $(a_1,a_2,a_3)= \id_{A^3}$ we have
    \begin{align*}
        (\rho_1,m_{12}) \circ f_1(a_1,a_2,a_3) = (\zeta_\psi(a_1,a_3), \zeta_\psi(a_1,a_2))= f_3(a_3,a_2,a_1)
    \end{align*}
    And therefore $(\rho_1,m_{12}) \circ f_1 = f_3 \circ \rho_{321}$. Therefore, by \cref{definition of the m_ij lemma}, we have
    \begin{align*}
        m_{13}\circ (\rho_1,m_{12}) \circ f_1 = m_{13} \circ f_3 \circ \rho_{321} = \zeta_\psi \circ \rho_{12} \circ \rho_{321} = \zeta_\psi \circ \rho_{32}
    \end{align*}
    as required.
\end{proof}

\begin{lem} \label{associativity verification part}
    In \cref{Hypothesis for constructing the map}, with the $m_{ij}$ defined as in \cref{definition of the m_ij lemma}, we have equalities of $S$-rational morphisms from $Z(\psi)^3$ to $Z(\psi)$:
    \begin{align*}
        m_{12} \circ (m_{12} \times \id_{Z(\psi)}) = m_{12} \circ (\id_{Z(\psi)} \times m_{12}) \\
        m_{12} \circ (m_{13} \times \id_{Z(\psi)}) = m_{13} \circ (\id_{Z(\psi)} \times m_{12})
    \end{align*}

    In particular, the partial magma $(Z(\psi),m_{12}|_{\dom m_{12}})$ is strongly associative.
\end{lem}
\begin{proof}
    Let us verify
    \begin{align*}
        m_{12} \circ (m_{13} \times \id_{Z(\psi)}) = m_{13} \circ (\id_{Z(\psi)} \times m_{12}).
    \end{align*}
    The other identity is similar, indeed easier.

    Since the $S$-rational morphism $(\zeta_\psi)^3 := \zeta_{\psi} \times \zeta_{\psi}\times \zeta_{\psi}$ is faithfully flat, by \cref{faithfully flat morphisms are S-rational epimorphisms} it is enough to show that both sides become equal after precomposition with $(\zeta_{\psi})^3$.

    First note that, if $(a_1,\dots,a_4) = \id_{A^4}$, we have
    \begin{align*}
          & \zeta_{(\widetilde{\psi})^\dagger* \widetilde{\psi}}(\zeta_{\psi}(a_1,a_2),\zeta_\psi(a_3,a_4))                                                                                                                           \\
        = & \zeta_{(\tau^*(\widetilde{\psi^\dagger}))* \widetilde{\psi}}(\tau \circ \zeta_{\psi^\dagger}(a_1,a_2),\zeta_\psi(a_3,a_4)) \tag{By \cref{Psi inverse identities}}                                                         \\
        = & \zeta_{(\tau \times \id_{Z(\psi)})*(\widetilde{\psi^\dagger} * \widetilde{\psi} ) }(\tau \circ \zeta_{\psi^\dagger}(a_1,a_2),\zeta_\psi(a_3,a_4)) \tag{\cref{Pullback commutes with composition of S-rational families:}} \\
        = & (\tau \times \id_{Z(\psi)})^*(\zeta_{(\widetilde{\psi^\dagger} * \widetilde{\psi} ) })(\tau \circ \zeta_{\psi^\dagger}(a_1,a_2),\zeta_\psi(a_3,a_4)) \tag{\cref{Canonical S-rational families and pullback.}}             \\
        = & \zeta_{(\widetilde{\psi^\dagger} * \widetilde{\psi} ) }(\zeta_{\psi^\dagger}(a_1,a_2),\zeta_\psi(a_3,a_4))                                                                                                                \\
        = & \zeta_{\psi^\dagger * \psi}(a_1,a_2,a_3,a_4) \tag{by \cref{Canonical families and composition}}
    \end{align*}

    Now let $(a_1,\dots,a_6) = \id_{A^6}$. Then
    \begin{align*}
          & m_{12} \circ (m_{13} \times \id_{Z(\psi)}) \circ (\zeta_\psi)^3(a_1,\dots,a_6)                                                                                                           \\
        = & \zeta_{\widetilde{\psi}*\widetilde{\psi}}(\zeta_{(\widetilde{\psi})^\dagger* \widetilde{\psi}}(\zeta_{\psi}(a_1,a_2),\zeta_\psi(a_3,a_4)), \zeta_\psi(a_5,a_6))                          \\
        = & \zeta_{\widetilde{\psi}*\widetilde{\psi}}(\zeta_{\psi^\dagger * \psi}(a_1,a_2,a_3,a_4), \zeta_\psi(a_5,a_6)) \tag{By above}                                                              \\
        = & \zeta_{\widetilde{(\psi^\dagger * \psi)}* \widetilde{\psi}}(\zeta_{\psi^\dagger * \psi}(a_1,a_2,a_3,a_4), \zeta_\psi(a_5,a_6)) \tag{By part \ref{equality of canonical families result}} \\
          & = \zeta_{\psi^\dagger * \psi * \psi}(a_1,\dots,a_6). \tag{by \cref{Canonical families and composition}}
    \end{align*}

    Similarly, one shows that
    \begin{align*}
        m_{13} \circ (\id_{Z(\psi)} \times m_{12}) \circ (\zeta_\psi)^3 = \zeta_{\psi^\dagger * \psi * \psi}
    \end{align*}
    and we conclude.

    Finally, since $Z(\psi)$ is separated (being a closed subscheme of a Hilbert scheme, which is separated), we conclude by \cref{domain for S-rational morphisms to separated target} that the identities hold wherever both sides are defined, so the partial magma $(Z(\psi),m_{12}|_{\dom m_{12}})$ is strongly associative.
\end{proof}

Now the proof of \cref{construction of the partial magma} is quick:
\begin{proof}[Proof of \cref{construction of the partial magma}]
    By \cref{mij inverses statement}, there exists an $S$-dense open $U \subset \dom m_{12}$ such that the morphisms $(\rho_1,m_{12})_U$ and $(\rho_2,m_{12})_U$ are isomorphisms onto their images, which are $S$-dense opens of $Z(\phi)^2$. Then $(Z(\phi), m_{12}|_{U})$ is cancellative. It is associative by \cref{associativity verification part} and \cref{associativity and cancellativity preserved by passing to subobjects.}.
\end{proof}

%% file: Chapters/Group_Chunks_from_S-Rational_Families/Final_group_chunk_construction.tex
\section{Constructing the Group Chunk}\label{section: Constructing the Group Chunk}
In this section we show that, with one extra condition, one can find a group chunk inside the partial magma we constructed in the previous section.

\begin{prop}\label{geometrically integral fibers on canonical family main proposition for group chunk}
    In the situation of \cref{Hypothesis for constructing the map}, take $Z(\phi)$, $m_{12}$, and $U$ from the conclusion of \cref{construction of the partial magma}, and suppose that the fibers of $Z(\psi)$ over $S$ are all geometrically integral. Then there exist $S$-dense opens $D \subset Z(\psi)$ and $W\subset D^3 \cap \Gamma_{m_{12U}}$ such that the pair $(D,W)$ is an Artin group chunk (\cref{Artin Group Chunk definition}).
\end{prop}
\begin{proof}
    The idea is to apply \cref{geometrically integral fibers gives X dense open}. Observe that the pair $(Z(\phi), \Gamma_{m_{12U}})$ satisfy all the hypotheses we need except that $Z(\psi)$ may not be faithfully flat over $S$. We can fix this by shrinking $Z(\psi)$.

    By \cref{characterization of faithfully flat S-rational morphisms}, there exist $S$-dense opens $A' \subset A$ and $B\subset Z(\psi)$ such that $A'$ is faithfully flat over $B$. Then $A'$ is fppf over $S$; indeed it is flat  and of finite presentation over $S$ since $A$ is fppf over $S$, and it surjects onto $S$ since it is $S$-dense in $A$. By \cref{if f is faithfully flat and g circ f is flat then g is flat} we conclude that $B$ is faithfully flat over $S$, and the fibers over $S$ are separated and geometrically integral since the same is true of $Z(\psi)$ and $B \subset Z(\psi)$ is $S$-dense.

    Now let $W = \Gamma_{m_{12U}} \cap B^3$. By \cref{associativity and cancellativity preserved by passing to subobjects.} we have that $(B,W)$ gives rise to a strongly associative cancellative partial magma. Since $W' \subset  \Gamma_{m_{12U}}$ is an $S$-dense open, and the projections $\rho_{ij}\colon \Gamma_{m_{12U}} \to Z(\phi)^2$ are $S$-dense open immersions (as the images of $(\rho_1,m_{12})_U$ and $(\rho_2,m_{12})_U$ are $S$-dense), it follows that the projections $\rho_{ij}\colon W \to B^2$ are $S$-dense open immersions. Thus $(B, \Gamma_{m_{12U}} \cap B^3)$ satisfies the assumptions of \cref{geometrically integral fibers gives X dense open}, so we conclude.
\end{proof}

Finally, we can prove the main result of this chapter. We restate it here for convenience:
\begin{thm}\label{main result: most general case}
    Let $S$ be a Noetherian scheme, and let $A,X,Y$ be schemes over $S$, with $A$ fppf and with geometrically integral fibers over $S$ and such that $Y$ is locally Noetherian and $X$ is $S$-birationally projective.

    Let $\phi\colon AX \dashrightarrow Y$ be an invertible $S$-rational family over $A$, with inverse $\phi^\dagger$. Let $\psi = \phi^\dagger * \phi$. Suppose
    \begin{enumerate}[label = (\roman*).]
        \item The family $\psi$ is tame (\cref{tame definition})
        \item The closed graphs $\Lambda_{\widetilde{\psi} * \widetilde{\psi}}$, $\Lambda_{(\widetilde{\psi})^\dagger * \widetilde{\psi}}$ and $\Lambda_{\widetilde{\psi} * (\widetilde{\psi})^\dagger}$ are $S$-generically flat over $Z(\psi)^2$, for some choice of $S$-birational projective model of $X$ witnessing tameness of $\psi$.
        \item The $S$-rational morphisms
              \begin{align*}
                  q_1\colon A^2 \dashrightarrow AZ(\psi); &  & (a_1,a_2) \in A^2(T) \mapsto (a_1, \zeta_\psi(a_1,a_2)) \\
                  q_2\colon A^2 \dashrightarrow AZ(\psi); &  & (a_1,a_2) \in A^2(T) \mapsto (a_2, \zeta_\psi(a_1,a_2))
              \end{align*}
              are both faithfully flat.
    \end{enumerate}

    Then composition of $\psi$ with itself induces an $S$-rational morphism $m_{12}\colon Z(\psi)^2 \dashrightarrow Z(\psi)$, and there exist $S$-dense opens $U\subset \dom(m_{12})$, $D \subset Z(\psi)$ and $W\subset D^3 \cap \Gamma_{m_{12U}}$ (where $\Gamma_{m_{12U}}$ denotes the graph of the representative of $m_{12}$ with domain $U$) such that the pair $(D,W)$ is an Artin group chunk (\cref{Artin Group Chunk definition}).
\end{thm}

\begin{proof}[Proof of \cref{main result: most general case}]
    Fix an $S$-birational projective model for $X$ as in \cref{Hypothesis for constructing the map}. In order to apply \cref{geometrically integral fibers on canonical family main proposition for group chunk}, we just need to verify that the fibers of $Z(\psi)$ over $S$ are geometrically integral.

    Indeed, we have that $\zeta_{\psi}$ is $S$-dominant by \cref{a faithfully flat morphism is S-dominant}. It follows that for any point $s \in S$, the pullback $A^2_s \dashrightarrow Z(\psi)_s$ is schematically dominant, and in particular $Z(\psi)_s = \Schim(A^2_s)$. Since $A$ has geometrically integral fibers, it follows that $A^2_s$ is geometrically integral, and then the fiber $Z(\psi)_s$ is geometrically integral by \cref{schematic image of a geometrically integral variety}, so we conclude.
\end{proof}
\begin{rem}
    Following on from \cref{multiple components remark}, one should get analogous versions of the above results assuming $Z(\psi)$ or $A$ is a finite union of schemes over $S$ such that the fibers are all geometrically integral.
\end{rem}